\documentclass[a4paper, reqno, 10pt]{amsart}

\usepackage{fullpage}

\usepackage{mathrsfs}
\usepackage{mathtools}
\usepackage{amsmath, amscd}
\usepackage{amssymb, amsfonts}
\usepackage{amsthm}
\usepackage{mathabx}
\usepackage{tikz}
\usepackage[all]{xy}
\usepackage{tikz-cd}
\usetikzlibrary{matrix,arrows, patterns, positioning, calc, intersections}

\usetikzlibrary{decorations.markings,arrows}
\tikzset{>=stealth}
\tikzset{
    dot/.style={circle,draw,fill,inner sep=1pt},
    arrow/.style={->,thick,shorten <=2pt,shorten >=2pt},
    twoarrow/.style={double,double distance=1.5pt,shorten <=9pt,shorten >=10pt,decoration={markings,mark=at position -8pt with {\arrow[scale=2]{>}}},preaction={decorate}},
    twoarrowlonger/.style={double,double distance=1.5pt,shorten <=5pt,shorten >=6pt,decoration={markings,mark=at position -4pt with {\arrow[scale=2]{>}}},preaction={decorate}},
}

\usepackage{enumitem}
\usepackage{bbm}
\usepackage{url}

\usepackage{comment}
\usepackage{romanbar}

\usepackage{hyperref}

\DeclareFontFamily{OT1}{pzc}{}
\DeclareFontShape{OT1}{pzc}{m}{it}{<-> s * [1.10] pzcmi7t}{}
\DeclareMathAlphabet{\mathpzc}{OT1}{pzc}{m}{it}

\newtheorem{thm}{Theorem}[section]
\newtheorem*{thm*}{Theorem}
\newtheorem{prop}[thm]{Proposition}
\newtheorem*{prop*}{Proposition}
\newtheorem{cor}[thm]{Corollary}
\newtheorem{lemma}[thm]{Lemma}

\theoremstyle{definition}
\newtheorem{defn}[thm]{Definition}

\newtheorem*{defn*}{Definition}
\newtheorem*{warning*}{Warning}
\newtheorem{ex}[thm]{Example}
\newtheorem{example}[thm]{Example}

\newtheorem{rem}[thm]{Remark}
\newtheorem{hypothesis}[thm]{Hypothesis}

\newtheorem{notation}[thm]{Notation}

\usepackage[parfill]{parskip}

\makeatletter
\def\part{\@startsection{part}{0}%
  \z@{\linespacing\@plus\linespacing}{.5\linespacing}%
  {\normalfont\scshape\centering}}
\makeatother

\makeatletter
\def\subsubsection{\@startsection{subsubsection}{3}%
  \z@{.5\linespacing\@plus.7\linespacing}{-.5em}%
  {\normalfont\bfseries}}
\makeatother

\makeatletter
\def\paragraph{\@startsection{paragraph}{4}%
  \z@\z@{-\fontdimen2\font}%
  {\normalfont\bfseries}}
\makeatother

\makeatletter
\let\orgdescriptionlabel\descriptionlabel
\renewcommand*{\descriptionlabel}[1]{%
  \let\orglabel\label
  \let\label\@gobble
  \phantomsection
  \edef\@currentlabel{#1}%
  \let\label\orglabel
  \orgdescriptionlabel{#1}%
}
\makeatother

\newcommand*{\longhookrightarrow}{\ensuremath{\lhook\joinrel\relbar\joinrel\rightarrow}}

\newcommand{\R}{{\mathbb{R}}}

\DeclareMathOperator{\Hom}{Hom}
\DeclareMathOperator{\Fun}{Fun}

\newcommand{\h}{\mathrm{h}}
\DeclareMathOperator{\Iso}{Iso}
\DeclareMathOperator{\Space}{\mathcal{S}\!\mathpzc{pace}}
\DeclareMathOperator{\sSet}{\mathpzc{s}\mathcal{S}\!\mathpzc{et}}
\DeclareMathOperator{\sSpace}{\mathpzc{s}\mathcal{S}\!\mathpzc{pace}}
\DeclareMathOperator{\SeSp}{\mathcal{S}\!\mathpzc{e}\mathcal{S}\!\mathpzc{p}}
\DeclareMathOperator{\SeCat}{\mathcal{S}\!\mathpzc{e}\mathcal{C}\!\mathpzc{at}}
\DeclareMathOperator{\ICat}{\mathcal{IC}\!\mathpzc{at}}

\DeclareMathOperator{\CSSp}{\mathcal{CSS}\!\mathpzc{p}}
\DeclareMathOperator{\CSSP}{\mathbf{CSSp}}
\DeclareMathOperator{\Man}{Man}
\DeclareMathOperator{\CMan}{\mathcal{M}\!\mathpzc{an}}

\newcommand{\C}{\mathcal{C}}

\renewcommand{\S}{\mathcal{S}}
\newcommand{\W}{\mathcal{W}}
\newcommand{\pr}{\mathrm{pr}}

\newcommand{\bullets}{{\bullet,\ldots, \bullet}}
\newcommand{\unit}{\mathbbm{1}}

\newcommand{\Int}{\mathrm{Int}}
\newcommand{\TInt}{\mathfrak{Int}}

\DeclareMathOperator{\PBord}{PBord}
\DeclareMathOperator{\PPBord}{\mathbb{P}Bord}
\newcommand{\TPBord}{\mathfrak{PBord}}
\DeclareMathOperator{\Bord}{Bord}
\DeclareMathOperator{\Cob}{\!Cob}

\newcommand{\ptpos}{
\begin{tikzpicture}[scale=0.3]
\fill (0.85,0.35) circle (0.1);
\draw[->, red] (0.85,0.35) -- (1.85,0.35) node[anchor = west] {\tiny 1};
\draw[->, blue] (0.85,0.35) -- (1.55, 1.05) node[anchor = west] {\tiny 2};
\end{tikzpicture}
}

\newcommand{\ptframe}{
\begin{tikzpicture}[scale=0.3]
\fill (0,0) circle (0.1);
\draw[->, red] (0,0) -- (1,0) node[anchor = west] {\tiny 1};
\draw[->, blue] (0,0) -- (0,1) node[anchor = south] {\tiny 2};
\end{tikzpicture}
}

\def\vdotssmall{\vbox{\baselineskip=2.5pt \lineskiplimit=0pt 
\kern6pt \hbox{.}\hbox{.}\hbox{.}}}

\newcommand{\overunder}[3][{}]{\overset{#2}{\underset{#1}{#3}}}
\newcommand{\ul}[1]{{\underline{#1}}}
\newcommand{\oul}[1]{{\overline{\underline{#1}}}}

\newcommand{\<}[1]{{\langle #1 \rangle}}

\newcommand{\myloopnop}[2][]{\Omega^{#1}(#2)}
\newcommand{\myloop}[3][]{\Omega^{#1}_{#3}(#2)}
\newcommand{\deloop}[3][]{\mathscr{L}^{#1}_{#3}(#2)}

\newcommand{\bound}[1]{B(c,#1)}

\newcommand{\eDelta}[1]{{|\Delta^{#1}|_e}}

\protected\def\hiddenmath{$(\infty,n)$}

\title{A note on the \texorpdfstring{\hiddenmath}{(infty,n)}-category of bordisms}

\author{Damien Calaque}
\address{IMAG, Universit\'e Montpellier, CNRS, 34095 Montpellier, France \& Institut Universitaire de France}
\email{damien.calaque@umontpellier.fr}

\author{Claudia Scheimbauer}
\address{Mathematical Insitute, University of Oxford, OX2 6GG Oxford, UK}
\email{scheimbauer@maths.ox.ac.uk}

\begin{document}

\begin{abstract}
In this extended note we give a precise definition of fully extended topological field theories \`a la Lurie. Using complete $n$-fold Segal spaces as a model, we construct an $(\infty,n)$-category of $n$-dimensional bordisms, possibly with tangential structure. We endow it with a symmetric monoidal structure and show that we can recover the usual category of bordisms.
\end{abstract}

\maketitle

\tableofcontents

\section*{Introduction}

Topological field theories (TFTs) arose as toy models for physical quantum field theories and have proven to be of mathematical interest, notably because they are a fruitful tool for studying topology. An $n$-dimensional TFT is a symmetric monoidal functor from the category of bordisms, which has closed $(n-1)$-dimensional manifolds as objects and $n$-dimensional bordisms as morphisms, to any other symmetric monoidal category, which classically is taken to be the category of vector spaces or chain complexes. 

A classification of 1- and 2-dimensional TFTs follows from classification theorems for 1- and 2-dimensional compact manifolds with boundary, cf.~\cite{Abrams}. In order to obtain a classification result for larger values of $n$ one needs a suitable replacement of the classification of compact $n$-manifolds with boundary used in the low-dimensional cases. Moreover, as explained in \cite{BaezDolan}, this approach requires passing to ``extended'' topological field theories. Here extended means that we need to be able to evaluate the $n$-TFT not only at $n$- and $(n-1)$-dimensional manifolds, but also at $(n-2)$-,..., 1-, and 0-dimensional manifolds. Thus, an extended $n$-TFT is a symmetric monoidal functor out of a higher category of bordisms. In light of the hope of computability of the invariants determined by an $n$-TFT, e.g.~by a triangulation, it is natural to include this data. Furthermore, Baez and Dolan conjectured that, analogously to the 1-dimensional case, extended $n$-TFTs are fully determined by their value at a point, calling this the Cobordism Hypothesis. A definition of a suitable bicategory of $n$-bordisms and a proof of a classification theorem of extended TFTs for dimension 2 was given in \cite{Schommer}.

In his expository manuscript \cite{Lurie}, Lurie suggested passing to $(\infty, n)$-categories for a proof of the Cobordism Hypothesis in arbitrary dimension $n$. He gave a detailed sketch of such a proof using a suitable higher category of bordisms, which, informally speaking, has zero-dimensional manifolds as objects, bordisms between objects as 1-morphisms, bordisms between bordisms as 2-morphisms, etc., and for $k>n$ there are only invertible $k$-morphisms given by diffeomorphisms and their isotopies. However, finding an explicit model for such a higher category poses one of the difficulties in rigorously defining these $n$-dimensional TFTs, which are called ``fully extended''. 

In \cite{Lurie},  Lurie gave a short sketch of a definition of this $(\infty, n)$-category using complete $n$-fold Segal spaces as a model. Instead of using manifolds with corners and gluing them, his approach was to conversely use embedded closed (not necessarily compact) manifolds, following along the lines of \cite{GMTW, galatius06, BoekstedtMadsen}, and to specify points where they are cut into bordisms of which the embedded manifold is a composition. Whitney's embedding theorem ensures that every $n$-dimensional manifold $M$ can be embedded into some large enough vector space and suitable versions for manifolds with boundary can be adapted to obtain an embedding theorem for bordisms, see Section \ref{hocat}. Moreover, the rough idea behind the definition of the $n$-fold Segal space is that it includes the data, for $k_1,\ldots, k_n$, of the 
classifying space for diffeomorphisms of, in the $i$th direction $k_i$-fold, composable $n$-bordisms. Lurie's idea was to use the fact that the space of embeddings of $M$ into $\R^\infty$ is contractible to justify the construction. 

Modifying this approach, the main goal of this note is to provide a detailed construction of such an $(\infty, n)$-category of bordisms, suitable for explicitly constructing an example of a fully extended $n$TFT, which will be the content of a subsequent paper \cite{CalaqueScheimbauer2}. As we explain in Section \ref{Lurie's vs}, Lurie's sketch does not lead to an $n$-fold Segal space, as the essential constancy condition is violated.  In our Definition \ref{def PBord}, we propose a stronger condition on elements in the levels of the Segal space. We show that this indeed yields an $n$-fold Segal space $\PBord_n$. Its completion $\Bord_n$ defines an $(\infty, n)$-category of $n$-bordisms and thus is a corrigendum to Lurie's $n$-fold simplicial space of bordisms from \cite{Lurie}.

Furthermore, we endow it with a symmetric monoidal structure and also consider bordism categories with additional structure, e.g.~orientations and framings, which allows us, in Section \ref{TFT}, to rigorously define fully extended topological field theories.

Our main motivation to have a precise definition of the $(\infty, n)$-category of bordisms was the following: in the subsequent paper \cite{CalaqueScheimbauer2} we explicitly construct an example of a fully extended topological field theory. Given an $E_n$-algebra $A$ we show that factorization homology with coefficients in $A$ leads to a fully extended $n$-dimensional topological field theory with target category a suitable $(\infty,n)$-Morita category with $E_n$-algebras as objects, bimodules as 1-morphisms, bimodules between bimodules as 2-morphisms, etc.

\paragraph{Organization of the paper}

In Part I, consisting of the first three sections, we recall the necessary tools from higher category theory needed to define fully extended TFTs. 

Section \ref{sec CSS} reviews the model for $(\infty,1)$-categories given by complete Segal spaces and recalls some useful information about other models. In Section \ref{sec CSSn} we explain the model for $(\infty,n)$-categories given by complete $n$-fold Segal spaces and introduce a model which is a hybrid between complete $n$-fold Segal spaces and Segal $n$-categories.

We propose two equivalent definitions of symmetric monoidal structures on complete $n$-fold Segal spaces in Section \ref{sec monCSSn}; one as a $\Gamma$-object in complete $n$-fold Segal spaces following \cite{ToenVezzosi} and one as a tower of suitable $(n+k)$-fold Segal spaces with one object, 1-morphism,..., $(k-1)$-morphism for $k\geq0$ following the Delooping Hypothesis.

Part II is devoted to the construction of $\Bord_n$. 

Our construction of the $(\infty,n)$-category $\Bord_n$ of higher bordisms is based on a simpler complete Segal space $\Int$ of closed intervals, which we introduce in Section \ref{sec Int}. The closed intervals correspond to places where we are allowed to cut the manifold into the bordisms it composes. The fact that we prescribe closed intervals instead of just a point corresponds to fixing collars of the bordisms.

Section \ref{sec Bord} is the central part of this article and consists of the construction of the complete $n$-fold Segal space $\Bord_n$ of bordisms. We discuss variants of $\Bord_n$, including $(\infty, d)$-categories of bordisms and tangles for arbitrary $d$, and compare our definition to Lurie's sketch in Section \ref{sec variants}.

In Section \ref{sec Bord symm} we endow $\Bord_n$ with a symmetric monoidal structure, both as a $\Gamma$-object and as a tower and compare the two definitions.

In Section \ref{hocat} we elaborate on the interpretation of the objects in $\Bord_n$ as $n$-bordisms. Furthermore we show that the homotopy category of the $(\infty,1)$-category of bordisms is what one should expect, namely the homotopy category of the $(\infty,1)$-category of $n$-bordisms $\Bord_n^{(\infty,1)}$ gives back the classical bordism category $n\Cob$.

Finally, in Section \ref{sec decorations} we consider bordism categories with additional structure such as orientations, denoted by $\Bord_n^{or}$, and framings, denoted by $\Bord_n^{fr}$, which allows us to define fully extended $n$-dimensional topological field theories in Section \ref{TFT}.

\paragraph{Conventions}

\begin{enumerate}
\item Let $\Space$ denote the category of simplicial sets with its usual model structure. By {\em space} we mean a fibrant object in $\Space$, i.e.~a Kan complex.
\item We denote the simplex category by $\Delta$. Objects are finite ordered sets denoted by $[m]=(0<\cdots <m)$ and morphisms are monotone maps. As an ordered set, we can view $[m]$ as a category. Functors between the associated categories arise exactly from monotone maps. Thus, we can take the nerve of $[m]$ which we will denote by $\Delta^m$.
\item The geometric realization $|\Delta^l|$ is the standard geometric simplex $\{(x_0,\ldots, x_l)\in \R^{l+1}: \sum_i x_i=1, x_i\geq 0\}$. We denote the extended simplex $\{(x_0,\ldots, x_l)\in \R^{l+1}: \sum_i x_i=1\}$ by $\eDelta{l}$.
\end{enumerate}

\subsection*{Acknowledgements.}
We would like to thank Mathieu Anel, David Li-Bland, Chris Schommer-Pries, and Hiro L.~Tanaka for helpful discussions, and Giovanni Felder and Bertrand To\"en for their careful proof-reading.  We thank the referee for his/her thorough reading and for the extensive comments which improved the exposition of the paper.
This paper is extracted from the first part of the second author's PhD thesis and was partially supported by the SNF grants 200021\textunderscore 137778 and P2EZP2\textunderscore 159113. CS thanks IH\'ES and FIM at ETH Zurich for the hospitality during her stays in fall 2014 and September 2015 during which parts of this paper were written. DC acknowledges the support of the Institut Universitaire de France and the ANR grant ``SAT'' ANR-14-CE25-0008. 

 

\part{Symmetric monoidal \texorpdfstring{$(\infty, n)$}{(infty,n)}-categories}

A higher category, or $n$-category for $n\geq0$, has not only objects and (1-)morphisms, but also $k$-morphisms between $(k-1)$-morphisms 
for $1\leq k\leq n$. Strict higher categories can be rigorously defined, however, most higher categories which occur in nature are not strict. 
Thus, we need to weaken some axioms and coherences between the weakenings become rather involved to formulate explicitly. 
Things turn out to become somewhat easier when using a geometric definition, in particular when furthermore allowing $k$-morphisms for all 
$k\geq 1$, which for $k\geq n$ are invertible. Such a higher category is called an $(\infty,n)$-category. There are several models for such 
$(\infty,n)$-categories, e.g.~Segal $n$-categories (cf.~\cite{HirschowitzSimpson}), $\Theta_n$-spaces (cf.~\cite{Re2}), and complete $n$-fold 
Segal spaces (cf.~\cite{BarwickCSS}), which all are equivalent in an appropriate sense (cf.~\cite{Toen05, BarwickSP}). For our purposes, the latter model turns out to 
be well-suited and in this Part we recall some basic facts about complete $n$-fold Segal spaces as higher categories. 
This is not at all exhaustive, and more details can be found in e.g.~\cite{BergnerRezk}. We also refer to \cite{Bergner}, especially for their role in the proof of the Cobordism Hypothesis in \cite{Lurie}.

Symmetric monoidal structures on $(\infty,n)$-categories \textit{per se} have not been very much studied in the literature for $n>1$ (even though they are particular instances of commutative monoids in $\infty$-categories, which are extensively studied in \cite{LurieHA}). We provide a brief review of these in Section \ref{sec monCSSn} and describe them in two different, but equivalent, ways: as $\Gamma$-objects on the one hand and using the Delooping Hypothesis on the other hand. A comparison follows from \cite{GepnerHaugseng}.

\section{Models for \texorpdfstring{$(\infty,1)$}{(infty,1)}-categories}\label{sec CSS}

We start with $n=1$. An $(\infty,1)$-category should be a category up to coherent homotopy which is encoded in the invertible higher morphisms. In this section, we will mention and give references for several models for $(\infty,1)$-categories we will use in the later sections. A good overview on different models for $(\infty,1)$-categories and their comparison can be found in \cite{Bergner10}. It should be mentioned that by \cite{Toen05} up to equivalence, there is essentially only one theory of $(\infty,1)$-categories; explicit equivalences between the models mentioned here have been proved e.g.~in \cite{DwyerKanSmith, BergnerThreeModels, BarwickKanRelCat, HorelModel}. One additional model which should be mentioned is that of Joyal's quasi-categories. It has been intensively studied, most prominently in \cite{LurieHTT}.

\subsection{The homotopy hypothesis and \texorpdfstring{$(\infty,0)$}{(infty,0)}-categories}

The basic hypothesis upon which $\infty$-category theory is based goes back to Grothendieck \cite{Grothendieck} and is the following:
\begin{hypothesis}[Homotopy hypothesis]\label{homotopy hypothesis}
Spaces are models for $\infty$-groupoids, also referred to as $(\infty,0)$-categories.
\end{hypothesis}
Given a space $X$, its points, i.e.~0-simplices, are thought of as objects of the $(\infty,0)$-category, paths between points as $1$-morphisms, 
homotopies between paths as $2$-morphisms, homotopies between homotopies as $3$-morphisms, and so forth. With this interpretation, it is clear that all $n$-morphisms are invertible up to homotopies, which are higher morphisms.

We take this hypothesis as the basic definition, and model ``spaces'' with simplicial sets rather than with topological spaces. 
\begin{defn}
An {\em $(\infty,0)$-category}, or {\em $\infty$-groupoid}, is a space. According to our conventions, it is a fibrant simplicial set, i.e.~a Kan complex.
\end{defn}

\subsection{Topologically and simplicially enriched categories}\label{sec top cat}

 Two particularly simple, but quite rigid models are topologically or simplicially enriched categories.
\begin{defn}
A {\em topological category} is a category enriched in topological spaces. A {\em simplicial category} is a category enriched in simplicial sets.
\end{defn}
Topological and simplicial categories are discussed and used in \cite{LurieHTT, ToenVezzosiHAG1}. However, for our applications they turn out to be too rigid. We would also like to allow some flexibility for objects, not only morphisms, thus also requiring spaces of objects.

\subsection{Segal spaces}

Complete Segal spaces, first introduced by Rezk in \cite{Rezk} as a model for $(\infty,1)$-categories, turn out to be very well-suited for geometric applications. We recall the definition in this section.

\begin{defn}
A ({\em $1$-fold}) {\em Segal space} is a simplicial space $X=X_\bullet$ which satisfies the {\em Segal condition}: 
for any $n,m\geq0$ the commuting square 
$$
\xymatrix{
X_{m+n}\ar[r] \ar[d] & X_m \ar[d] \\
X_n \ar[r] & X_0
}
$$
induced by the maps $[m]\to [m+n]$, $(0<\cdots<m) \mapsto (0<\cdots< m)$, and $[n]\to [m+n]$, $(0<\cdots<n) \mapsto (m<\cdots < m+n)$, 
is a homotopy pullback square.
In other words, the induced map
$$
X_{m+n}\longrightarrow X_m\underset{X_0}{\overset{h}{\times}}X_n
$$
is a weak equivalence. 

Defining a {\em map of Segal spaces} to be a map of the underlying simplicial spaces gives a category of Segal spaces 
$\SeSp=\SeSp_1$.
\end{defn}

\begin{rem}\label{rem Segal g}
For any $m\geq 1$, consider the maps $g_\beta:[1]\to [m]$, $(0<1)\mapsto (\beta-1<\beta)$ for $1\leq \beta\leq m$. Note that requiring the Segal condition is equivalent to requiring the condition that the maps
$$X_m\longrightarrow X_1\underset{X_0}{\overset{h}{\times}} \cdots \underset{X_0}{\overset{h}{\times}} X_1$$
induced by $g_1,\ldots, g_m$ are weak equivalences.
\end{rem}

\begin{rem}
Following \cite{Lurie} we omit the Reedy fibrancy condition which often appears in the literature. 
In particular, this condition would guarantee that for $m,n\geq 0$ the canonical map
$$X_m\underset{X_0}{\times}X_n\longrightarrow X_m\underset{X_0}{\overset{h}{\times}}X_n$$
is a weak equivalence. Our definition corresponds to the choice of the projective model structure instead of the injective (Reedy) model structure, which is slightly different (though Quillen equivalent) compared to \cite{Rezk}. We will explain this in more detail in Section \ref{sec css fibrant}.
\end{rem}

\begin{defn}
We will refer to the spaces $X_n$ as the {\em levels} of the Segal space $X$.
\end{defn}

\begin{ex}\label{ex: nerve of top category}
Let $\mathcal C$ be a small topological category. 
Recall that its nerve is the simplicial set
$$N(\mathcal C)_n = \Hom([n], \mathcal C) =\bigsqcup_{x_0,\ldots, x_n\in \mathrm{Ob} \,\C} \Hom_\C(x_0, x_1)\times\cdots \Hom_\C(x_{n-1}, x_n),$$
with face maps given by composition of morphisms, and degeneracies by insertions of identities. 
The nerve $N(\mathcal C)$ is a Segal space. 
Moreover, a simplicial set, viewed as a simplicial space with discrete levels, satisfies the Segal condition if and only if it is the nerve of an (ordinary) category.
\end{ex}

\subsubsection{Segal spaces as \texorpdfstring{$(\infty,1)$}{(infty,1)}-categories}

The above example motivates the following interpretation of Segal spaces as models for $(\infty,1)$-categories.
If $X_\bullet$ is a Segal space then we view the set of 0-simplices of the space $X_0$ as the set of objects. For $x,y\in X_0$ we view 
$$\Hom_X(x,y) = \{x\}\times^h_{X_0}X_1\times^h_{X_0}\{y\}$$
as the $(\infty,0)$-category, i.e.~the space, of arrows from $x$ to $y$. More generally, we view $X_n$ as the $(\infty,0)$-category, i.e.~the space, of $n$-tuples of composable arrows together with a composition. Note that given an $n$-tuple of composable arrows, the Segal condition implies that the corresponding fiber of the Segal map $X_n\to X_1\times^h_{X_0} \cdots \times^h_{X_0} X_1$ is a contractible space. The map $X_n\to X_1$ determined by the functor $[1]\to [n], 0<1\mapsto 0<n$ can be thought of as ``composition'', and thus we can think of the $n$-tuple as having a contractible space of possible compositions. Moreover, one can interpret paths in the space $X_1$ of 1-morphisms as 2-morphisms, which are invertible up to homotopies, which in turn are 3-morphisms, and so forth.

\subsubsection{The homotopy category of a Segal space}

To a higher category one can intuitively associate an ordinary category, its {\em homotopy category}, which has the same objects 
and whose morphisms are $2$-isomorphism classes of $1$-morphisms. For Segal spaces, one can realize this idea as follows.

\begin{defn}
The {\em homotopy category} $h_1(X)$ of a Segal space $X=X_\bullet$ is the (ordinary) category whose objects are the 0-simplices of the space $X_0$ and whose morphisms between objects $x,y\in X_0$ are
\begin{align*}
\Hom_{h_1(X)}(x,y) &=\pi_0\left(\Hom_X(x,y)\right)\\
&=\pi_0\left(\{x\}\underset{X_0}{\overset{h}{\times}} X_1 \underset{X_0}{\overset{h}{\times}} \{y\}\right).
\end{align*}
For $x,y,z\in X_0$, the following diagram induces the composition of morphisms, as weak equivalences induce bijections on $\pi_0$.
\begin{eqnarray*}
\left(\{x\}\underset{X_0}{\overset{h}{\times}} X_1 \underset{X_0}{\overset{h}{\times}} \{y\}\right)\times
\left(\{y\}\underset{X_0}{\overset{h}{\times}} X_1 \underset{X_0}{\overset{h}{\times}} \{z\}\right)
& \longrightarrow & \{x\}\underset{X_0}{\overset{h}{\times}} X_1 \underset{X_0}{\overset{h}{\times}}X_1 \underset{X_0}{\overset{h}{\times}} \{z\} \\
& \overset{\simeq}{\longleftarrow} & \{x\}\underset{X_0}{\overset{h}{\times}} X_2 \underset{X_0}{\overset{h}{\times}} \{z\} \\
& \longrightarrow & \{x\}\underset{X_0}{\overset{h}{\times}} X_1 \underset{X_0}{\overset{h}{\times}} \{z\}\,.
\end{eqnarray*}
\end{defn}

\begin{ex}
Given a small (ordinary) category $\C$, the homotopy category of its nerve, viewed as a simplicial space with discrete levels, is equivalent to $\C$,
$$h_1(N(\C))\simeq \C.$$
\end{ex}

The above example motivates the following definition of equivalences of Segal spaces.
\begin{defn}\label{DK equiv}
A map $f:X\to Y$ of Segal spaces is a {\em Dwyer-Kan equivalence} if
\begin{enumerate}
\item the induced map $h_1(f): h_1(X) \to h_1(Y)$ on homotopy categories is essentially surjective, and
\item for each pair of objects $x,y\in X_0$ the induced map $\Hom_X(x,y)\to \Hom_{Y}(f(x), f(y))$ is a weak equivalence.
\end{enumerate}
\end{defn}

\subsection{Complete Segal spaces}\label{complete}

We would like the equivalences of Segal spaces to be the Dwyer-Kan equivalences. However, instead of considering all Segal spaces and their the Dwyer-Kan equivalences, it turns out that we can instead consider a full subcategory of Segal spaces which satisfy an extra condition called {\em completeness}, for which Dwyer-Kan equivalences have an equivalent, simpler, description. To make sense of this, we need to first introduce the model categories involved.

\subsubsection{The model structures of Segal spaces}

We now describe various model structures on the category $\sSpace$ of simplicial spaces in this section. Ultimately, the goal is to have a model category whose fibrant objects deserve to be called ``$(\infty,1)$-categories'' and whose equivalences are analogs of equivalences of categories. We will first introduce model categories whose fibrant objects are Segal spaces. Then, in the next step, we will fix the weak equivalences. We refer to \cite{Rezk} and \cite{HorelModel} for more details.

Let us first consider the injective and projective model structures on the category of simplicial spaces, denoted by $\sSpace_c$ and $\sSpace_f$, respectively. Note that the fibrant objects in $\sSpace_f$ are the levelwise fibrant ones, while the fibrant objects of $\sSpace_c$ turn out to be the {\em Reedy fibrant} simplicial spaces\footnote{See for example \cite[Theorem 15.8.7]{Hirschhorn} for a proof that the injective and Reedy model structures coincide.}.  Conversely, every object in $\sSpace_c$ is cofibrant, see for example \cite[Corollary 15.8.8.]{Hirschhorn}. These model categories are Quillen equivalent (via the identity functor). 

In the first step we perform left Bousfield localizations of the previous model structures $\sSpace_c$ and $\sSpace_f$ with respect to the morphisms 
$$
\Delta^1\coprod_{\Delta^0}\cdots\coprod_{\Delta^0}\Delta^1\longrightarrow\Delta^n.
$$
This provides two model categories, denoted $\sSpace_c^{Se}$ and $\sSpace_f^{Se}$, which still are Quillen equivalent. For the injective model structure, it is immediate that fibrant objects in $\sSpace_c^{Se}$ satisfy $X_n\xrightarrow{\simeq}  X_1\times_{X_0}\cdots\times_{X_0}X_1$ and thus are Reedy fibrant Segal spaces. For the projective model structure, it follows from \cite{HorelModel} that the fibrant objects in $\sSpace_f^{Se}$ satisfy $X_n\xrightarrow{\simeq} X_1\times^h_{X_0}\cdots\times^h_{X_0}X_1$ and thus are Segal spaces\footnote{Note that this terminology is not consistent throughout the literature: often ``Segal space'' includes the Reedy fibrancy condition. Our examples will not be Reedy fibrant, which is the reason for our choice of terminology.}.

\subsubsection{Complete Segal spaces}

Even though the model categories $\sSpace_c^{Se}$ and $\sSpace_f^{Se}$ have the (Reedy fibrant) Segal spaces as their fibrant objects, there are not enough weak equivalences: every weak equivalence between Segal spaces is indeed a Dwyer-Kan equivalence, but there are more Dwyer-Kan equivalences. 

This problem can be circumvented by further localizing the model structures. For this new model structure, the weak equivalences between Segal spaces turn out to be exactly the Dwyer-Kan equivalences. We will see that these further localized model structures have fewer fibrant objects, which are the {\em complete} (Reedy fibrant) Segal spaces. We will focus on the case of the projective model structure, since the other case can be found spelled out in great detail in many references, for example the original \cite{Rezk}, but to our knowledge the former has so far only appeared in \cite{HorelModel}. Moreover, although we will phrase it for the projective model structure, the first part works the same in the injective case. The difference appears when computing the involved mapping spaces explicitly, see the remark below.

Intuitively, the condition we would like to impose is that the underlying $\infty$-groupoid of invertible morphisms of the Segal space $X_\bullet$ is already encoded by the space $X_0$. To translate this, we first need to understand what the space of (homotopy) invertible morphisms of $X_\bullet$ is.

Let $f$ be an element in $X_1$ with source and target $x$ and $y$, i.e.~its images under the two face maps $X_1\rightrightarrows X_0$ are $x$ and $y$. It  is called {\em invertible} if its image under
$$
\{x\}\underset{X_0}{\times} X_1 \underset{X_0}{\times} \{y\}
\longrightarrow
\{x\}\underset{X_0}{\overset{h}{\times}} X_1 \underset{X_0}{\overset{h}{\times}} \{y\}
\longrightarrow
\pi_0\left(\{x\}\underset{X_0}{\overset{h}{\times}} X_1 \underset{X_0}{\overset{h}{\times}}\{y\}\right)
=\Hom_{\h_1(X)}(x,y)\,,
$$
is an invertible morphism in $h_1(X)$, i.e.~it has a left and right inverse.

To define the space of invertible morphisms, consider the walking isomorphism $I[1]$, which is the category with two objects and one invertible morphism between them,
$$\begin{tikzpicture}
\draw (0,0) node[dot] (A) {}
(1,0) node[dot] (B) {};
\draw[->, bend left=40] (A) to (B);
\draw[->, bend left=40] (B) to (A);
\draw (0.5,0) node {$\cong$};
\end{tikzpicture}
$$
Mapping the walking isomorphism into an arbitrary category $\C$ we get the isomorphisms of $\C$, and therefore the information about its underlying groupoid. Mimicking this procedure for a Segal space $X_\bullet$, we consider the derived mapping space
$$ \mathrm{Map}_{\sSpace_f^{Se}} ( N(I[1]), X) .$$

Moreover, an analog of \cite[Lemma 5.8]{Rezk} shows that if an element in $X_1$ is invertible, any element in the same connected component will also be invertible. Thus we define the {\em space of invertible morphisms in $X_\bullet$} to be the homotopy pullback\footnote{To compare with the definition in \cite{HorelModel}, note that the pullback is a homotopy pullback since the map $X_1\to \pi_0(X_1)$ is a fibration.}
$$
\begin{tikzcd}
X_1^{inv} \arrow{r} \arrow{d}  &  X_1\arrow{d}\\
\pi_0 \mathrm{Map}_{\sSpace_f^{Se}} ( N(I[1]), X)  \arrow{r}& \pi_0 X_1 = \pi_0 \mathrm{Map}_{\sSpace_f^{Se}} ( \Delta^1, X)  \arrow[ul, phantom, "\mathrm{h} \lrcorner", very near end]
\end{tikzcd}
$$
Here, the bottom arrow arises from the obvious functor $[1]\to I[1]$.

Finally, identity morphisms in $X_\bullet$ should be invertible. Indeed, the degeneracy map $s_0:[1] \to [0]$ factors as $[1]\to I[1] \rightarrow [0]$ and induces a map
$$X_0 \to X_1^{inv}.$$

\begin{defn}
A Segal space $X_\bullet$ is {\em complete} if the map $X_0\to X_1^{inv}$ is a weak equivalence. 
We denote the full subcategory of $\SeSp$ whose objects are complete Segal spaces by $\CSSp=\CSSp_1$.
\end{defn}

\begin{ex}
Let $\C$ be a category. Then $N(\C)$ is a complete Segal space if and only if there are no non-identity isomorphisms in $\C$, i.e.~the underlying groupoid of $\C$ is a set (viewed as a category with only identity morphisms).
\end{ex}

In order to compute $X_1^{inv}$ explicitly, we have to be able to describe the (derived) mapping space $\mathrm{Map}_{\sSpace_f^{Se}} ( N(I[1]), X)$. 
\begin{lemma}
We have a homotopy pullback square 
$$
\begin{tikzcd}
\mathrm{Map}_{\sSpace_f^{Se}} ( N(I[1]), X) \arrow{r} \arrow[swap]{d}{\{0,2\}\amalg \{1,3\}}  &  X_3 \arrow{d}{\{0,2\}\amalg \{1,3\}}\\
X_0 \times X_0  \arrow{r}& X_1 \times X_1 \,. \arrow[ul, phantom, "\mathrm{h} \lrcorner", very near end]
\end{tikzcd}
$$
\end{lemma}
\begin{proof}
Note that since $X_\bullet$ was assumed to be a Segal space, it is fibrant, but $N(I[1])$ might not be cofibrant\footnote{Note that for the injective model structure, it is cofibrant and therefore $X_1^{inv}$ is just the subspace of $X_1$ of invertible morphisms.}. So to compute the desired mapping space, we cofibrantly replace $N(I[1])$ and then compute the mapping space in the underlying category,
$$\mathrm{Map}_{\sSpace_f^{Se}} ( N(I[1]), X) \simeq  \mathrm{Map}_{\sSpace} (\mathrm{cof}\left( N(I[1])\right), X).$$

To compute the cofibrant replacement, the crucial observation (originally by \cite{Rezk}, reformulated by \cite{BarwickSP}) is that the nerve of $I[1]$ can be obtained by the pushout of simplicial sets
$$K = \Delta^3 \amalg_{\Delta^{\{0,2\}} \amalg \Delta^{\{1,3\}}} (\Delta^0 \amalg \Delta^0).$$
This can be seen as contracting the edges $\{0,2\}$ and $\{1,3\}$ in the 3-simplex:
$$
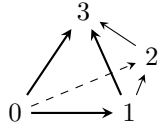
\begin{tikzpicture}[scale=1.5]
\path (0,0) node (0) {0} -- (1,0) node (1) {1}  -- (0.6, 0.9) node (3) {3} -- (1.2, 0.5) node (2) {2};
\draw[->, thick] (0) -- (1);
\draw[->] (1) -- (2);
\draw[->] (2) -- (3);
\draw[->, thick] (1) -- (3);
\draw[->, thick] (0) -- (3);
\draw[->, dashed] (0) -- (2);
\end{tikzpicture}
$$

We use an argument similar to that in \cite[Remark 3.4]{JF-S}, which observes the following: $K$ is given by a strict pushout along a diagram of cofibrant objects of which one arrow is an inclusion. By \cite[A.2.4.4]{LurieHTT}, this is a homotopy pushout in the injective model structure and therefore homotopy equivalent to the homotopy pushout in the projective model structure. So a cofibrant replacement of $K$ is given by taking the homotopy pushout of the same diagram,
$$\mathrm{cof}(K) = \Delta^3 \amalg^h_{\Delta^{\{0,2\}} \amalg \Delta^{\{1,3\}}} (\Delta^0 \amalg \Delta^0).$$
Finally, we obtain the space as the wanted homotopy pullback\footnote{This can be compared to Rezk's definition using the zig-zag category 
\begin{tikzpicture}[scale=0.5, inner sep=1pt]
\draw (1,0) node (2) {2}
(3,0) node (3) {3};
\draw (0,0) node {0} edge [->] (2);
\draw (2,0) node {1} edge[->] (2) edge[->] (3);
\end{tikzpicture}
and requiring the morphisms
\begin{tikzpicture}[scale=0.5, inner sep=1pt]
\draw (1,0) node (2) {2};
\draw (0,0) node {0} edge [->] (2);
\end{tikzpicture}
and
\begin{tikzpicture}[scale=0.5, inner sep=1pt]
\draw (3,0) node (3) {3};
\draw (2,0) node {1} edge[->] (3);
\end{tikzpicture}
to be identities.
}
\end{proof}

\subsubsection{Complete Segal spaces as fibrant objects}\label{sec css fibrant}

There is a further model structure on the category of simplicial spaces which implements completeness. It is obtained by a further left Bousfield localization, with respect to the morphism
$$
\Delta^0\longrightarrow N(I[1]).
$$
This provides two Quillen equivalent model categories, denoted $\sSpace_c^{CSe}$ and $\sSpace_f^{CSe}$. 
Fibrant objects in $\sSpace_c^{CSe}$, respectively $\sSpace_f^{CSe}$, are Reedy fibrant complete Segal spaces, respectively complete Segal spaces.

Summarizing, we have the following diagram
$$
\begin{tikzcd}
\sSpace_c \arrow[bend right =10]{r} & \sSpace_f \arrow[bend right =10]{l}\\
\sSpace_c^{Se} \arrow[bend right =10]{r} \arrow{u} & \sSpace_f^{Se} \arrow[bend right =10]{l} \arrow{u}\\
\sSpace_c^{CSe} \arrow[bend right =10]{r} \arrow{u}& \sSpace_f^{CSe} \arrow[bend right =10]{l} \arrow{u}
\end{tikzcd}
$$
where the horizontal arrows are Quillen equivalences induced by the identity and the vertical arrows are localizations.

The following Proposition shows that in the localized model structure Dwyer-Kan equivalences of Segal spaces indeed are weak equivalence, and therefore we have fixed the concern mentioned above. We refer to \cite[Theorem 5.15]{HorelModel} for a proof, which makes substantial use of the analogous result for Reedy fibrant Segal spaces in $\sSpace_c^{CSe}$ from \cite[Theorem 7.7]{Rezk}.
\begin{thm}
Let $X$ and $Y$ be Segal spaces. A morphism $f:X\to Y$ is a weak equivalence in $\sSpace_f^{CSe}$ if and only if 
it is a Dwyer-Kan equivalence. 
\end{thm}

As a consequence the obvious inclusions induce the following equivalences of categories:
$$
\CSSp[\mathpzc{lwe}^{-1}]\longrightarrow \SeSp[\mathpzc{DK}^{-1}]\longrightarrow Ho(\sSpace_f^{CSe})\,,
$$
where $\mathpzc{DK}$ and $\mathpzc{lwe}$ stand for the subcategory of Dwyer-Kan and levelwise weak equivalences, respectively. 

This justifies the following definition.
\begin{defn}
An {\em $(\infty,1)$-category} is a complete Segal space.
\end{defn}

\begin{rem}
We denote the category of Reedy fibrant complete Segal spaces by $\CSSp_c$, that is to say the fibrant objects in $\sSpace_c^{CSe}$. Remember that $\sSpace_c^{CSe}$ and $\sSpace_f^{CSe}$ are Quillen equivalent, 
so that the embedding $\CSSp_c\subset\CSSp$ induces an equivalence 
$\CSSp_c[\mathpzc{lwe}^{-1}]\to \CSSp[\mathpzc{lwe}^{-1}]$
 of which an inverse is given by the Reedy fibrant replacement 
functor $(-)^R$. Sometimes it turns out to be more useful to work in the model category $\sSpace_c^{CSe}$ as every object is cofibrant. Note that the Reedy fibrant replacement functor does not change the homotopy type of the levels.
\end{rem}

\begin{defn}\label{sec completion}
The fibrant replacement functor in the model category $\sSpace_f^{CSe}$ sending a Segal space to its fibrant replacement is called {\em completion}. In \cite{Rezk} Rezk gave a rather explicit construction of the completion of Segal spaces. He showed that there is a completion functor 
which to every Segal space $X$ associates a complete Segal space $\widehat{X}$ together with a map 
$i_X: X\to\widehat{X}$, which is a Dwyer-Kan equivalence. 
\end{defn}

\begin{rem}\label{rem completion}
The completeness condition says that all invertible morphisms essentially are just identities up to the choice of a path. In this sense, one might like to think of complete Segal spaces as a homotopical version of skeletal\footnote{A category is called {\em skeletal} if each isomorphism class contains just one element, see for example \cite{RiehlCat}.}
or {\em reduced} (see also \cite{Joyal}) category, and, since any category is equivalent to a reduced one, assuming this extra condition is harmless.
However, the information on the invertible morphisms is merely encoded in a different way, namely, in the spatial structure. Also, we would like to remark that in the homotopical situation, this intuition might be misleading: indeed, instead of thinking of a complete Segal space as having few invertible morphisms, it is better to think of a complete Segal space as having a ``maximal'' space of objects. This is illustrated by \cite[Corollary 6.6]{Rezk}. A good example to keep in mind is a special case of \cite[Remark 14.1]{Rezk}: given a group $G$, we can view as a category with one object, and consider its nerve. Its completion is the constant simplicial space $BG$. 
\end{rem}

\begin{rem}
It is worth noticing that $\sSpace_f$, $\sSpace_c$, $\sSpace_f^{Se}$, $\sSpace_c^{Se}$, $\sSpace_f^{CSe}$, and $\sSpace_c^{CSe}$ are all Cartesian closed simplicial model categories. 
In particular, for any simplicial space $X$ and any complete Segal space $Y$, the simplicial space $Y^X$ is a complete Segal space. 
\end{rem}

\subsubsection{The classification diagram -- the Rezk or relative nerve}

Many examples of (complete) Segal spaces arise by a construction due to Rezk \cite{Rezk} which produces a (complete) Segal space from a simplicial model category. More generally, several authors \cite{BarwickKan,LM-G, Low} proved that this construction also gives a complete Segal space for far-reaching generalizations of model categories, namely, for relative categories with certain weak conditions. For instance, categories of fibrant objects in the sense of Brown satisfy the conditions to obtain a Segal space; if they additionally are saturated, they lead to complete Segal spaces.
\begin{defn}\label{defn rel cat}
A {\em relative category} is a a pair $(\C, \W)$ consisting of a category $\C$ and a subcategory $\W\subseteq\C$ containing all objects of $\C$. The morphisms in $\W$ are called {\em weak equivalences}. A {\em relative functor} between two relative categories is a functor which preserves weak equivalences. Together they form a category $\mathpzc{RelCat}$.
\end{defn}

\begin{defn}\label{Rezk construction}
Let $(\mathcal C,\mathcal W)$ be a relative category. Consider the simplicial object in categories $\mathcal C_\bullet$ given by $\mathcal C_n:=\Fun\big([n],\mathcal C)$. It has a subobject $\mathcal C^{\mathcal W}_\bullet$, where $\mathcal C^{\mathcal W}_n\subset \mathcal C_n$ is the subcategory which has the same objects and whose morphisms consist only of those composed of those in $\mathcal W$. Taking its nerve we obtain a simplicial space $N(\mathcal C,\mathcal W)_\bullet$ with
$$N(\C, \mathcal W)_n = N(\C^{\mathcal W}_n).$$
It is proved in \cite{LM-G} that this simplicial space satisfies the Segal condition if $(\C, \mathcal W)$ admits a suitable homotopical three-arrow calculus. Moreover, it is complete if it additionally is saturated, i.e.~a morphism is a weak equivalence if and only if it is an isomorphism in the homotopy category. However, it is not level-wise\footnote{Strictly speaking we should call this a (complete) Segal simplicial set, since we defined a space to be fibrant.} fibrant unless we started with an $\infty$-groupoid. Its level-wise fibrant replacement is called the {\em Rezk or relative nerve} or the {\em classification diagram}, which, by abuse of notation, we again denote by $N(\mathcal C,\mathcal W)$.
\end{defn}

\begin{ex}
Let $C$ be a small category. Then it is straightforward to see that $N(\C,\Iso\C)$ is a complete Segal space. Note that the natural morphism $N(\C) \to N(\C, \Iso \C)$ is a Dwyer-Kan equivalence. This exhibits $N(\C,\Iso\C)$ as a completion of $N(\C)$.
\end{ex}

Now we can apply this construction to the model category of complete Segal spaces from the previous section:
\begin{defn}
The {\em $(\infty,1)$-category of $(\infty,1)$-categories} is $N(\CSSp, \mathpzc{lwe})$.
\end{defn}

\begin{rem}
The inclusions of the relative categories of cofibrant-and-fibrant objects and of fibrant objects in a simplicial model category lead to equivalences of the classification diagrams. Thus, the inclusions of relative categories $(\CSSp, \mathpzc{lwe}) \subset (\SeSp,\mathpzc{DK}) \subset (\sSpace_f^{CSe}, \W)$ induce level-wise equivalences of complete Segal spaces, i.e.~equivalences of $(\infty,1)$-categories
$$N(\CSSp, \mathpzc{lwe}) \to N(\SeSp,\mathpzc{DK}) \to N(\sSpace_f^{CSe}, \W).$$
\end{rem}

\begin{rem}
Note that morphisms in the homotopy category of $(\infty,1)$-categories are more subtle: they are zig-zags $X\to X_1 \xleftarrow{\simeq} X_2 \to \cdots \to Y$, where the wrong-way-arrows are weak equivalences and therefore more flexible.
\end{rem}

We finally observe that a Quillen equivalence 
$$\begin{tikzcd}
L: \mathcal M \arrow[shift right=0.5ex]{r}& \mathcal N :R\arrow[shift right = 0.5ex]{l}.
\end{tikzcd}$$
induces a weak equivalence of complete Segal spaces 
$$
X:=N(\mathcal M,\mathpzc{we}_{\mathcal M})\simeq N(\mathcal N,\mathpzc{we}_{\mathcal N})=:Y
$$
between the associated classification diagrams. Indeed, the left derived functor $\mathbb{L}L=L\circ Q$ ($Q$ being a cofibrant replacement functor) preserves weak equivalence and thus induces a morphism $\mathbf{L}:X\to Y$, which can be proven to be a Dwyer-Kan equivalence (this essentially follows from \cite{Dwyer-Kan3}).

\subsection{Some other models for \texorpdfstring{$(\infty,1)$}{(infty,1)}-categories}

We very briefly recall some other models for $(\infty,1)$-categories in this section which were the motivation for some definitions later on. 

\subsubsection{Segal categories}

Let us mention another model for $(\infty,1)$-categories given by certain Segal spaces, which avoids having to require completeness by instead requiring a discrete set of objects. This will be the motivation for the definitions of ``hybrid $n$-fold Segal spaces'' in section \ref{sec complete SS_n}.

\begin{defn}
A Segal ($1$-)category is a Segal space $X=X_\bullet$ such that $X_0$ is discrete, i.e.~constant as a simplicial set. We denote $\SeCat$ the full subcategory of $\SeSp$ consisting of Segal categories. 
\end{defn}
Segal categories also are the fibrant objects in a certain model category that is Quillen equivalent to $\sSpace_f^{CSe}$, see the above mentioned \cite{Bergner10} or \cite{LurieGoodwillie} for more details and references.
In particular, the embedding  $\SeCat\subset \SeSp$ induces an equivalence of complete Segal spaces 
$$
N(\SeCat,\mathpzc{DK})\longrightarrow N(\SeSp,\mathpzc{DK})\,.
$$

\subsubsection{Relative categories}\label{relative cats}

Following \cite{BarwickKanRelCat} a rather weak notion of $(\infty,1)$-category is given by relative categories from Definition \ref{defn rel cat}. One should think of the weak equivalences $\mathcal{W}$ as being ``formally inverted''. We have already implicitly used this to define the $(\infty,1)$-category of $(\infty,1)$-categories.

\begin{ex}
Let $\C=\mathrm{Ch}_R$ be the category of chain complexes over a ring $R$ and let $\W\subseteq\C$ be the subcategory of chain complexes and quasi-isomorphisms.
\end{ex}
$\mathpzc{RelCat}$ admits a model structure exhibiting it as a model for $(\infty, 1)$-categories: in \cite{BarwickKan} the model structure of $\sSpace_c^{CSe}$ is transferred along a slight modification of the relative nerve,
$$N_\xi: \mathpzc{RelCat} \longrightarrow \sSpace_c^{CSe} :K_\xi$$
thus making the above adjunction into a Quillen equivalence.

\subsubsection{Categories internal to simplicial sets}\label{sec internal categories}

Instead of enriching categories in a category of spaces as in Section \ref{sec top cat}, for certain applications it turns out to be useful to also have a space of objects (thus allowing more flexibility than in topological categories), but keeping strict composition (and thus having more rigidity than in Segal spaces). This philosophy is implemented when considering categories internal to spaces. We will use this model to construct examples of Segal spaces.

\begin{defn}
Let $\mathcal{S}$ be a category with finite limits. 
A {\em category internal to $\mathcal{S}$} or for short, an {\em internal category in $\mathcal{S}$} consists of objects $\C_0, \C_1$ together with source and target morphisms $s,t:\C_1\rightrightarrows \C_0$, a degeneracy morphism $d:\C_0\to\C_1$ satisfying $s\circ d=t\circ d=id_{\C_0}$, and a composition morphism $\circ: \C_1\times_{\C_0} \C_1\to \C_1$ satisfying associativity and such that for any $x\in \C_0$, the maps $-\circ d(x)$ and $d(x)\circ -$ are the identity. Let $\mathcal{IC}\!\mathpzc{at}(\mathcal{S})$ denote the category of categories internal to $\mathcal{S}$, where morphisms from $(\C_0,\C_1)$ to $(\mathcal D_0, \mathcal D_1)$ are pairs of morphisms $\C_i\to \mathcal D_i$ for $i=0,1$ which are compatible with the additional structure in the obvious way.

For short, we call an {\em internal category} a category internal to $\mathcal{S}=\sSet$. In this case, we use the notation $\mathcal{IC}\!\mathpzc{at} = \mathcal{IC}\!\mathpzc{at}(\sSet)$.
\end{defn}

Note that there is an equivalence of categories $\mathcal{IC}\!\mathpzc{at} \to \mathcal{C}\!\mathpzc{at}^{\Delta^{op}}.$ Composition with the level-wise nerve and swapping the simplicial directions gives a functor
$$N: \mathcal{IC}\!\mathpzc{at} \longrightarrow \mathcal{C}\!\mathpzc{at}^{\Delta^{op}} \overset{N(-)}{\longrightarrow} \Space^{\Delta^{op}} \overset{swap}{\longrightarrow} \sSpace.$$

In \cite{HorelModel}, similarly to $\mathpzc{RelCat}$, the model structures of $\sSpace_f^{Se}$ and $\sSpace_f^{CSe}$ are transferred along $N$ to endow $\mathcal{IC}\!\mathpzc{at}$ with a model structure, the latter exhibiting it as a model for $(\infty, 1)$-categories. Examples of fibrant objects in the former model category are given by the following strongly Segal internal categories (\cite[Proposition 5.13]{HorelModel}): their nerves are Segal spaces.
\begin{defn}
A {\em strongly Segal internal category} is a category $\mathcal C=(\mathcal C_0,\mathcal C_1)$ internal to $\S= \Space\subset\sSet$ such that the source and target maps $s,t:\mathcal C_1\to \mathcal C_0$ are fibrations of simplicial sets. We denote by $\ICat^{Se}$ the category of strongly Segal internal categories. 
\end{defn}

\begin{rem}
The condition that $s,t$ are fibrations ensures that the pullback along them are homotopy pullbacks, and therefore $N\mathcal C$ is a Segal space. This condition is sufficient, but not necessary for an internal category to be a fibrant object for the transferred model structure. On the other hand, the condition that $\mathcal C_0,\mathcal C_1$ be Kan complexes is necessary.
\end{rem}

Since the model structure was transferred, there is a Quillen equivalence given by the nerve,
$$N:\mathcal{IC}\!\mathpzc{at} \longrightarrow \sSpace_f^{Se}.$$
Moreover, categorical equivalences of strongly Segal internal categories are (by definition) precisely the morphisms that are sent to Dwyer-Kan equivalences by the nerve. Thus, the induced morphism 
$$
N(\ICat_{Se},\mathpzc{cat.eq.})\longrightarrow N(\SeSp,\mathcal{DK})
$$
is an equivalence of complete Segal spaces.

\section{Models for \texorpdfstring{$(\infty,n)$}{(infty,n)}-categories}\label{sec CSSn}

As a model for $(\infty,n)$-categories, we will use complete $n$-fold Segal spaces, which were first introduced by Barwick in his thesis and appeared prominently in Lurie's \cite{Lurie}. Details can be found e.g.~in \cite{LurieGoodwillie, BarwickSP, BergnerRezk}. As mentioned above, $(\infty,n)$-categories are homotopical versions of weak $n$-categories. Recall that $n$-categories are inductively built by taking categories (weakly) enriched in $(n-1)$-categories. For $n=2$ these are known as 2-categories (strict) or bicategories (weak). Alternatively, one could choose to consider categories {\em internal} to $(n-1)$-categories, i.e.~they have a whole $(n-1)$-category of objects. For $n=2$ these were first introduced under the name of double categories by Ehresmann in \cite{EhresmannDoublecat} and have been thoroughly studied in category theory. Therefore we will call the higher versions thereof {\em $n$-uple categories}\footnote{This is non-standard: usually they are called $n$-fold categories. However, by an unfortunate choice of terminology, complete $n$-fold Segal spaces will correspond to $n$-categories. In order to hopefully reduce confusion we will instead be consistent in using ``uple'' for internal versions and reserve ``fold'' for the enriched, globular version.}.
Even though we present our main example as an $n$-fold Segal space in the next part, it actually arises from such an ``$n$-uple'' version as we will see later on.

Moreover, it even comes from a more rigid model, namely from internal $n$-uple categories, which are $n$-uple categories internal to simplicial sets. This model is the easiest to define, which is why we start with it.

\subsection{Internal \texorpdfstring{$n$}{n}-uple categories}\label{sec internal n-uple categories}

Iterating the approach in \cite{HorelModel}, one obtains a model for $(\infty,n)$-uple-categories given by $n$-uple categories internal to simplicial sets, i.e.~categories internal to the category of $(n-1)$-uple categories internal to simplicial sets. Unravelling the definition for $n=2$, there is a space of objects, a space of ``horizontal'' 1-morphisms, a space of ``vertical'' 1-morphisms, and a space of 2-morphisms, together with unit maps and composition maps. For larger $n$, there is a space of objects and suitable spaces of higher morphisms ``in all directions'', again together with unit maps and composition maps. Equivalently, an $n$-uple category internal to simplicial sets is a simplicial object in (strict) $n$-fold categories. This model has been discussed in \cite{CavigliaHorel}.

Our bordism category defined in the next part secretly is such an internal $n$-uple category, however, details on this model were not available at the time of writing this article, so we will present it in a different way here.

\begin{rem}
Note that composition is well-defined on the nose, as opposed to the models we will consider in the next sections.
\end{rem}

\subsection{\texorpdfstring{$n$}{n}-uple and \texorpdfstring{$n$}{n}-fold Segal spaces}

Recall that an $n$-uple\footnote{Again, usually, this is called an $n$-fold simplicial space, but we use this terminology to emphasize the difference.} simplicial space is a functor $X:(\Delta^{op})^{\times n} \to \Space$.
An $n$-uple Segal space is an $n$-uple simplicial space with an extra condition ensuring it is the $\infty$-analog of an $n$-uple category.
\begin{defn}\label{defn n-uple SSp}
An {\em $n$-uple Segal space} is an $n$-uple simplicial space $X=X_{\bullet,\dots,\bullet}$ such that
for every $1\leq i\leq n$, and every $k_1,\ldots, k_{i-1},k_{i+1},\ldots, k_n\geq0$,
$$X_{k_1,\ldots, k_{i-1},\bullet, k_{i+1},\ldots, k_n}$$
is a Segal space.

Defining a {\em map of $n$-uple Segal spaces} to be a map of the underlying $n$-uple simplicial spaces gives a category of $n$-uple Segal spaces, $\SeSp^n$.
\end{defn}

Imposing an extra globularity condition leads to a model for $\infty$-analogs of $n$-categories:
\begin{defn}
An $n$-uple simplicial space $X_{\bullet,\dots,\bullet}$ is {\em essentially 
constant} if the map from the constant $n$-uple simplicial space $X_{0,\ldots,0}$ given by the degeneracy maps
$$ X_{0,\ldots,0} \longrightarrow X$$
is a weak equivalence of $n$-uple simplicial spaces.
\end{defn}

\begin{defn}\label{defn n-fold SSp}
An {\em $n$-fold Segal space} is an $n$-uple Segal space $X=X_{\bullet,\dots,\bullet}$ such that
for every $1\leq i\leq n$, and every $k_1,\ldots, k_{i-1}\geq0$, the $(n-i)$-uple simplicial space
$$X_{k_1,\ldots, k_{i-1},0, \bullet, \ldots, \bullet}$$
is essentially constant.\footnote{To be consistent with our choice of ``uple'' versus ``fold'', we could call an $n$-uple simplicial space which satisfies this extra condition an $n$-fold simplicial space.}

Defining a {\em map of $n$-fold Segal spaces} to be a map of the underlying $n$-uple simplicial spaces gives a category of $n$-fold Segal spaces, $\SeSp_n$.
\end{defn}

\begin{rem}\label{rem iterative Segal}
Alternatively, one can formulate the conditions iteratively. First, an {\em $n$-uple Segal space} is a simplicial object $Y_\bullet$ in $(n-1)$-uple Segal spaces which satisfies the Segal condition. Then, an $n$-fold Segal space is a simplicial object $Y_\bullet$ in $(n-1)$-fold Segal spaces which satisfies the Segal condition and such that $Y_0$ is essentially constant (as an $(n-1)$-fold Segal space). To get back the above definition, the ordering of the indices is crucial: $X_{k_1,\ldots, k_n} = (Y_{k_1})_{k_2,\ldots, k_n}$.
\end{rem}

\subsubsection{Interpretation as higher categories}

An $n$-fold Segal space can be thought of as a higher category in the following way.

The first condition means that this is an $n$-uple category, i.e.~there are $n$ different ``directions'' in which we can ``compose''. An element of $X_{k_1,\ldots, k_n}$ should be thought of as a composition consisting of $k_i$ composable morphisms in the $i$th direction. 

The second condition imposes that we indeed have a higher $n$-category, i.e.~an $n$-morphism has as source and target two $(n-1)$-morphisms which themselves have the ``same'' (in the sense that they are homotopic) source and target.

For $n=2$ one can think of this second condition as ``fattening'' the objects in a bicategory. A 2-morphism in a bicategory can be depicted as
$$
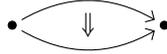
\begin{tikzpicture}
\filldraw (-1,0) circle (1.5pt);
\node (A) at (-1,0) {};
\filldraw (1,0) circle (1.5pt);
\node (B) at (1,0) {};
\node at (0,0) {$\Downarrow$};
\path[->,font=\scriptsize,>=angle 90]
(A) edge [bend left] node[above] {} (B)
edge [bend right] node[below] {} (B);
\end{tikzpicture}
$$
The top and bottom arrows are the source and target, which are 1-morphisms between the same objects.

In a 2-fold Segal space $X_{\bullet, \bullet}$, an element in $X_{1,1}$ can be depicted as
$$
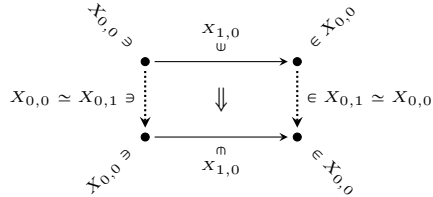
\begin{tikzpicture}
\filldraw (3,0.5) circle (1.5pt) node[anchor= south east] {\tiny \rotatebox{-45}{$X_{0,0}\ni$}};
\node (A') at (3,0.5) {};
\filldraw (5,0.5) circle (1.5pt) node[anchor= south west] {\tiny \rotatebox{45}{$\in X_{0,0}$}};
\node (B') at (5,0.5) {};

\filldraw (3,-0.5) circle (1.5pt) node[anchor= north east] {\tiny \rotatebox{45}{$X_{0,0}\ni$}};
\node (C') at (3,-0.5) {};
\filldraw (5,-0.5) circle (1.5pt) node[anchor= north west] {\tiny \rotatebox{-45}{$\in X_{0,0}$}};
\node (D') at (5,-0.5) {};

\node at (4,0) {$\Downarrow$};
\path[->]
(A') edge node[above] {\tiny $\substack{X_{1,0}\\ \rotatebox{90}{$\in$}}$} (B');
\path[->]
(C') edge  node[below] {\tiny $\substack{ \rotatebox{90}{$\ni$}\\X_{1,0}}$} (D');
\path[->, densely dotted, thick]
(A') edge node[left] {\tiny $X_{0,0}\simeq X_{0,1} \ni$} (C');
\path[->, densely dotted, thick]
(B') edge node[right] {\tiny $\in X_{0,1}\simeq X_{0,0}$} (D');
\end{tikzpicture}
$$
The images under the source and target maps in the first direction $X_{1,1} \rightrightarrows X_{1,0}$ are 1-morphisms which are depicted by the horizontal arrows. The images under the source and target maps in the second direction $X_{1,1} \rightrightarrows X_{0,1}$ are 1-morphisms, depicted by the dashed vertical arrows, which are essentially just identity maps, up to homotopy, since $X_{0,1} \simeq X_{0,0}$. Thus, here the source and target 1-morphisms (the horizontal ones) themselves do not have the same source and target anymore, but up to homotopy they do.

The same idea works with higher morphisms, in particular one can imagine the corresponding diagrams for $n=3$. A 3-morphism in a tricategory can be depicted as 
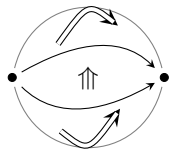
\begin{center}
\begin{tikzpicture}
\filldraw (-1,0) circle (1.5pt);
\node (A) at (-1,0) {};
\filldraw (1,0) circle (1.5pt);
\node (B) at (1,0) {};

\node at (0,0) {\rotatebox{90}{$\Rrightarrow$}};

\draw[->] (A) to [out=30, in=135] (B);
\draw[->] (A) to [out=-45, in=-150] (B);

\draw[gray] (A) .. controls (-0.7, 1.2) and (0.7, 1.2) .. (B);
\draw[gray] (A) .. controls (-0.7, -1.2) and (0.7, -1.2) .. (B);

\draw[twoarrowlonger] (-0.5, 0.3) .. controls (0,1) .. (0.5,0.6);
\draw[twoarrowlonger] (-0.5, -0.6) .. controls (0,-1) .. (0.5,-0.3);
\end{tikzpicture}
\end{center}
whereas a 3-morphism, i.e.~an element in $X_{1,1,1}$ in a 3-fold Segal space $X$ can be depicted as
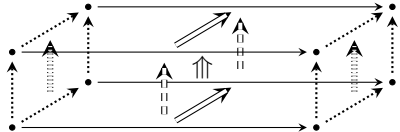
\begin{center}
\begin{tikzpicture}[scale=2]
\node (a) at (-1,0) {};
\node (A) at (1,0) {};
\node (b) at (-1,0.5) {};
\node (B) at (1,0.5) {};
\node (c) at (-0.5, 0.3) {};
\node (C) at (1.5,0.3) {};
\node (d) at (-0.5,0.8) {};
\node (D) at (1.5,0.8) {};

\filldraw (a) circle (.5pt);
\filldraw (A) circle (.5pt);
\filldraw (b) circle (.5pt);
\filldraw (B) circle (.5pt);
\filldraw (c) circle (.5pt);
\filldraw (C) circle (.5pt);
\filldraw (d) circle (.5pt);
\filldraw (D) circle (.5pt);

\draw[twoarrowlonger] (0, 0.5) -- (0.5, 0.8);
\draw[twoarrowlonger, densely dashed] (0.5, 0.3) -- (0.5, 0.8);
\draw[twoarrowlonger, densely dotted] (-0.75, 0.15) -- (-0.75, 0.65);

\draw[->] (a) -- (A);
\draw[->] (b) -- (B);
\draw[->] (c) -- (C);
\draw[->] (d) -- (D);

\draw[->, densely dotted, thick]
(a) -- (c);
\draw[->, densely dotted, thick]
(a) -- (b);
\draw[->, densely dotted, thick]
(b) -- (d);
\draw[->, densely dotted, thick]
(c) -- (d);

\draw[->, densely dotted, thick]
(A) -- (C);
\draw[->, densely dotted, thick]
(A) -- (B);
\draw[->, densely dotted, thick]
(B) -- (D);
\draw[->, densely dotted, thick]
(C) -- (D);

\draw[twoarrowlonger] (0, 0) -- (0.5, 0.3);
\draw[twoarrowlonger, densely dashed] (0, 0) -- (0, 0.5);
\draw[twoarrowlonger, densely dotted] (1.25, 0.15) -- (1.25, 0.65);
\node at (0.25,0.4) {\rotatebox{90}{$\Rrightarrow$}};
\end{tikzpicture}
\end{center}
Here the dotted arrows are those in $X_{0,1,1}\simeq X_{0,0,1}\simeq X_{0,0,0}$ and the dashed ones are those in $X_{1,0,1}\simeq X_{1,0,0}$.

Thus, we should think of the set of 0-simplices of the space $X_{0,\ldots,0}$ as the objects of our category, and elements of $X_{1,\ldots,1,0,\ldots,0}$ as $i$-morphisms, where $0<i\leq n$ is the number of 1's. Pictorially, they are the $i$-th ``horizontal" arrows. Moreover, the other ``vertical" arrows are essentially just identities of lower morphisms. Similarly to before, paths in $X_{1,\ldots, 1}$ should be thought of as $(n+1)$-morphisms, which therefore are invertible up to a homotopy, which itself is an $(n+2)$-morphism, and so forth.

\subsubsection{The homotopy bicategory of a 2-fold Segal space}\label{homotopy bicategory}

To any higher category one can intuitively associate a bicategory having the same objects and $1$-morphisms, and with $2$-morphisms being $3$-isomorphism classes 
of the original $2$-morphisms. 

\begin{defn}
The {\em homotopy bicategory} $\h_2(X)$ of a $2$-fold Segal space $X=X_{\bullet,\bullet}$ is defined as follows: 
objects are the points of the space $X_{0,0}$ and
$$
\Hom_{\h_2(X)}(x,y)=\h_1\big(\Hom_{X}(x,y)\big)
=\h_1\left(\{x\} \underset{X_{0,\bullet}}{\overset{h}{\times}} X_{1,\bullet} \underset{X_{0,\bullet}}{\overset{h}{\times}} \{y\}\right)
$$
as Hom categories. Horizontal composition is defined as follows:  
\begin{eqnarray*}
\left(\{x\}\underset{X_{0,\bullet}}{\overset{h}{\times}} X_{1,\bullet} \underset{X_{0,\bullet}}{\overset{h}{\times}} \{y\}\right)\times
\left(\{y\}\underset{X_{0,\bullet}}{\overset{h}{\times}} X_{1,\bullet} \underset{X_{0,\bullet}}{\overset{h}{\times}} \{z\}\right)
& \longrightarrow & \{x\}\underset{X_{0,\bullet}}{\overset{h}{\times}} X_{1,\bullet} \underset{X_{0,\bullet}}{\overset{h}{\times}}X_{1,\bullet} \underset{X_{0,\bullet}}{\overset{h}{\times}} \{z\} \\
& \tilde\longleftarrow & \{x\}\underset{X_{0,\bullet}}{\overset{h}{\times}} X_{2,\bullet} \underset{X_{0,\bullet}}{\overset{h}{\times}} \{z\} \\
& \longrightarrow & \{x\}\underset{X_{0,\bullet}}{\overset{h}{\times}} X_{1,\bullet} \underset{X_{0,\bullet}}{\overset{h}{\times}} \{z\}\,.
\end{eqnarray*}
The second arrow happens to go in the wrong way but it is a weak equivalence. Therefore after taking $\h_1$ it turns out to be an equivalence of categories, 
and thus to have an inverse (assuming the axiom of choice). 
\end{defn}

A proof that this definition indeed gives a bicategory will be the subject of a subsequent article.

\subsection{Complete and hybrid \texorpdfstring{$n$}{n}-fold Segal spaces}\label{sec complete SS_n}

As with (1-fold) Segal spaces, we need to impose an extra condition to ensure that invertible $k$-morphisms are paths in the space of $(k-1)$-morphisms. Again, there are several ways to include its information.
\begin{defn}
Let $X$ be an $n$-fold Segal space and $1\leq i,j \leq n$. It is said to satisfy
\begin{description}
\item[$CSS^i$\label{CSS^i}] if for every $k_1,\ldots, k_{i-1}\geq0$,\
\begin{equation*}
X_{k_1,\ldots, k_{i-1},\bullet, 0,\ldots, 0}
\end{equation*}
is a complete Segal space.
\item[$SC^j$\label{SC^j}] if for every $k_1,\ldots,k_{j-1}\geq0$,\
\begin{equation*}
X_{k_1,\ldots, k_{j-1}, 0,\bullet,\ldots,\bullet}
\end{equation*}
is discrete, i.e.~a discrete space viewed as a constant $(n-j+1)$-fold Segal space.
\end{description}
\end{defn}

\begin{defn} An $n$-fold Segal space is
\begin{enumerate}
\item {\em complete}, if for every $1\leq i\leq n$, $X$ satisfies \eqref{CSS^i}.
\item a {\em Segal $n$-category} if for every $1\leq j\leq n$, $X$ satisfies \eqref{SC^j}.
\item {\em $m$-hybrid} for $m\geq0$ if condition \eqref{CSS^i} is satisfied for $i> m$ and condition \eqref{SC^j} is satisfied for $j\leq m$.
\end{enumerate}
Denote the full subcategory of $\SeSp_n$ of complete $n$-fold Segal spaces by $\CSSp_n$.
\end{defn}

\begin{rem}
Note that an $n$-hybrid $n$-fold Segal space is a Segal $n$-category, while an $n$-fold Segal space is $0$-hybrid if and only if it is complete.
\end{rem}

For our purposes, the model of complete $n$-fold Segal spaces is well-suited, which leeds us to the following definition.
\begin{defn}
An $(\infty,n)$-category is a complete $n$-fold Segal space.
\end{defn}

\subsubsection{The underlying model categories}\label{model cat CSS_n}

Similarly to subsection \ref{sec css fibrant} there are model categories running in the background. We can consider either the injective or projective model structure on the category of $n$-uple simplicial spaces $\sSpace^n$, which we denote by $\sSpace^n_{c}$ respectively $\sSpace^n_{f}$. Bousfield localizations at the analogs of the Segal maps give model structures whose fibrant objects are (Reedy fibrant) $n$-uple Segal spaces, further localizing at maps governing essential constancy, the fibrant objects become (Reedy fibrant) $n$-fold Segal spaces, and a third localization at a map imposing completeness gives model structures $\sSpace_{n,c}^{CSe}$ respectively $\sSpace_{n,f}^{CSe}$ whose fibrant objects are (Reedy fibrant) complete $n$-fold Segal spaces, see \cite{LurieGoodwillie,BarwickSP} and \cite[Appendix]{JF-S}. Note that again, the identity map induces a Quillen equivalence between $\sSpace^n_{c}$ and $\sSpace^n_{f}$ which descends to the localizations.

Alternatively, and by \cite[Appendix, Proposition A.9]{JF-S} equivalently, the construction of complete Segal objects for absolute distributors from \cite{LurieGoodwillie} provides an iterative definition of these model categories by considering simplicial objects in a suitable model category (which is taken to be the appropriate localization of $\sSpace_{n-1,c}$ respectively $\sSpace_{n-1, f}$) and localizing at the maps governing the Segal condition, essential constancy, and/or completeness in the new simplicial direction.

\cite{LurieGoodwillie} also provides a model category whose fibrant objects are Segal category objects in some suitable underlying model category, thus allowing an iteration of the construction of Segal categories as well. Applying this construction $m$ times to the above one for complete $(n-m)$-fold Segal spaces provides a model category whose fibrant objects are $m$-hybrid $n$-fold Segal spaces.

One can show (see e.g.~in \cite{BarwickCSS, LurieGoodwillie, BergnerRezk, BergnerRezkII, Yan}) that equivalences between (possibly non-complete) $n$-fold Segal spaces for this model structure are exactly the {\em Dwyer-Kan equivalences}, which are defined inductively. For this we need the following inductive definition of the homotopy category of an $n$-fold Segal space:
\begin{defn}
The {\em homotopy category} $h_1(X)$ of an $n$-fold Segal space $X_\bullets$ is the following category: its objects are the 0-simplices, i.e.~the points of $X_{0,\ldots, 0}$. For $x,y$ two objects, we let 
$$
\Hom_X(x,y)_{\bullet,\dots,\bullet}:=\{x\} \underset{X_{0,\bullet,\dots,\bullet}}{\overset{h}{\times}} X_{1,\bullet,\dots,\bullet} \underset{X_{0,\bullet,\dots,\bullet}}{\overset{h}{\times}} \{y\}
$$
be the $(n-1)$-fold Segal space of morphisms\footnote{We will revisit this notion in \ref{sec morphisms}.} from $x$ to $y$. 
Now let morphisms in $h_1(X)$ from $x$ to $y$ be the set of isomorphism classes of objects in $h_1(\Hom_X(x,y)_{\bullet,\dots,\bullet})$, which is already defined by induction. Composition is defined using the Segal condition in the first index.
\end{defn}

\begin{defn}
A morphism $f:X\to Y$ of $n$-fold Segal spaces is a {\em Dwyer-Kan equivalence} if
\begin{enumerate}
\item the induced functor $h_1(f):h_1(X)\to h_1(Y)$ is essentially surjective. 
\item for each pair of objects $x,y\in X_{0,\dots,0}$, the induced morphism $\Hom_X(x,y)\to \Hom_Y(f(x),f(y))$ is a Dwyer-Kan equivalence of 
$(n-1)$-fold Segal spaces. 
\end{enumerate}
\end{defn}

Again we obtain equivalences of complete Segal spaces
$$
N(\CSSp_n,\mathpzc{lwe})\longrightarrow N(\SeSp_n,\mathpzc{DK})\longrightarrow N(\sSpace_n,\mathcal W_f^{CSe})\,,
$$
where $\mathcal W_f^{CSe}$ is the subcategory of weak equivalences in the localization $\sSpace_{n,f}^{CSe}$. 

\begin{rem}
Note that $\CSSp_n$ is the subcategory of fibrant objects for a left Bousfield localization of $\sSpace_{n, f}$ and weak equivalences of complete $n$-fold Segal spaces are level-wise weak equivalences. Denoting the category of fibrant objects in $\sSpace_{n,c}^{CSe}$, the Reedy fibrant complete $n$-fold Segal spaces, by $\CSSp_{n,c}$, the Quillen equivalence between $\sSpace_{n, c}$ and $\sSpace_{n, f}$ induces an equivalence
$N(\CSSp_{n,c}, \mathpzc{lwe}) \longrightarrow N(\CSSp_n, \mathpzc{lwe}),$
whose inverse is given by Reedy fibrant replacement $(-)^R$.
\end{rem}

Recall from Remark \ref{rem iterative Segal} that we can think of an $n$-fold Segal space in an iterative way: we can view an $n$-fold Segal space as a Segal object in $(n-1)$-fold Segal spaces, which we in turn can think of a Segal object in Segal objects in $(n-2)$-fold Segal spaces, etc. Then condition ($CSS^i$) above means that the $i$th iteration is a {\em complete} Segal space object. For more on this point of view, see \cite{LurieGoodwillie, Haugseng}
\begin{defn}
Given an $n$-fold Segal space $X_\bullets$, one can apply the completion functor iteratively to obtain a complete $n$-fold Segal space $\widehat{X}_\bullets$, its {\em ($n$-fold) completion}. This yields a map $X\to\widehat{X}$, the {\em completion map}, which is universal among all maps (in the homotopy category) to complete $n$-fold Segal spaces. It is a left adjoint to the embedding of $\CSSp_n[\mathpzc{lwe}^{-1}]$ into $\SeSp_n[\mathpzc{lwe}^{-1}]$.

If an $n$-fold Segal space $X_\bullets$ satisfies ($SC^j$) for $j\leq m$, we can apply the completion functor just to the last $(n-m)$ indices to obtain an $m$-hybrid $n$-fold Segal space $\widehat{X}^m_\bullets$, its {\em $m$-hybrid completion}.
\end{defn}

\subsection{Constructions of \texorpdfstring{$n$}{n}-fold Segal spaces}\label{sec constructions}

We describe several intuitive constructions of $(\infty,n)$-categories in terms of (complete) $n$-fold Segal spaces.

\subsubsection{Truncation}\label{sec truncation}
{\em 
Given an $(\infty,n)$-category, for $k\leq n$ its {\em $(\infty,k)$-truncation, or $k$-truncation,} is the $(\infty,k)$-category obtained by discarding the non-invertible $m$-morphisms for $k<m\leq n$. 
}

In terms of $n$-fold Segal spaces, there is a functor $\tau_k: \SeSp_n\to\SeSp_k$ sending $X=X_{\bullet,\dots,\bullet}$ to its $k$-truncation, the $k$-fold Segal space
$$
\tau_kX=X_{\underbrace{\bullet,\dots,\bullet}_{k\textrm{ times}},\underbrace{0,\dots,0}_{n-k\textrm{ times}}}\,.
$$

If $X$ is $m$-hybrid then so is $\tau_kX$ by the definition of the conditions \eqref{CSS^i}  and \eqref{SC^j}. In particular, if $X$ is complete, then $\tau_kX$ is as well, and thus, the truncation of an $(\infty, n)$-category is an $(\infty,k)$-category.

\begin{warning*}
Truncation does not behave well with respect to completion, i.e.~the truncation of the completion is not the completion of the truncation. However, we get a map in one direction:
$$
\xymatrix{
\tau_k(X) \ar[r] \ar[d] & \tau_k (\widehat{X})\\
\widehat{\tau_k(X)} \ar@{-->}[ur]
}
$$
In general, this map is not an equivalence. So in general one should always complete an $n$-fold Segal space before truncating it. For example, for $n=1$ and a non-complete Segal Space $X$, the truncation $\tau_1(X)=X_0$ is just the zeroth space, but the truncation of the completion will be equivalent to the underlying $\infty$-groupoid $X_1^{inv}$. The map in this case is given by the degeneracy map. In the example $X=N(G)$ from Remark \ref{rem completion}, the former is $N(G)_0 = \{*\}$ and the latter is $BG$, which are not equivalent in general.
\end{warning*}

\begin{rem}
As explained above, the $k$-truncation of an $(\infty,n)$-category $X$ should be the maximal $(\infty,k)$-category contained in $X$. However, the image of the degeneracy
$$X_{\underbrace{\scriptstyle 1,\ldots, 1}_{k},0,\ldots ,0}\hookrightarrow X_{\underbrace{\scriptstyle 1,\ldots, 1}_{m},0,\ldots, 0}$$
consists exactly of the invertible $m$-morphisms for $k<m\leq n$ if and only if $X$ satisfies \eqref{CSS^i} for $k<i\leq n$. For example, if $X=X_\bullet$ is a ($1$-fold) Segal space then $X_0$ is the underlying $\infty$-groupoid of invertible morphisms if and only if $X$ is complete. 
\end{rem}

\subsubsection{Extension}

{\em Any $(\infty,n)$-category can be viewed as an $(\infty,n+1)$-category with only identities as $(n+1)$-morphisms.}

In terms of $n$-fold Segal spaces, any $n$-fold Segal space can be viewed as a constant simplicial object in $n$-fold Segal spaces, i.e.~an $(n+1)$-fold Segal space which is constant in the first index. Explicitly, if $X_{\bullet,\ldots,\bullet}$ is an $n$-fold Segal space, then $\varepsilon(X)_{\bullet,\ldots,\bullet}$ is the constant simplicial object in the category of Segal spaces given by $X$, i.e.~it is the $(n+1)$-fold Segal space such that for every $k\geq0$,
$$\varepsilon(X)_{\bullet,\ldots,\bullet, k}=X_{\bullet,\ldots,\bullet}$$
and the face and degeneracy maps in the last index are identity maps. 

\begin{lemma}
If $X$ is complete, then $\varepsilon(X)$ is complete.
\end{lemma}

\begin{proof}
Since $X$ is complete, it satisfies ($CSS^i$) for $i>1$. For $i=0$, we have to show that $\varepsilon(X)_{\bullet,0,\ldots,0}$ is complete. This is satisfied because
$$(\varepsilon(X)_{1,0,\ldots, 0})^{inv}=\varepsilon(X)_{1,0,\ldots, 0}=X_{0,\ldots, 0}=\varepsilon(X)_{0,0,\ldots, 0},$$
since morphisms between two elements $x,y$ in the homotopy category of $\varepsilon(X)_{\bullet, k_2,\ldots, k_n}$ are just connected components of the space of paths in $X_{k_2,\ldots, k_n}$, and thus are always invertible.
\end{proof}

We call $\varepsilon$ the extension functor, which is left adjoint to $\tau_{n}$. Moreover, the unit 
$\mathrm{id}\to\tau_1\circ\varepsilon$ of the adjunction is the identity.

\subsubsection{Inverting}

{\em 
Given an $(\infty,n)$-category, for $k\leq n$ we obtain an $(\infty,k)$-category by inverting the non-invertible $m$-morphisms for $k<m\leq n$. 
}

We saw that the extension functor $\varepsilon$ had a right adjoint $\tau_n$. It also has a left adjoint $\eta$, which formally inverts all $(n+1)$-morphisms. For an $n$-fold Segal space $X$, this is given by realizing the last index,
$$(\eta X)_{k_1,\ldots, k_n} = |X_{k_1,\ldots, k_n, \bullet}|.$$
Here geometric realization amounts to taking the diagonal of the bisimplicial set $X_{k_1,\ldots, k_n, \bullet}$. Since the following diagram of right adjoints commutes, the diagram of left adjoints commutes as well. Therefore, completion and inverting commute.
$$
\begin{tikzcd}[row sep=huge]
\SeSp_{n+1}  \arrow[d, draw=none, "\dashv" description] \arrow[r, draw=none, "{\rotatebox[origin=c]{90}{$\dashv$}}" description] \arrow[bend right, swap, d, "\widehat{(-)}"] \arrow[bend right]{r}{\eta} & \SeSp_n \arrow[d, draw=none, "\dashv" description] \arrow[bend right]{l}{\varepsilon} \arrow[bend right, swap, d, "\widehat{(-)}"] \\
\CSSp_{n+1} \arrow[r, draw=none, "{\rotatebox[origin=c]{90}{$\dashv$}}" description] \arrow[hook, bend right]{u} \arrow[bend right]{r}{\eta} & \CSSp_n \arrow[bend right]{l}{\varepsilon} \arrow[hook, bend right]{u}
\end{tikzcd}
$$

\subsubsection{The higher category of morphisms and loopings}\label{sec morphisms}

{\em Given two objects $x,y$ in an $(\infty,n)$-category, morphisms from $x$ to $y$ should form an $(\infty,n-1)$-category.}

This can be realized for $n$-fold Segal spaces, which is one of the main advantages of this model for $(\infty,n)$-categories.

\begin{defn}
Let $X=X_{\bullet,\cdots,\bullet}$ be an $n$-fold Segal space. As we have seen above one should think of objects as vertices of the space $X_{0,\ldots,0}$. Let $x,y\in X_{0,\ldots,0}$. The {\em $(n-1)$-fold Segal space of morphisms from $x$ to $y$} is 
$$
\Hom_X(x,y)_{\bullet,\cdots,\bullet}=
\{x\} \underset{X_{0,\bullet,\cdots,\bullet}}{\overset{h}{\times}} X_{1,\bullet,\cdots,\bullet} 
\underset{X_{0,\bullet,\cdots,\bullet}}{\overset{h}{\times}} \{y\}\,.
$$
\end{defn}

\begin{rem}\label{Hom (m-1)-hybrid}
Note that if $X$ is $m$-hybrid, then $\Hom_{X}(x,y)$ is $(m-1)$-hybrid.
\end{rem}

\begin{ex}[Compatibility with extension]
Let $X$ be an $(\infty,0)$-category, i.e.~a space, viewed as an
$(\infty,1)$-category, i.e.~a constant (complete) Segal space $\varepsilon(X)_\bullet$, $\varepsilon(X)_k=X$. For any two objects $x,y\in \varepsilon(X)_0=X$ the $(\infty,0)$-category, i.e.~the space, of morphisms from $x$ to $y$ is 
$$
\Hom_{\varepsilon(X)}(x,y)=
\{x\} \underset{\varepsilon(X)_0}{\overset{h}{\times}} \varepsilon(X)_{1}\underset{\varepsilon(X)_0}{\overset{h}{\times}} \{y\}
=\{x\}\underset{X}{\overset{h}{\times}}\{y\}=\mathrm{Path}_X(x,y)\,,
$$
the path space in $X$, which coincides with what one expects by the interpretation of paths, homotopies, homotopies between homotopies, etc.~being higher invertible morphisms. 
\end{ex}

\begin{defn}
Let $X$ be an $n$-fold Segal space, and $x\in X_0$ an object in $X$. Then the {\em looping of $X$ at $x$} is the $(n-1)$-fold Segal space
$$\myloop{X}{x}_{\bullet,\ldots,\bullet} = \Hom_{X}(x,x)_\bullets = \{x\}\times^h_{X_{0,\bullet,\ldots,\bullet}} X_{1,\bullet,\ldots,\bullet} \times^h_{X_{0,\bullet,\ldots,\bullet}} \{x\}.$$
In the following, it will often be clear at which element we are looping, e.g.~if there essentially is only one element, or at a unit for a monoidal structure, which we define in the next section. Then we omit the $x$ from the notation and just write
$$\Omega X=\myloopnop{X}=\myloop{X}{x}.$$
\end{defn}

We can iterate this procedure as follows.
\begin{defn}\label{def iterated looping}
Let $\myloop[0]{X}{x} = X$. For $1\leq k\leq n$, let the {\em $k$-fold iterated looping} be the $(n-k)$-fold Segal space
$$\myloop[k]{X}{x}=\myloop{\myloop[k-1]{X}{x} }{x},$$
where we view $x$ as a trivial $k$-morphism via the degeneracy maps, i.e.~an element in $\myloop[k-1]{X}{x}_{0\ldots, 0} \to X_{\underbrace{\scriptstyle 1,\ldots, 1}_k,0,\ldots,0}$.
\end{defn}

Looping $k$ times commutes with taking the $k$-hybrid completion up to weak equivalence, since completion is taken index by index:

Let $X$ be a $k$-hybrid $n$-fold Segal space. Then for the $k$-hybrid completion $\hat X$, which is the completion in the last $(n-k)$ variables, we have that $\myloopnop[k]{\hat X} \xrightarrow{\simeq} \hat X_{\underbrace{1,\ldots 1}_{k},\bullets}$ is complete, so by the universal property of completion, the horizontal map in the following diagram exists:
$$
\begin{tikzcd}
\widehat{\myloopnop[k]{X}} \arrow[dashed]{r}& \myloopnop[k]{\hat X}\\
\myloopnop[k]{X} \arrow{u}{}\arrow[swap]{ur}{}
\end{tikzcd}
$$

\begin{lemma}\label{lemma looping completion}
Let $X$ be a $k$-hybrid $n$-fold Segal space. Then the induced map
$$\widehat{\myloopnop[k]{X}}  \xrightarrow{\simeq} \myloopnop[k]{\hat X}$$
is a level-wise weak equivalence.
\end{lemma}

\begin{proof}
In the diagram
$$
\begin{tikzcd}
\widehat{\myloopnop[k]{X}} \arrow[dashed]{r}& \myloopnop[k]{\hat X}\\
\myloopnop[k]{X} \arrow{u}{}\arrow[swap]{ur}{}
\end{tikzcd}
$$
we know that the vertical map is a DK-equivalence, since completions are DK-equivalences. 
Moreover, since $X$ is hybrid, we have that $\myloopnop[k]{X} \xrightarrow{\simeq} X_{\underbrace{1,\ldots 1}_{k},\bullets}$ and $\myloopnop[k]{\hat X} \xrightarrow{\simeq} \hat X_{\underbrace{1,\ldots 1}_{k},\bullets}$, and by definition of (hybrid) completion, $X_{\underbrace{1,\ldots 1}_{k},\bullets} \to \hat X_{\underbrace{1,\ldots 1}_{k},\bullets}$ is just a completion, so it is a DK-equivalence. Thus, in the diagram
$$
\begin{tikzcd}
\widehat{X_{\underbrace{1,\ldots 1}_{k},\bullets}} \arrow[dashed]{r}& \hat X_{\underbrace{1,\ldots 1}_{k},\bullets}\\
X_{\underbrace{1,\ldots 1}_{k},\bullets} \arrow{u}{}\arrow[swap]{ur}{}
\end{tikzcd}
$$
by the two-out-of-three property, the horizontal morphism is as well. But since both $\widehat{X_{\underbrace{1,\ldots 1}_{k},\bullets}}$ and $\hat X_{\underbrace{1,\ldots 1}_{k},\bullets}$ are complete, it is a level-wise equivalence.
\end{proof}

\subsubsection{\texorpdfstring{$n$}{n}-fold from \texorpdfstring{$n$}{n}-uple  Segal spaces}\label{sec n-fold from n-uple}

We can extract the maximal $n$-fold Segal space from an $n$-uple one by the following procedure. Let us recall and introduce some notation for various model structures on the category of $n$-uple simplicial spaces. 
\begin{itemize}
\item $\sSpace^{(C)Se}_{n,f}$, where fibrant objects are (complete) $n$-fold Segal spaces. 
\item $\sSpace^{(C)Se}_{n,c}$, where fibrant objects are Reedy fibrant (complete) $n$-fold Segal spaces. 
\item $\sSpace_{Se}^{n,f}$, where fibrant objects are $n$-uple Segal spaces. 
\item $\sSpace_{Se}^{n,c}$, where fibrant objects are Reedy fibrant $n$-uple Segal spaces. 
\end{itemize}
From now, let $*\in\{c,f\}$. There are (two) Quillen adjunctions 
$$
\sSpace^{Se}_{n,*} \overunder[\mathrm{id}]{\mathrm{id}}{\leftrightarrows} \sSpace_{Se}^{n,*}\,.
$$
Let us denote (in a rather unusual way) $\mathbf{L}:=\mathbb{R}\mathrm{id}:N(\sSpace^{Se}_{n,*},w.e.)\to N(\sSpace_{Se}^{n,*},w.e.)$. 
Observe that on fibrant objects, $\mathbf{L}$ is nothing but the inclusion of (possibly Reedy fibrant) $n$-fold Segal spaces into (possibly Reedy fibrant) $n$-uple Segal spaces. 
After \cite[Proposition 4.12]{Haugseng}, we know it has a right adjoint $\mathbf{R}$. For the given (possibly Reedy fibrant) $n$-uple Segal space $X$, we wish to compute $\mathbf{R}(X)$. By adjunction, we know that 
$$
\mathbf{R}(X)_{\bullets}\simeq \mathrm{Map}^h_{\sSpace^{Se}_{n,*}}\big(\Delta^{\vec{\bullet}},\mathbf{R}(X)\big)\simeq \mathrm{Map}^h_{\sSpace_{Se}^{n,*}}\big(\mathbf{L}(\Delta^{\vec{\bullet}}),X\big)\,,
$$
where $\Delta^{\vec{k}}$ for $\vec{k}=(k_1,\ldots, k_n)$ is the n-fold simplicial set represented by $[k_1]\times\cdots\times[k_n]\in \Delta^{\times n}$, and $\mathrm{Map}^h$ denotes the derived mapping space. 

We will now find an explicit way to compute $\mathbf{R}(X)$ by finding cofibrant replacements of  $\mathbf{L}(\Delta^{\vec{k}})$. We start by recalling certain strict $n$-categories of the desired shapes, which are all objects in Joyal's category $\Theta_n$ \cite{JoyalTheta, Re2}.

For $\vec{k}=(k_1,\ldots, k_n)$, let $\Theta^{\vec{\bullet}}$ be the \emph{walking $\vec{k}$-tuple of $n$-morphisms} which is the strict $n$-category from \cite[Definition 5.1]{JF-S}.
We do not want to recall the full definition here, but rather the intuition:
\begin{itemize}
\item For $\vec{k}=(1,0,\ldots, 0)$, the category $\Theta^{\vec{k}} =
\tikz[baseline=(base)]{ \path node[dot] (l) {} +(0,-2.5pt) coordinate (base) +(.5,-0pt) node[dot](r) {}; \draw[arrow] (l) -- (r);} 
$ is the walking 1-morphism.
\item For $\vec{k}=(2,0,\ldots, 0)$, the category $\Theta^{\vec{k}} =
\tikz[baseline=(base)]{ \path node[dot] (l) {} +(0,-2.5pt) coordinate (base) +(.5,-0pt) node[dot](r) {} +(1,-0pt) node[dot](t) {}; \draw[arrow] (l) -- (r); \draw[arrow] (r) -- (t);} 
$ is the walking composable pair of 1-morphisms.
\item For $\vec{k}=(2,1,\ldots, 0)$, the strict 2-category $\Theta^{\vec{k}} =
\tikz[baseline=(base)]{
\path node[dot] (l) {} +(0,-2.5pt) coordinate (base) +(1,-0pt) node[dot](r) {} +(2,-0pt) node[dot](t) {};
\draw[arrow] (l) to[bend left=60] (r) node[midway] {} (A);
\draw[arrow] (l) to[bend right=60] (r) node[midway] {} (B); 
\draw[arrow] (r) to[bend left=60] (t) node[midway] {} (C);
\draw[arrow] (r) to[bend right=60] (t) node[midway] {} (D);
\draw[twoarrow] (0.5, -0.5) -- (0.5,0.5);
\draw[twoarrow] (1.5, -0.5) -- (1.5,0.5);
} $ is the walking horizontally composable pair of 2-morphisms.
\item For $\vec{k}=(3,2,\ldots, 0)$, we have the strict 2-category $\Theta^{\vec{k}} =
\tikz[baseline=(base), scale=1.7]{
\path node[dot] (l) {} +(0,-2.5pt) coordinate (base) +(1,-0pt) node[dot](r) {} +(2,-0pt) node[dot](t) {} +(3,0) node[dot] (s) {};
\draw[arrow] (l) to[bend left=70] (r) ;
\draw[arrow] (l) to[bend left=0] (r) ;
\draw[arrow] (l) to[bend right=70] (r) ; 
\draw[arrow] (r) to[bend left=70] (t) ;
\draw[arrow] (r) to[bend left=0] (t) ;
\draw[arrow] (r) to[bend right=70] (t) ;
\draw[arrow] (t) to[bend left=70] (s) ;
\draw[arrow] (t) to[bend left=0] (s) ;
\draw[arrow] (t) to[bend right=70] (s) ;
\draw[twoarrow] (0.5, -0.15) -- (0.5,0.45);
\draw[twoarrow] (0.5, -0.45) -- (0.5,0.15);
\draw[twoarrow] (1.5, -0.15) -- (1.5,0.45);
\draw[twoarrow] (1.5, -0.45) -- (1.5,0.15);
\draw[twoarrow] (2.5, -0.15) -- (2.5,0.45);
\draw[twoarrow] (2.5, -0.45) -- (2.5,0.15);
} 
$
\item More generally, for $\vec{k}=(k_1, \ldots, k_n)$, the strict $n$-category $\Theta^{\vec{\bullet}}$ has $k_1\cdots k_n$ $n$-morphisms which are composable following the pattern of a grid of dimension $k_1\times\cdots\times k_n$.
\end{itemize}

The elementary building blocks for these categories are $\Theta^{(n)}$, where $(n)=(\underbrace{1,\ldots, 1}_{n}, 0,\ldots, 0)$. All others are built by gluing these in a grid of of dimension $k_1\times\cdots\times k_n$. In \cite{BarwickSP} Barwick--Schommer-Pries use the following definition, which can been easily seen to be equivalent to the one in \cite{JF-S} by induction:
\begin{defn}
Let $C^1$ be the walking 1-morphism, i.e.~the category with two objects and one non-identity morphism from one object to the other,
$C^1 = \{\tikz{ \path node[dot] (l) {} +(.8,0) node[dot](r) {} +(0,-3pt) coordinate (base); \draw[arrow] (l) -- (r);}\}$.
The strict $n$-category $\Theta^{(n)}$ is defined inductively by the pushout square
$$
\begin{tikzcd}
\{0,1\}\times \Theta^{(n-1)} \arrow[hook]{r} \arrow{d} & C^1\times \Theta^{(n-1)} \arrow{d}\\
\{0,1\}\times \{*\} \arrow{r} & \Theta^{(n)} .
\end{tikzcd}
$$
Note that this immediately implies the existence of a surjective ``collapse'' map $c_n: C^n \to \Theta^{(n)}$, where $C^n=(C^1)^{\times n}$ is the walking $n$-morphism as a strict $n$-uple category.
\end{defn}

The $n$-fold nerve of $\Theta^{\vec{k}}$ is 
\begin{itemize}
\item levelwise fibrant (because $\Theta^{\vec{k}}$ is discrete). 
\item a Segal space (because $\Theta^{\vec{k}}$ is a strict $n$-category). 
\item complete (because $\Theta^{\vec{k}}$ is reduced). 
\end{itemize}
Let us thus abuse notation and still write $\Theta^{\vec{k}}$ for this (complete) $n$-fold Segal space. Now we can write the formula for the cofibrant replacement, and therefore the recipe for finding the underlying $n$-fold Segal space.
\begin{thm}\label{thm underlying n-fold}
Given a $n$-uple Segal space $X$, its maximal underlying $n$-fold Segal space has levels, for $\vec{k}=(k_1,\ldots, k_n)\in(\Delta^{op})^n$,
$$\mathbf{R}(X)_{\vec{k}}=\mathrm{Map}^h_{\sSpace_{Se}^{n,*}}(\Theta^{\vec{k}},X).$$
Since $\Theta^{\vec{\bullet}}$ is an $n$-fold cosimplicial object in strict $n$-categories (see \cite{JF-S}), this defines a (complete) $n$-fold Segal space.
\end{thm}

To prove this Theorem, we need to understand what the cofibrant replacement $\mathbf{L}(\Delta^{\vec{\bullet}})$ is. The first step is a tool to compute the right hand expression in the Theorem, namely, an explicit cofibrant replacement of $\Theta^{\vec{k}}$.
\begin{prop}\label{prop cofibrant Theta}
For $n=1$, the category $\Theta^{(1)}$, or rather its nerve, is cofibrant in the projective model structure $\sSpace_{CSe}^{1,f}$. For $n>1$, a cofibrant replacement of $\Theta^{(n)}$ in the projective model structure of $n$-uple Segal spaces $\sSpace_{Se}^{n,f}$ is given inductively by replacing the pushouts in the definition by homotopy pushouts and $\Theta^{(n-1)}$ by its (inductively already defined) cofibrant replacement.
\end{prop}

\begin{proof}
Similarly to Section \ref{complete}, we use an argument similar to that in \cite{JF-S}, Remark 3.4., which observes the following: $\Theta^{(2)}$ is given by a strict pushout along a diagram of cofibrant objects of which one arrow is an inclusion. By \cite[A.2.4.4]{LurieHTT}, this is a homotopy pushout in the injective model structure and therefore homotpy equivalent to the homotopy pushout in the projective model structure. So a cofibrant replacement of $\Theta^{(2)}$ is given by taking the homotopy pushout of the same diagram,
$$
\begin{tikzcd}
\{0,1\}\times C^1 \arrow[hook]{r} \arrow{d} \arrow[dr, phantom, "\mathrm{h} \lrcorner", very near end] & C^2 \arrow{d}{} \\
\{0,1\}\times \{*\} \arrow{r} & \mathrm{cof}(\Theta^{(2)})
\end{tikzcd}
$$

Now we proceed by induction. Assume we have shown the statement for $k<n$ and we have a cofibrant replacement $\mathrm{cof}(\Theta^{(k)})$ given as in the Proposition. Then, since the map $\{0,1\}\hookrightarrow C^1$ is a cofibration in the projective model structure, the map $\{0,1\}\times \mathrm{cof}(\Theta^{(n-1)}) \hookrightarrow C^1 \times \mathrm{cof}(\Theta^{(n-1)})$ is a cofibration. Moreover, $\{0,1\}\times \mathrm{cof}(\Theta^{(n-1)})$,  $C^1 \times \mathrm{cof}(\Theta^{(n-1)})$, and $\{0,1\}\times \{*\}$ are all cofibrant, so we can use the above-mentioned \cite[A.2.4.4]{LurieHTT}, again to see that the strict pushout, which is weakly equivalent to $\Theta^{(n)}$, is a homotopy pushout, and moreover cofibrant. Summarizing, it is a cofibrant replacement of $\Theta^{(n)}$.
\end{proof}

\begin{rem}
Similarly, we can obtain cofibrant replacements for $\Theta^{\vec k}$ as defined in \cite{JF-S} by replacing the pushouts in the definition by homotopy pushouts.
\end{rem}

The remaining ingredient in the proof of the Theorem is the following Lemma.
\begin{lemma}\label{lem Theta Delta}
The natural map $\Delta^{\vec{k}}\to \Theta^{\vec{k}}$ is a weak equivalence in $\sSpace^{Se}_{n,*}$. 
\end{lemma}
\begin{proof}
We need to show that for any fibrant object $Y$ in $\sSpace^{Se}_{n,*}$ the induced map $\mathrm{Map}^h_{\sSpace^{Se}_{n,*}}(\Theta^{\vec{k}},Y)\to\mathrm{Map}^h_{\sSpace^{Se}_{n,*}}(\Delta^{\vec{k}},Y)$ is a weak equivalence of simplicial sets.

We show the claim for $\vec k=(k)$ proceeding by induction using the explicit cofibrant replacement from the previous Proposition. For $k=1$, this is true, since $\Theta^{(1)} = \Delta^{(1)} =\Delta^1$. Assume we have proven the statement for $l<k$. Then
\begin{align*}
\mathrm{Map}^h(\Theta^{(k)},Y) \overset{\simeq}{\longrightarrow} & \mathrm{Map}^h(C^1\times \Theta^{(k-1)},Y)\overunder[\mathrm{Map}^h(\{0,1\}\times \Theta^{(k-1)},Y)]{h}{\times} \mathrm{Map}^h(\{0,1\},Y)\\
\simeq &\,\, \mathrm{Map}^h(C^1\times \Theta^{(k-1)},Y)\overunder[Y_{0,\bullets}^{\times 2}]{h}{\times} Y_{0,\ldots, 0}^{\times 2}\\
\simeq &\,\, \mathrm{Map}^h(C^1\times \Theta^{(k-1)},Y)\\
\simeq &\,\, \mathrm{Map}^h(\Theta^{(k-1)}, \Hom(C^1,Y)).
\end{align*}
Here the first equivalence uses that the cofibrant replacement of $\Theta^{(k)}$ is the homotopy pushout as described in the previous Proposition, the next equivalence computes the mapping spaces on the right and below the times symbol, the third equivalence uses essential constancy of $Y$, i.e.~condition (ii) in Definition \ref{defn n-fold SSp}, and the last one uses that $n$-fold Segal spaces are Cartesian closed.

By the induction hypothesis, the natural map $\Delta^{(k-1)}\to \Theta^{(k-1)}$ induces an equivalence
$$\mathrm{Map}^h(\Theta^{(k-1)}, \Hom(C^1,Y)) \xrightarrow{\simeq} \mathrm{Map}^h(\Delta^{(k-1)}, \Hom(C^1,Y)) \simeq \mathrm{Map}^h(\Delta^{(k)}, Y) \simeq Y_{(k)}.$$
A similar argument works for general $\vec{k}$.
\end{proof}

\begin{rem}
The above Lemma is equivalent to the observation that the model structure $\sSpace^{Se}_{n,*}$ can be obtained as the left Bousfield localization of $\sSpace_{Se}^{n,*}$ along $\Delta^{\vec{k}}\to\Theta^{\vec{k}}$. 
\end{rem}

\begin{proof}[Proof of Theorem \ref{thm underlying n-fold}]
The following equivalences are compatible with the cosimplicial structure of $\Delta^{\vec{\bullet}}$ and $\Theta^{\vec{\bullet}}$:
$$\mathbf{R}(X)_{\vec{k}}\cong \mathrm{Map}^h \big(\Delta^{\vec{k}},\mathbf{R}(X)\big)\simeq \mathrm{Map}^h\big(\mathbf{L}(\Delta^{\vec{k}}),X\big) \underset{\simeq}{\xleftarrow{Lemma \,\ref{lem Theta Delta}}} \mathrm{Map}^h\big(\mathbf{L}(\Theta^{\vec{k}}),X\big) \simeq \mathrm{Map}^h(\Theta^{\vec{k}},X) \,.
$$
\end{proof}

\section{Symmetric monoidal structures}\label{sec monCSSn}

\subsection{Definition {\em via} \texorpdfstring{$\Gamma$}{Gamma}-objects}\label{Gamma}

Following \cite{toen, ToenVezzosi}, we define a symmetric monoidal $n$-fold Segal space in analogy to Segal's $\Gamma$-spaces from \cite{Segal}. This is a special case of commutative monoid in an arbitrary $(\infty,1)$-category as defined in \cite{LurieHA}.

\begin{defn}\label{def Gamma}
Segal's category $\Gamma$ is the category whose objects are the finite sets
$$\<{m} = \{0,\ldots, m\}, $$
for $m\geq0$ which are pointed at 0. Morphisms are pointed functions, i.e.~for $k,m\geq0$, functions
$$f: \<{m}\longrightarrow \<{k}, \quad f(0)=0.$$
For every $m\geq0$, there are $m$ canonical morphisms
$$\gamma_\beta:\<{m}\longrightarrow \<{1}, \quad j \longmapsto \delta_{\beta j}$$
for $1\leq \beta\leq m$, called the {\em Segal morphisms}.
\end{defn}

\begin{rem}
Note that $\Gamma$ is a skeleton of the category of finite pointed sets $\mathrm{Fin}_*$. In his original paper \cite{Segal}, Segal defined $\Gamma$ to be the opposite category of $\mathrm{Fin}_*$. However, in the literature, $\Gamma$ has often appeared in the above convention.
\end{rem}

Recall from Section \ref{model cat CSS_n} that the $(\infty,1)$-category of $(\infty,n)$-categories is presented by a model category in which the fibrant objects are complete $n$-fold Segal spaces. More precisely, the $(\infty,1)$-category of $(\infty,n)$-categories is defined to be the complete Segal space 
$$
N(\CSSp_n,\mathpzc{lwe})\simeq N(\SeSp_n,\mathpzc{DK}) \simeq N(\sSpace_n,\mathcal W_f^{CSe}).
$$
We would now like to define a symmetric monoidal $(\infty,n)$-category to be an $(\infty,1)$-functor from $\Gamma$, viewed as an $(\infty,1)$-category, e.g.~as $N(\Gamma, \Iso\Gamma)$, to the $(\infty,1)$-category of $(\infty,n)$-categories satisfying certain properties. 

Using the strictification theorem of To\"en-Vezzosi from \cite{ToenVezzosiSegaltopoi} the $(\infty,1)$-category of such functors can be computed using the model category $(\sSpace_{n, f}^{CSe})^\Gamma$ of $\Gamma$-diagrams in $\sSpace_{n, f}^{CSe}$ endowed with the projective model structure,
$$N((\sSpace_{n, f}^{CSe})^\Gamma, \mathcal{W}) \xrightarrow{\simeq} \mathrm{Map}^h_{\sSpace_{n, f}^{CSe}} (N(\Gamma, \Iso\Gamma), N(\sSpace_{n}, \mathcal{W}_f^{CSe})).$$
Fibrant objects in the former are strict functors
from $\Gamma$ to $\CSSp_n$. Thus the following definition suffices.

\begin{defn}\label{defn symm mon Gamma}
A {\em symmetric monoidal complete $n$-fold Segal space} is a (strict) functor from $\Gamma$ to the (strict) category of complete $n$-fold Segal spaces $\CSSp_n$,
$$
A:\Gamma \longrightarrow \CSSp_n
$$
such that for every $m\geq0$, the induced map
$$
A\big(\prod_{1\leq \beta\leq m} \gamma_\beta \big): A\<m\longrightarrow (A\<1)^m
$$
is an equivalence of complete $n$-fold Segal spaces.

The complete $n$-fold Segal space $X = A\<1$ is called the complete $n$-fold Segal space {\em underlying $A$}, and by abuse of language we will sometimes call a complete $n$-fold Segal space $X$ symmetric monoidal, if there is a symmetric monoidal complete $n$-fold Segal space $A$ such that $A\<1 = X$.
\end{defn}

\begin{rem}
Note that in particular, for $m=0$, this implies that $A\<0$ is levelwise equivalent to a point, viewed as a constant $n$-fold Segal space, which we will denote by $\unit$.
\end{rem}

\begin{rem}
We can define symmetric monoidal $n$-fold Segal spaces in a similar way, by replacing $\CSSp_n$ be $\SeSp_n$.
\end{rem}

\begin{defn}
The $(\infty,1)$-category, i.e.~complete Segal space, of functors from $\Gamma$ to $\CSSp_n$, which as mentioned above can be computed using the model category of $\Gamma$-diagrams in $\sSpace_{n, f}^{CSe}$, has a full sub-$(\infty,1)$-category of symmetric monoidal complete $n$-fold Segal spaces. Similarly to Section \ref{model cat CSS_n}, this $(\infty,1)$-category can be realized as the localization of the projective model structure on $(\sSpace_{n, f}^{CSe})^\Gamma$ with respect to the Segal morphisms, see \cite[Example A.11]{JF-S}. A $1$-morphism in this $(\infty,1)$-category is called a {\em symmetric monoidal functor of $(\infty,n)$-categories}.
\end{defn}

The completion map $X\to \widehat{X}$ is a weak equivalence. Moreover, since Dwyer-Kan equivalences are closed under products, completion commutes with finite products of Segal spaces (up to weak equivalence). Therefore we obtain the following Lemma.
\begin{lemma}\label{symm mon completion}
If $A:\Gamma \longrightarrow \SeSp_n$ is a symmetric monoidal $n$-fold Segal space, then
\begin{align*}
\widehat{A}: \Gamma &\longrightarrow \CSSp_n,\\
\<m &\longmapsto \widehat{A\<m}
\end{align*}
is a symmetric monoidal complete $n$-fold Segal space.
\end{lemma}

\begin{ex}\label{symm mon h_1}
Let $A:\Gamma \longrightarrow \SeSp$ be a symmetric monoidal Segal space.
Consider the product of maps $\gamma_1\times \gamma_2$ and the map $\gamma:\<2\to\<1; 1,2\mapsto 1$. They induce a span
$$A\<1\times A\<1 \overunder[\simeq]{A(\gamma_1)\times A(\gamma_2)}{\xleftarrow{\hphantom{A(\gamma_1)\times A(\gamma_2)}}} A\<2 \overset{A(\gamma)}{\longrightarrow} A\<1.$$
Passing to the homotopy category, we obtain a map
$$h_1(A\<1)\times h_1(A\<1)\longrightarrow h_1(A\<1).$$
To\"en and Vezzosi showed in \cite{ToenVezzosi} that this is a symmetric monoidal structure on the category $h_1(A\<1)$. Roughly speaking, this uses functoriality of $A$. Associativity uses the Segal space $A\<3$; $A\<0$ corresponds to the unit; and the map $c:\<2\to\<2; 1\mapsto 2, 2 \mapsto 1$ induces the braiding and commutativity.
\end{ex}

\begin{ex}
Truncations and extensions of symmetric monoidal $(\infty,n$)-categories are again symmetric monoidal. Let $A$ be a symmetric monoidal $n$-fold (complete) Segal space. Since $\tau_k$ and $\varepsilon$ are functorial, and preserve weak equivalences and products (since they are right adjoints), the assignments
$$\tau_k(A)\<m = \tau_k( A\<m), \quad \varepsilon(A)\<m = \varepsilon( A\<m)$$
can be extended to functors $\tau_k(A)$ and $\varepsilon(A)$, and the images of $A\big(\prod_{1\leq \beta\leq m} \gamma_\beta \big)$ are again weak equivalences. Thus, they endow the $k$-truncation and extension with a symmetric monoidal structure.
\end{ex}

\begin{ex}\label{ex symm mon looping}
Given a symmetric monoidal (possibly complete) $n$-fold Segal space $A:\Gamma \longrightarrow \SeSp_n$, recall that $A\<0$ is weakly equivalent to the point $\unit$ viewed as a constant $n$-fold Segal space. For every $m\geq0$ there is a unique map $\<0\to\<m$, which induces a map $\unit\simeq A\<0\to A\<m$ which picks out a distinguished object $\unit_\<m\in A\<m$. The looping of $A$ with respect this object is also symmetric monoidal, with
$$
\myloopnop{A}\<m = \myloop{A\<m}{\unit_\<m}\,,
$$
which extends to a symmetric monoidal structure similarly to in the previous example. Note that since the space of choices for the unit $\unit_\<m$ is contractible, different choices lead to equivalent loopings.
\end{ex}

\begin{ex}
Important examples come from the classification diagram construction. Let $\C$ be a small symmetric monoidal category and let $\W=\Iso\C$. As we saw in Section \ref{Rezk construction}, this gives a complete Segal space $N(\C, \W)$. The symmetric monoidal structure of $\C$ endows $N(\C, \W)$ with the structure of a symmetric monoidal complete Segal space:

First note that $\W^{\times m} = \Iso(\C^{\times m})$ for every $m$. On objects, let $A:\Gamma \longrightarrow \CSSp$ be given by $A\<m= N(\C^{\times m},\W^{\times m})_\bullet$. We explain the image of the map $\<2 \to \<1; 1,2\mapsto 1$, which should be a map $A\<2 \to A\<1$. The image of an arbitrary map $\<m\to\<l$ can be defined similarly.

An $l$-simplex in $A\<2\>_0 = N(\C\times \C,\W\times\W)_0$ is a pair
$$C_0\xrightarrow{w_1}\cdots \xrightarrow{w_l} C_l, \qquad D_0\xrightarrow{w'_1}\cdots \xrightarrow{w'_l} D_l,$$
and is sent to
$$C_0\otimes D_0 \xrightarrow{w_1\otimes w_1'} \ldots \xrightarrow{w_l\otimes w_l'} C_l\otimes D_l.$$
Observe that $w_i\otimes w_i'$ is again in $\W$. More generally, an $l$-simplex in
$$A\<2_k = N(\C\times \C,\W\times\W)_k$$
is a pair of diagrams
\begin{equation*}
\begin{tikzcd}
C_{0,0} \arrow{r}{f_{10}} \arrow{d}{w_{01}} & C_{1,0} \arrow{r}{f_{20}} \arrow{d}{w_{11}} & \ldots \arrow {r}{f_{k0}} & C_{k,0} \arrow{d}{w_{k1}} &  D_{0,0} \arrow{r}{g_{10}} \arrow{d}{v_{01}} & D_{1,0} \arrow{r}{g_{20}} \arrow{d}{v_{11}} & \ldots \arrow {r}{g_{k0}} & D_{k,0} \arrow{d}{v_{k1}} \\
C_{0,1} \arrow{r}{f_{11}} \arrow{d}{w_{02}} & C_{1,1} \arrow{r}{f_{21}} \arrow{d}{w_{21}} & \ldots \arrow {r}{f_{k1}} & C_{k,1} \arrow{d}{w_{k2}} &  D_{0,1} \arrow{r}{g_{11}} \arrow{d}{v_{02}} & D_{1,1} \arrow{r}{g_{21}} \arrow{d}{v_{21}} & \ldots \arrow {r}{g_{k1}} & D_{k,1} \arrow{d}{v_{k2}} \\
\vdots \arrow{d}{w_{0l}} & \vdots \arrow{d}{w_{1l}} & & \vdots \arrow{d}{w_{kl}}  & \vdots \arrow{d}{v_{0l}} & \vdots \arrow{d}{v_{1l}} & & \vdots \arrow{d}{v_{kl}} \\
C_{0,l} \arrow{r}{f_{1l}} & C_{1,l} \arrow{r}{f_{2l}} & \ldots \arrow{r}{f_{k,l}} & C_{k,l}  & D_{0,l} \arrow{r}{g_{1l}} & D_{1,l} \arrow{r}{g_{2l}} & \ldots \arrow{r}{g_{k,l}} & D_{k,l}
\end{tikzcd}
\end{equation*}
which is sent to the diagram
\begin{equation*}
\begin{tikzcd}[column sep=large]
C_{0,0}\otimes D_{0,0} \arrow{r}{f_{10}\otimes g_{10}} \arrow{d}{w_{01}\otimes v_{01}} & C_{1,0} \otimes D_{1,0} \arrow{r}{f_{20}\otimes g_{20}} \arrow{d}{w_{11}\otimes v_{11}} & \ldots \\
C_{0,1}\otimes D_{0,1} \arrow{r}{f_{11}\otimes g_{11}} \arrow{d}{} & C_{1,1}\otimes D_{1,1} \arrow{r}{f_{21}\otimes g_{21}} \arrow{d}{} & \ldots \\
\vdots & \vdots 
\end{tikzcd}
\end{equation*}
of component-wise tensor products.

Finally, we need to check that $A\big(\prod_{1\leq \beta\leq m} \gamma_\beta \big)$ is a weak equivalence. This follows from the fact that
$$
(A\<m)_k = N(\C^{\times m},\W^{\times m})_k = \big(N(\C,\W)_k\big)^{\times m} = \big(A\<1_k\big)^m\,.
$$
\end{ex}

\begin{rem}
More generally, if we start with a symmetric monoidal relative category $(\C,\W)$ (a definition can e.g.~be found in \cite{campbell}), such that all $N(\C^{\times m},\W^{\times m})$ are (complete) Segal spaces, then the above construction for $(\C,\W)$ yields a symmetric monoidal (complete) Segal space $N(\C,\W)$.
\end{rem}

\subsection{Definition {\em via} towers of \texorpdfstring{$(n+i)$}{(n+i)}-fold Segal spaces}\label{tower}

Recall that a monoidal category can be seen as a bicategory with just one object. Similarly, a $k$-monoidal $n$-category should be the same as a 
connected $(k+n)$-category with only one object, one 1-morphism, one 2-morphism, and so on up to one $(k-1)$-morphism.

We will base our definitions for (symmetric) monoidal $(\infty,n)$-categories in this section on this guiding principle, which often goes by the name ``Delooping Hypothesis''.

Moreover, in our simplicial setting this principle turns out to be true almost by definition: we will use that associative monoids in a (higher) category $\C$ can be described as simplicial objects in $\C$ satisfying Segal conditions. This motivates the following definition of a ($k$-)monoidal complete $n$-fold Segal space.

\subsubsection{Monoidal \texorpdfstring{$n$}{n}-fold complete Segal spaces}

To implement the above idea, we first need to explain what ``having (essentially) one object'' means.
\begin{defn}
A {\em connected} or {\em 0-connected} $n$-fold Segal space $X$ is a pointed object in $n$-fold Segal spaces, i.e.~a morphism $*\to X$ from the constant $n$-fold Segal space consisting of a point, to $X$, such that the map
$$*\longrightarrow X_{0,\bullets}$$
is a weak equivalence of $(n-1)$-fold Segal spaces. In particular, a connected $n$-fold Segal space has a contractible space of objects.
\end{defn}

\begin{defn}\label{def monoidal complete segal}
A {\em monoidal} complete $n$-fold Segal space is a $1$-hybrid $(n+1)$-fold Segal space $X^{(1)}$ which is connected. Note that as $X^{(1)}$ is 1-hybrid, $X^{(1)}_{0,\bullets}$ is constant with values a discrete space. Thus, to be connected implies that $X^{(1)}_{0,\bullets}$ is equal to the point viewed as a constant $n$-fold Segal space; we again denote the unique object by $*$. We say that this endows the complete $n$-fold Segal space
$$X= \myloopnop{ X^{(1)} } = \myloop{X^{(1)}}{*}$$
with a monoidal structure and that $X^{(1)}$ is a {\em delooping} of $X$.
\end{defn}

\begin{rem}
Without the completeness condition, we could define a monoidal $n$-fold Segal space to be an $(n+1)$-fold Segal space $X^{(1)}$ which is connected. Then $\myloop{ X^{(1)} }{*} = \Hom_{X^{(1)}}(*,*)$ is independent of the choice of point $*\in X_{0,\ldots,0}$ and we can say that this endows the $n$-fold Segal space $X = \myloopnop{X^{(1)} } = \myloop{ X^{(1)} }{ * }$ with a monoidal structure.
However, for a complete Segal space $X$, the space $X_{0,\ldots,0}$ will not be contractible (unless it is trivial). Thus, we need a model for $(\infty,n+k)$-categories which can have a point as the set of objects, 1-morphisms, et cetera. This motivates our use of hybrid Segal spaces. 
\end{rem}

\begin{rem}\label{rem hybrid easy looping}
Let $X$ be an $m$-hybrid $n$-fold Segal space with $m>0$ which is connected. Then $X_{0,\bullet,\ldots,\bullet} = *$, and therefore the looping just is
$$\myloopnop{X}_\bullets = \{*\} \overunder[\{*\}]{h}{\times} X_{1,\bullets} \overunder[\{*\}]{h}{\times} \{*\} \simeq  X_{1,\bullets}.$$
\end{rem}

A similar definition works for hybrid Segal spaces.
\begin{defn}\label{def monoidal hybrid segal}
A {\em monoidal} $m$-hybrid $n$-fold Segal space is an $(m+1)$-hybrid $(n+1)$-fold Segal space $X^{(1)}$ which is connected. We say that this endows the $m$-hybrid $n$-fold Segal space
$$X=\myloopnop{ X^{(1)} }$$
with a monoidal structure and that $X^{(1)}$ is a {\em delooping} of $X$.
\end{defn}

\begin{rem}\label{rem monoids} 
Definitions \ref{def monoidal complete segal} and \ref{def monoidal hybrid segal} are special cases of the following more general construction of monoids in a model category. 
Let $\mathcal M$ be a left proper cellular model category, and consider the projective model structure on the category $\mathcal M^{\Delta^{op}}$ 
of simplicial objects in $\mathcal M$.
By the strictification theorem by To\"en and Vezzosi from \cite{ToenVezzosiSegaltopoi}, the $(\infty,1)$-category of $(\infty,1)$-functors between the $(\infty,1)$-categories represented by $\Delta^{op}$ and $\mathcal M$ is equivalent to $N\big(\mathcal M^{\Delta^{op}},\mathpzc{lwe}\big)$.  
We say that an object $X_\bullet \in \mathcal M^{\Delta^{op}}$ is a {\em weak monoid} if the Segal maps 
$$
X_n\longrightarrow X_1^n
$$
are weak equivalences. One can show that the $(\infty,1)$-category of monoids in $N(\mathcal M,\mathpzc{we})$, which is, as usual, obtained by a localization of the model structure on $\mathcal M^{\Delta^{op}}$ with respect to the maps governing the Segal morphisms, is equivalent to the relative nerve 
of the relative category of weak monoids in $\mathcal M$ and levelwise weak equivalences. 
Monoidal $m$-hybrid $n$-Segal spaces are exactly the weak monoids in $m$-hybrid $n$-Segal spaces. 
\end{rem}

\begin{ex}
Let $\C$ be a small monoidal category and let $\W=\Iso\C$. As we saw in Section \ref{Rezk construction}, this gives a complete Segal space $N(\C, \W)$. The monoidal structure of $\C$ endows $N(\C, \W)$ with the structure of a monoidal complete Segal space:

Recall that $\mathcal C_\bullet$ was the simplicial object in categories given by $\mathcal C_n:=\Fun\big([n],\mathcal C)$. Let $\C_{m,n} = \C_n^{\otimes m}$ be the category which has objects of the form
$$
\xymatrix{
C_{01} \otimes \cdots \otimes C_{0m} \ar[r]^{c_1\!\!\!\!\!\!} & \cdots\cdots\cdots \ar[r]^{\!\!\!\!\!\!c_n} & C_{n0}\otimes \cdots \otimes C_{nm}
}
$$
for $c_i=c_{i1}\otimes \cdots \otimes c_{im}$ and morphisms of the form 
$$
\xymatrix{
C_{01} \otimes \cdots \otimes C_{0m} \ar[r]^{c_1\!\!\!\!\!\!} \ar[d]^{f^0} & \cdots\cdots\cdots \ar[r]^{\!\!\!\!\!\!c_n} & 
C_{n0}\otimes \cdots \otimes C_{nm} \ar[d]_{f^n} \\
D_{01} \otimes \cdots \otimes D_{0m} \ar[r]^{d_1\!\!\!\!\!\!} & \cdots\cdots\cdots \ar[r]^{\!\!\!\!\!\!d_n} & 
D_{n0}\otimes \cdots \otimes D_{nm},
}
$$
where $c_1,\ldots, c_n, d_1,\ldots, d_n$, and $f^0,\ldots, f^n$ are products of $m$ morphisms in $\C$.

Consider its subcategory $\C_{m,n}^\W \subset \C_{m,n}$ which has the same objects, and vertical morphisms involving only the ones in $\W=\Iso\C$, i.e.~$f^0,\ldots, f^n$ are products of morphisms in $\W$.

Now let
$$\C^{(1)}_{m,n}= N(\C_{m,n}^\W)$$
be the (ordinary) nerve. By a direct verification one sees that the collection $\C^{(1)}_{\bullet,\bullet}$ is a 2-fold Segal space. Moreover,
\begin{enumerate}
\item $\C^{(1)}_{0,n} =N(\C_n^{\otimes 0})=*$, so $\C^{(1)}_{0,\bullet}$ is discrete and equal to the point viewed as a constant Segal space, and 
\item for every $m\geq0$, we get that $\C^{(1)}_{m,\bullet}=N(\C_{m,\bullet}^\W)=N((\C_{\bullet}^{\otimes m})^\W)$ is a complete Segal space.
\end{enumerate}
Summarizing, $\C^{(1)}$ is a 1-hybrid 2-fold Segal space which is connected and endows $\myloopnop{ \C^{(1)} }_\bullet \simeq \C^{(1)}_{1,\bullet} \simeq N(\C, \W)_\bullet$ with the structure of a monoidal complete Segal space.
\end{ex}

\subsubsection{\texorpdfstring{$k$}{k}-monoidal \texorpdfstring{$n$}{n}-fold complete Segal spaces}

To encode braided or symmetric monoidal structures, we can push this definition even further.
\begin{defn}
An $n$-fold Segal space $X$ is called {\em j-connected} if
$$X_{\underbrace{\scriptstyle1,\ldots,1,}_{j}0,\bullets}$$
is weakly equivalent to the point viewed as a constant $n$-fold Segal space.
\end{defn}

\begin{rem}
Note that being $j$-connected implies being $i$-connected for every $0\leq i <j$.
\end{rem}

\begin{defn}
A {\em $k$-monoidal} $m$-hybrid $n$-fold Segal space is an $(m+k)$-hybrid $(n+k)$-fold Segal space $X^{(k)}$ which is $(k-1)$-connected.
\end{defn}

\begin{rem}
Note that as $X^{(k)}$ is $(m+k)$-hybrid, $X^{(k)}_{\underbrace{\scriptstyle1,\ldots,1,}_{i}0,\bullets}$ is discrete for every $0\leq i< k$. Thus, being $(k-1)$-connected implies that $X^{(k)}_{\underbrace{\scriptstyle1,\ldots,1,}_{i}0,\bullets}$ is equal to the point viewed as a constant $(n-i+1)$-fold Segal space for every $0\leq i< k$.
\end{rem}

By the following proposition this definition satisfies the delooping hypothesis. In practice we can use it to define a $k$-monoidal $n$-fold complete Segal space step-by-step by defining a tower of monoidal $i$-hybrid $(n+i)$-fold Segal spaces for $0\leq i<k$.

\begin{prop}\label{prop k-monoidal}
The data of a $k$-monoidal $n$-fold complete Segal space is the same as a tower of monoidal $i$-hybrid $(n+i)$-fold Segal spaces $X^{(i+1)}$ for $0\leq i < k$ together with weak equivalences
$$X^{(j)} \simeq \myloopnop{ X^{(j+1)} }$$
for every $0\leq j<k-1$.
\end{prop}

\begin{defn}
We say that these equivalent data endow the complete $n$-fold Segal space
$$X=X^{(0)} \simeq \myloopnop{ X^{(1)} }$$
with a $k$-monoidal structure. The $(n+i+1)$-fold Segal space $X^{(i+1)}$ is called an {\em $i$-fold delooping of $X$}.
\end{defn}

Before proving the proposition, we need the following lemmas.
\begin{lemma}\label{lemma k-l monoidal}
If $X$ is a $k$-monoidal $m$-hybrid $n$-fold Segal space, and $0\leq l \leq k$, then $X$ is an $l$-monoidal $(m+k-l)$-hybrid $(n+k-l)$-fold Segal space.
\end{lemma}

\begin{proof}
Since $X$ is a $k$-monoidal $m$-hybrid $n$-fold Segal space, $X$ is an $(m+k)$-hybrid $(n+k)$-fold Segal space such that
$$X_{\underbrace{\scriptstyle1,\ldots,1,}_{k-1}0,\ldots,0}=*.$$
This implies that $X_{\underbrace{\scriptstyle1,\ldots,1,}_{l-1}0,\ldots,0}=*$.
\end{proof}

\begin{lemma}\label{Hom monoidal}
Let $X$ be a $k$-monoidal $m$-hybrid $n$-fold Segal space. Then $\myloopnop{X}=\myloop{X}{*}$ is a $(k-1)$-monoidal $(m-1)$-hybrid $n$-fold Segal space.
\end{lemma}

\begin{proof}
This follows from
$$\myloopnop{X}_\bullets=\Hom_X(*,*)_{\bullet,\ldots,\bullet}= \{*\}\times^h_{X_{0,\bullet,\ldots,\bullet}} X_{1,\bullet,\ldots,\bullet} \times^h_{X_{0,\bullet,\ldots,\bullet}} \{*\} \simeq X_{1,\bullet,\ldots,\bullet},$$
since $X_{0,\bullet,\ldots,\bullet}=\{*\}$.
\end{proof}

\begin{proof}[Proof of Proposition \ref{prop k-monoidal}.]
Let $Y$ be a $k$-monoidal $n$-fold complete Segal space. By Lemma \ref{lemma k-l monoidal} it is a monoidal $(k-1)$-hybrid $(n+k-1)$-fold Segal space and we define the top layer of our tower to be $X^{(k)}=Y$.

Now let $X^{(k-1)}=\myloopnop{ X^{(k)} }$. By Lemmas \ref{Hom monoidal} and \ref{lemma k-l monoidal}, this is a monoidal $(k-2)$-hybrid $(n+k-2)$-fold Segal space.

Inductively, define $X^{(i)}=\myloopnop{ X^{(i+1)} }$ for $1\leq i \leq k-1$. Similarly to above, by Lemmas \ref{Hom monoidal} and \ref{lemma k-l monoidal}, this is a monoidal $(i-1)$-hybrid $(n+i-1)$-fold Segal space.

Conversely, assume we are given a tower $X^{(i)}$ as in the proposition. Since $Y=X^{(k)}$ is a monoidal $(k-1)$-hybrid $(n+k-1)$-fold Segal space,
\begin{equation}\label{eq X_0}
Y_{0,\bullets}=X^{(k)}_{0,\bullets}=*.
\end{equation}
Since $X^{(k-1)}$ is a monoidal $(k-2)$-hybrid $(n+k-2)$-fold Segal space and by \eqref{eq X_0},
\begin{equation}\label{eq X_1}
\begin{split}
Y_{1,0,\bullets}=X^{(k)}_{1,0,\bullets} &= \{*\}\times^h_{X^{(k)}_{0,0,\bullets}} X^{(k)}_{1,0,\bullets} \times^h_{X^{(k)}_{0,0,\bullets}} \{*\} \\
 &=  \myloopnop{ X^{(k)} }_{0,\bullets} \\
 &\simeq X^{(k-1)}_{0,\bullets} = *.
\end{split}
\end{equation}
Since $X^{(k)}$ is $k$-hybrid, $Y_{1,0,\bullets}=X^{(k)}_{1,0,\bullets}$ is discrete and so $Y_{1,0,\bullets}=*$.

Inductively, for $0\leq i<k$, since $X^{(k-i)}$ is a monoidal $(k-i-1)$-hybrid $(n+k-i-1)$-fold Segal space and by \eqref{eq X_0}, \eqref{eq X_1},\dots
\begin{equation*}
\begin{split}
Y_{\underbrace{\scriptstyle1,\ldots,1,}_{i}0,\bullets} &=X^{(k)}_{\underbrace{\scriptstyle1,\ldots,1,}_{i}0,\bullets}\\
 &\simeq \{*\}\times^h_{X^{(k)}_{0,\underbrace{\scriptscriptstyle1,\ldots,1,}_{i-1}0,\bullets}} X^{(k)}_{\underbrace{\scriptstyle1,\ldots,1,}_{i}0,\bullets} \times^h_{X^{(k)}_{0,\underbrace{\scriptscriptstyle1,\ldots,1,}_{i-1}0,\bullets}} \{*\} \\
 &=  \myloopnop{ X^{(k)} }_{\underbrace{\scriptstyle1,\ldots,1,}_{i-1}0,\bullets} \\
 &\simeq X^{(k-1)}_{\underbrace{\scriptstyle1,\ldots,1,}_{i-1}0,\bullets} = \ldots \simeq X^{(k-i)}_{0,\bullets} = *.
\end{split}
\end{equation*}
Again, since $X^{(k)}$ is $k$-hybrid, $Y_{\underbrace{\scriptstyle1,\ldots,1,}_{i}0,\bullets}=X^{(k)}_{\underbrace{\scriptstyle1,\ldots,1,}_{i}0,\bullets}$ is discrete and so $Y_{\underbrace{\scriptstyle1,\ldots,1,}_{i}0,\bullets}=*$.
\end{proof}

Given a bicategory $\C$ and an object $x$ in $\C$, the endomorphism, or loop category $\mathrm{End}_\C(x) = \myloop{\C}{x}$ is monoidal. Its monoidal structure comes from composition of endomorphisms, which is encoded in the full sub-bicategory of $\C$ which has only the object $x$. In analogy with topology, one can call this delooping $B\myloop{\C}{x}$. We now prove that a similar statement holds for $n$-fold Segal spaces. With the definition of ``symmetric monoidal'' appearing in the next section it will become clear that this provides an analog of Example \ref{ex symm mon looping} in this setting.

Recall from Section \ref{sec constructions} the constructions of the truncation of an $n$-fold Segal space to an $(n-1)$-fold Segal space and its left adjoint, extension. Truncation also has a right adjoint, which is taking the 0th coskeleton:
$$(\mathrm{cosk}_0(X) )_{k_1, \bullets} = X_{\bullets}^{(k_1+1)};$$
face and degeneracy maps are given by partial projections and partial diagonals. Given an $n$-fold Segal space $X$ and $1\leq l\leq n$ we can first truncate $l$ times and then take the coskeleton $l$ times to obtain an $n$-fold Segal space which we abbreviate by $\mathbf{cosk}_0^l (X)$.
\begin{defn}
Fix $1\leq l\leq n$. Let $X_\bullets$ be an $n$-fold Segal space and $x\in X_{0,\ldots, 0}$. The object $x$ determines a map $x: * \to \mathbf{cosk}_0^l (X)$ using the degeneracy maps.
We define a new $n$-fold Segal space $pre\deloop[l]{X}{x}$ as the homotopy pullback
$$
\begin{tikzcd}
pre\deloop[l]{X}{x}_{\bullets} \arrow{d} \arrow{r} &X_{\bullets} \arrow{d}{S}\\
*\arrow{r}{x} & \mathbf{cosk}_0^l (X).
\end{tikzcd}
$$
For $pre\deloop[l]{X}{x}$ we have that for $1\leq i\leq l$,
$$pre\deloop[l]{X}{x}_{k_1,\ldots,k_{i-1}, 0, k_{i+1},\ldots,k_n} \simeq  *\cong\{x\} .$$
To obtain an $l$-hybrid Segal space, we discretize these spaces, i.e.~we define
$$\deloop[l]{X}{x}_{k_1,\ldots,k_{i-1}, 0, k_{i+1},\ldots,k_n} = 
\begin{cases}
*\cong\{x\} & \mbox{if } k_i=0 \mbox{ for } 1\leq i\leq l,\\
pre\deloop[l]{X}{x}_{k_1,\ldots,k_n} & else,
\end{cases}
 $$
with the obvious modified face and degeneracy maps.
\end{defn}

\begin{rem}
Unravelling the coskeleton, for $k_1,\ldots, k_l>0$ the $(n-l)$-fold Segal space $\deloop[l]{X}{x}_{k_1,\ldots, k_l,\bullets}$ is a homotopy fiber:
$$
\begin{tikzcd}
\deloop[l]{X}{x}_{k_1,\ldots, k_l,\bullets} \arrow{d} \arrow{r} &X_{k_1,\ldots, k_l, \bullets} \arrow{d}{S}\\
*\arrow{r}{x} & X_{0,\ldots,0, \bullets}^{\times(k_1+1)\cdots(k_n+1)}
\end{tikzcd}
$$
where $S:X_{k_1,\ldots, k_l,\bullets}\to X_{0,\ldots,0, \bullets}^{\times(k_1+1)\cdots(k_n+1)}$ is the product of all maps arising from the maps $f_i:[0]\to[k_i]$.
The remaining face maps send everything to the point $*$, which we identify with $x$, or, more precisely, its image under the appropriate composition of degeneracy maps. The remaining degeneracy maps $d_{\bullets}:\deloop[l]{X}{x}_{k_1,\ldots,k_l,\bullets}\to \deloop[l]{X}{x}_{k_1,\ldots, k_i-1, \ldots, k_l,\bullets}$ satisfy $d_{\bullets}(*)=d_{\bullets}(x)$, where again we identify $x$ with its image under the appropriate composition of degeneracy maps. Since $X$ is an $n$-uple simplicial space, $\deloop[l]{X}{x}$ is well-defined as an $n$-uple simplicial space. The Segal condition is preserved, and, if $X$ satisfied condition \eqref{CSS^i} for some $i> l$, then $\deloop[l]{X}{x}$ does too.
\end{rem}

\begin{lemma}\label{lemma loop monoidal}
Let $1\leq l\leq n$ and let $X_\bullets$ be an $n$-fold Segal space which satisfies \eqref{CSS^i} for $i>l$. Then for any $x\in X_{0,\ldots, 0}$, the $n$-fold Segal space $\deloop[l]{X}{x}$ is an $l$-monoidal complete $(n-l)$-fold Segal space which endows $\myloop[l]{X}{x}$ with an $l$-monoidal structure.
\end{lemma}

\begin{proof}
By construction $\deloop[l]{X}{x}$ is $(l-1)$-connected. Since $X$ satisfies \eqref{CSS^i} for $i>l$, it is an $l$-hybrid $n$-fold Segal space. Finally, $\myloop[l]{X}{x}= \myloop[l]{\deloop[l]{X}{x} }{x}$ by the following lemma.
\end{proof}

\begin{lemma}\label{lemma looping of fiber}
For an $n$-fold Segal space $X$, an object $x\in X_{0,\ldots,0}$, and $0\leq l \leq n$, we have an equivalence
$$\myloop[l]{\deloop[l]{X}{x} }{x} \simeq \myloop[l]{X}{x}.$$
\end{lemma}

\begin{proof}
This can be checked level-wise: exactly the parts of $X$ which involve $x$ remain in $\deloop[l]{X}{x}$, and when looping at $x$ that's the part that is seen.
\end{proof}

This lemma gives a method for finding a $k$-monoidal structure as a tower.
\begin{prop}\label{prop tower without pointed}
Let $Y^{(0)}$ be an $n$-fold Segal space. Assume we are given, for $1\leq l\leq k$, an $(n+l)$-fold Segal space $Y^{(l)}$ together with an object $y_l\in Y^{(l)}$ such that
$$\myloop{Y^{(l)}}{y_l} \simeq Y^{(l-1)}.$$
Then $Y^{(0)}$ has a $k$-monoidal structure. If all $Y^{(l)}$ satisfy \eqref{CSS^i} for $i>l$, then $Y^{(0)}$ is a $k$-monoidal complete Segal space.
\end{prop}

\begin{proof}
The monoidal $(n+l-1)$-fold Segal space $\deloop{Y^{(l)}}{y_l}$ endows $Y^{(l-1)}$ with a monoidal structure. Proposition \ref{prop k-monoidal} finishes the proof.
\end{proof}

\subsubsection{Symmetric monoidal \texorpdfstring{$n$}{n}-fold complete Segal spaces}

The Stabilization Hypothesis, first formulated in \cite{BaezDolan}, states that an $n$-category which is monoidal of a sufficiently 
high degree cannot be made ``more monoidal'' and moreover is symmetric monoidal.
For Tamasani's weak $n$-categories, a proof was given by Simpson in \cite{SimpsonStabilization}; for general $n$-categories a proof follows from Lurie's proof of Dunn's additivity in \cite{LurieHA}, see \cite{GepnerHaugseng} for details. 

For $(\infty,n)$-categories, we cannot expect stabilization: for instance, $k$-monoidal $(\infty,0)$-categories  are $(\infty, k)$-categories with one object, one morphism, etc.~up to one $(n-1)$-morphism, which, in turn, are $E_k$-algebras (in $\Space$). Note that since there are $E_k$-algebras which are not $E_{k+1}$-algebras, there are $k$-monoidal $(\infty,0)$-categories which are not $(k+1)$-monoidal. However, this motivates the following definition:
\begin{defn}
A {\em symmetric monoidal structure} on a complete $n$-fold Segal space $X$ is a tower of monoidal $i$-hybrid $(n+i)$-fold Segal spaces $X^{(i+1)}$ for $i>0$ such that if we set $X^{(0)}=X$, we have that for every $i\geq0$,
$$X^{(i)} \simeq \myloopnop{ X^{(i+1)} }.$$
\end{defn}

\subsection{Comparing the two definitions}\label{sec monoidal comparison}

In this section we show that every symmetric monoidal $(\infty,n)$-category defined as in Section \ref{Gamma} gives one as defined in Section \ref{tower}. The converse is also true, but we do not go into the details here. Essentially this is a consequence of Dunn's additivity, see \cite{LurieHA}: starting with the definition via a tower of $(\Delta^k)^{op}$-monoids, one can replace them by $E_k$-monoids, which in turn, when letting $k$ go to $\infty$, lead to a commutative monoid. 
See also \cite[Corollary 6.3.13]{GepnerHaugseng}.

We start with a symmetric monoidal $(\infty,n)$-category defined as in Section \ref{Gamma}, a symmetric monoidal complete $n$-fold Segal space $X:\Gamma \to \CSSp_n$. We will precompose it with the functor
\begin{align*}
f:\Delta^{op} &\longrightarrow \Gamma,\\
[m] &\longmapsto \<m\>,
\end{align*}
which sends a map $(f:[n] \to [m] )$ in $\Delta$ to $\tilde f:\<m\>\to\<n\>$, where $\tilde f (0)=0$ and for $j\neq0$,
$$\tilde f (j)= \begin{cases} \min\{i:f(i)=j\}, & \mbox{if it exists},\\
0,& \mbox{otherwise}.
\end{cases}$$

The composition
$$\tilde X^{(1)}: \Delta^{op} \overset{f}{\longrightarrow} \Gamma \overset{X}{\longrightarrow} \CSSP_n$$
is an $(n+1)$-fold simplicial space. Moreover, since $f$ sends the maps $g_\beta$ from Remark \ref{rem Segal g} to the Segal morphisms $\gamma_\beta$ from Definition \ref{def Gamma}, $\tilde X^{(1)}$ is an $(n+1)$-fold Segal space. It satisfies \eqref{CSS^i} for $i> 1$. Moreover, $\tilde X^{(1)}$ is connected.
 However, it does not satisfy \eqref{SC^j} for $j=1$ since $X^{(1)}_{0,\bullets}$ may not be discrete. We can easily remedy this problem: choose an object $x\in X^{(1)}_{0,\ldots,0}$ and consider the $(n+1)$-fold Segal space $X^{(1)} = \deloop{\tilde X^{(1)}}{x}$. Unravelling the definition, we have that
$$X^{(1)}_{k_1,\bullets} = \begin{cases} * & k_1=0\\
\tilde X^{(1)}_{k_1,\bullets} & k_1\neq 0
\end{cases}$$
as complete $n$-fold Segal spaces. Note that choosing different $x$'s leads to equivalent complete $n$-fold Segal spaces. Lemma \ref{lemma loop monoidal} implies that $X^{(1)}$ is a monoidal complete $n$-fold Segal space.

The higher layers of the tower are obtained from the maps $\Gamma^k\to\Gamma$ coming from taking the smash product of finite pointed sets, i.e.~taking their product and identifying anything containing a base point. Then, composing with $f^k$ we obtain
$$\tilde X^{(k)}: (\Delta^{op})^k\overset{f^k}{\longrightarrow} \Gamma^k \longrightarrow \Gamma \longrightarrow \CSSp_n.$$
Similarly, $\tilde X^{(k)}$ is $(k-1)$-connected, but might not satisfy \eqref{SC^j} for $j\leq k$. Choosing any object $x\in X^{(k)}_{0,\ldots,0}$, then $X^{(k)} = \deloop[k]{\tilde X^{(k)} }{x}$ is the desired $k$-monoidal complete $n$-fold Segal space.

 

\part{The \texorpdfstring{$(\infty, n)$}{(infty,n)}-category of bordisms}

To rigorously define fully extended topological field theories we need a suitable $(\infty, n)$-category of bordisms, which, informally speaking, has zero-dimensional manifolds as objects, bordisms between objects as 1-morphisms, bordisms between bordisms as 2-morphisms, etc., and for $k>n$ there are only invertible $k$-morphisms. Finding an explicit model for such a higher category, i.e.~defining a complete $n$-fold Segal space of bordisms, is the main goal of this part and this paper. We endow it with a symmetric monoidal structure and also consider bordism categories with additional structure, e.g.~orientations and framings, which allows us, in Section \ref{TFT}, to rigorously define fully extended topological field theories.

\section{The complete \texorpdfstring{$n$}{n}-fold Segal space of closed intervals}\label{sec Int}
In this section we define a complete Segal space $\Int_\bullet$ of closed intervals in $\R$ which will form the basis of the $n$-fold Segal space of bordisms. It will be a tool to record where (in the time direction) the bordisms can be cut. In particular, there will be a forgetful functor from bordisms to these closed intervals. We start by defining an internal category of closed intervals in $\R$, whose nerve will give a complete Segal space of certain tuples of closed intervals. However, for our model of the bordism category, to avoid having to deal with manifolds with corners, we will instead want to interpret the tuples of intervals as being closed in an open interval of finite length (instead of $\R$). This will be explained in \ref{sec Int2}. Finally, we could have chosen that open interval to always be $(0,1)$ and thus fix the ``length'' in the time direction of the bordism and its collars to be 1. This choice requires rescaling and will be explained in \ref{sec Int (0,1)}.

\subsection{\texorpdfstring{$\Int^c$}{Intc} as an internal category}\label{sec TInt}

We first define a category internal to topological spaces $\TInt_c$ which gives rise to a strongly Segal internal category $\Int^c$ of closed intervals in $\R$.

The topological space of objects of $\TInt_c$ is
\begin{equation}\label{eqn Int0}
\TInt^c_0 =\{(a,b): a<b\} \subset \R^2
\end{equation}
with the standard topology from $\R^2$. We interpret an element $(a,b)\in \TInt^c_0$ as the closed interval $I=[a,b]$. This interpretation gives a bijection from the set of points of the topological space $\TInt^c_0$ to the set of closed bounded intervals:
$$\TInt^c_0\longleftrightarrow \{\mbox{closed bounded intervals }I=[a,b]\mbox{ in }\R \mbox{ with non-empty interior}\}$$
which we use as an identification. In fact, $\TInt_0^c$ is a submanifold of $\R^2$ and to get the desired Kan complex $\Int^c_0$, we take smooth singular simplices (see e.g.~\cite{Lee}), i.e.~for $l\geq 0$, the $l$-simplices are pairs of smooth maps $a, b: \eDelta{l}\to \R$ such that $a(s)< b(s)$ for every $s\in \eDelta{l}$. Faces and degeneracies are the usual ones. We view such an $l$-simplex as a {\em closed interval bundle} and denote it by $[a,b]\to \eDelta{l}$ or $(I(s))_{s\in\eDelta{l}} =(a(s),b(s))_{s\in\eDelta{l}}$.

The topological space of morphisms of $\TInt_c$ is
\begin{equation}\label{eqn Int1}
\TInt^c_1 =\{(a_0, a_1,b_0, b_1): a_j<b_j \mbox{ for }j=0,1, \mbox{ and } a_0\leq a_1, b_0 \leq b_1\} \subset \R^4 ,
\end{equation}
again with the standard topology from $\R^4$. Now we interpret an element $(a_0, a_1,b_0, b_1)\in \TInt^c_1$ as a pair of ordered closed intervals $I_0\leq I_1$, where $I_0=[a_0,b_0]$ and $I_1=[a_1,b_1]$. Here ``ordered'' means that $a_0\leq a_1$ and $b_0\leq b_1$. This gives an identification of the points of the topological space with certain pairs of intervals:
$$\TInt^c_1 \longleftrightarrow \{I_0\leq I_1: I_j=[a_j, b_j] \mbox{ with }a_j<b_j \mbox{ for }j=0,1, \mbox{ and } a_0\leq a_1, b_0 \leq b_1\}.$$
As above $\TInt^c_1$ has the structure of a submanifold of $\R^4$ and by taking smooth singular simplices we obtain a Kan complex $\Int^c_1$ whose $l$-simplices now are quadruples of smooth maps $a_0, a_1, b_0, b_1: \eDelta{l}\to \R$ such that $a_j(s)< b_j(s)$ for $j=0,1$,  $a_0(s)\leq a_1(s)$, and $b_0(s) \leq b_1(s)$ for every $s\in \eDelta{l}$. We view such an $l$-simplex as a {\em closed interval bundle with two closed subintervals} and denote it by $\left([a_0,b_0] \leq [a_1, b_1] \right)\to \eDelta{l}$ or $(I_0(s)\leq I_1(s))_\eDelta{l}$.

The face and degeneracy maps
$$\begin{tikzcd} \TInt^c_0 \arrow{r}[description]{d} & \TInt^c_1 \arrow[bend left=20]{l}{t} \arrow[bend right=20]{l}[swap]{s} \end{tikzcd}$$
arise from forgetting and repeating an interval, respectively:
\begin{align*}
s:[a_0, b_0] \leq [a_1, b_1] & \longmapsto [a_0, b_0], \\
t: [a_0, b_0] \leq [a_1, b_1] & \longmapsto [a_1, b_1],
\intertext{and}
d: [a,b] & \longmapsto [a,b] \leq [a,b].
\end{align*}
Composition is given by remembering the outer intervals:
$$\left([a_0, b_0] \leq [a_1, b_1]\right) \circ \left([a_1, b_1] \leq [a_2, b_2]\right) = \left([a_0, b_0] \leq [a_2, b_2]\right).$$
Here $s,t$, and $d$ are smooth maps, so $\TInt^c$ is a category internal to manifolds. Thus, when taking smooth singular simplices to get $\Int^c$, all above assignments are well-defined for $l$-simplices as well and commute with the faces and degeneracies. 
Moreover, $s$ and $t$ are fibrations since they are restrictions of projections.

\begin{rem}
Note that even though we like to think of the $l$-simplices in $\Int^c_0$ and $\Int^c_1$ as ``closed interval bundles'', we do not treat them as such: face and degeneracy maps are not defined to be pullbacks of the bundles, which would only be defined up to isomorphism; instead, they are defined explicitly at the level of spaces to ensure that simplicial functoriality holds.
\end{rem}

Summarizing, we obtain
\begin{lemma}
$\Int^c$ is a strongly Segal internal category.
\end{lemma}

Moreover, the spaces of objects and morphisms are contractible:
\begin{lemma}\label{lemma contractible}
$\Int^c_0 \simeq \Int^c_1 \simeq *$.
\end{lemma}

\begin{proof}
The underlying topological space is contractible as a subspace of $\R^{2k}$, so the associated Kan complex given by taking smooth simplices is also contractible.
\end{proof}

\subsection{$\Int^c$ as a complete Segal space}\label{sec Segal space Int}

We defined $\Int^c$ as a strongly Segal internal category in the previous section. Its nerve, constructed in Section \ref{sec internal categories}, is a Segal space $\Int^c_\bullet =N(\Int^c)_\bullet$. Let us spell out this Segal space in more detail to become more familiar with it.

For an integer $k\geq0$, let
\begin{equation}\begin{split}
\TInt^c_k=\{(\ul{a}, \ul{b})=(a_0,\ldots, a_k, b_0,\ldots, b_k) : \,&a_j<b_j \mbox{ for }0\leq j\leq k, \mbox{ and } \\
&a_{j-1} \leq a_j \mbox{ and } b_{j-1} \leq b_j \mbox{ for }1\leq j\leq k\} \subset \R^{2k} 
\end{split}
\label{int}\end{equation}
with the subspace topology. As above, one can extract Kan complexes $\Int^c_k$ by taking smooth simplices. Note that for $k=0,1$ this coincides with \eqref{eqn Int0} and \eqref{eqn Int1} above. As before, we interpret an element $(\ul{a}, \ul{b})$ as an ordered $(k+1)$-tuple of closed intervals $\ul{I}=I_0\leq\cdots\leq I_k$ with left endpoints $a_j$ and right endpoints $b_j$ such that $I_j$ has non-empty interior. 
By ``ordered'', i.e.~$I_j\leq I_{j'}$, we mean that the endpoints are ordered, i.e.~$a_j\leq a_{j'}$ and $b_j\leq b_{j'}$ for $j\leq j'$.

\paragraph{Spatial structure of the levels}
The spatial structure of a level $\Int^c_k$ comes from taking smooth singular simplices of the submanifold of $\R^{2k}$. Thus, an $l$-simplex consists of smooth maps
$$\eDelta{l}\to \R, \quad s\mapsto a_j(s), b_j(s)$$
for $j=0,\ldots , k$ such that for every $s\in \eDelta{l}$, the following inequalities hold:
\begin{align*}
a_i(s)< b_i(s), & \quad\mbox{for }i=0,\ldots,k\\
a_{i-1}(s)\leq a_i(s), & \quad\mbox{and}\\
b_{i-1}(s) \leq b_i(s) & \quad\mbox{for } i=1,\ldots,k.
\end{align*}
We denote an $l$-simplex by $(I_0\leq\cdots\leq I_k) \to \eDelta{l}$ or $(I_0(s) \leq \cdots \leq I_k(s))_{s \in |\Delta^l|}$ and call it a {\em closed interval bundle with $(k+1)$ subintervals}.

For a morphism $f:[m]\to [l]$ in the simplex category $\Delta$, i.e.~a (weakly) order-preserving map, let $|f|: \eDelta{m} \to \eDelta{l}$ be the induced map between standard simplices. Let $f^\Delta$ be the map sending an $l$-simplex in $\Int^c_k$ to the $m$-simplex in $\Int^c_k$ given by precomposing with $|f|$,
$$f^\Delta: \big(I_0(s) \leq \cdots \leq I_k(s)\big)_{s \in \eDelta{l}} \longmapsto \big(I_0(|f|(s)) \leq\ldots\leq I_k(|f|(s)\big)_{s\in \eDelta{m}}.$$

\begin{notation}
We denote the {\em spatial face and degeneracy maps} of $\Int^c_k$ by $d^\Delta_j$ and $s^\Delta_j$ for $0\leq j \leq l$.
\end{notation}

The following Lemma is a straightforward generalization of Lemma \ref{lemma contractible}.
\begin{lemma}\label{lemma contractible level}
Each level $\Int^c_k$ is a contractible Kan complex.
\end{lemma}

\paragraph{Simplicial structure -- the simplicial space $\Int^c_\bullet$}

By construction, since $\Int^c$ was strongly Segal, its nerve is a functor $\Int^c_\bullet: \Delta^{op} \to \Space$. Let us recall that to a morphism $g:[m]\to [k]$ in $\Delta$, it assigns
\begin{eqnarray*}
\mathrm{Int}_{k} & \overset{g^*}{\longrightarrow} &\mathrm{Int}_{m},\\
(I_0(s) \leq \cdots \leq I_k(s))_{s\in\eDelta{l}} & \longmapsto & ( I_{g(0)}(s)\leq\cdots \leq I_{g(m)}(s)))_{s\in\eDelta{l}}.
\end{eqnarray*}

One could alternatively see this directly by observing that the assignment is clearly functorial and $f^\Delta$ and $g^*$ commute for all morphisms $f,g$ in $\Delta$.

\begin{notation}
We denote the {\em simplicial face and degeneracy maps} by $d_j$ and $s_j$ for $0\leq j \leq k$.

Explicitly, they are given by the following formulas. The $j$th degeneracy map is given by doubling the $j$th interval, and the $j$th face map is given by deleting the $j$th interval,
\begin{align*}
\mathrm{Int}_{k} & \overset{s_j}{\longrightarrow} \mathrm{Int}_{k+1}, & \mathrm{Int}_{k} & \overset{d_j}{\longrightarrow} \mathrm{Int}_{k-1},\\
I_0\leq\cdots\leq I_k & \longmapsto I_0\leq\cdots \leq I_j\leq I_j \leq \cdots \leq I_k, & I_0\leq\cdots\leq I_k & \longmapsto  
I_0\leq\cdots \leq \hat{I_j}\leq \cdots \leq I_k.
\end{align*}
\end{notation}

\paragraph{The complete Segal space $\Int^c_\bullet$}

\begin{prop}
$\Int^c_\bullet$ is a complete Segal space. Moreover, the inclusion $* \hookrightarrow \Int^c_\bullet$ given by degeneracies, where $*$ is seen as a constant complete Segal space, is an equivalence of complete Segal spaces.
\end{prop}

\begin{proof}
We have seen in Lemma \ref{lemma contractible level} that every $\Int^c_k$ is contractible. This ensures the Segal condition, namely that
$$\Int^c_k \overset{\simeq}{\longrightarrow} \Int^c_1 \overunder[\Int^c_0]{h}{\times} \cdots \overunder[\Int^c_0]{h}{\times} \Int^c_1,$$
completeness, and ensures that the given inclusion is a level-wise equivalence.
\end{proof}

\subsection{The internal category or complete Segal space \texorpdfstring{$\Int$}{Int} of ordered closed intervals in an open one}\label{sec Int2}

We now change our interpretation of the spaces \eqref{int}: we do not identify them with the spaces of ordered closed bounded intervals $I_0\leq\cdots\leq I_k$ anymore, but as ordered intervals which are closed in $(a_0, b_k)$, i.e.~we interpret the elements as 
$$\tilde I_0\leq\cdots\leq \tilde I_k,$$
where $\tilde I_j=I_j\cap (a_0, b_k)$ for $0\leq j\leq k$. Thus, in the generic case when $a_j\neq a_0$ for $0<j\leq k$ and $b_j\neq b_k$ for $0\leq j<k$, then $\tilde I_0\leq \cdots \leq \tilde I_k$ are the half-open or closed intervals 
$$(a_0, b_0] \leq [a_1, b_1]\leq \cdots \leq [a_{k-1}, b_{k-1}]\leq [a_k, b_k).$$
If we view the elements in \eqref{int} in this way, we will denote the internal category (or analogously the Segal space) by $\Int$.

Note that the identity gives an isomorphism of complete Segal spaces describing the change of interpretation:
\begin{align*}
\Int^c_k & \longrightarrow \Int_k\\
(I_0\leq \cdots \leq I_k) &\longmapsto (\tilde I_0\leq \cdots \leq \tilde I_k),
\end{align*}
where $\tilde I_j=I_j\cap (a_0, b_k)$ for $j=0,\ldots ,k$. Conversely, $I_j=\mathrm{cl}_{\R}({\tilde{I_j}})$, the closure of $\tilde I_j$ in $\R$.

\begin{defn}
Let
$$\Int^n_\bullets = (\Int_\bullet)^{\times n}.$$
We denote an element in $\Int^n_{k_1,\ldots, k_n}$ by
$$\oul{I}= (\oul{a},\oul{b})= (I^i_0\leq\cdots\leq I^i_{k_i})_{1\leq i\leq n}.$$
\end{defn}

\begin{lemma}
 The $n$-fold simplicial space $\Int^n_\bullets$ is a complete $n$-fold Segal space. Moreover, the inclusion $*\hookrightarrow \Int^n_\bullets$ given by degeneracies, where $*$ is seen as a constant complete Segal space, is an equivalence of complete $n$-fold Segal spaces.
\end{lemma}

\begin{proof}
The Segal condition and completeness follow from the Segal condition and completeness for $\Int_\bullet$. Since every $\Int_k$ is contractible by Lemma \ref{lemma contractible level}, $(\Int_\bullet)^{\times n}$ satisfies essential constancy, so $\Int^n$ is a complete $n$-fold Segal space. It also ensures that the given inclusion is a level-wise equivalence.
\end{proof}

\subsection{The boxing maps}\label{sec boxing}

We will need the following maps for convenience later:
\begin{defn}
Fix $k\geq 0$. The map of spaces
\begin{align*}
B: \Int_k &\longrightarrow \Int_0\\
\ul{I}=(I_0\leq\cdots\leq I_k)\to\eDelta{l} &\longmapsto B(\ul{I})=B(\ul{a}, \ul{b})=(a_0, b_k) \to \eDelta{l}
\end{align*}
\end{defn}
is called the {\em boxing map}.

Its $n$-fold product gives, for every $k_1,\ldots, k_n\geq0$, a map $B:\Int^n_{k_1,\ldots, k_n}\to \Int_0^n$ which sends an $l$-simplex to the (family of) smallest open box(es) containing all intervals,
$$\oul{I}= (I^i_0\leq\cdots\leq I^i_{k_i})_{1\leq i\leq n}\to\eDelta{l} \longmapsto B(\oul{I})=B(\oul{a}, \oul{b}) = (a^1_0, b^1_{k_1})\times\cdots \times (a^n_0, b^n_{k_n})\to\eDelta{l}.$$
We will usually view the total space of $B(\oul{I})\to\eDelta{l}$ as sitting inside $\R^n\times \eDelta{l}$ as $\bigcup_{s\in\eDelta{l}} B(\oul{I}(s)) \times \{s\}$.

We will also require the following rescaling maps.
\begin{defn}\label{defn box rescaling}
For an element $\oul{I}\in \Int^n_{k_1,\ldots, k_n}$, let $\rho(\oul{I}): B(\oul{I}) \to (0,1)^n$ be the restriction of the product of the affine maps $\R\to \R$ sending $a^i_0$ to 0 and $b^i_k$ to 1. We call it the {\em box rescaling map}.
\end{defn}

\subsection{A variant: closed intervals in \texorpdfstring{$(0,1)$}{(0,1)}}\label{sec Int (0,1)}

One might prefer to restrict to intervals which lie in $(0,1)$, modifying the definition to 
\begin{equation*}\begin{split}
\Int^{(0,1)}_k=\{(\ul{a}, \ul{b})=(a_0,\ldots, a_k, b_0,\ldots, b_k) : a_j<b_j \mbox{ for }0\leq j\leq k,  0= \,&a_0\leq a_1\leq\cdots \leq a_k \\ \mbox{ and } &b_0\leq\cdots\leq b_{k-1}\leq b_k=1\} \subset \Int_k
\end{split}\end{equation*}

The simplicial structure now has to be modified to ensure that the outer endpoints always are $0$ and $1$. This is provided by composition with an affine rescaling map: Let $g:[m]\to [k]$ be a morphism in $\Delta$. Then, let
\begin{eqnarray*}
\Int^{(0,1)}_{k} & \overset{g^*}{\longrightarrow} &\Int^{(0,1)}_{m},\\
(I_0\leq\cdots\leq I_k) \to \eDelta{l} & \longmapsto & \rho_g ( I_{g(0)}\leq\cdots \leq I_{g(m)} ) \to \eDelta{l},
\end{eqnarray*}
where the rescaling map $\rho_g=\rho( I_{g(0)} \leq \cdots \leq I_{g(m)} )$ is the unique affine transformation $\R\to\R$ sending $a_{g(0)}$ to 0 and $b_{g(m)}$ to 1.

\begin{lemma}
$\Int^{(0,1)}_\bullet$ is a complete Segal space.
\end{lemma}

\begin{proof}
The only thing which is not completely analogous to $\Int^c$ is checking that it is a simplicial space. Given two maps $[m]\overset{g}{\to} [k] \overset{\tilde g}{\to} [p]$, and $I_0\leq\cdots\leq I_p$, the rescaling map $\rho_{\tilde g\circ g}$ and the composition of the rescaling maps $\rho_{\tilde g}\circ \rho_g$ both send $a_{\tilde g\circ g(0)}$ to 0 and $b_{\tilde g\circ g(m)}$ to 1 and, since affine transformations $\R\to\R$ are uniquely determined by the image of two points, this implies that they coincide. Thus, this gives a functor $\Delta^{op}\to \Space$.
\end{proof}

Note that the degeneracy maps are the same ones, given by repeating an interval. However, the face maps need to modified: after deleting an end interval we have to rescale the remaining intervals linearly to $(0,1)$. Explicitly, for $j=0$, the rescaling map is the affine map $\rho_0$ sending $(a_1,1)$ to $(0,1)$, $\rho_0(x)=\frac{x-a_1}{1-a_1}$ and for $j=k$, it is the affine map $\rho_k: (0,b_{k-1}) \to (0,1)$, $\rho_k(x)=\frac{x}{b_{k-1}}$. Then,
\begin{eqnarray*}
\Int^{(0,1)}_{k} & \overset{d_j}{\longrightarrow} &\Int^{(0,1)}_{k-1},\\
I_0\leq\cdots\leq I_k & \longmapsto & 
\begin{cases}I_0\leq\cdots \leq \hat{I_j}\leq \cdots \leq I_k, & j\neq 0, k,\\
(0,\frac{b_1 - a_1}{1 - a_1}] \leq \cdots \leq [\frac{a_k - a_1}{1 - a_1}, 1), & j=0,\\
(0,\frac{b_0}{b_{k-1}}] \leq \cdots \leq [\frac{a_{k-1}}{b_{k-1}}, 1), & j=k.
\end{cases}
\end{eqnarray*}

\begin{rem}
An advantage of this ``reduced'' version is that the space of objects is just a point: for $k=0$, the condition on the endpoints of the intervals becomes $a_0=0$ and $b_0=1$, so the only element is $(0,1)\in\Int_0$. In particular, $\Int_0$ is discrete.
\end{rem}

\begin{rem}\label{rem on box rescaling}
Note that the boxing maps applied to $\Int_k^{(0,1)}$ are trivial: for $\ul{I}=I_0\leq\cdots\leq I_k$, we always have that $B(\ul{I})=(0,1)$. Moreover, $\Int_k^{(0,1)}$ is the preimage of $(0,1)$ under the boxing maps. Finally, note that the simplicial structure is defined exactly as the composition
$$\Int^{(0,1)}_k \xrightarrow{\iota} \Int_k \xrightarrow{g^*} \Int_m \xrightarrow{\rho}  \Int^{(0,1)}_m,$$
where $\rho: \ul{I} \mapsto \left(\rho(\ul{I})\right)(\ul{I})$ consists of applying the box rescaling maps. Moreover, since $\rho\circ \iota = id$, the diagram
$$
\begin{tikzcd}
\Int_k \arrow[shift left = 0.5ex]{r}{\rho} \arrow{d}{g^*}& \Int_k^{(0,1)} \arrow[hook, shift left = 0.5ex]{l}{\iota} \arrow{d}{g^*}\\
\Int_m \arrow{r}{\rho} & \Int_m^{(0,1)}
\end{tikzcd}
$$
commutes and shows that the simplicial structure is defined exactly in a way to ensure that we a natural transformation of simplicial spaces
$$ \rho: \Int \longrightarrow \Int^{(0,1)},$$
which is a weak equivalence of complete Segal spaces.
\end{rem}

\section{The \texorpdfstring{$(\infty,n)$}{(infty,n)}-category of bordisms \texorpdfstring{$\Bord_n$}{Bordn}}\label{sec Bord}

In this section we define an $n$-fold Segal space $\PBord_n$ in several steps. However, it will turn out not to be complete in general. By applying the completion functor we obtain a complete $n$-fold Segal space, the $(\infty,n)$-category of bordisms $\Bord_n$.

Let $V$ be a finite-dimensional vector space. We first define the levels relative to $V$ with elements being certain submanifolds of the (finite-dimensional) vector space $V\times\R^n \cong V\times B$, where $B$ is an open box, i.e.~a product of $n$ bounded open intervals in $\R$. Then we vary $V$, i.e.~we take the limit over all finite-dimensional vector spaces lying in some fixed infinite-dimensional vector space, e.g.~$\R^\infty$. The idea behind this process is that by Whitney's embedding theorem, every manifold can be embedded in some large enough vector space, so in the limit, we include representatives of every $n$-dimensional manifold. We use $V\times B$ instead of $V\times\R^n$ as in this case the spatial structure is easier to write down explicitly.

\subsection{The sets of 0-simplicies of \texorpdfstring{$(\PBord_n^V)_{k_1,\ldots, k_n}$}{PBordn}}

The intuition behind the following definition should be the following. An element (i.e.~0-simplex) in the space $(\PBord_n^V)_{1,\ldots, 1}$ should be an $n$-fold bordism, i.e.~a manifold for which there are $n$ ``time'' directions singled out and whose boundary is decomposed into an incoming and an outgoing part in each of these time directions. This is a picture of a simple example for $n=2$.
$$
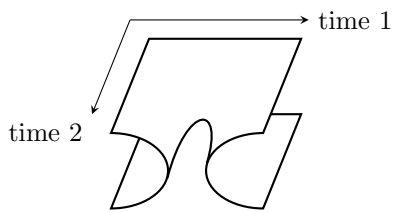
\begin{tikzpicture}[scale=0.5]
\draw[->] (-1.5, 4) -- (3.2, 4) node[anchor=west] {time 1};
\draw[->] (-1.5, 4) -- (-2.5, 1.5) node[anchor=north east]{time 2};
\draw[thick] (-1.25, 0.875) -- (-2, -1) arc (-90:90:1.5 and 1) -- (-1, 3.5) -- (3, 3.5) -- (2,1) arc (90:270:1.5 and 1) -- (3,1.5) -- (2.2, 1.5);
\draw[thick] (-0.52,-0.2) .. controls (0, 1.8) and (1, 1.8) .. (0.52,0.2);
\end{tikzpicture}
$$
An element in the space $(\PBord_n^V)_{k_1,\ldots, k_n}$ should be an $n$-fold bordism, which is the composition of $k_1$ bordisms in the first ``time'' direction, $k_2$ bordisms in the second direction, and so on. This is a picture of an example for $n=2$ and $k_1=k_2=2$.
$$
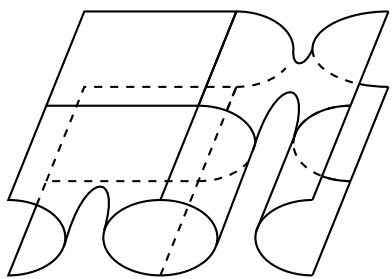
\begin{tikzpicture}[scale=0.5, thick]
\draw[thick] (-1.25, 0.875) -- (-2, -1) arc (-90:90:1.5 and 1) -- (-1, 3.5) -- (3, 3.5) -- (2,1) arc (90:270:1.5 and 1);
\draw[dashed] (2,-1) -- (3,1.5) -- (-1, 1.5) -- (-1.25, 0.875);
\draw[thick] (-0.52,-0.2) .. controls (0, 1.8) and (1, 1.8) .. (0.52,0.2);
\draw (2,-1) arc (-90:90:1.5 and 1);
\draw (3, 3.5) arc (90 : 0: 1.5 and 1);
\draw [dashed] (3,1.5) arc (-90 : -30: 1.5 and 1);
\draw (3.48, -0.2) -- (4.5, 2.5);
\draw (-1, 3.5) -- (0, 6) -- (4,6) -- (3,3.5);
\draw[dashed] (-1, 1.5) -- (0, 4) -- (4,4) -- (3,1.5);

\draw (4,6) arc (90:0: 1.5 and 1);
\draw[dashed] (4,4) arc (-90:-30: 1.5 and 1);
\draw (6, 5) arc (180:90: 2 and 1);
\draw[dashed] (6, 5) arc (-180:-115: 2 and 1);
\draw[thick] (8, 4) arc (-90:-112: 2 and 1);

\begin{scope}[shift={(5,2.5)}]
\draw[thick] (-2, 1) -- (-1, 3.5)
 (3, 3.5) -- (2,1);
\draw[thick] (2,1) arc (90:180:1.5 and 1);
\draw[dashed] (2,1) arc (90:270:1.5 and 1);
\draw[thick] (2,-1) arc (-90: -120: 1.5 and 1);
\draw[thick] (2,-1) -- (3,1.5);
\draw[dashed] (-1, 1.5) -- (-1.25, 0.875);
\draw[thick] (-0.56,-0.18) .. controls (0, 1.8) and (1, 1.8) .. (0.52,0.2);
\draw (0.5, 2.5) .. controls (0.45, 2) and (0.8, 2) .. (1,2.5);

\begin{scope}[shift={(-1, -2.5)}]
\draw[thick] (2,-1) -- (3,1.5);
\draw[thick] (2,1) -- (3,3.5);
\draw [thick] (2,-1) arc (270:90: 1.5 and 1);
\draw [thick] (0.52, 0.17) -- (1.52, 2.67);
\end{scope}
\end{scope}
\end{tikzpicture}
$$
The pictures both depict the bordisms as embedded into $\R$ times the two time directions. We would like to point out that the time directions have a preferred ordering, as we will discuss in more detail later.

More generally, we will choose the bordisms to be equipped with an embedding into some finite dimensional real vector space $V$ times $n$ time directions, which we single out to track where the bordism is allowed to be cut into the individual composed bordisms. Furthermore, to keep track of the ``cuts'', we need to remember the data of the grid in the time directions.  
$$
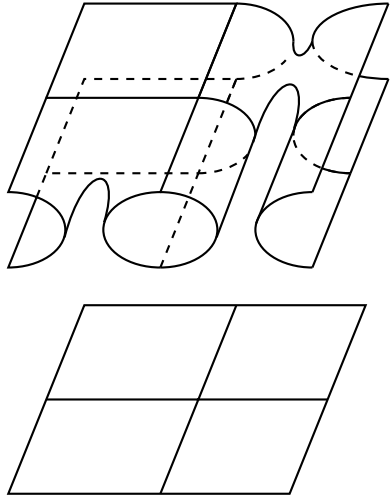
\begin{tikzpicture}[scale=0.5, thick]

\begin{scope}[black, yshift=-6cm]
\draw (-2, -1) -- (0,4) -- (7.4, 4) -- (5.4, -1) -- cycle;
\draw (2, -1) -- (4,4)
(-1, 1.5) -- (6.4, 1.5);
\end{scope}

\draw[thick] (-1.25, 0.875) -- (-2, -1) arc (-90:90:1.5 and 1) -- (-1, 3.5) -- (3, 3.5) -- (2,1) arc (90:270:1.5 and 1);
\draw[dashed] (2,-1) -- (3,1.5) -- (-1, 1.5) -- (-1.25, 0.875);
\draw[thick] (-0.52,-0.2) .. controls (0, 1.8) and (1, 1.8) .. (0.52,0.2);
\draw (2,-1) arc (-90:90:1.5 and 1);
\draw (3, 3.5) arc (90 : 0: 1.5 and 1);
\draw [dashed] (3,1.5) arc (-90 : -30: 1.5 and 1);
\draw (3.48, -0.2) -- (4.5, 2.5);
\draw (-1, 3.5) -- (0, 6) -- (4,6) -- (3,3.5);
\draw[dashed] (-1, 1.5) -- (0, 4) -- (4,4) -- (3,1.5);

\draw (4,6) arc (90:0: 1.5 and 1);
\draw[dashed] (4,4) arc (-90:-30: 1.5 and 1);
\draw (6, 5) arc (180:90: 2 and 1);
\draw[dashed] (6, 5) arc (-180:-115: 2 and 1);
\draw[thick] (8, 4) arc (-90:-112: 2 and 1);

\begin{scope}[shift={(5,2.5)}]
\draw[thick] (-2, 1) -- (-1, 3.5)
 (3, 3.5) -- (2,1);
\draw[thick] (2,1) arc (90:180:1.5 and 1);
\draw[dashed] (2,1) arc (90:270:1.5 and 1);
\draw[thick] (2,-1) arc (-90: -120: 1.5 and 1);
\draw[thick] (2,-1) -- (3,1.5);
\draw[dashed] (-1, 1.5) -- (-1.25, 0.875);
\draw[thick] (-0.56,-0.18) .. controls (0, 1.8) and (1, 1.8) .. (0.52,0.2);
\draw (0.5, 2.5) .. controls (0.45, 2) and (0.8, 2) .. (1,2.5);

\begin{scope}[shift={(-1, -2.5)}]
\draw[thick] (2,-1) -- (3,1.5);
\draw[thick] (2,1) -- (3,3.5);
\draw [thick] (2,-1) arc (270:90: 1.5 and 1);
\draw [thick] (0.52, 0.17) -- (1.52, 2.67);
\end{scope}
\end{scope}
\end{tikzpicture}
$$
In practice, we will keep track of little intervals surrounding the grid instead of the grid itself. This should be thought of as remembering little collars around the cuts rather than the cuts themselves.
$$
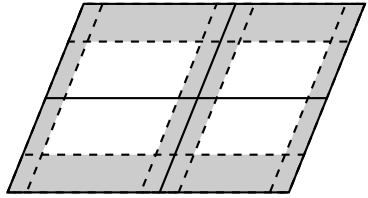
\begin{tikzpicture}[scale=0.5, thick]
\fill[fill= gray!40] (-2, -1) -- (0,4) -- (0.5, 4) -- (-1.5, -1) -- cycle;
\fill[fill= gray!40] (7.4, 4) -- (5.4, -1) -- (4.9, -1) -- (6.9, 4) -- cycle;
\fill[fill= gray!40] (-2, -1) -- (5.4, -1) -- (5.8, 0) -- (-1.6, 0) -- cycle;
\fill[fill= gray!40] (7.4, 4) -- (0,4) -- (-0.4, 3) -- (7, 3) -- cycle;
\fill[fill= gray!40] (2.5, -1) -- (4.5,4) -- (3.5,4) -- (1.5, -1) -- cycle;

\draw[dashed] (-2, -1) -- (0,4) -- (0.5, 4) -- (-1.5, -1) -- cycle;
\draw[dashed] (7.4, 4) -- (5.4, -1) -- (4.9, -1) -- (6.9, 4) -- cycle;
\draw[dashed] (-2, -1) -- (5.4, -1) -- (5.8, 0) -- (-1.6, 0) -- cycle;
\draw[dashed] (7.4, 4) -- (0,4) -- (-0.4, 3) -- (7, 3) -- cycle;
\draw[dashed] (2.5, -1) -- (4.5,4) -- (3.5,4) -- (1.5, -1) -- cycle;

\draw (-2, -1) -- (0,4) -- (7.4, 4) -- (5.4, -1) -- cycle;
\draw (2, -1) -- (4,4)
(-1, 1.5) -- (6.4, 1.5);
\end{tikzpicture}
$$
We will explain how to recover the cuts and how to interpret the following definition in the example and remark right after the definition. Moreover, we will relate it to more classical definitions of (higher) bordisms in section \ref{hocat}.

For $S\subseteq\{1,\ldots,n\}$ denote the projection from $\R^n$ onto the coordinates indexed by $S$ by $\pi_S:\R^n\to\R^S$. We will now define the sets of $0$-simplices of $(\PBord_n^V)_{k_1,\ldots, k_n}$ and denote them by $(\PPBord_n^V)_{k_1,\ldots, k_n}$ to avoid adding an extra index. This notation will only appear in this and the next subsection.

\begin{defn}\label{def PBord}
Let $V$ be a finite-dimensional $\R$-vector space, which we identify with some $\R^r$. For every $n$-tuple $k_1,\ldots, k_n\geq0$, let $(\PPBord_n^V)_{k_1,\ldots, k_n}$ be the collection of tuples $(M,\oul{I}=(I^i_0\leq\cdots\leq I^i_{k_i})_{1\leq i\leq n})$, satisfying the following conditions:
\begin{enumerate}
\item For $1\leq i\leq n$,
$$(I^i_0\leq\cdots\leq I^i_{k_i})\in \Int_{k_i}.$$
\item\label{cond 1} $M$ is a closed and bounded $n$-dimensional submanifold of $V\times B(\oul{I})$ and the composition $\pi: M \hookrightarrow V\times B(\oul{I})\twoheadrightarrow B(\oul{I})$ is a proper map.\footnote{Recall the boxing map from Section \ref{sec boxing}.}
\item\label{cond 3} For every $S\subseteq\{1,\ldots, n\}$, let $p_S:M\xrightarrow{\pi}B(\oul{I})\xrightarrow{\pi_S}\R^S$ be the composition of $\pi$ with the projection $\pi_S$ onto the $S$-coordinates. Then for every $1\leq i\leq n$ and $0\leq j_i\leq k_i$,  at every $x\in p_{\{i\}}^{-1}(I^i_{j_i})$, the map $p_{\{i,\ldots,n\}}$ is submersive.
\end{enumerate}
\end{defn}

\begin{ex}
An example of an element in $(\PPBord_1^\R)$ is depicted below. It represents a composition of three 1-bordisms, the first one of which is ``degenerate'', i.e.~a trivial 1-bordism between two points.
$$
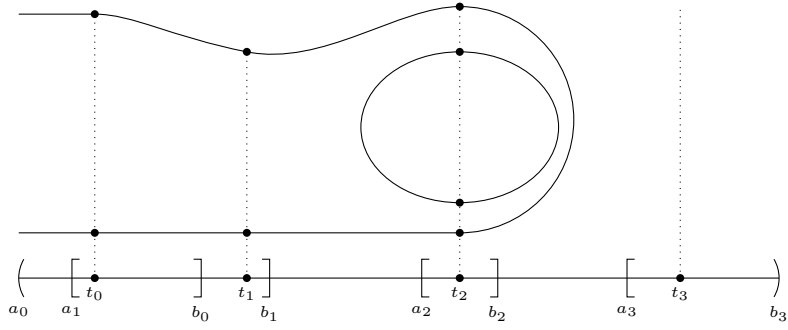
\begin{tikzpicture}
\draw (-5,0) -- (5,0);

\draw (5, -0.25) node [anchor=north] {\tiny $b_3$};
\draw (5,0) arc (0:30:0.5);
\draw (5,0) arc (0:-30:0.5);

\draw (-5, -0.25) node [anchor=north] {\tiny $a_0$};
\draw (-5,0) arc (0:30:-0.5);
\draw (-5,0) arc (0:-30:-0.5);

\draw (-2.7, 0.25) -- (-2.6, 0.25) -- (-2.6, -0.25) node[anchor=north] {\tiny $b_0$} -- (-2.7, -0.25);

\draw (-4.2, 0.25) -- (-4.3, 0.25) -- (-4.3, -0.25) node[anchor=north] {\tiny $a_1$} -- (-4.2, -0.25);

\draw (-1.8, 0.25) -- (-1.7, 0.25) -- (-1.7, -0.25) node[anchor=north] {\tiny $b_1$} -- (-1.8, -0.25);

\draw (0.4, 0.25) -- (0.3, 0.25) -- (0.3, -0.25) node[anchor=north] {\tiny $a_2$} -- (0.4, -0.25);

\draw (1.2, 0.25) -- (1.3, 0.25) -- (1.3, -0.25) node[anchor=north] {\tiny $b_2$} -- (1.2, -0.25);

\draw (3.1, 0.25) -- (3, 0.25) -- (3, -0.25) node[anchor=north] {\tiny $a_3$} -- (3.1, -0.25);

\draw (0.8, 2) ellipse (1.3 and 1);

\draw (-5, 0.6) -- (0.8, 0.6) arc [start angle = -90, end angle = 90, x radius = 1.5, y radius = 1.5] .. controls (0,3.6) and (-1, 2.8) .. (-2, 3) .. controls (-3, 3.2) and (-3.5, 3.5) .. (-4, 3.5) .. controls (-4.3, 3.5) .. (-5, 3.5);

\fill (3.7, 0) node[anchor=north] {\tiny $t_3$} circle (0.15em);
\draw[dotted] (3.7, 0) -- (3.7, 3.6);

\fill (0.8, 0) node[anchor=north] {\tiny $t_2$} circle (0.15em);
\draw[dotted] (0.8, 0) -- (0.8, 3.6);

\fill (0.8, 0.6) circle (0.15em);
\fill (0.8, 1) circle (0.15em);
\fill (0.8, 3) circle (0.15em);
\fill (0.8, 3.6) circle (0.15em);

\fill (-2, 0) node[anchor=north] {\tiny $t_1$} circle (0.15em);
\draw[dotted] (-2,0) -- (-2,3);
\fill (-2, 0.6) circle (0.15em);
\fill (-2, 3) circle (0.15em);

\fill (-4, 0) node[anchor=north] {\tiny $t_0$} circle (0.15em);
\draw[dotted] (-4,0) -- (-4,3.5);
\fill (-4, 0.6) circle (0.15em);
\fill (-4, 3.5) circle (0.15em);

\end{tikzpicture}
$$
\end{ex}

\begin{rem}\label{rem composed bordisms}
For $k_1,\ldots, k_n\geq0$, one should think of an element in $(\PPBord_n^V)_{k_1,\ldots,k_n}$ as a collection of $k_1\cdots k_n$ composed bordisms, with $k_i$ composed bordisms with collars in the $i$th direction. They can be understood as follows.
\begin{itemize}
\item Condition \ref{cond 3} in particular implies that for every $1\leq i\leq n$, at every $x\in p_{\{i\}}^{-1}(I^i_{j})$, the map $p_{\{i\}}$ is submersive. So if we choose $t^i_j\in I^i_j$, it is a regular value of $p_{\{i\}}$, and therefore $p_{\{i\}}^{-1}(t^i_j)$ is an $(n-1)$-dimensional manifold. The embedded manifold $M$ should be thought of as a composition of $n$-bordisms and $p_{\{i\}}^{-1}(t^i_j)$ is one of the $(n-1)$-bordisms (or a composition therof) in the composition.
$$
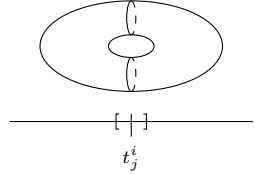
\begin{tikzpicture}[scale=0.2]
\draw (0,0) circle (6 and 3);
\draw (0,0) circle (1.5 and 0.75);

\node (C1) at ($(135: 6 and 3)$) {};
\path let \p1 = (C1) in node[inner sep=0] (c1) at (\x1,0) {};
\path (C1) arc (93: 130: 0.5 and 0.707*3) node[inner sep=0] (P1) {};
\path (C1) arc (93: 210: 0.5 and 0.707*3) node[inner sep=0] (P3) {};
\path (C1) arc (453: 30: 0.5 and 0.707*3) node[inner sep=0] (P2) {};
\path (C1) arc (453: 310: 0.5 and 0.707*3) node[inner sep=0] (P4) {};

\draw (0, 3) arc (90: 270: 0.3 and 1.125);
\draw[dashed] (0, 3) arc (90: -90: 0.3 and 1.125);
\path (0, 3) arc (90: 160 : 0.3 and 1.125) node[inner sep=0] (pp1) {};
\path (0, 3) arc (90: 340 : 0.3 and 1.125) node[inner sep=0] (pp2) {};

\draw[dashed] (0, -3) arc (-90: 90: 0.3 and 1.125);
\draw (0, -3) arc (270: 90: 0.3 and 1.125);
\path (0, -3) arc (-90: 160 : 0.3 and 1.125) node[inner sep=0] (pp3) {};
\path (0, -3) arc (-90: 340 : 0.3 and 1.125) node[inner sep=0] (pp4) {};

\node (C2) at ($(45: 6 and 3)$) {};
\path let \p1 = (C2) in node[inner sep=0] (c2) at (\x1,0) {};
\path (C2) arc (87: 130: 0.5 and 0.707*3) node[inner sep=0] (p1) {};
\path (C2) arc (87: 210: 0.5 and 0.707*3) node[inner sep=0] (p3) {};
\path (C2) arc (447: 30: 0.5 and 0.707*3) node[inner sep=0] (p2) {};
\path (C2) arc (447: 310: 0.5 and 0.707*3) node[inner sep=0] (p4) {};

\draw (-8, -5) -- (8,-5);
\draw (-0.8, -5.5) -- (-1, -5.5) -- (-1, -4.5) -- (-0.8, -4.5);
\draw (0.8, -5.5) -- (1, -5.5) -- (1, -4.5) -- (0.8, -4.5);

\draw (0, -4.5)  -- (0, -6) node[anchor=north] {$\scriptstyle t_j^i$};
\end{tikzpicture}
$$

\item For any $t^{n-1}_j\in I^{n-1}_{j}$ and $t^{n-1}_l \in I^{n-1}_l$, there is an inclusion of the preimages
$$p_{\{n-1,n\}}^{-1} \Big( (t^{n-1}_j, t^n_l)\Big) \subset p_{\{n-1\}}^{-1}(I^{n-1}_{j}),$$
and by condition \ref{cond 3} the map $p_{\{n-1,n\}}$ is submersive there. Therefore $p_{\{n-1,n\}}^{-1} \Big( (t^{n-1}_j, t^n_l)\Big)$ is an $(n-2)$-dimensional manifold, which should be thought of as one of the $(n-2)$-bordisms which are connected by the composition of $n$-bordisms $M$. Moreover, again since $p_{\{n-1,n\}}$ is submersive everywhere in $p_{\{n-1\}}^{-1} (I^{n-1}_j)$, a variant of Ehresmann's fibration theorem shows that the preimage $p_{\{n-1\}}^{-1} (t^{n-1}_j)$ is a trivial fibration and thus a trivial $(n-1)$-bordism between the $(n-2)$-bordisms it connects.

$$
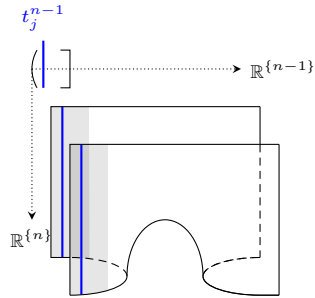
\begin{tikzpicture}[scale=2.5]
\begin{scope}[yscale=1,xscale=1]

\draw[densely dotted, ->] (1.8, 3.2) -- (2.9, 3.2) node[anchor=west] {$\scriptstyle \R^{\{n-1\}}$};
\draw[densely dotted, ->] (1.8, 3.2) -- (1.8, 2.4) node[anchor=north] {$\scriptstyle \R^{\{n\}}$};

\draw (1.8, 3.2) arc (0:30:-0.2);
\draw (1.8, 3.2) arc (0:-30:-0.2);

\draw (1.95, 3.1) -- (2, 3.1) -- (2, 3.3) -- (1.95, 3.3);
\draw[blue, thick] (1.86, 3.1) -- (1.86, 3.35) node[anchor=south] {$\scriptstyle t_j^{n-1}$};

\fill[very nearly transparent] (2,2) arc (-90: -48.2: 0.3cm and 0.1cm) --(2.2,2.8) -- (2, 2.8) -- (2,2);
\fill[very nearly transparent] (1.9, 2.2) -- (2, 2.2) arc (90 : 70:0.3cm and 0.1cm) -- (2.1, 3) -- (1.9, 3) -- cycle;

\draw[blue, thick] (1.96, 2.2) -- (1.96, 3);
\draw[blue, thick] (2.06, 2.004) -- (2.06, 2.8);

\draw (2,2) arc (-90: 0: 0.3cm and 0.1cm);
\draw [densely dashed] (2.3, 2.1) arc (0:90: 0.3cm and 0.1cm);
\draw (2, 2.2) -- (1.9, 2.2) -- (1.9, 3) -- (3,3);
\draw (2,2) -- (2, 2.8) -- (3.1, 2.8) 
	 (3.1, 2) -- (3,2);
\draw (3,2) arc (270: 180: 0.3cm and 0.1cm);
\draw [densely dashed]  (3, 2.2) arc (90: 180: 0.3cm and 0.1cm) (3, 2.2);

\draw (2.3, 2.1) arc (180: 0: 0.2cm and 0.3cm);

\draw (2.7, 2.1) arc (180:270: 0.3cm and 0.1cm) --+(0.1,0);
\draw [densely dashed] (2.7, 2.1) arc (180: 90: 0.3cm and 0.1cm);

\draw (3.1, 2.8) -- (3.1, 2);
\draw (3,3)--(3,2.8);
\draw[densely dashed] (3, 2.8) -- (3, 2.2);
\end{scope}
\end{tikzpicture}
$$

\item Similarly, for $(t^k_{j_k}, \ldots, t^n_{j_n})\in I^k_{j_k}\times \cdots \times I^n_{j_n}$, the preimage
$$p_{\{k,\ldots, n\}}^{-1}\Big( (t^k_{j_k}, \ldots, t^n_{j_n}) \Big)$$
is a $(k-1)$-dimensional manifold, which should be thought of as one of the $(k-1)$-bordisms which is connected by the composition of $n$-bordisms $M$.
\item Moreover, the following proposition shows that different choices of ``cutting points'' $t^i_j\in I^i_j$ lead to diffeomorphic bordisms. In  the case when $b_j^i < a_{j+1}^i$ one should thus think of the $n$-bordisms we compose as $\pi^{-1}(\prod_{i=1}^n[b^i_j, a^i_{j+1}])$, and the preimages of the specified intervals as collars of the bordisms along which they are composed. Otherwise, one should think of that $n$-bordism in the composition as being ``degenerate'', i.e.~of being a trivial $n$-bordism.
\end{itemize}
We will come back to this interpretation in Section \ref{hocat} when we compute homotopy categories.
\end{rem}

\begin{prop}
Let $(M, \oul{I})\in (\PPBord_n^V)_{k_1,\ldots, k_n}$. Fix $1\leq i\leq n$ and $0\leq j\leq j'\leq k_i$. Then for any $u^i_j, v^i_j \in I^i_j$ and $u^i_{j'}, v^i_{j'} \in I^i_{j'}$ such that $u^i_j< u^i_{j'}$ and $v^i_j <v^i_{j'}$ there is a diffeomorphism
$$p_{\{i\}}^{-1}([u^i_j, u^i_{j'}])\longrightarrow p_{\{i\}}^{-1}([v^i_j, v^i_{j'}]).$$
\end{prop}

\begin{proof}
Since the map $p_{\{i\}}$ is submersive in $I^i_j$ and $I^i_{j'}$, we can apply the Morse lemma, which we recall in \ref{Morse lemma}, to $p_{\{i\}}$ twice to obtain diffeomorphisms
$$p_{\{i\}}^{-1}([u^i_j, u^i_{j'}])\longrightarrow p_{\{i\}}^{-1}([v^i_j, u^i_{j'}])\longrightarrow p_{\{i\}}^{-1}([v^i_j, v^i_{j'}]).$$
\end{proof}

Applying the proposition successively for $i=1,\ldots, n$ yields
\begin{cor}\label{cor composed bordisms}
Let $(M, \oul{I})\in (\PPBord_n^V)_{k_1,\ldots, k_n}$ and let $B_1, B_2 \subseteq \R^n$ be products of non-empty closed bounded intervals with endpoints lying in the same specified intervals, i.e.~$B_1=\prod_i [u^i_j, u^i_{j'}]$ and $B_2=\prod_i [v^i_j, v^i_{j'}]$, where $0\leq j\leq j'\leq k_i$ and $u^i_j, v^i_j \in I^i_j$ and $u^i_{j'}, v^i_{j'} \in I^i_{j'}$ such that $u^i_j< u^i_{j'}$ and $v^i_j <v^i_{j'}$ for every $1\leq i\leq n$. Then there is a diffeomorphism
$$\pi^{-1}(B_1) \longrightarrow \pi^{-1}(B_2).$$
\end{cor}

\subsection{The spaces \texorpdfstring{$(\PBord_n)_{k_1\ldots, k_n}$}{PBordn}}\label{sec space PBord}

The level sets $(\PPBord_n^V)_{k_1,\ldots, k_n}$ form the underlying sets of 0-simplices of spaces which we construct in this subsection. Ultimately, we want the space to encode the diffeomorphisms of $n$-fold bordisms which are the composition of $k_i$ bordisms in the $i$th direction. More precisely, it should be the disjoint union of classifying spaces thereof. It will only become apparent that the space we define is the desired one in \ref{sec Morse}, in particular Proposition \ref{prop BDiff}.

\subsubsection{The topological space \texorpdfstring{$(\TPBord_n^V)_{k_1,\ldots, k_n}$}{TPBord}}\label{def PBord top}

We endow the set $(\PPBord_n^V)_{k_1,\ldots, k_n}$ with the following topology coming from modifications of the Whitney $C^\infty$-topology on $\mathrm{Emb}(M, V\times (0,1)^n)$.

In \cite{galatius06}, spelled out in more detail in \cite{GRW}, a topology is constructed\footnote{\cite{galatius06, GRW} use the notation $\Psi\big(V\times (0,1)^n\big)= \mathrm{Sub}(V\times(0,1)^n)$.} on the set of closed (not necessarily compact) $n$-dimensional submanifolds $M \subseteq V\times (0,1)^n$, which we identify with the quotient
$$\mathrm{Sub}(V\times(0,1)^n) \overset{\simeq}{\longleftarrow} \bigsqcup_{[M]}\mathrm{Emb}(M, V\times (0,1)^n)/\mathrm{Diff}(M),$$
where the coproduct is taken over diffeomorphism classes of $n$-manifolds. It is given by defining the neighborhood basis at $M$ to be
$$\{N\subset V\times(0,1)^n : N\cap K = j(M) \cap K, j\in W\},$$
where $K\subset V\times(0,1)^n$ is compact and $W\subseteq \mathrm{Emb}(M,V\times(0,1)^n)$ is a neighborhood of the inclusion $M \hookrightarrow V\times(0,1)^n$ in the Whitney $C^\infty$-topology. Thus we obtain a topology on
$$\mathrm{Sub}(V\times(0,1)^n)\times \TInt^n_{k_1,\ldots, k_n},$$
where we view $\TInt^n_{k_1,\ldots, k_n}$ as a (topological) subspace of $\R^{2k}$ as in \ref{sec TInt}.

For an element $\oul{I}\in \TInt^n_{k_1,\ldots, k_n}$, recall from Definition \ref{defn box rescaling} the box rescaling map $\rho(\oul{I}): B(\oul{I}) \to (0,1)^n$. Then we identify an element $(M, \oul{I})\in (\TPBord_n^V)_{k_1,\ldots, k_n}$ whose underlying submanifold is the image of an embedding $\iota: M\hookrightarrow V\times B(\oul{I})$ with the element $\left([\rho(\oul{I})\circ\iota], \rho(\oul{I})\right)$ in the above space. This identification gives an inclusion
$$(\TPBord_n^V)_{k_1,\ldots, k_n} \subseteq \mathrm{Sub}(V\times(0,1)^n)\times \TInt^n_{k_1,\ldots, k_n},$$
which we use to topologize the left-hand side.

\subsubsection{The space \texorpdfstring{$(\PBord_n^V)_{k_1,\ldots, k_n}$}{PBordn}}

 To model the levels of the bordism category as spaces, i.e.~as Kan complexes, we can start with the above version as a topological space and take singular simplices of this topological space. However, {\em smooth maps} from a smooth manifold $X$ to $\mathrm{Sub}(V\times(0,1)^n)$ as defined in \cite[Definition 2.16, Lemma 2.17]{GRW} are easier to handle. By Lemma 2.18 in the same paper, every continuous map from a smooth manifold, in particular from $\eDelta{l}$, to $(\TPBord_n^V)_{k_1,\ldots, k_n}$ can be perturbed to a smooth one, so the homotopy type when considering smooth singular simplices does not change.

We could directly define the space $(\PBord_n^V)_{k_1,\ldots, k_n}$ to be the smooth singular space of $(\TPBord_n^V)_{k_1,\ldots, k_n}$. However, we will first give a very explicit description of it.
\begin{defn}\label{def PBord simplices}
An $l$-simplex of $(\PBord_n^V)_{k_1,\ldots, k_n}$ consists of tuples $(M, \oul{I}(s)= (I^i_0(s) \leq\cdots\leq I^i_{k_i}(s))_{s\in\eDelta{l}}$ such that
\begin{enumerate}
\item $\oul{I}=(I^i_0 \leq\cdots\leq I^i_{k_i})_{1\leq i \leq n} \to \eDelta{l}$ is an $l$-simplex  in $\Int^n_{k_1,\ldots, k_n}$,
\item $M$ is a closed and bounded $(n+l)$-dimensional submanifold of $V\times B(\oul{I}(s))_{s\in\eDelta{l}}\subset V\times \R^n\times \eDelta{l}$ such that\footnote{Recall that we view the total space of $B(\oul{I})\to\eDelta{l}$ as sitting inside $\R^n\times \eDelta{l}$ as $\bigcup_{s\in\eDelta{l}} B(\oul{I}(s)) \times \{s\}$.}
\begin{enumerate}
\item\label{cond simplex pi} the composition $\pi: M \hookrightarrow V\times B(\oul{I}(s))_{s\in\eDelta{l}} \twoheadrightarrow B(\oul{I}(s))_{s\in\eDelta{l}}$ of the inclusion with the projection is proper,
\item\label{cond simplex eDelta} its composition with the projection onto $\eDelta{l}$ is a submersion $M\to \eDelta{l}$ which is trivial outside $|\Delta^l|\subset \eDelta{l}$, and
\end{enumerate}
\item\label{cond simplex 3} for every $S\subseteq\{1,\ldots, n\}$, let $p_S:M\xrightarrow{\pi}B(\oul{I}(s))_{s\in\eDelta{l}} \subset \R^n \times \eDelta{l} \xrightarrow{\pi_S}\R^S \times \eDelta{l}$ be the composition of $\pi$ with the projection $\pi_S$ onto the $S$-coordinates. Then for every $1\leq i\leq n$ and $0\leq j_i\leq k_i$,  at every $x\in p_{\{i\}}^{-1}(\bigcup_{s\in\eDelta{l}} I^i_{j_i}(s) \times \{s\} )$, the map $p_{\{i,\ldots,n\}}$ is submersive.
\end{enumerate}
\end{defn}

From the Definition of smooth map in \cite[Definition 2.16, Lemma 2.17]{GRW} we immediately get:
\begin{lemma}
An $l$-simplex of $(\PBord_n^V)_{k_1,\ldots, k_n}$ is exactly a smooth $l$-simplex of $(\TPBord_n^V)_{k_1,\ldots, k_n}$.
\end{lemma}

\begin{rem}
Note that for $l=0$ we recover Definition \ref{def PBord}. Moreover, for every $s\in\eDelta{l}$ the fiber $M_s$ of $M\to\eDelta{l}$ determines an element in $(\PBord_n^V)_{k_1,\ldots, k_n}$
$$(M_s)=(M_s\subset V\times B(\oul{I}(s)), \oul{I}(s)).$$
We will use the notation $\pi_s: M_s\to B(\oul{I}(s))$ for the composition of the embedding and the projection.
\end{rem}

\begin{rem}\label{rem simplex trivial fiber bundle}
The conditions \eqref{cond simplex pi}, \eqref{cond simplex eDelta}, and \eqref{cond simplex 3} imply that $M\to \eDelta{l}$ is a smooth fiber bundle, and, since $\eDelta{l}$ is contractible, even a trivial fiber bundle. The proof is a more elaborate version of the argument after Definition 2.6 in \cite{GMTW}.
\end{rem}

We now use the simplicial maps of the space $\Int^{n}_{k_1,\ldots, k_n}$ to explain those of $(\PBord_n^V)_{k_1,\ldots, k_n}$.
\begin{defn}
Fix $k\geq 0$ and let $f:[m]\to [l]$ be a morphism in the simplex category $\Delta$, i.e.~a (weakly) order-preserving map. Then let $|f|: \eDelta{m} \to \eDelta{l}$ be the induced map between standard simplices.

Let $f^\Delta$ be the map sending an $l$-simplex in $(\PBord_n^V)_{k_1,\ldots, k_n}$ to the $m$-simplex which consists of
\begin{enumerate}
\item for $1\leq i\leq n$, the $m$-simplex in $\Int_{k_i}$ obtained by applying $f^\Delta$,
$$f^\Delta \Big( (I^i_0(s) \leq\cdots\leq I^i_{k_i}(s) )_{s\in \eDelta{l}} \Big) = \big(I_0(|f|(s)) \leq\ldots\leq I_k(|f|(s)\big)_{s\in \eDelta{m}};$$
\item The $(n+m)$-dimensional submanifold $f^\Delta M \subseteq V\times B(\oul{I}(s))_{s\in\eDelta{m}}$ obtained by the pullback of $M\to \eDelta{l}$ along $|f|$. Note that its fiber at $s\in \eDelta{m}$ is $(f^\Delta M)_s=M_{|f|(s)}$ and
$$f^\Delta M= \bigcup_{s\in\eDelta{m} } M_{|f|(s)}\times \{s\}.$$
\end{enumerate}
\end{defn}
The above assignment is indeed well-defined since the underlying assignment for the underlying intervals is well-defined and since the map $|f|$ is a submersion, the pullback of $M\to \eDelta{l}$ along $|f|$ is also a submersion. Moreover, the assignment is functorial, since pullback commutes contravariantly with composition, and thus $(\PBord_n^V)_{k_1,\ldots, k_n}$ is a simplicial set.

\begin{prop}\label{prop PBord Kan}
$(\PBord_n^V)_{k_1,\ldots, k_n}$ is the smooth singular space of $(\TPBord_n^V)_{k_1,\ldots, k_n}$. In particular, it is a space.
\end{prop}

\begin{proof}
By definition the simplicial maps $f^\Delta$ are induced precisely by the maps~{$|f|:\eDelta{m} \to \eDelta{l}$}.
\end{proof}

\begin{notation}
We denote the {\em spatial face and degeneracy maps} of $(\PBord_n^V)_{k_1,\ldots, k_n}$ by $d^\Delta_j$ and $s^\Delta_j$ for~$0\leq j \leq l$.
\end{notation}

\begin{ex}\label{ex cutoff path}
We now construct an example of a path. It shows that cutting off part of the collar of a bordism yields an element which is connected to the original one by a path.

Let $(M) = (M, \oul{I}=(I^i_0\leq\cdots\leq I^i_{k_i})_{i=1,\ldots, n})\in (\PBord_n^V)_{k_1,\ldots, k_n}$ and fix $1\leq i\leq n$.
We show that cutting off a short enough piece in the $i$th direction at an end of an element of $(\PBord_n^V)_{k_1,\ldots, k_n}$ leads to an element which is connected by a path to the original one. Fix $1\leq i \leq n$ and let $\varepsilon < b^i_0-a^i_0$.

Choose a smooth, increasing, bijective function $[0,1]\to [0,\varepsilon], s\mapsto \varepsilon(s)$ with vanishing derivative at the endpoints.

For $0\leq j\leq k_i$ and $s\in [0,1]\subset \eDelta{1}$ let
$$I^i_j(s) = (a^i_0+\varepsilon(s), b^i_{k_i}) \cap I^i_j,$$
and then $B(\oul{I}(s)) = (a^i_0+\varepsilon(s), b^i_{k_i}) \subset B(\oul{I})$. For $s\leq 0$ and $s\geq 1$ let the family be constant.
Then let $M(\varepsilon)$ be the preimage of the subset $\bigcup_{s\in\eDelta{1}}B(\oul{I}(s))\times\{s\} \subseteq B(\oul{I})\times \eDelta{1}$ of $M\times \eDelta{1} \to B(\oul{I})\times \eDelta{1}$, i.e.~the submanifold
$$
\begin{tikzcd}
M(\varepsilon) \arrow[hook]{r}\arrow{d} & M\times \eDelta{1}\arrow{d} \\ 
\bigcup_{s\in\eDelta{1}}B(\oul{I}(s))\times\{s\} \arrow[hook]{r}&B(\oul{I})\times \eDelta{1}
\end{tikzcd}
$$

Then $(M(\varepsilon), \oul{I}(s) )$ is a 1-simplex in $(\PBord_n^V)_{k_1,\ldots, k_n}$ with fibers $M(\varepsilon)_s = p_{\{i\}}^{-1}\big((a^i_0+\varepsilon(s), b^i_k)\big)$.
\end{ex}

\begin{rem}\label{rem cutoff}
In the above example we constructed a path from an element in $(\PBord_n^V)_{k_1,\ldots, k_n}$ to its {\em cutoff}, where we cut off the preimage of $p_i^{-1}((a^i_0,\varepsilon])$ for suitably small $\varepsilon$. Note that the same argument holds for cutting off the preimage of $p_i^{-1}([b^i_{k_i}-\delta, b^i_{k_i}))$ for suitably small $\delta$. Moreover, we can iterate the process and cut off $\varepsilon_i, \delta_i$ strips in all $i$ directions. Choosing $\varepsilon_i = \frac{b^i_0-a^i_0}{2}, \delta_i=\frac{b^i_{k_i}-a^i_{k_i}}{2}$ yields a path to its {\em cutoff} with underlying submanifold
$$cut(M) = \pi^{-1}\Big(\prod_{i=1}^n (\frac{a^i_0+b^i_0}{2}, \frac{a^i_{k_i}+b^i_{k_i}}{2})\Big).$$
\end{rem}

\subsection{The \texorpdfstring{$n$}{n}-fold simplicial space \texorpdfstring{$(\PBord_n)_{\bullet,\cdots,\bullet}$}{PBordn}}\label{sec simp space PBord}

We make the collection of spaces $(\PBord_n^V)_{\bullet,\ldots,\bullet}$ into an $n$-fold simplicial space by lifting the simplicial structure of $\Int^{\times n}_\bullets$.
We first need to extend the assignment
$$([k_1], \ldots , [k_n]) \longmapsto (\PBord_n^V)_{k_1,\ldots,k_n}$$
to a functor from $(\Delta^{op})^n$.

\begin{defn}
For every $1\leq i\leq n$, let $g_i:[m_i]\to [k_i]$ be a morphism in $\Delta$, and denote by $g=( g_i )_i$ their product in $\Delta^n$. Then
$$ (\PBord_n^V)_{k_1,\ldots ,k_n}  \overset{g^*}{\longrightarrow} (\PBord_n^V)_{m_1,\ldots ,m_n}.
$$
applies $g_i^*$ to the $i$th tuple of intervals and perhaps cuts the manifold. Explicitly, on $l$-simplices, $g^*$ sends an element
$$(M\subset V\times B(\oul{I}(s))_{s\in\eDelta{l}}, \oul{I}(s)= (I^i_0(s)\leq\dots\leq I^i_{k_i}(s))_{i=1}^n)$$
to
$$
\big(g^*M=\pi^{-1}\big(B(g^*\oul{I}(s))_{s\in\eDelta{l}})\big)\subset V \times B(\oul{I}(s))_{s\in\eDelta{l}},
g^*(\oul{I})(s) = (I^i_{g(0)}(s)\leq\dots\leq I^i_{g(m_i)}(s))_{i=1}^n) \big),
$$
where $\pi: M\subset V\times B(\oul{I}(s))_{s\in\eDelta{l}} \twoheadrightarrow B(\oul{I}(s))_{s\in\eDelta{l}}$. Note that $(g^*M)_s = g^*M_s$.

\end{defn}
Note that as the manifold $g^*M$ is the preimage of the new box, we just cut off the part of the manifold outside the new box. This is functorial, as it is functorial on the intervals, and, if $\tilde g_i:[k_i]\to [\tilde k_i]$ and $\tilde g=( \tilde g_i )_i$, the following diagram commutes by construction:
$$\begin{tikzcd}
M \arrow{d}{\pi} \arrow[phantom]{r}{\supseteq} &  g^* M\arrow{d}{\pi} \arrow[phantom]{r}{\supseteq} &  \tilde g^*g^* M\arrow{d}{\pi}\\
B(\oul{I}(s))_{s\in\eDelta{l}} \arrow[phantom]{r}{\supseteq}& B(g^*(\oul{I}(s))_{s\in\eDelta{l}} \arrow[phantom]{r}{\supseteq}& B(\tilde g^*g^*(\oul{I}(s))_{s\in\eDelta{l}}
\end{tikzcd}$$

\begin{notation}
We denote the {\em (simplicial) face and degeneracy maps} by $d^i_j:(\PBord_n^V)_{k_1,\ldots,k_n}\to (\PBord_n^V)_{k_1,\ldots, k_i-1,\ldots, k_n}$ and $s^i_j:(\PBord_n^V)_{k_1,\ldots,k_n}\to (\PBord_n^V)_{k_1,\ldots, k_i+1,\ldots, k_n}$ for $0\leq j \leq k_i$.
\end{notation}

\begin{notation}\label{rem composed bordisms 3}
Recall from remark \ref{rem composed bordisms} that for $k_1,\ldots, k_n\geq0$, one should think a 0-simplex in $(\PBord_n^V)_{k_1,\ldots,k_n}$ as a collection of $k_1\cdots k_n$ composed bordisms with $k_i$ composed bordisms with collars in the $i$th direction. These composed collared bordisms are the images under the maps
$$D(j_1,\ldots, j_k): (\PBord_n^V)_{k_1,\ldots,k_n} \longrightarrow (\PBord_n^V)_{1,\ldots,1}$$
for $(1\leq j_i \leq k_i)_{1\leq i\leq n}$ arising as compositions of inert face maps, i.e.~$D(j_1,\ldots, j_k)$ is the map determined by the maps
$$d(j_i): [1]\to [k_i], \quad (0<1)\mapsto (j_i-1<j_i)$$
in the category $\Delta$. This should be thought of as sending an element to the $(j_1,\ldots, j_k)$-th collared bordism in the composition. Moreover, we will later use the notation
$$D^i(j_i): (\PBord_n^V)_{k_1,\ldots,k_n} \longrightarrow (\PBord_n^V)_{k_1,\ldots,1,\ldots, k_n}$$
for the maps induced by just $d(j_i)$. By abuse of notation, we will denote the submanifold $d(j_i)^*M$ by $D^i(j_i)(M)$.
\end{notation}

\begin{prop}
The spatial and simplicial structures of $(\PBord_n^V)_\bullets$ are compatible, i.e.~for $f:[l]\to[p]$, $g_i:[m_i]\to [k_i]$ for $1\leq i \leq n$, the induced maps
$$f^\Delta\mbox{ and } g^*$$
commute. We thus obtain an $n$-fold simplicial space $(\PBord_n^V)_{\bullet,\cdots,\bullet}$.
\end{prop}

\begin{proof}
Since $\Int^n$ is a simplicial space, it is enough to show that the maps commute on the manifold part, i.e.~
$$g^*f^\Delta M= f^\Delta g^* M.$$
This follows from the commuting of the following diagram, in which all sides arise from taking preimages. The preimages are taken over $B(g^*\oul{I}(s))_{s\in\eDelta{m}}\subset B(\oul{I}(s))_{s\in\eDelta{m}}$ and $|f|:\eDelta{m}\to \eDelta{l}$, respectively, which affect different components of $V\times \bigcup_{s\in\eDelta{m}} (B(\oul{I}(s)) \times\{s\} ) \subset V\times \R^n \times \eDelta{m}$, so they commute.
$$\begin{tikzcd}[row sep=scriptsize, column sep=scriptsize]
& V\times B(\oul{I}(s))_{s\in\eDelta{m}} \arrow{dd}[near start]{id\times|f|} \arrow[hookleftarrow]{rr} && V\times B(g^*\oul{I}(s))_{s\in\eDelta{m}} \arrow{dd}[near start]{id\times|f|}\\
f^\Delta M \arrow[hook, dashed, green!70!black]{ur} \arrow[dashed, green!70!black]{dd} & & g^*f^\Delta M = f^\Delta g^*M \arrow[densely dashdotdotted, hook, red]{ur} \arrow[densely dashdotdotted, crossing over, red]{ll} &\\
& V\times B(\oul{I}(s))_{s\in\eDelta{l}} \arrow[hookleftarrow]{rr}&& V\times B(g^*\oul{I}(s))_{s\in\eDelta{l}} \\
M \arrow[hook]{ur} && g^*M \arrow[hook, dashdotted, blue]{ur} \arrow[dashdotted, blue]{ll} \arrow[leftarrow, densely dashdotdotted, crossing over, red]{uu} &
\end{tikzcd}$$
\end{proof}

\subsection{The complete \texorpdfstring{$n$}{n}-fold Segal space \texorpdfstring{$\Bord_n$}{Bordn}}

We will now prove that $\PBord_n^V$ leads to an $(\infty,n)$-category, i.e.~a complete $n$-fold Segal space of bordisms. 
\begin{prop}\label{prop PBord Ssp}
$(\PBord_n^V)_{\bullet,\ldots,\bullet}$ is an $n$-fold Segal space.
\end{prop}

\begin{proof} We need to prove that the Segal condition is satisfied and globularity.
\paragraph{The Segal condition is satisfied.} 
Fix fixed $k_1,\ldots, k_n\geq 0$. We need to show that for every $1\leq i\leq n$, and $k_i=m+l$, the Segal map
$$\gamma_{m,l}: (\PBord_n^V)_{k_1,\ldots, k_i,\ldots ,k_{n}}\overset{}{\longrightarrow} (\PBord_n^V)_{k_1,\ldots, m,\ldots, k_n}\overset{h}{\underset{(\PBord_n^V)_{k_1,\ldots, 0,\ldots, k_n}}{\times}} (\PBord_n^V)_{k_1,\ldots, l,\ldots, k_n}$$
is a weak equivalence. From now on we will often omit writing out the indices for $\alpha\neq i$ for clarity.

Since every level set $(\PBord_n^V)_{k_1,\ldots, k_n}$ is a Kan complex by proposition \ref{prop PBord Kan}, i.e.~fibrant, the homotopy fiber product on the right hand side can be chosen to be the space of triples consisting of two points and a path between their target and source, respectively.

Note that an element in this space is given by a triple consisting of
$$\begin{array}{l}
(M, \oul{I}) = (\iota : M \subset V\times B(\oul{I}), \oul{I}=\left(I^i_0\leq \cdots\leq I^i_m, I^j_0\leq \cdots\leq I^j_{k_j}\right)_{1\leq j\leq n, j\neq i}),\\
(N, \oul{J}) = (\kappa : N \subset V\times B(\oul{J}), \oul{J}=\left(J^i_0\leq \cdots\leq J^i_l, J^j_0\leq \cdots\leq J^j_{k_j}\right)_{1\leq j\leq n, j\neq i}),
\end{array}
$$
together with a path $h$ from the target $D^i(m)(M, \oul{I}) =\left(D^i(m)(M), I^i_m, (I^j_0\leq \cdots\leq I^j_{k_j})_{1\leq j\leq n, j\neq i}\right)$ of $(M,\oul{I})$ in the $i$th direction to the source $D^i(1)(N,\oul{J}) = \left(D^i(1)(N), J^i_0, (J^j_0\leq \cdots\leq J^j_{k_j})_{1\leq j\leq n, j\neq i}\right)$ of $(N,\oul{J})$ in the $i$th direction (using Notation \ref{rem composed bordisms 3}).

The Segal map $\gamma_{m,l}$ factors as a composition
\begin{equation}\label{eqn Segal factors}
\begin{tikzcd}
(\PBord_n^V)_{k_i} \arrow{r}{\gamma_{m,l}} \arrow{d} & (\PBord_n^V)_{m}\overset{h}{\underset{(\PBord_n^V)_{0}}{\times}} (\PBord_n^V)_{l} \\
(\PBord_n^V)^{m,l} \arrow[hook]{r} & \mathrm{P}^{m,l}_{k_i}, \arrow[hook]{u}
\end{tikzcd}
\end{equation}
as follows: Informally, the lower right hand corner is the subspace of triples for which, for the directions besides the $i$th, the tuples of intervals agree and the path of intervals is constant. The lower left hand corner is the subspace thereof, for which in addition in the $i$th direction $I^i_m = J^i_0$, and along the path this interval stays constant.
We will define these spaces below. Our strategy to prove that $\gamma_{m,l}$ is a weak equivalence is to show that all three maps are weak equivalences. Here the left vertical map is the main step of the proof -- this is where ``composing'' the bordisms happens, as we will see below. That the bottom and right vertical map are weak equivalences follows from a rescaling procedure.
Let us first define the two spaces in question.

For the lower right hand corner, for $1\leq j\leq n$ and $j\neq i$, consider the $j$th forgetful map
$$\PBord_n^V \longrightarrow \Int, \quad (M,\oul{I})\longmapsto \ul{I}^j.$$
The canonical maps from the pullback to the homotopy pullback $\Int_\bullet \cong \Int_\bullet \overunder[\Int_\bullet]{}{\times} \Int_\bullet \to \Int_\bullet \overunder[\Int_\bullet]{h}{\times} \Int_\bullet$ (which is a weak equivalence since a deformation retract is straightforward to write down and rescales the second tuple of intervals) for varying $j$ induce a (strict) pullback square
$$
\begin{tikzcd}
(\PBord_n^V)_{\bullets,m,\bullets}\overset{h}{\underset{(\PBord_n^V)_{\bullets,0,\bullets}}{\times}} (\PBord_n^V)_{\bullets,l,\bullets}  \arrow{r} & \Int_\bullets^{\times (n-1)} \overunder[\Int_\bullets^{\times (n-1)}]{h}{\times} \Int_\bullets^{\times (n-1)} \\
P^{m,l}_\bullets \arrow{r} \arrow{u} & \Int_\bullets^{\times (n-1)}. \arrow{u}{\simeq}
\end{tikzcd}
$$
The strict pullback of this diagram consists of exactly those pairs whose $j$th tuples of intervals agree for every $j\neq i$, and is constant along the path (but the embedded manifold can still vary).\footnote{Note that since the right vertical map is a weak equivalence, if the diagram were also a homotopy pullback diagram, we would immediately see that the left vertical map is a weak equivalence as well. However, neither map in the diagram is a fibration (or not even a ``sharp map'' \`a la Rezk \cite{RezkSharp}), so we need to find a different strategy.}

For the lower left hand corner, consider the canonical map $\Int_m\overunder[\Int_0]{}{\times} \Int_l \to \Int_m\overunder[\Int_0]{h}{\times} \Int_l$ (which is a weak equivalence since both sides are contractible). Now form the (strict) pullback
$$
\begin{tikzcd}
P^{m,l}_{\bullets}  \arrow{r} & \Int_m\overunder[\Int_0]{h}{\times} \Int_l \\
(\PBord_n^V)_\bullets^{m,l}  \arrow{r} \arrow{u} & \Int_m\overunder[\Int_0]{}{\times} \Int_l. \arrow{u}{\simeq}
\end{tikzcd}
$$
It consists of exactly those pairs whose $j$th tuples of intervals agree for every $j\neq i$ and is constant along the path (but the embedded manifold can still vary), and, in addition, in the $i$th direction, the last interval of the first element is the first interval of the second element.\footnote{Again, the right vertical map is a weak equivalence, and it would be more convenient to take the homotopy pullback. However, the same problem appears as in the previous situation.}

{\em The left vertical map in \eqref{eqn Segal factors} is a weak equivalence:} We first fix once and for all a ``smoothed diagonal'' $D\subset [0,1]^2$: it is the graph of a map $\varsigma:[0,1]\to[0,1]$, which has vanishing derivative in $[0,\frac13]$ and $[\frac23, 1]$ (we could also chose fixed shorter intervals) and is bijective with smooth inverse in $[\frac13, \frac23]$, for example
$$
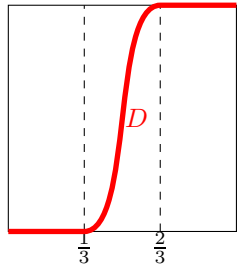
\begin{tikzpicture}[scale=3, inner sep = 0]
\draw (0,0) node (A) {} rectangle (1,1) node (D) {};
\draw[dashed] (0.333, 0) node (B) {} node[anchor=north] {$\frac13$} -- (0.333, 1)
(0.666, 0) node[anchor=north] {$\frac23$} -- (0.666, 1) node (C) {};
\draw[line width=2pt,color=red] (0,0) -- (0.333,0) .. controls (0.55,0) and (0.45,1) .. node[anchor=west] {$D$} (0.666,1) -- (1,1);
\end{tikzpicture}
$$

We will use this to define a deformation retract of $\gamma^{m,l}$ which we suggestively call $glue$. The homotopy exhibiting the deformation retract will use the following two modified functions for $\tau\in[0,1]$. Let
$$\varsigma^s_\tau = \tau \cdot \varsigma \quad \text{and} \quad \varsigma^t_\tau =  1 + \tau \cdot (\varsigma -1).$$
Then for $\tau=1$ we have that $\varsigma = \varsigma^s_1 = \varsigma^t_1$, and for $\tau =0$ we have $\varsigma^s_0 = 0$ and $\varsigma^t_0 = 1$. Moreover, for every $\tau$, both $\varsigma^s_\tau$ and $\varsigma^t_\tau$ are smooth and bijective onto its image. These give ``flatter'' diagonals $D_{s,\tau}$ and $D_{t, \tau}$.

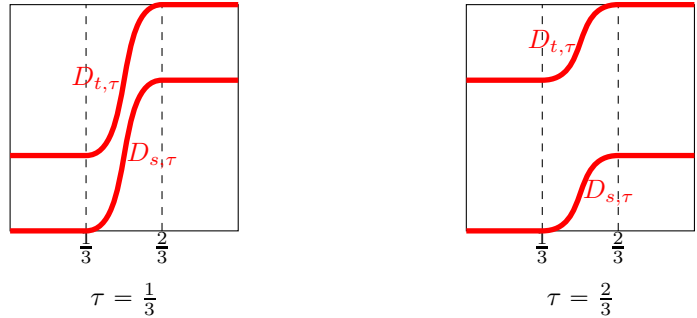
$$\begin{tikzpicture}[scale=3, inner sep = 0]
\draw (0,0) node (A) {} rectangle (1,1) node (D) {};
\draw[dashed] (0.333, 0) node (B) {} node[anchor=north] {$\frac13$} -- (0.333, 1)
(0.666, 0) node[anchor=north] {$\frac23$} -- (0.666, 1) node (C) {};
\draw[line width=2pt,color=red] (0,0) -- (0.333,0) .. controls (0.55,0) and (0.45,0.666) .. node[anchor=west] {$D_{s,\tau}$} (0.666,0.666) -- (1,0.666);
\draw[line width=2pt,color=red] (0,0.333) -- (0.333,0.333) .. controls (0.55,0.333) and (0.45,1) .. node[anchor=east] {$D_{t, \tau}$} (0.666,1) -- (1,1);
\draw (0.5, -0.3) node {$\tau = \frac13$};

\begin{scope}[xshift = 2cm]
\draw (0,0) node (A) {} rectangle (1,1) node (D) {};
\draw[dashed] (0.333, 0) node (B) {} node[anchor=north] {$\frac13$} -- (0.333, 1)
(0.666, 0) node[anchor=north] {$\frac23$} -- (0.666, 1) node (C) {};
\draw[line width=2pt,color=red] (0,0) -- (0.333,0) .. controls (0.55,0) and (0.45,0.333) .. node[anchor=west] {$D_{s,\tau}$} (0.666,0.333) -- (1,0.333);
\draw[line width=2pt,color=red] (0,0.666) -- (0.333,0.666) .. controls (0.55,0.666) and (0.45,1) .. node[anchor=east] {$D_{t, \tau}$} (0.666,1) -- (1,1);
\draw (0.5, -0.3) node {$\tau = \frac23$};
\end{scope}
\end{tikzpicture}
$$

Recall from above that an element in $(\PBord_n^V)_\bullets^{m,l}$ is given by a pair $(M, \oul{I})$ and $(N,\oul{J})$ and a path $h$ from the target of the former to the source of the latter, along which the interval is constant.
We will use this path $h$ to glue the embedded manifolds $M$ and $N$. A similar argument works for $l$-simplices in $(\PBord_n^V)_\bullets^{m,l}$.

The 1-simplex $h$ by definition is a submanifold $P$ of\footnote{Actually, of $V\times(c,b)\times B((I^j_0\leq \cdots\leq I^j_{k_j})_{1\leq j\leq n, j\neq i})\times \eDelta{1}.$} $V\times(c,b)\times \eDelta{1}$ such that the composition with the projection $\pi_{\{i\}}: P\to (c,b)\times \eDelta{1}$ is a submersion. We rescale the fixed smoothed diagonal $D$ linearly to obtain a smooth diagonal $D^{c,b}$ in $(c,b)\times \eDelta{1}$.

Consider the preimage $P_{diag}$ of $\pi_{\{i\}}$ of $D^{c,b}$.
Since the projection $\pi_{\{i\}}: P \to (c,b)\times \eDelta{1}$ is submersive, a Morse lemma style argument shows that this preimage $P_{diag}$ is diffeomorphic to both $D(m)(M)$ and $D(1)(N)$. Thus we glue the manifolds $M$ and $N$ over $P_{diag}$ to obtain $M\cup_{P_{diag}} N$. We realize it as a submanifold of $V\times \R \times (a,d)$ by using
\begin{itemize}
\item $M \cong M\times \{0\} \subset V\times \{0\} \times (a,b) \subset V\times \R \times (a,d)$,
\item $N \cong N\times \{1\} \subset V\times \{1\} \times (c,d) \subset V\times \R \times (a,d)$,
\end{itemize}
and, using the coordinate in $\eDelta{1}\cong \R$,
\begin{itemize}
\item $P_{diag} \subset V\times \R \times (c,b) \subset V\times \R \times (a,d)$.
\end{itemize}
However, note that the extra copy of $\R$ introduced above is not necessary: let
$$\bar{D} = (\{0\} \times (a,c]) \cup D^{c,b} \cup (\{1\} \times [b,d) )\subset \R\times (a,d).$$
Then the projection onto the second coordinate induces a diffeomorphism $\bar{D}\cong (a,d)$. Thus, composing the embedding of the submanifold into $V\times \R \times (a,d)$ with the projection onto $V\times (a,d)$ still is an embedding:
$$M\cup_{P_{diag}} N  \hookrightarrow  V \times (a,d).$$

The same construction works for $l$-simplices: the same argument goes through with $(M, \oul{I})$ and $(N, \oul{J})$ now being $l$-simplices, and thus submanifolds of $V\times (a,b)\times \eDelta{l}$ and $V\times (c,d)\times \eDelta{l}$, respectively, and $P$ a submanifold of $V\times (c,b)\times \eDelta{l+1}$.
Moreover, since the shape $D$ was chosen once and for all, this construction commutes with the spatial structure and indeed gives a map of spaces
$$glue:  (\PBord_n^V)_\bullets^{m,l} \longrightarrow (\PBord_n^V)_{\bullets, k_i, \bullets}.$$
We claim that this is a deformation retract of $\gamma^{m,l}$: Indeed, $glue \circ \gamma^{m,l}$ is the identity, since the path between the source and target in the image of $\gamma^{m,l}$ is constant. As for the other composition $\gamma^{m,l}\circ glue$, this sends a pair of elements (or $l$-simplices) $(M, \oul{I})$ and $(N,\oul{J})$ together with a path $h$ from the target to the source to a pair $(\tilde M, \oul{I})$ and $(\tilde N,\oul{J})$ which is {\em not} the original one (In fact, the latter pair has a constant path $\tilde h$). However, there is a homotopy from $\gamma^{m,l}\circ glue$ to the identity as follows: for $\tau\in[0,1]$, send $(M, \oul{I}), (N,\oul{J}), h$ to the the following pair: modify the above construction by using $D_{s, \tau}$ and $D_{t, \tau}$ instead to obtain $P_{diag}^{s,\tau}$ and $P_{diag}^{t,\tau}$. Now one can glue $M$ with $P_{diag}^{s,\tau}$ and $N$ with $P_{diag}^{t,\tau}$ and embed each as above to obtain $(M_\tau, \oul{I})$ and $(N_\tau, \oul{J})$. A path $h_\tau$ between their target and source is given by the restriction of $P$ to (i.e.~the preimage of) the part between $D_{s, \tau}$ and $D_{t, \tau}$. For $\tau = 0$ this is the identity map, and for $\tau = 1$, this is exactly $\gamma^{m,l}\circ glue$.

{\em ``Rescaling'' -- the bottom and right vertical maps in \eqref{eqn Segal factors} are weak equivalences:} Both maps are part of a deformation retraction. Let us describe the right vertical map first.

The idea of ``rescaling'' is illustrated in the following picture for $n=2$, $i=1$, $l=m=1$, and $k_2=2$. Note that we just depict the cutting lines, not the intervals around them. The rescaling is performed on the right hand piece.
$$
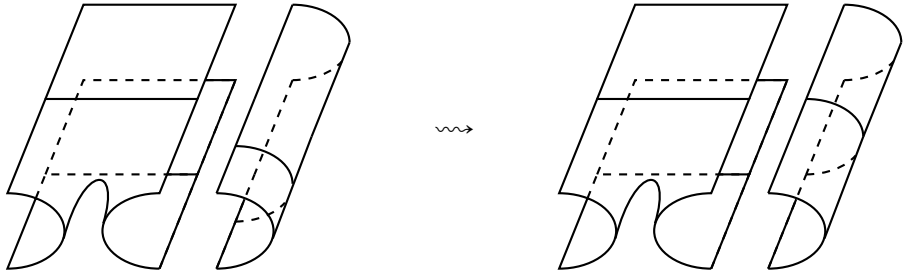
\begin{tikzpicture}[scale=0.5, thick, baseline=(current bounding box.center)]
\draw[] (2,-1) -- (3,1.5) -- (-1, 1.5);
\draw[] (-1, 1.5) -- (0, 4) -- (4,4) -- (3,1.5);
\fill[color=white] (2,1) -- (4,6) -- (0,6) -- (-2,1) -- cycle;
\draw[thick] (-1.25, 0.875) -- (-2, -1) arc (-90:90:1.5 and 1) -- (-1, 3.5) -- (3, 3.5) -- (2,1) arc (90:270:1.5 and 1);

\draw[thick] (-0.52,-0.2) .. controls (0, 1.8) and (1, 1.8) .. (0.52,0.2);
\draw (-1, 3.5) -- (0, 6) -- (4,6) -- (3,3.5);

\draw[dashed] (2,-1) -- (3,1.5) -- (-1, 1.5) -- (-1.25, 0.875);
\draw[dashed] (-1, 1.5) -- (0, 4) -- (4,4) -- (3,1.5);

\begin{scope}[xshift=1.5cm]
\draw (2,-1) -- (3,1.5);
\filldraw[color= white] (3.5,0) arc (0:90:1.5 and 1)  -- (3, 3.5) arc (90 : -10: 1.5 and 1) -- cycle;
\draw[thick]  (3, 3.5) -- (2,1);
\draw[dashed] (2,-1) -- (3,1.5);

\draw (2,-1) arc (-90:90:1.5 and 1);
\draw (2.5, 2.25) arc (90 : 0: 1.5 and 1);
\draw [dashed] (2.5,0.25) arc (-90 : -30: 1.5 and 1);
\draw (3.48, -0.2) -- (4.5, 2.5);
\draw (4,6) -- (3,3.5);
\draw[dashed]  (4,4) -- (3,1.5);
\draw (4,6) arc (90:0: 1.5 and 1);
\draw[dashed] (4,4) arc (-90:-30: 1.5 and 1);
\draw[thick]  (4.5, 2.5) -- (5.5, 5);
\end{scope}
\end{tikzpicture}
\hspace{1cm}\rightsquigarrow\hspace{1cm}
\begin{tikzpicture}[scale=0.5, thick, baseline=(current bounding box.center)]
\draw[] (2,-1) -- (3,1.5) -- (-1, 1.5);
\draw[] (-1, 1.5) -- (0, 4) -- (4,4) -- (3,1.5);
\fill[color=white] (2,1) -- (4,6) -- (0,6) -- (-2,1) -- cycle;
\draw[thick] (-1.25, 0.875) -- (-2, -1) arc (-90:90:1.5 and 1) -- (-1, 3.5) -- (3, 3.5) -- (2,1) arc (90:270:1.5 and 1);

\draw[thick] (-0.52,-0.2) .. controls (0, 1.8) and (1, 1.8) .. (0.52,0.2);
\draw (-1, 3.5) -- (0, 6) -- (4,6) -- (3,3.5);

\draw[dashed] (2,-1) -- (3,1.5) -- (-1, 1.5) -- (-1.25, 0.875);
\draw[dashed] (-1, 1.5) -- (0, 4) -- (4,4) -- (3,1.5);

\begin{scope}[xshift=1.5cm]
\draw (2,-1) -- (3,1.5);
\filldraw[color= white] (3.5,0) arc (0:90:1.5 and 1)  -- (3, 3.5) arc (90 : -10: 1.5 and 1) -- cycle;
\draw[thick]  (3, 3.5) -- (2,1);
\draw[dashed] (2,-1) -- (3,1.5);

\draw (2,-1) arc (-90:90:1.5 and 1);
\draw (3, 3.5) arc (90 : 0: 1.5 and 1);
\draw [dashed] (3,1.5) arc (-90 : -30: 1.5 and 1);
\draw (3.48, -0.2) -- (4.5, 2.5);
\draw (4,6) -- (3,3.5);
\draw[dashed]  (4,4) -- (3,1.5);
\draw (4,6) arc (90:0: 1.5 and 1);
\draw[dashed] (4,4) arc (-90:-30: 1.5 and 1);
\draw[thick]  (4.5, 2.5) -- (5.5, 5);
\end{scope}
\end{tikzpicture}
$$
The deformation retract is given as follows: we observed above that the canonical map
$$\Int_\bullet \cong \Int_\bullet \overunder[\Int_\bullet]{}{\times} \Int_\bullet \longrightarrow \Int_\bullet \overunder[\Int_\bullet]{h}{\times} \Int_\bullet$$
level-wise has a deformation retraction. We will lift this to the desired deformation retraction. 

An element (or $l$-simplex) in the right hand side is given by a triple $(\ul{I},\ul{J}, h)$, where $h$ is a 1-simplex (or $(l+1)$-simplex) from $\ul{I}$ to $\ul{J}$, which we denote by $\ul{I}\to \eDelta{1}$. The later determines a family of diffeomorphisms $B(\ul{J}) \to B(\ul{I}(s))$ and we send 
a triple $\left( (M,\oul{I}), (N,\oul{J}), h\right)$ to a triple $\left( (M,\oul{I}), (N_s,\oul{J}_s), h_s\right)$, where $(N_s,\oul{J}_s)$ is given by the composition
$$N\subset V\times B(\oul{J}) \to V\times B(\oul{I}(s)).$$
We need the family of diffeomorphisms to have the following property: if for every $s\in [s,1]$, the cardinality $|I_j(s)\cap I_{j+1}(s) |$ is 0 or 1, then $b_j(1)\mapsto b_j(s)$ and $a_{j+1}(1)\mapsto a_{j+1}(s)$. 
Such maps are easily defined in a piece-wise linear way. However, we need them to be diffeomorphisms and vary smoothly in the parameter $s$, which requires smoothing. One explicit way of doing this smoothing uses flows along vector fields as in the proof of 1.~in Theorem \ref{thm time Morse lemma}.

As for the horizontal map, the rescaling in the $i$th direction, let $B(\ul{I^i})=(a,b)$ and $a^i_j$ and $b^i_j$ the left and right endpoints of $I^i_j$; and $B(\ul{J^i})=(c,d)$ and $c^i_j$ and $d^i_j$ the left and right endpoints of $J^i_j$. Similarly to above, by rescaling $(N, \oul{J})$, we can assume that we have rescaled the embeddings and intervals such that
$I^i_m = J^i_0 = (a^i_m, b) = (c, d^i_0)$, and along the path this interval stays constant. This assumption implies the the intervals can be ``glued'' (or rather, concatenated) to obtain an element in $\Int_{k_i}$.
$$\begin{tikzpicture}
\draw (-5,0) -- (5,0);

\draw (5, -0.25) node [anchor=north] {\tiny $d$};
\draw (5,0) arc (0:30:0.5);
\draw (5,0) arc (0:-30:0.5);

\draw (-5, -0.25) node [anchor=north] {\tiny $a_0^i=a$};
\draw (-5,0) arc (0:30:-0.5);
\draw (-5,0) arc (0:-30:-0.5);

 (-2,6,0)
\draw (-3.7, 0.25) -- (-3.6, 0.25) -- (-3.6, -0.25) node[anchor=north] {\tiny $b_0^i$} -- (-3.7, -0.25);

\draw (-1.7, -0.25) node[anchor=north] {$\dots$};

\draw (0.4, 0.25) -- (0.3, 0.25) -- (0.3, -0.25) node[anchor=north] {\tiny $a_m^i=c$} -- (0.4, -0.25);

\draw (1.2, 0.25) -- (1.3, 0.25) -- (1.3, -0.25) node[anchor=north] {\tiny $b_m^i=b$} -- (1.2, -0.25);

\draw (2.4, -0.25) node[anchor=north] {$\dots$};

\draw (3.6, 0.25) -- (3.5, 0.25) -- (3.5, -0.25) node[anchor=north] {\tiny $c_l^i$} -- (3.6, -0.25);
\end{tikzpicture}$$

Similarly to above, this can be implemented using a deformation retraction of $\Int_m\overunder[\Int_0]{}{\times} \Int_l \to \Int_m\overunder[\Int_0]{h}{\times} \Int_l$, which is lifted to one of the inclusion.

\paragraph{For every $i$ and every $k_1,\ldots,k_{i-1}$, the $(n-i)$-fold Segal space $(\PBord_n^V)_{k_1,\ldots,k_{i-1},0,\bullet,\cdots,\bullet}$
is essentially constant.}

We show that the degeneracy inclusion map
$$(\PBord_n^V)_{k_1,\ldots,k_{i-1},0,0,\ldots,0}\longhookrightarrow (\PBord_n^V)_{k_1,\ldots,k_{i-1},0,k_{i+1},\ldots,k_{n}}$$
admits a deformation retraction and thus is a weak equivalence.

Consider the assignment sending a pair consisting of $t\in [0,1]$ and an $l$-simplex
$$
\Big(M\subset V\times B(\oul{I}(s)), \big((\ul{I}^\beta(s))_{1\leq\beta< i}, (a^i_0(s),b^i_0(s)), (\ul{I}^\alpha(s))_{i<\alpha\leq n}\big)_{s\in \eDelta{l}}\Big),
$$
in $(\PBord_n^V)_{k_1,\ldots,k_{i-1},0,k_{i+1},\ldots,k_{n}}$ to
$$
\Big(M \subset V\times B(\oul{I}(s)), \big((\ul{I}^\beta(s))_{1\leq\beta< i}, (a^i_0(s),b^i_0(s)),
(\ul{I}^\alpha(s,t))_{i<\alpha\leq n}\big)_{(s,t)\in \eDelta{l}\times [0,1]}\Big),
$$
where for $\alpha>i$ and every $0\leq j\leq k_\alpha$,
\begin{align*}
a^\alpha_j(s,t)&=(1-\varepsilon(t))a^\alpha_j(s) +\varepsilon(t)a^\alpha_0(s),\\
b^\alpha_j(s,t)&=(1-\varepsilon(t))b^\alpha_j(s)+\varepsilon(t)b^\alpha_{k_\alpha}(s).
\end{align*}
for a smooth, increasing, bijective $\varepsilon:[0,1]\to[0,1]$ with vanishing derivative at the endpoints.
This is a homotopy $H: [0,1] \times (\PBord^V_n)_{k_1,\ldots,k_{i-1},0,k_{i+1},\ldots,k_{n}} \to (\PBord^V_n)_{k_1,\ldots,k_{i-1},0,k_{i+1},\ldots,k_{n}}$ exhibiting the deformation retract\footnote{To be precise, we take $t\in \eDelta{1}$ and extend the assignment so that it is constant outside $[0,1]$.}. Note that $B(\oul{I}(s,t))=B(\oul{I}(s))$ for every $t\in[0,1]$. Moreover, for $t=0$ we have that $I^\alpha_j(s,0)=I^\alpha_j(s)$ and the $l$-simplex is sent to itself. For $t=1$ we have $I^\alpha_j(s,1)=(a^\alpha_0(s), b^\alpha_{k_\alpha}(s))$, so the image lies in $(\PBord^V_n)_{k_1,\ldots,k_{i-1},0,0,\ldots,0}$.

It suffices to check that for every $t\in [0,1]$ the image indeed is an $l$-simplex in $(\PBord^V_n)_{k_1,\ldots,k_{i-1},0,k_{i+1},\ldots,k_{n}}$.
Since $(M,\oul{I}(s))\in (\PBord^V_n)_{k_1,\ldots,k_{i-1},0,k_{i+1},\ldots,k_{n}}$, this reduces to checking
\begin{quote}
{\em For every $i< \alpha\leq n$ and $0\leq j\leq k_\alpha$, at every $x\in p_{\{\alpha\}}^{-1}(I^\alpha_j(s,t)_{s\in\eDelta{l}})$, the map $p_{\{\alpha,\ldots,n\}}$ is submersive.}
\end{quote}

Since in the $i$th direction we only have one interval, we have that $p_{\{i\}}^{-1}((a^i_0(s), b^i_0(s))_{s\in\eDelta{l}})=M$, so in particular, $p_{\{i\}}^{-1}((a^i_0(s), b^i_0(s))_{s\in\eDelta{l}})\supset p_{\{\alpha\}}^{-1}(I^\alpha_j(s,t)_{s\in\eDelta{l}})$. Therefore, condition \eqref{cond simplex 3} in \ref{def PBord simplices} on $(M)$ for $i$ implies, that $p_{\{i,\ldots,n\}}$ is a submersion in $p_{\{i\}}^{-1}((a^i_0(s), b^i_0(s))_{s\in\eDelta{l}})=M\supset p_{\{\alpha\}}^{-1}(I^\alpha_j(s,t)_{s\in\eDelta{l}})$, so $p_{\{\alpha,\ldots,n\}}$ is submersive there as well. 
\end{proof}

\begin{rem}
It is much easier to see that the ``strict'' Segal condition also holds, i.e.~that 
$$(\PBord_n^V)_{k_1,\ldots, k_i+k_i',\ldots ,k_{n}}\overset{\sim}{\longrightarrow} (\PBord_n^V)_{k_1,\ldots, k_i,\ldots ,k_n}\underset{(\PBord_n^V)_{k_1,\ldots, 0,\ldots ,k_n}}{\times} (\PBord_n^V)_{k_1,\ldots, k_i',\ldots ,k_n}.$$
An element in the right hand side is a pair of submanifolds $M\subset V \times (a_0, b_{k_i})$ and $N \subset V\times (\tilde{a}_0, \tilde{b}_{k_i'})$ which coincide on the intersection $V \times (a_{k_i}, b_{k_i})$ together with intervals $I_0\leq \cdots \leq I_{k_i}$ and $\tilde{I}_0\leq\cdots\leq \tilde{I}_{k_i'}$ such that $I_{k_i}=\tilde{I}_0$. So we can glue them together to form a submanifold $M\cup N \subset V \times (a_0, \tilde{b}_{k_i'})$, and concatenate the intervals $I_0\leq \cdots \leq I_{k_i}\leq \tilde{I}_1 \leq \cdots \leq \tilde{I}_{k_i'}$. Thus the above strict Segal map even is a homeomorphism.
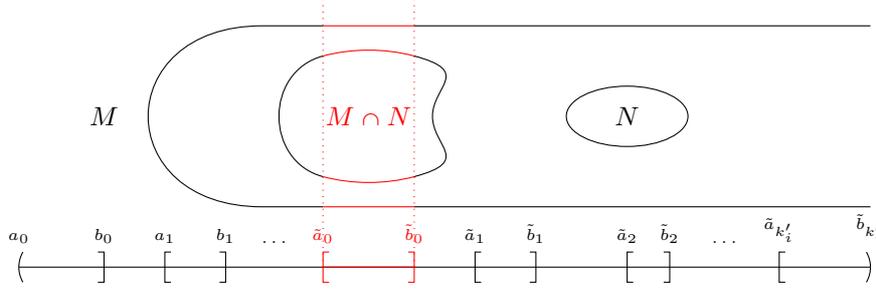
\begin{center}
\begin{tikzpicture}[scale=0.8]
\draw (-7,0) -- (7,0);
\draw (7, 0.25) node [anchor=south] {\tiny $\tilde{b}_{k_i'}$};
\draw (7,0) arc (0:30:0.5);
\draw (7,0) arc (0:-30:0.5);

\draw (-7, 0.25) node [anchor=south] {\tiny $a_0$};
\draw (-7,0) arc (0:30:-0.5);
\draw (-7,0) arc (0:-30:-0.5);

\draw (-5.7, 0.25) -- (-5.6, 0.25) node[anchor=south] {\tiny $b_0$} -- (-5.6, -0.25) -- (-5.7, -0.25);
\draw (-4.5, 0.25) -- (-4.6, 0.25) node[anchor=south] {\tiny $a_1$} -- (-4.6, -0.25) -- (-4.5, -0.25);
\draw (-3.7, 0.25) -- (-3.6, 0.25) node[anchor=south] {\tiny $b_1$} -- (-3.6, -0.25) -- (-3.7, -0.25);

\draw (-2.8, 0.25) node[anchor=south] {\tiny $\ldots$};

\draw[red] (-1.9, 0.25) -- (-2, 0.25) node[anchor=south] {\tiny $\tilde{a}_0$} -- (-2, -0.25) -- (-1.9, -0.25);

\draw[red] (-0.6, 0.25) -- (-0.5, 0.25) node[anchor=south] {\tiny $\tilde{b}_0$} -- (-0.5, -0.25) -- (-0.6, -0.25);

\draw[red] (-2, 0) -- (-0.5,0);

\draw(0.6, 0.25) -- (0.5, 0.25) node[anchor=south] {\tiny $\tilde{a}_1$} -- (0.5, -0.25) -- (0.6, -0.25);

\draw (1.4, 0.25) -- (1.5, 0.25) node[anchor=south] {\tiny $\tilde{b}_1$} -- (1.5, -0.25) -- (1.4, -0.25);
\draw (3.1, 0.25) -- (3, 0.25) node[anchor=south] {\tiny $\tilde{a}_2$} -- (3, -0.25) -- (3.1, -0.25);

\draw(3.6, 0.25) -- (3.7, 0.25) node[anchor=south] {\tiny $\tilde{b}_2$} -- (3.7, -0.25) -- (3.6, -0.25);

\draw (4.6, 0.25) node[anchor=south] {\tiny $\ldots$};

\draw(5.6, 0.25) -- (5.5, 0.25) node[anchor=south] {\tiny $\tilde{a}_{k_i'}$} -- (5.5, -0.25) -- (5.6, -0.25);

\draw (7, 1) -- (-0.5, 1)
	(-2, 1) -- (-3, 1) .. controls (-5.5, 1) and (-5.5, 4) .. (-3, 4) -- (-2, 4)
	(-0.5, 4) -- (7, 4);
\draw[red] (-2, 1) -- (-0.5 , 1)
	(-2, 4) -- (-0.5, 4);

\draw (-5.6, 2.5) node {$M$};
\draw (3, 2.5) node {$N$};
\draw[red] (-1.25, 2.5) node {$M\cap N$};

\draw[red, dotted] (-2, 0) -- (-2, 4.5)
(-0.5, 0) -- (-0.5, 4.5);

\draw
(-0.5, 1.5) ..controls +(15: 1) and +(-90:0.5) .. (-0.2, 2.5) .. controls +(90: 0.5) and +(-15 :1) .. (-0.5, 3.5);
\draw[red] (-0.5, 3.5) .. controls +(165: 0.5) and +(15: 0.5) .. (-2, 3.5)
(-0.5, 1.5) .. controls +(-165: 0.5) and +(-15: 0.5) .. (-2, 1.5);
\draw (-2, 1.5) ..controls +(165: 1) and +(-165: 1) .. (-2, 3.5);

\draw (3, 2.5) ellipse (1 and 0.5);
\end{tikzpicture}
\end{center}
Note that this construction also extends to $l$-simplices: we glue submanifolds of $V \times (a_0(s), b_{k_i}(s))_{s\in\eDelta{l}}$ and $V \times (\tilde{a}_0(s), \tilde{b}_{k_i'}(s))_{s\in \eDelta{l}}$ to form one of $V \times (a_0(s), \tilde{b}_{k_i'}(s))_{s\in\eDelta{l}}$.
\end{rem}

\begin{rem}
It can be checked that the maps $(\PBord_n^V)_{k_1,\ldots, k_i,\ldots ,k_n}\to(\PBord_n^V)_{k_1,\ldots, 0,\ldots ,k_n}$ are not fibrations unless $k_1=\cdots =k_n=0$. If this were true, this together with the previous Remark would have simplified the proof of the previous Proposition. However, we can still conclude that $\PBord_n^V$ fits into a stricter model for $(\infty,n)$-categories; for $n=1$, it is a fibrant object in the model category of internal objects from \cite{HorelModel} and \ref{sec internal categories} (even though it is not strongly Segal). For $n>1$, it is a fibrant object in the model category of internal $n$-uple categories from \cite{CavigliaHorel} and \ref{sec internal n-uple categories}, although it seems not to be a fibrant internal $n$-fold category, since the identities are not strict.
\end{rem}

So far the definition of $\PBord_n^V$ depended on the choice of the vector space $V$. However, in the bordism category we would like to consider all (not necessarily compact) $n$-dimensional manifolds. By Whitney's embedding theorem any such manifold can be embedded into some finite-dimensional vector space $V$, so we need to allow big enough vector spaces. 
\begin{defn}
Fix some countably infinite-dimensional vector space\footnote{Note that the definition does not depend on the choice of the countably infinite-dimensional vector space; any such is the colimit over all finite dimensional vector spaces.}, e.g.~$\R^{\infty}$. Then we define $\PBord_n$ to be the homotopy colimit of $n$-fold Segal spaces\footnote{Note that the identity map from the model category of $n$-fold simplicial spaces to the model category of $n$-fold Segal spaces is a left adjoint (since it is a localization) and therefore preserves homotopy colimits. Thus, the homotopy colimit can be computed in $n$-fold simplicial spaces.}
$$\PBord_n=\varinjlim_{V\subset \R^{\infty}} \PBord_n^V = \operatorname{hocolim}_{V\subset \R^{\infty}} \PBord_n^V.$$
\end{defn}

\begin{rem}
If the vector space $V$ is $r$-dimensional, the $n$-fold Segal space $\PBord_n^V$ is also interesting in its own right. It describes codimension $r$ {\em tangles}, see also \cite[Section 4.4]{Lurie}. For example, if $n=1$ and $r=2$, we obtain a Segal space of 1-dimensional tangles in $\R^3$. We will elaborate on this more in the next section. Moreover, we have not used that $V$ is a vector space. Instead, one could take $V$ to be some fixed manifold (as in e.g.~\cite{GRW}). This requires some extra care which we will not pursue here.
\end{rem}

The last condition necessary to be a good model for the $(\infty,n)$-category of bordisms is completeness, which $\PBord_n$ in general does not satisfy. However, we can pass to its completion $\Bord_n$.
\begin{defn}
The {\em $(\infty, n)$-category of bordisms $\Bord_n$} is the $n$-fold completion $\widehat{\PBord_n}$ of $\PBord_n$, which is a complete $n$-fold Segal space.
\end{defn}

\begin{rem}\label{rem PBord incomplete}
For $n \geq 6$, $\PBord_n$ is not complete, see the full explanation in \cite{Lurie}, 2.2.8. For $n=1$ and $n=2$, by the classification theorems of one- and two-dimensional manifolds, $\PBord_n$ is complete, and therefore $\Bord_n=\PBord_n$.
\end{rem}

\section{Variants of \texorpdfstring{$\Bord_n$}{Bordn} and comparison with Lurie's definition}\label{sec variants}

\subsection{The \texorpdfstring{$(\infty, d)$}{(infty,d)}-category of \texorpdfstring{$n$}{n}-bordisms and tangles for any \texorpdfstring{$d\geq 0$}{d}}\label{sec PBord_n^l}

For $d\geq 0$ we define a $d$-fold Segal space whose top, i.e.~$d$-morphisms are $n$-dimensional bordisms. For $d < n$ this amounts to extending the category of $n$-dimensional bordisms only down to $(n-d)$-dimensional objects.

\begin{defn}\label{def PBord_n^l} Let $V$ be a finite-dimensional $\R$-vector space, which we identify with some $\R^r$. Let $n\geq0, d=n+l\geq 0$. For every $d$-tuple $k_1,\ldots, k_d\geq0$, we let $(\PBord_n^{l,V})_{k_1,\ldots, k_d}$ be the collection of tuples $(M, \oul{I}=(I^i_0\leq\ldots\leq I^i_{k_i})_{1\leq i \leq d})$ satisfying conditions analogous to (1)-(3) in Definition \ref{def PBord}, i.e.
\begin{enumerate}
\item For $1\leq i\leq d$,
$$(I^i_0\leq\cdots\leq I^i_{k_i})\in \Int_{k_i}.$$
\item $M$ is a closed and bounded $n$-dimensional submanifold of $V\times B(\oul{I})$ and the composition $\pi: M \hookrightarrow V\times B(\oul{I})\twoheadrightarrow B(\oul{I})$ is a proper map.
\item For every $S\subseteq\{1,\ldots, d\}$, let $p_S:M\xrightarrow{\pi}B(\oul{I})\xrightarrow{\pi_S}\R^S$ be the composition of $\pi$ with the projection $\pi_S$ onto the $S$-coordinates. Then for every $1\leq i\leq d$ and $0\leq j_i\leq k_i$,  at every $x\in p_{\{i\}}^{-1}(I^i_{j_i})$, the map $p_{\{i,\ldots,d\}}$ is submersive.
\end{enumerate}
We make $(\PBord_n^{l,V})_{k_1,\ldots, k_d}$ into a space similarly to $(\PBord_n^V)_{k_1,\ldots, k_n}$.
\end{defn}

\begin{prop}
$(\PBord_n^{l,V})_{\bullet,\cdots,\bullet}$ is a $d=(n+l)$-fold Segal space.
\end{prop}

\begin{proof}
The proof is completely analogous to the proof of Proposition \ref{prop PBord Ssp}.
\end{proof}

Again we take the homotopy colimit of $n$-fold Segal spaces, i.e.~in the model category $\sSpace_{n,f}^{Se}$ over all finite-dimensional vector spaces in a given infinite-dimensional vector space, say $\R^\infty$:
$$\PBord_n^l=\operatorname{colim}_{V\subset\R^\infty} \PBord_n^{l,V}.$$

\begin{defn}
For $l \leq 0$ let $d=n+l\leq n$ and let $\Bord_n^l$, which we will also denote by $\Bord_n^{(\infty,d)}$, be the $d$-fold completion of $\PBord_n^l$, the {\em $(\infty, d)$-category of $n$-bordisms}.
\end{defn}

\begin{notation}
For $l \leq 0$ let $d=n+l\leq n$ and let $\Bord_n^{(\infty,d),V}$ be the $d$-fold completion of $\PBord_n^{l, V}$.
If $V$ is $r$-dimensional and $l=0$, this is the unframed version of what Lurie calls the {\em $(\infty,n)$-category of $n$-tangles $\mathrm{Tang}^V_{n, n+r}$} in \cite{Lurie}.
\end{notation}

\begin{rem}\label{rem hybrid}
For $l>0$, the underlying submanifold of objects of $\PBord_n^l$, i.e.~elements in $(\PBord_n^l)_{0,\ldots,0}$, are $n$-dimensional manifolds $M$ which have a submersion onto $\R^{n+l}$. This implies that $M=\emptyset$. Thus, the only objects are $(\emptyset, I^1_0,\ldots, I^{n+l}_0)$ and $(\PBord_n^l)_{0,\ldots,0} \cong \Int^n_{0,\ldots,0} \simeq *$. Similarly, $(\PBord_n^l)_{0,k_2,\ldots, k_{n+l}} \cong \Int^n_{0,k_2,\ldots, k_{n+l}} \simeq *$. Thus, $(\PBord_n^l)_{0,\bullets}$ is equivalent to the point viewed as a constant $(n-1)$-fold Segal space. Similarly, $(\PBord_n^l)_{1,\ldots, 1, 0,\bullets}$, with $(l-1)$ 1's, is equivalent to the point viewed as a constant $(n-l)$-fold Segal space. These will appear again later in section \ref{tower Bord}.
\end{rem}

\subsection{Unbounded submanifolds, \texorpdfstring{$(0,1)$}{(0,1)} as a parameter space, and cutting points}

\subsubsection{Unbounded submanifolds}\label{sec unbounded bordisms}

We could have omitted the condition that $M$ be bounded in condition \eqref{cond 1} in Definitions \ref{def PBord} and \ref{def PBord simplices}, requiring it only to be closed. This modification leads to an $n$-fold simplicial space $\PBord_n^{unb}$ together with a level-wise inclusion
$$\PBord_n \longhookrightarrow \PBord_n^{unb}.$$
Recall from \ref{rem cutoff} that for every element in $(\PBord_n)_{k_1,\ldots, k_n}$, we constructed a path to its cutoff. There is a similar cutoff path for every element in $(\PBord_n^{unb})_{k_1,\ldots, k_n}$ to an element whose underlying submanifold
$$cut(M) = \pi^{-1}\Big(\prod_{i=1}^n (\frac{b^i_0}{2}, \frac{a^i_{k_i}}{2})\Big)$$
is bounded in the $V$-direction. Moreover, this construction extends to $l$-simplices. Altogether, it shows that the inclusion is a level-wise equivalence of $n$-fold simplicial spaces.
Finally, since $\PBord_n$ is an $n$-fold Segal space, $\PBord_n^{unb}$ is as well.

\subsubsection{Restricting the boxing \texorpdfstring{to  $(0,1)^n$}{}}\label{sec Bord (0,1)}

Instead of basing $\PBord_n$ and $\PBord_n^l$ on $\Int$, we could instead use $\Int^{(0,1)}$ from Section \ref{sec Int (0,1)}. This approach leads to $(n+l)$-fold Segal spaces $\PBord_n^{l, (0,1)}$ using the box rescaling maps $\rho(\oul{I}): B(\oul{I}) \to (0,1)^{n+l}$ and the functor $\rho$ from Definition \ref{defn box rescaling} and Remark \ref{rem on box rescaling}. It fits exactly into a commuting diagram
$$
\begin{tikzcd}
\PBord_n^l \arrow{r}{\rho} \arrow{d} & \PBord_n^{l, (0,1)} \arrow{d}\\
\Int^{n+l} \arrow{r}{\rho} & (\Int^{(0,1)})^{ \times (n+l)}
\end{tikzcd}
$$
where the vertical maps are the forgetful maps. Any simplex
$$\left(M\subset V\times B(\oul{I}), \oul{I}= (I^i_0\leq\cdots\leq I^i_{k_i})_{1\leq i\leq n+l}\right)$$
is sent to
$$\left(M\subset V\times B(\oul{I}) \xrightarrow{\rho(\oul{I})} V\times (0,1)^{n+l}, \big(\rho(\oul{I})(I^i_0)\leq\cdots\leq \rho(\oul{I})(I^i_{k_i})\big)_{1\leq i\leq n+l}\right).$$
On a fixed level, i.e.~for fixed $k_1,\ldots, k_{n+l}$, there is an inclusion of spaces $\iota:(\PBord_n^{l, (0,1)})_{k_1,\ldots, k_{n+l}} \hookrightarrow (\PBord_n^l)_{k_1,\ldots, k_{n+l}}$ and the above map is a retract of the inclusion. 
The $(n+l)$-fold simplicial structure needs to be modified by rescaling maps to ensure that the boxing stays $(0,1)^{n+l}$: For a morphism $g=\prod g_i$ in $\Delta^{n+l}$, the associated morphism of spaces is $\rho\circ g\circ\iota$. Since $\rho$ only involves rescalings, $\PBord_n^l \overset{\rho}{\to} \PBord_n^{l, (0,1)}$ is a level-wise weak equivalence, so in particular also a DK-equivalence of $n$-fold Segal spaces. We leave it to the reader to fill in the details.

\subsubsection{Cutting points}\label{sec PBord^t} Another variant of an $n$-fold Segal space of bordisms can be obtained by replacing the intervals $I^i_j$ in Definition \ref{def PBord} of $\PBord_n$ by specified ``cutting points'' $t^i_j$, which correspond to where we cut our composition of bordisms. Equivalently, we can say that in this case the intervals are replaced by intervals consisting of just one point, i.e.~$a^i_j=b^i_j=:t^i_j$. The levels of this $n$-fold Segal space $\PBord_n^t$ can be made into spaces as we did for $\PBord_n$, but we now need to add paths between 0-simplices which coincide inside the boxing of $t$'s, i.e.~over $[t^1_0, t^1_{k_1}]\times\cdots\times[t^n_0,t^n_{k_n}]$. However, for $\PBord_n^t$ the Segal condition is more difficult to prove, as in this case we do not specify the collar along which we glue. Since the space of collars is contractible, sending an interval $I$ with endpoints $a$ and $b$ to its midpoint $t=\frac12(a+b)$ induces an equivalence of $n$-fold Segal spaces from $\PBord_n$ to $\PBord_n^t$. We will not elaborate more on this variant and leave details to the interested reader.

\subsection{Comparison with Lurie's definition of bordisms}\label{Lurie's vs}

In \cite{Lurie}, Lurie defined the $n$-fold Segal space of bordisms as follows:

\begin{defn}\label{def PBord^L}
Let $V$ be a finite-dimensional vector space. For every $n$-tuple $k_1,\ldots, k_n\geq0$, let $(\PBord_n^{V,L})_{k_1,\ldots, k_n}$ be the collection of tuples $(M,(t^i_0\leq\ldots\leq t^i_{k_i})_{i=1,\ldots n})$, where
\begin{enumerate}
\item[$1.$] For $1 \leq i \leq n$,
$$t^i_0 \leq \dots \leq t^i_{k_i}$$
is an ordered $(k_i+1)$-tuple of elements in $\R$.
\item[$2.$] $M$ is a closed $n$-dimensional submanifold of $V\times\R^n$ and the composition $\pi: M \hookrightarrow V\times \R^n \twoheadrightarrow \R^n$ is a proper map.
\item[$\tilde 3.$] For every $S\subseteq\{1,\ldots, n\}$ and for every collection $\{j_i\}_{i\in S}$, where $0\leq j_i \leq k_i$, the composition $p_S:M\overset{\pi}{\to}\R^n\to\R^S$ does not have $(t_{j_i})_{i\in S}$ as a critical value.
\item[$\tilde 4.$] For every $x\in M$ such that $p_{\{i\}}(x)\in\{ t^i_0,\ldots, t^i_{k_i}\}$, the map $p_{\{i+1,\ldots,n\}}$ is submersive at $x$.
\end{enumerate}
It is endowed with a topology coming from the Whitney topology similar to what we described in Section \ref{def PBord top}. Similarly to before, let
$$\PBord_n^L=\varinjlim_{V\subset \R^{\infty}} \PBord_n^{V,L}$$
\end{defn}

Comparing this definition with Definition \ref{def PBord} and $\PBord_n^t$ from \ref{sec PBord^t} above, our condition \eqref{cond 3} on $\PBord_n^t$ is replaced by the two strictly weaker conditions $(\tilde 3)$ and $(\tilde 4)$ on $\PBord_n^L$, which are both implied by \eqref{cond 3}:

\begin{lemma}\label{lemma ours implies Lurie}
Let $M$ be a closed $n$-dimensional manifold and $\pi:M\to \R^n$. Moreover, for $1\leq i\leq n$ let $(t^i_0\leq\ldots\leq t^i_{k_i})$ be an ordered $(k_i+1)$-tuple of elements in $\R$. Denote for $S\subseteq\{1,\ldots, n\}$ the composition $M\overset{\pi}{\to} \R^n \twoheadrightarrow \R^S$ by $p_S$. Assume that condition \eqref{cond 3} from Definition \ref{def PBord} holds, i.e.~for every $1\leq i\leq n$ and $0\leq j_i\leq k_i$ for $x\in M$ such that $p_{\{i\}}(x)=t^i_{j_i}$ the map $p_{\{i,\ldots, n\}}$ is submersive at $x$. Then,
\begin{enumerate}
\item [$\tilde 3.$] For every $S\subseteq\{1,\ldots, n\}$ and for every collection $\{j_i\}_{i\in S}$, where $0\leq j_i \leq k_i$, the composition $p_S:M\overset{\pi}{\to}\R^n\to\R^S$ does not have $(t_{j_i})_{i\in S}$ as a critical value.
\item [$\tilde 4.$] For every $x\in M$ such that $p_{\{i\}}(x)\in\{ t^i_0,\ldots, t^i_{k_i}\}$, the map $p_{\{i+1,\ldots,n\}}$ is submersive at $x$.
\end{enumerate} 
\end{lemma}

\begin{proof}
Let $i_0=\min S$. Consider the following diagram
$$
\begin{tikzcd}[column sep = large]
M \arrow{r}{p_{\{i_0,\ldots, n\}}} \arrow[swap]{dr}{p_S} & \R^{\{i_0,\ldots,n\}} \arrow{d}{proj}\\
& \R^S
\end{tikzcd}
$$
For $\tilde 3.$ let $x\in p_S^{-1}\left((t_{j_i})_{i\in S}\right)$. Then $p_{\{i_0\}}(x)=t^{i_0}_{j_{i_0}}$, so by assumption the map $p_{\{i_0,\ldots, n\}}$ is submersive at $x$. Since $proj$ is submersive, $p_S = proj \circ p_{\{i_0,\ldots, n\}}$ is also submersive at $x$.

For $\tilde 4.$ note that if $p_{\{i,\ldots, n\}}$ is submersive at $x$ then $p_{\{i+1,\ldots,n\}}$ is submersive at $x$.
\end{proof}

However, Lurie's $n$-fold simplicial space $\PBord^L_n$ is not an $n$-fold Segal space as we will see in the example below. Thus, our $\PBord_n^t$ is a corrigendum of Lurie's $\PBord^L_n$ from \cite{Lurie}.

\begin{example}\label{ex: Lurie torus false}
\hspace{1em}
\newline
\begin{minipage}{0.1\textwidth}
\begin{tikzpicture}
\draw[style=densely dashed] (0,0) circle (10pt);
\draw[style=densely dashed] (0,0) circle (20pt);
\draw (-5pt, 30pt) node[anchor=east]{\tiny $t^1_0$} -- (-5pt, -30pt);
\end{tikzpicture}
\end{minipage}
\hfill
\begin{minipage}{0.85\textwidth}
Consider the 2-dimensional torus $T$ in $\R\times\R^2$, embedded such that the projection onto $\R^2$ is an annulus, and consider the tuple $(T\subset\R\times\R^2, t^1_0, t^2_0\leq\ldots\leq t^2_{k_2})$, where $t^1_0$ is indicated in the picture of the projection plane $\R^2$ on the left.
Then, because of condition $(\tilde 3)$, $t^2_0\leq\ldots\leq t^2_{k_2}$ can be chosen everywhere such that any $(t^1_0,t^2_j)$ is not a point where the vertical ($t^1_0$-)line intersects the two circles in the picture. Thus, if $t^2_j$ and $t^2_{\tilde \jmath}$ are in two different connected components of this line minus these forbidden points, there is no path connecting this point to an element in the image of the degeneracy map. However, it satisfies the conditions 1., 2., $\tilde 3.$, and $\tilde 4.$ in the definition of $(\PBord^L_2)_{0,k_2}$, so $(\PBord^L_2)_{0,\bullet}$ is not essentially constant.
\end{minipage}
\end{example}

\subsection{The \texorpdfstring{$n$}{n}-uple Segal space \texorpdfstring{$\PBord_n^{uple}$}{of bordisms}}\label{uple Bord}

Consider the following interval version of condition $\tilde 3.$ in Definition \ref{def PBord^L}:
\begin{enumerate}
\item[$\tilde 3.$] For every $S\subseteq\{1,\ldots, n\}$ and for every collection $\{j_i\}_{i\in S}$, where $0\leq j_i \leq k_i$, the composition $p_S:M\overset{\pi}{\to}\R^n\to\R^S$ does not have any critical value in $(I_{j_i})_{i\in S}$.
\end{enumerate}
It ensures that the fibers $p_S^{-1}\left( (t^i_{j_i})_{i\in S} \right)$ for $t^i_{j_i}\in I^i_{j_i}$ are $(n-|S|)$-dimensional smooth manifolds. This can be interpreted similarly to in Remark \ref{rem composed bordisms} as follows: Choosing $t^i_{j_i} < t^i_{j_i'}$, we picture the composed bordism as
$$\pi^{-1}(\prod [t^i_{j_i} , t^i_{j_i'}]),$$
which, by condition $\tilde 3.$ is a ``cubical'' $n$-dimensional bordism whose sources and targets in each direction are themselves ``cubical'' $(n-1)$-dimensional bordisms, and furthermore, its diffeomorphism class does not depend on the choice of cutting points $t^i_{j_i}$, see Section \ref{sec (cubical) bordisms} for more details.
Our condition \eqref{cond 3}, which we saw in Lemma \ref{lemma ours implies Lurie} to imply $\tilde 3.$, ensures in addition the globularity condition, namely, that the source and target of the $n$-dimensional bordism are $(n-1)$-dimensional bordisms which themselves have the same source and target. This is reflected in the essential constancy condition, namely that we have an ``$n$-category'' instead of an ``$n$-uple category''. Thus, relaxing condition \eqref{cond 3} in Definition \ref{def PBord} to $(\tilde 3)$, we obtain:
\begin{defn}
The $n$-uple Segal space $\PBord_n^{uple}$ has levels whose elements are tuples of $(M,\oul{I}=(I^i_0\leq\cdots\leq I^i_{k_i})_{1\leq i\leq n})$, satisfying conditions 1., 2., and $\tilde 3$.
\end{defn}

Using the construction in Section \ref{sec n-fold from n-uple} one can see that $\PBord_n$ indeed is the maximal $n$-fold Segal space underlying the $n$-uple Segal space $\PBord_n^{uple}$.

\begin{prop}
The maximal underlying $n$-fold Segal space of $\PBord_n^{uple}$ is equivalent to $\PBord_n$.
\end{prop}

\begin{proof}
Recall from Section \ref{sec n-fold from n-uple} that the levels of the maximal underlying $n$-fold Segal space of a $n$-uple Segal space $X$ are given by
$$\mathbf{R}(X)_{\vec{k}}=\mathrm{Map}^h(\Theta^{\vec{k}},X).$$
The remaining structure arises from that of $X$.

The level-wise inclusion $\iota: \PBord_n \hookrightarrow \PBord_n^{uple}$ and the weak equivalence from Lemma \ref{lem Theta Delta} give a map
$$\PBord_n \simeq  \mathrm{Map}^h(\Delta^{\vec{k}}, \PBord_n) \simeq  \mathrm{Map}^h(\Theta^{\vec{k}}, \PBord_n) \longrightarrow \mathrm{Map}^h(\Theta^{\vec{k}}, \PBord_n^{uple}) = \mathbf{R}(\PBord_n^{uple}).$$

By the Segal condition, it suffices to show that this map is an equivalence at every $(i)=(\underbrace{1,\ldots, 1}_{i}, 0,\ldots, 0) \in (\Delta^{op})^k$.
To compute the correct mapping spaces, we use the cofibrant replacement of $\Theta^{(i)}$ from Proposition \ref{prop cofibrant Theta}.

Let $C^i$ be the $n$-uple category, or rather its nerve, the n-fold simplicial space which is the image of the Yoneda embedding of $([1],\ldots, [1])\in \Delta^{\times i}$ and thus is cofibrant for the projective model structure. Then
$$\mathrm{Map}^h(C^i,\PBord_n^{uple}) = (\PBord_n^{uple})_{(i)}.$$
Now we proceed by induction. The first index for which there is something to show is $i=2$. Using the explicit cofibrant replacement from Proposition \ref{prop cofibrant Theta}, we compute that
$$\mathbf{R}(\PBord_n^{uple})_{(2)} = \mathrm{Map}^h(\Theta^{(2)},\PBord_n^{uple})$$
is the homotopy pullback of the diagram
$$
\begin{tikzcd}
\mbox{} & (\PBord_n^{uple})_{(2)} \arrow{d} \\
(\PBord_n^{uple})_{0} \times (\PBord_n^{uple})_{0} \arrow{r} & (\PBord_n^{uple})_{(1)}\times (\PBord_n^{uple})_{(1)} .
\end{tikzcd}$$
From this, we see that an element in $\mathbf{R}(\PBord_n^{uple})_{(2)}$ is an element in $(\PBord_n^{uple})_{(2)}$ together with a path from its source and target to elements in $(\PBord_n^{uple})_{0}$. By Theorem \ref{thm time Morse lemma} such a path determines a diffeomorphism between the composed bordisms associated to the start and the end (the source and target, respectively) of the path, which in turn are elements in $(\PBord_n^{uple})_{0}$. Thus they satisfy condition $\tilde 3.$ everywhere, i.e.~the projection $\pi$ is a submersion. This property is preserved by diffeomorphisms, which implies that the source and target of the element in $(\PBord_n^{uple})_{(2)}$ also satisfy this condition. So the original element must satisfy condition $3.$ and thus lies in $(\PBord_n)_{(2)}$.

For $i>2$, again using the explicit cofibrant replacement, there is a homotopy pullback diagram
$$\begin{tikzcd}
\mathbf{R}(\PBord_n^{uple})_{(i)} \arrow{r}\arrow{d} \arrow[dr, phantom, "\ulcorner \mathrm{h}", very near start]& \mathrm{Map}(C^1\times \mathrm{cof}(\Theta^{(i-1)}),\PBord_n^{uple}) \arrow{d} \\
(\PBord_n)^{uple}_{0} \times (\PBord_n)^{uple}_{0} \arrow{r} & (\PBord_n)^{uple}_{(i-1)}\times (\PBord_n)^{uple}_{(i-1)} .
\end{tikzcd}$$
Since there is a surjection $C^i\twoheadrightarrow C^1\times \Theta^{(i-1)}$, elements in the top right corner arise from elements in $\mathrm{Map}^h(C^i,\PBord_n^{uple}) = (\PBord_n^{uple})_{(i)}$. Moreover,
$$\mathrm{Map}(C^1\times \mathrm{cof}(\Theta^{(i-1)}),\PBord_n^{uple}) = \mathbf{R}((\PBord_n^{uple})_{1,\bullets})_{(i-1)}.$$
The induction hypothesis implies that elements therein are elements in $(\PBord_n^{uple})_{(i)}$ which satisfy condition \eqref{cond 3} in Definition \ref{def PBord} or $i >1$.

Thus elements in $\mathbf{R}(\PBord_n^{uple})_{(i)}$ are such elements together with a path from their source and target to elements in $(\PBord_n^{uple})_{0}$. Unravelling this condition as in the base case, we see that this implies condition \eqref{cond 3} for $i=1$.
\end{proof}

\begin{ex}[The torus as a composition]

The difference between the $n$-fold and the $n$-uple Segal spaces can be seen when decomposing the torus, viewed as a 2-morphism in the respective $n$-(uple) categories. We will omit drawing the intervals outside of the torus and just draw the ``cutting lines'', which should be understood as actually extending to small closed intervals around them. 

The torus as a 2-morphism in $\Bord_2^{uple}$ can be decomposed simultaneously in both directions. One possible decomposition into in some sense elementary pieces is the following:
$$
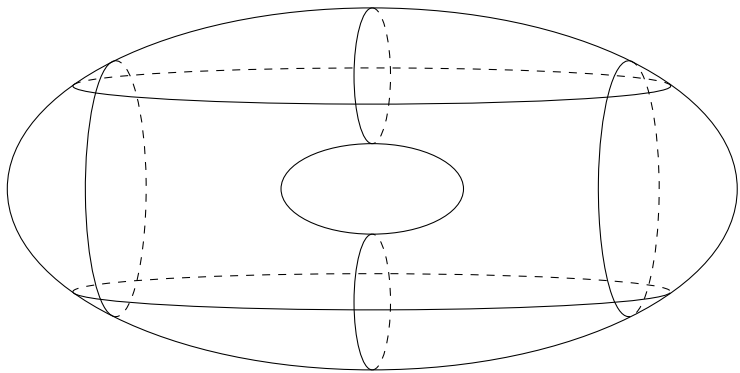
\begin{tikzpicture}[scale=0.8]
\draw (0,0) circle (6 and 3);
\draw (0,0) circle (1.5 and 0.75);

\node (C1) at ($(135: 6 and 3)$) {};
\node (D1) at ($(145: 6 and 3)$) {};

\draw (D1) arc (177 : 363 : 0.819*6 and 0.3); 
\draw[dashed] (D1) arc (177 : 3 : 0.819*6 and 0.3);

\node (D2) at ($(215: 6 and 3)$) {};

\draw (D2) arc (183 : 357 : 0.819*6 and 0.3); 
\draw[dashed] (D2) arc (183 : -3 : 0.819*6 and 0.3);

\draw (C1) arc (93: 267: 0.5 and 0.707*3);
\draw[dashed] (C1) arc (453: 267: 0.5 and 0.707*3);
\path (C1) arc (93: 130: 0.5 and 0.707*3) node[inner sep=0] (P1) {};
\path (C1) arc (93: 210: 0.5 and 0.707*3) node[inner sep=0] (P3) {};
\path (C1) arc (453: 30: 0.5 and 0.707*3) node[inner sep=0] (P2) {};
\path (C1) arc (453: 310: 0.5 and 0.707*3) node[inner sep=0] (P4) {};

\draw (0, 3) arc (90: 270: 0.3 and 1.125);
\draw[dashed] (0, 3) arc (90: -90: 0.3 and 1.125);
\path (0, 3) arc (90: 160 : 0.3 and 1.125) node[inner sep=0] (pp1) {};
\path (0, 3) arc (90: 340 : 0.3 and 1.125) node[inner sep=0] (pp2) {};

\draw[dashed] (0, -3) arc (-90: 90: 0.3 and 1.125);
\draw (0, -3) arc (270: 90: 0.3 and 1.125);
\path (0, -3) arc (-90: 160 : 0.3 and 1.125) node[inner sep=0] (pp3) {};
\path (0, -3) arc (-90: 340 : 0.3 and 1.125) node[inner sep=0] (pp4) {};

\node (C2) at ($(45: 6 and 3)$) {};
\path let \p1 = (C2) in node[inner sep=0] (c2) at (\x1,0) {};
\draw (C2) arc (87: 273: 0.5 and 0.707*3);
\draw[dashed] (C2) arc (447: 273: 0.5 and 0.707*3);
\path (C2) arc (87: 130: 0.5 and 0.707*3) node[inner sep=0] (p1) {};
\path (C2) arc (87: 210: 0.5 and 0.707*3) node[inner sep=0] (p3) {};
\path (C2) arc (447: 30: 0.5 and 0.707*3) node[inner sep=0] (p2) {};
\path (C2) arc (447: 310: 0.5 and 0.707*3) node[inner sep=0] (p4) {};
\end{tikzpicture}
$$
However, similar to the argument in Example \ref{ex: Lurie torus false}, this decomposition is not a valid decomposition in $\Bord_2$, as condition \ref{cond 3} in Definition \ref{def PBord} is violated.

The torus as a 2-morphism in $\Bord_2$ can only be decomposed ``successively'', so we first decompose it in the first direction, i.e.~the first coordinate, e.g.~as
$$
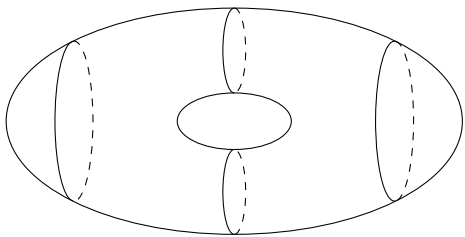
\begin{tikzpicture}[scale=0.5]
\draw (0,0) circle (6 and 3);
\draw (0,0) circle (1.5 and 0.75);

\node (C1) at ($(135: 6 and 3)$) {};
\path let \p1 = (C1) in node[inner sep=0] (c1) at (\x1,0) {};
\draw (C1) arc (93: 267: 0.5 and 0.707*3);
\draw[dashed] (C1) arc (453: 267: 0.5 and 0.707*3);
\path (C1) arc (93: 130: 0.5 and 0.707*3) node[inner sep=0] (P1) {};
\path (C1) arc (93: 210: 0.5 and 0.707*3) node[inner sep=0] (P3) {};
\path (C1) arc (453: 30: 0.5 and 0.707*3) node[inner sep=0] (P2) {};
\path (C1) arc (453: 310: 0.5 and 0.707*3) node[inner sep=0] (P4) {};

\draw (0, 3) arc (90: 270: 0.3 and 1.125);
\draw[dashed] (0, 3) arc (90: -90: 0.3 and 1.125);
\path (0, 3) arc (90: 160 : 0.3 and 1.125) node[inner sep=0] (pp1) {};
\path (0, 3) arc (90: 340 : 0.3 and 1.125) node[inner sep=0] (pp2) {};

\draw[dashed] (0, -3) arc (-90: 90: 0.3 and 1.125);
\draw (0, -3) arc (270: 90: 0.3 and 1.125);
\path (0, -3) arc (-90: 160 : 0.3 and 1.125) node[inner sep=0] (pp3) {};
\path (0, -3) arc (-90: 340 : 0.3 and 1.125) node[inner sep=0] (pp4) {};

\node (C2) at ($(45: 6 and 3)$) {};
\path let \p1 = (C2) in node[inner sep=0] (c2) at (\x1,0) {};
\draw (C2) arc (87: 273: 0.5 and 0.707*3);
\draw[dashed] (C2) arc (447: 273: 0.5 and 0.707*3);
\path (C2) arc (87: 130: 0.5 and 0.707*3) node[inner sep=0] (p1) {};
\path (C2) arc (87: 210: 0.5 and 0.707*3) node[inner sep=0] (p3) {};
\path (C2) arc (447: 30: 0.5 and 0.707*3) node[inner sep=0] (p2) {};
\path (C2) arc (447: 310: 0.5 and 0.707*3) node[inner sep=0] (p4) {};
\end{tikzpicture}
$$
which is an element in $(\Bord_2)_{4, 1}$ and then decompose the two middle pieces, which are the images under the compositions of face maps $D^1(2), D^1(3): (\Bord_2)_{4, 1} \rightrightarrows (\Bord_2)_{1,1}$, as
$$
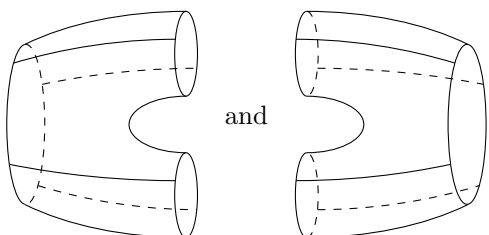
\begin{tikzpicture}[scale=0.5, baseline]
\draw (0,0.75) arc (90:270:1.5 and 0.75);

\node (C1) at ($(135: 6 and 3)$) {};
\draw (C1) arc (135: 90: 6 and 3);

\node (C3) at ($(225: 6 and 3)$) {};
\draw (C3) arc (225: 270: 6 and 3);

\path let \p1 = (C1) in node[inner sep=0] (c1) at (\x1,0) {};
\draw (C1) arc (93: 267: 0.5 and 0.707*3);
\draw[dashed] (C1) arc (453: 267: 0.5 and 0.707*3);
\path (C1) arc (93: 130: 0.5 and 0.707*3) node[inner sep=0] (P1) {};
\path (C1) arc (93: 210: 0.5 and 0.707*3) node[inner sep=0] (P3) {};
\path (C1) arc (453: 30: 0.5 and 0.707*3) node[inner sep=0] (P2) {};
\path (C1) arc (453: 310: 0.5 and 0.707*3) node[inner sep=0] (P4) {};

\draw (0, 3) arc (90: 270: 0.3 and 1.125);
\draw (0, 3) arc (90: -90: 0.3 and 1.125);
\path (0, 3) arc (90: 160 : 0.3 and 1.125) node[inner sep=0] (pp1) {};
\path (0, 3) arc (90: 340 : 0.3 and 1.125) node[inner sep=0] (pp2) {};

\draw (0, -3) arc (-90: 90: 0.3 and 1.125);
\draw (0, -3) arc (270: 90: 0.3 and 1.125);
\path (0, -3) arc (-90: 160 : 0.3 and 1.125) node[inner sep=0] (pp3) {};
\path (0, -3) arc (-90: 340 : 0.3 and 1.125) node[inner sep=0] (pp4) {};

\draw (P1) parabola[bend at end] (pp1); 
\draw [dashed] (P2) parabola[bend at end] (pp2); 

\draw (P3) parabola[bend at end] (pp3); 
\draw[dashed] (P4) parabola[bend at end] (pp4); 
\end{tikzpicture}
\quad \text{and} \quad
\begin{tikzpicture}[scale=0.5, baseline]
\draw (0,0.75) arc (90:-90: 1.5 and 0.75);
\draw (0, 3) arc (90: 270: 0.3 and 1.125);
\draw[dashed] (0, 3) arc (90: -90: 0.3 and 1.125);
\path (0, 3) arc (90: 160 : 0.3 and 1.125) node[inner sep=0] (pp1) {};
\path (0, 3) arc (90: 340 : 0.3 and 1.125) node[inner sep=0] (pp2) {};

\draw[dashed] (0, -3) arc (-90: 90: 0.3 and 1.125);
\draw (0, -3) arc (270: 90: 0.3 and 1.125);
\path (0, -3) arc (-90: 160 : 0.3 and 1.125) node[inner sep=0] (pp3) {};
\path (0, -3) arc (-90: 340 : 0.3 and 1.125) node[inner sep=0] (pp4) {};

\node (C2) at ($(45: 6 and 3)$) {};
\draw (C2) arc (45: 90: 6 and 3);
\draw (0,-3) arc (-90:-45: 6 and 3);
\path let \p1 = (C2) in node[inner sep=0] (c2) at (\x1,0) {};
\draw (C2) arc (87: 273: 0.5 and 0.707*3);
\draw (C2) arc (447: 273: 0.5 and 0.707*3);
\path (C2) arc (87: 130: 0.5 and 0.707*3) node[inner sep=0] (p1) {};
\path (C2) arc (87: 210: 0.5 and 0.707*3) node[inner sep=0] (p3) {};
\path (C2) arc (447: 30: 0.5 and 0.707*3) node[inner sep=0] (p2) {};
\path (C2) arc (447: 310: 0.5 and 0.707*3) node[inner sep=0] (p4) {};

\draw (pp1) parabola (p1); 
\draw [dashed] (pp2) parabola (p2); 

\draw (pp3) parabola (p3); 
\draw[dashed] (pp4) parabola (p4); 
\end{tikzpicture}
$$

Altogether a possible decomposition of the torus into elementary pieces in $\Bord_2$ is
$$
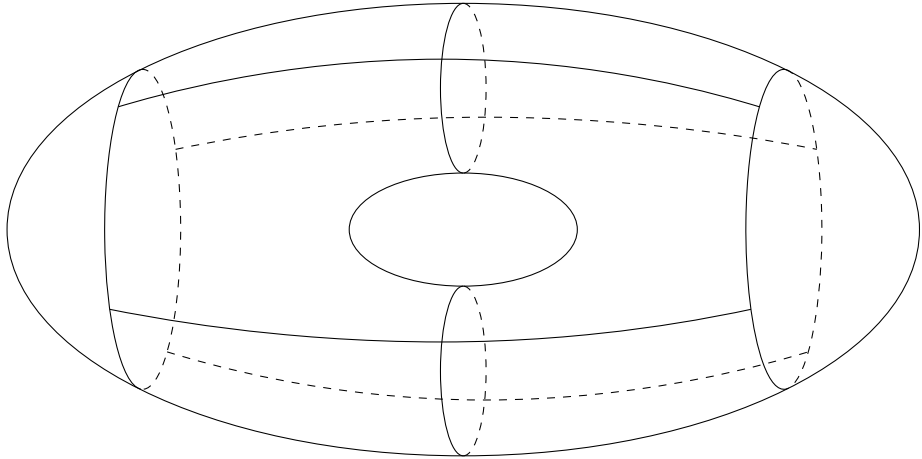
\begin{tikzpicture}
\draw (0,0) circle (6 and 3);
\draw (0,0) circle (1.5 and 0.75);

\node (C1) at ($(135: 6 and 3)$) {};
\path let \p1 = (C1) in node[inner sep=0] (c1) at (\x1,0) {};
\draw (C1) arc (93: 267: 0.5 and 0.707*3);
\draw[dashed] (C1) arc (453: 267: 0.5 and 0.707*3);
\path (C1) arc (93: 130: 0.5 and 0.707*3) node[inner sep=0] (P1) {};
\path (C1) arc (93: 210: 0.5 and 0.707*3) node[inner sep=0] (P3) {};
\path (C1) arc (453: 30: 0.5 and 0.707*3) node[inner sep=0] (P2) {};
\path (C1) arc (453: 310: 0.5 and 0.707*3) node[inner sep=0] (P4) {};

\draw (0, 3) arc (90: 270: 0.3 and 1.125);
\draw[dashed] (0, 3) arc (90: -90: 0.3 and 1.125);
\path (0, 3) arc (90: 160 : 0.3 and 1.125) node[inner sep=0] (pp1) {};
\path (0, 3) arc (90: 340 : 0.3 and 1.125) node[inner sep=0] (pp2) {};

\draw[dashed] (0, -3) arc (-90: 90: 0.3 and 1.125);
\draw (0, -3) arc (270: 90: 0.3 and 1.125);
\path (0, -3) arc (-90: 160 : 0.3 and 1.125) node[inner sep=0] (pp3) {};
\path (0, -3) arc (-90: 340 : 0.3 and 1.125) node[inner sep=0] (pp4) {};

\node (C2) at ($(45: 6 and 3)$) {};
\path let \p1 = (C2) in node[inner sep=0] (c2) at (\x1,0) {};
\draw (C2) arc (87: 273: 0.5 and 0.707*3);
\draw[dashed] (C2) arc (447: 273: 0.5 and 0.707*3);
\path (C2) arc (87: 130: 0.5 and 0.707*3) node[inner sep=0] (p1) {};
\path (C2) arc (87: 210: 0.5 and 0.707*3) node[inner sep=0] (p3) {};
\path (C2) arc (447: 30: 0.5 and 0.707*3) node[inner sep=0] (p2) {};
\path (C2) arc (447: 310: 0.5 and 0.707*3) node[inner sep=0] (p4) {};

\draw (P1) parabola[bend at end] (pp1) parabola (p1); 
\draw [dashed] (P2) parabola[bend at end] (pp2) parabola (p2); 

\draw (P3) parabola[bend at end] (pp3) parabola (p3); 
\draw[dashed] (P4) parabola[bend at end] (pp4) parabola (p4); 
\end{tikzpicture}
$$
This of course is also a valid decomposition in the $2$-fold category $\Bord_2^{uple}$.
\end{ex}

\section{The symmetric monoidal structure on bordisms and tangles}\label{sec Bord symm}

The $(\infty,n)$-category $\Bord_n$ is symmetric monoidal with its symmetric monoidal structure essentially arising from taking disjoint unions.
In this section we endow $\Bord_n$ with a symmetric monoidal structure in two ways. In Section \ref{Gamma Bord} the symmetric monoidal structure arises from a $\Gamma$-object. In Section \ref{tower Bord} a symmetric monoidal structure is defined using a tower of monoidal $i$-hybrid $(n+i)$-fold Segal spaces.

\subsection{The symmetric monoidal structure arising as a \texorpdfstring{$\Gamma$}{Gamma}-object}\label{Gamma Bord}

We construct a sequence of $n$-fold Segal spaces $(\PBord_n^V\langle m\rangle)_\bullets$ which form a $\Gamma$-object in $n$-fold Segal spaces which in turn endows $\Bord_n$ with a symmetric monoidal structure as defined in Section \ref{Gamma}.

\begin{defn}
Let $V$ be a finite-dimensional vector space.  For every $k_1,\ldots, k_n$, let $(\PBord_n^V\langle m\rangle)_{k_1,\ldots, k_n}$ be the collection of tuples
$$(M_1,\ldots, M_m, (I^i_0\leq\ldots\leq I^i_{k_i})_{i=1,\ldots, n}),$$
where each $(M_\beta, (I^i_0\leq\ldots\leq I^i_{k_i})_{i=1,\ldots n})$ is an element of $(\PBord_n^V)_{k_1,\ldots, k_n}$ and $M_1,\ldots , M_m$ are disjoint. It can be made into an $n$-fold simplicial space similarly to $\PBord_n^V$. Moreover, similarly to the definition of $\Bord_n$, we take the homotopy colimit over all $V\subset\R^\infty$ and complete to get an $n$-fold complete Segal space $\Bord_n[m]$.
\end{defn}

\begin{prop}
Let
\begin{align*}
\Gamma & \longrightarrow \SeSp_n,\\
\langle m\rangle & \longmapsto \PBord_n\langle m\rangle,
\end{align*}
where to a morphism $f:\langle m\rangle\to\langle k\rangle$ in $\Gamma$ we assign the map
\begin{align*}
\PBord_n\langle m\rangle & \longrightarrow \PBord_n\langle k\rangle,\\
(M_1,\ldots, M_m, I's) & \longmapsto (\coprod_{\beta\in f^{-1}(1)} M_\beta,\ldots, \coprod_{\beta\in f^{-1}(k)} M_\beta, I's).
\end{align*}
This assignment is functorial and endows $\Bord_n$ with a symmetric monoidal structure.
\end{prop}

\begin{proof}
By Lemma \ref{symm mon completion} it is enough to show that the assignment is a functor $\Gamma \to \SeSp_n$ with the property that for every $m\geq 0$ the map
$$\prod_{1\leq \beta\leq n} \gamma_\beta: \PBord_n\langle m\rangle \longrightarrow (\PBord_n\langle 1\rangle)^m$$
is an equivalence of $n$-fold Segal spaces.

Functoriality follows from the functoriality of taking preimages. For $m=0$ both sides are contractible. For $m>0$ the map $\prod_{1\leq \beta\leq n} \gamma_\beta$ is a level-wise inclusion  and we show that level-wise it is a weak equivalence.

First, we can show that for every $k_1,\ldots, k_n$, the space $(\PBord_n\langle 1\rangle)^m_{k_1,\ldots, k_n}$ is weakly equivalent to its pullback, which we will denote by $P^m$, along the diagonal map $\Int^n_{k_1,\ldots, k_n} \to (\Int^n_{k_1,\ldots, k_n})^m$. The argument is analogous to the proofs of the rescaling steps in the proof of the Segal condition in Proposition \ref{prop PBord Ssp}. Note that $P^m$ is the subspace of those elements for which the intervals coincide, and $\prod_{1\leq \beta\leq n} \gamma_\beta$ factors through $P^m$. 

Now, we will exhibit an explicit deformation retraction of $\prod_{1\leq \beta\leq n} \gamma_\beta: \PBord_n\langle m\rangle_{k_1,\ldots, k_n} \to P^m$, which shows that the two spaces are equivalent.

Consider the family of embeddings $\iota_s: V\to \R\otimes V$, $v\mapsto (s\alpha,v)$, parametrized by $s\in [0,1]$. Note that this also induces a family of embeddings $\R^\infty \to \R\oplus\R^\infty\cong\R^\infty$. 

Let $\Big( (M_1), \ldots, (M_m)\Big)$ be any $k$-simplex in the target $P^m$. We construct a $k$-simplex in $\PBord_n\langle m\rangle_{k_1,\ldots, k_n}$ together with a map $F: \Delta^k\times[0,1]\to P^m$ restricting to the original one at 1 and the new $k$-simplex at 0. The map $F$ is defined as follows: for fixed $s$, it consists of the $k$-simplex which is defined by composing the embedding $M_\alpha \hookrightarrow V \times B(\oul{I})$ with the embedding $\iota_s$. This depends smoothly on $s$ (and nothing else). For $s>0$, this indeed lands in $\PBord_n\langle m\rangle_{k_1,\ldots, k_n}$, since the images of $M_i$ are disjoint. The construction is compatible with the simplicial structure, since $\iota_s$ did not affect the copy of $B(\oul{I})$. Altogether this induces a strong homotopy equivalence between the above spaces.
\end{proof}

\begin{rem}\label{rem Gamma for d-cat Bord}
More generally, the same construction works for $\Bord_n^{(\infty,d)}$ for $d\leq n$ using a sequence of $d$-fold Segal spaces $\PBord_n^l\langle m\rangle$ for $l=n-d$.
\end{rem}

\subsection{Looping, the monoidal structure and the tower}\label{tower Bord}

Our goal for this section is to endow $\Bord_n$ and $\Bord_n^{(\infty,d)}$ with symmetric monoidal structures arising from a tower of monoidal $l$-hybrid $(n+l)$-fold Segal spaces $\Bord_n^{(l)}$ for $l\geq0$. We will prove a stronger statement first, namely, that the tangle categories $\Bord_n^{(\infty,d), V}$ are $r$-monoidal if $V$ is $r$-dimensional.

Recall from Section \ref{sec PBord_n^l} the $(n+l)$-fold Segal spaces $\PBord_n^l$ of $n$-dimensional bordisms. We saw in Remark \ref{rem hybrid} that $\PBord_n^l$ is $(l-1)$-connected if $l>0$. However, it does not have a discrete space of objects, 1-morphisms, \ldots, $(l-1)$-morphisms. For $l\leq0$ the situation is even worse as $\PBord_n^l$ is not even connected. However, in any $\PBord_n^l$, there is the distinguished object
$$\emptyset=\left(\emptyset, (0,1), \ldots, (0,1)\right),$$
and by Proposition \ref{prop tower without pointed} it suffices to prove the following theorem.

\begin{thm}\label{thm looping}
For $n+l \geq k\geq 0$, and an $(r-1)$-dimensional vector space $V$ there is a weak equivalence
\begin{equation*}
\begin{tikzpicture}
\node (a) at (0,0) {};
\draw (0,0) node[anchor=east] {$\myloop[k]{\PBord_n^{l, V} }{\emptyset}$};
\node (b) at (2,0) {};
\draw (2,0) node[anchor=west] {$\PBord_n^{l-k, V\oplus\R}$};

\draw[->] (a) edge node[above] {$u_r$} (b);
\end{tikzpicture}
\end{equation*}
which induces a weak equivalence $u: \myloop[k]{ \PBord_n^{l} }{\emptyset} \to \PBord_n^{l-k}$.
\end{thm}

Since looping and hybrid completion commute by Lemma \ref{lemma looping completion} the following corollary is immediate.
\begin{cor} Let $n\geq 0$, $d\leq n$, and $V$ be an $r$-dimensional vector space.
\begin{itemize}
\item The tangle categories $\Bord_n^V$ and $\Bord_n^{(\infty,d), V}$ are $r$-monoidal.
\item The bordism categories $\Bord_n$ and $\Bord_n^{(\infty,d)}$ are symmetric monoidal.
\end{itemize}
\end{cor}

We can extract the $k$-monoidal $(n+l)$-fold complete Segal spaces which form the tower for the symmetric monoidal structure:
\begin{defn}\label{defn Bord_n^{(l)}}
For $k>0$ and $d\geq 0$, the $(n+l)$-fold Segal space
$$\deloop[k]{ \PBord_n^{k-(n-d)} }{ \emptyset}$$
is $(k-1)$-connected and satisfies
$$\myloop[k]{ \deloop[k]{ \PBord_n^{k-(n-d)} }{ \emptyset}  }{\emptyset} \simeq \myloop[k]{ \PBord_n^{k-(n-d)} }{\emptyset} \simeq \PBord_n^{-(n-d)}$$
by the above Theorem. Its $k$-hybrid completion thus is a $k$-monoidal complete $(d+k)$-fold Segal space. 
The collection thereof, for $k\geq 0$, endows the complete $d$-fold Segal space $\Bord_n^{(\infty,d)} = \Bord_n^l$ for $d=n+l$ with a symmetric monoidal structure. For $d=n$, we obtain the symmetric monoidal structure on the complete $n$-fold Segal space $\Bord_n$.
\end{defn}

Since $\myloopnop[n-d]{\Bord_n}$ is complete, the universal property for the completion $\Bord_n^{(\infty,d)}$ of $\PBord_n^{d-n}$ applied to the map $\PBord_n^{d-n} \simeq \myloopnop[n-d]{\PBord_n} \to \myloopnop[n-d]{\Bord_n}$ gives the following corollary.
\begin{cor}\label{cor looping is again a Bord}
There is a morphism of symmetric monoidal $(\infty, d)$-categories
$$\Bord_n^{(\infty,d)} \longrightarrow \myloopnop[n-d]{\Bord_n}.$$
\end{cor}

\begin{rem}
Since completion and looping do not necessarily commute, this map is not necessariy an equivalence.
\end{rem}

\begin{proof}[Proof of Theorem \ref{thm looping}]
It is enough to show that 
$$\myloop{\PBord_n^{l, \R^{r-1}} }{ \emptyset }=\Hom_{\PBord_n^{l, \R^{r-1}}}(\emptyset,\emptyset) \simeq \PBord_n^{l-1, \R^r}.$$ The statement for general $k$ follows by induction.

We define a map
$$u_r: \myloopnop{\PBord_n^{l, \R^{r-1}}} \longrightarrow \PBord_n^{l-1, \R^r}$$
by sending an element in $\Hom_{\PBord_n^{l, \R^{r-1}}}(\emptyset,\emptyset)_{k_2,\ldots, k_{n+l}}$
$$(M_l)=\left(M\subseteq \R^{r-1}\times (a^1_0, b^1_1)\times B(\oul{I}), I^1_0 \leq I^1_1, \oul{I}= (I^i_0\leq\cdots\leq I^i_{k_i})_{i=2}^{n+l}\right)\in (\PBord_n^{l})_{1, k_2,\ldots, k_{n+l}}$$
to
$$(M_{l-1})=(M\subseteq (\underbrace{\R^{r-1}\times(a^1_0, b^1_1)}_{=:\tilde{V}})\times B(\oul{I}), \oul{I}=(I^i_0\leq\cdots\leq I^i_{k_i})_{i=2}^{n+l}),$$
so it ``forgets'' the first specified intervals. In the above, we view $\tilde V=\R^{r-1}\times (a^1_0, b^1_1)$ as a vector space using a diffeomorphism $(a^1_0, b^1_1) \cong \R$, and thus get an isomorphism $\tilde V\cong \R^r$.

First of all, we need to check that this map is well-defined, that is, that $(M_{l-1})\in (\PBord_n^{l-1, \R^r})_{k_2,\ldots, k_{n+l}}$.  The condition we need to check is the second one, i.e.~we need to check that $M\subset \tilde{V}\times B(\oul{I})$ is a bounded submanifold and $M\hookrightarrow \tilde{V}\times B(\oul{I}) \twoheadrightarrow B(\oul{I})$ is proper. Since $p_1^{-1}(I^1_0))= p_1^{-1}(I^1_1) =\emptyset$, we know that $M$ is bounded in the direction of the first coordinate, since $M = p_1^{-1}([b^1_0, a^1_1])$, and moreover, we know that $M\to B(I^1_0 \leq I^1_1, \oul{I})$ is proper. Together this implies the statements.

We claim that the homotopy fibers of this map are contractible. The homotopy fiber at a point 
$$(M_1)= (M_1\subseteq (\R^r \times B(\oul{I}), \oul{I}=(I^i_0\leq\cdots\leq I^i_{k_i})_{i=2}^{n+l}),$$
 in the target consists of pairs of:
\begin{itemize}
\item a 1-simplex $(M)$ in $(\PBord_n^{l-1, \R^r})_{k_2,\ldots, k_{n+l}}$ with endpoint $(M_1)$ together with,
\item for the source vertex $(M_0)$, a choice of pair of intervals which ``bound'' the manifold the the last coordinate in $\R^r$, i.e.~an element $I^1_0\leq I^1_1$ in $\Int_1$ such that
$$(M_0\subseteq \R^r \times B(\oul{I}(0), I^1_0\leq I^1_1, \oul{I}(0)=(I^i_0(0)\leq\cdots\leq I^i_{k_i}(0))_{i=2}^{n+l}) \in \myloopnop{\PBord_n^{l, \R^{r-1}}}_{k_2,\ldots, k_{n+l}}.$$
\end{itemize}
In presence of the 1-simplex, the choice of the intervals is equivalent to the choice of such intervals for the original element $(M_1)$, since they can be transported back and forth along the 1-simplex. Thus, the homotopy fiber is a product of the space of paths to the chosen $(M_1)$ (which is contractible) with the space of choices of pairs of intervals as in the second item, but for the original $(M_1)$.

We claim that this latter space is contractible as well: the fixed submanifold $M_1\subseteq \R^r \times B(\oul{I})$ is a bounded and closed submanifold. Therefore, it is bounded in the $\R^r$-direction. Therefore the image of the projection $p_{\{r\}}: M_1\to \R^{\{r\}}$ is bounded and the intervals can be chosen to be any intervals which lie of either side of the convex hull of this image. The complement of the convex hull has two connected components and each interval can be chosen arbitrarily in one of the connected components. The space of  subintervals of a given interval is contractible, and therefore the homotopy fibers are contractible.
\end{proof}

\subsection{Comparison of the symmetric monoidal structures}

Starting with the $\Gamma$-object $\PBord_n\langle -\rangle :\Gamma \to \SeSp_n$ we can extract, as explained in section \ref{sec monoidal comparison}, the $l$th layer of a tower for this symmetric monoidal structure on $\PBord_n$. It is the $(n+l)$-fold Segal space $\PBord_n\langle -\rangle_{\leq l}$ given by precomposing $\PBord_n\langle -\rangle$ with the map $(\Delta^{op})^l \overset{f^l}{\longrightarrow} \Gamma^l \overset{\wedge}{\longrightarrow} \Gamma$. We now show that it is equivalent to the layers of the tower constructed in the previous section.

\begin{prop}
The $(n+l)$-fold Segal spaces
$$\PBord_n\langle -\rangle_{\leq l}
\quad\mbox{and}\quad
\deloop[l]{ \PBord_n^{l} }{ \emptyset }
$$
are weakly equivalent.
\end{prop}

\begin{proof}
Recall from Section \ref{sec Bord (0,1)} the variant of bordisms which are constrained to the box $(0,1)^n$. The rescaling maps determine a weak equivalence
$$\PBord_n^l \overset{\rho}{\longrightarrow}  \PBord_n^{l, (0,1)}.$$
Similarly, rescaling determines a weak equivalence
$$\PBord_n\langle -\rangle \overset{\rho}{\longrightarrow} \PBord_n^{(0,1)}\langle -\rangle$$
to a $\Gamma$-object in $n$-fold Segal spaces, where $\PBord_n^{(0,1)}\langle m \rangle$ is the obvious modification constraining the bordisms to the box $(0,1)^n$ and using rescaling maps as in \ref{sec Bord (0,1)}.
We will construct a weak equivalence between $\PBord_n^{(0,1)}\langle -\rangle_{\leq l}$ and $\deloop[l]{ \PBord_n^{l, (0,1)} }{ \emptyset }$  as the colimit of a zig-zag of maps below. We need this intermediate step of constraining bordisms to the box to ensure that the maps in this zig-zag are indeed functorial.

To shorten notation, for an ascending sequence $V_0 \subset V_1\subset V_2 \subset \cdots$ of $r$-dimensional vector spaces $V_r$, let
$$Y_r=\PBord_n^{^{(0,1)}, V_r}\langle -\rangle_{\leq l}
\quad\mbox{and}\quad
X_r= \deloop[l]{ \PBord_n^{l, (0,1), V_r} }{ \emptyset }.
$$
We will work with $V_r=\R^r$ and use the standard inclusions $\R^r \cong \R^r\oplus \{0\} \hookrightarrow \R^{r+1}$. They induce level-wise inclusions $\iota_X: X_r \hookrightarrow X_{r+1}$ and $\iota_Y: Y_r\hookrightarrow Y_{r+1}$.

We will construct a sequence of maps
$$
\begin{tikzcd}[column sep=small]
\dots \arrow{dr}& & X_r \arrow{dr}{f_r} & & X_{r+l} \arrow{dr}{f_{r+l}} & & \dots\\
& Y_r \arrow{ur}{g_r} && Y_{r+l} \arrow{ur}{g_{r+l}} && Y_{r+2l} \arrow{ur}
\end{tikzcd}
$$
such that $f_r\circ g_r \sim \iota_Y$ and $g_{r+l}\circ f_r \sim \iota_X$. This will induce the weak equivalence on the homotopy colimits.

{\em The first map $f_r$ -- forgetting certain intervals:}
To define $f_r$, for $k_1,\ldots, k_n, j_1,\ldots, j_l\geq0$ consider a general element in $(X_{r+l})_{k_1,\ldots, k_n, j_1,\ldots, j_l}$ given by
$$\left( M \subset  V_r\times B(\oul{I})\times B(\oul{J}), \oul{I}, \oul{J}\right),$$
where $\oul{I}\in \Int^{(0,1),n}_{k_1,\ldots, k_n}$ and $\oul{J}\in \Int^{(0,1),l}_{j_1,\ldots, j_l}$. Note that $B(\oul{I})\times B(\oul{J}) =(0,1)^n\times (0,1)^l \subset \R^n\times \R^l$. The condition that the preimage of the projection map over the intervals in $\oul{J}$ be empty implies that $M$ is the disjoint union of the preimages
$$M_{m_1,\ldots, m_l} = D(m_1,\ldots, m_l)(M)$$
for $1\leq i \leq l$ and $1\leq m_i \leq j_i$. In brief, the image under $f_r$ will be same embedding, but tracking these disjoint manifolds, together with just the intervals $\oul{I}$.

To implement this, we order the $M_{m_1,\ldots, m_l}$'s in lexicographic ordering. Finally we relabel them from 1 to $j_1\cdot\cdots \cdot j_l$ according to this ordering, which amounts to setting $M_s = M_{m_1,\ldots, m_l}$ for $s=m_1+ \sum_{i=2}^l (m_i-1)\cdot j_1\cdots j_{i-1}$.

 The image under $f_r$ has to be an element in
$$(Y_{r+l})_{k_1,\ldots, k_n, j_1,\ldots, j_l} = (\PBord_n^{V_{r+l}}\langle j_1\cdots j_l\rangle)_{k_1,\ldots, k_n}.$$ 
For this, we use a fixed diffeomorphism $(0,1) \cong \R$ which induces a diffeomorphism $B(\oul{J}) \cong V_l$ to transfer the vector space structure. This in turn induces an isomorphism $V_{r+l}\cong V_r\times B(\oul{J})$. Then we let the image of the general element above be
$$ \left( M_1, \ldots,  M_{j_1\cdots j_l}, \oul{I}\right), \quad\mbox{where } M_1\amalg\cdots\amalg M_{j_1\cdots j_l} \subset  V_{r+l}\times B(\oul{I}).$$

We remark that this assignment is functorial in $k_1,\ldots, k_n, j_1,\ldots, j_l$, as the diffeomorphism is fixed once and for all and the map just forgets certain intervals and orders the manifolds in the disjoint union in a functorial way.

{\em The second map $g_r$ -- adding certain intervals back in:}
Conversely, to define $g_r: Y_r \to X_r$, start with a general element in $(Y_{r})_{k_1,\ldots, k_n, j_1,\ldots, j_l}$ given by
$$ \left( M_1, \ldots,  M_{j_1\cdots j_l}, \oul{I}\right), \quad\mbox{where } M_1\amalg\cdots\amalg M_{j_1\cdots j_l} \subset  V_{r}\times B(\oul{I})$$
Set
$$M = M_1\amalg\cdots\amalg M_{j_1\cdots j_l}, $$
and for $1\leq i \leq l$ and $0\leq m_i \leq j_i$, set
$$J^i_{m_i}= [\frac{2m_i}{ 2j_i+1}, \frac{2m_i+1}{2j_i+1}]\cap (0, 1),$$
so the intervals are equidistributed among $(0,1)$.
Now we realize $M$ as a submanifold of $V_{r}\times B(\oul{I})\times B(\oul{J}) = V_{r}\times (0,1)^n \times (0,1)^l$ in the following way:

For $1\leq s \leq j_1 \cdots j_l$ find the relabelling in the converse direction: find $m_1,\ldots, m_l$ such that $s=m_1+ \sum_{i=2}^l (m_i-1)\cdot j_1\cdots j_{i-1}$. For example, $m_l= \lfloor\frac{s}{j_1\cdots j_{l-1}}\rfloor$, etc. Then, we use the inclusion of midpoints of the $m_i$th interval $\{(\frac{2m_1 +\frac12}{ 2j_1+1},\ldots, \frac{2m_l +\frac12}{ 2j_l+1})\}\subset B(\oul{J})$ and set
$$M_s \cong M_s \times \{(\frac{2m_1 +\frac12}{ 2j_1+1},\ldots, \frac{2m_l +\frac12}{ 2j_l+1})\} \subset  V_{r}\times B(\oul{I}) \times \{(\frac{2m_1 +\frac12}{ 2j_1+1},\ldots, \frac{2m_l +\frac12}{ 2j_l+1})\} \subset  V_{r}\times B(\oul{I}) \times B(\oul{J}).$$
The images of the $M_s$ by construction will be disjoint, so together, we get
$$M \subset  V_{r}\times (0,1)^n\times (0,1)^l = V_{r}\times B(\oul{I}) \times B(\oul{J}).$$
Together with the intervals $\oul{I}$, $\oul{J}$, this gives an element in $X_r$.

Again, it is straightforward to check functoriality: the only choice was that of the extra intervals $\oul{J}$, and this choice was canonical because we chose them to be equidistributed in $(0,1)$. Moreover, they are sent to each other under the face and degeneracy maps of $\Int^{(0,1)}$ because of the extra rescaling step. This is the reason for constraining the bordisms to the box in this proof.

{\em The homotopies:}
It is easy to see that $f_r\circ g_r\sim \iota_Y$: For an element in the composition, the manifolds in the disjoint union are ``spread out'' over different points in $(0,1)^l=B(\oul{J})\subset \R^l$:
\begin{center}
$ \left( M_1, \ldots,  M_{j_1\cdots j_l}, \oul{I}\right), \quad\mbox{where } M_1\amalg\cdots\amalg M_{j_1\cdots j_l} \subset  V_{r}\times B(\oul{I})$\\
\rotatebox{-90}{$\mapsto$}\\
$\left(M, \oul{I}, \oul{J}\right), \quad\mbox{where } M= M_1\amalg\cdots\amalg M_{j_1\cdots j_l} \subset  V_{r}\times (0,1)^n\times (0,1)^l =V_{r}\times B(\oul{I}) \times B(\oul{J})$\\
\rotatebox{-90}{$\mapsto$}\\
$ \left( M_1, \ldots,  M_{j_1\cdots j_l}, \oul{I}\right), \quad\mbox{where } M_1\amalg\cdots\amalg M_{j_1\cdots j_l} \subset  V_{r+l}\times B(\oul{I})$
\end{center}
However, for the latter element, consider the homotopy $$H: [0,1]\times \R^{r+l} \to \R^{r+l}, (t, x_1,\ldots, x_r, y_1,\ldots, y_l)\mapsto (x_1,\ldots, x_r, t\cdot y_1,\ldots, t\cdot y_l).$$
For any $t$, 
$$\left( (H_t\times id_{(0,1)^n})(M) \subset \R^{r+l} \times (0,1)^n = \R^{r+l} \times B(\oul{I}), \oul{I} \right)$$
still is an element in $(Y_{r+l})_{k_1,\ldots, k_n, j_1,\ldots, j_l}$, and for $t=0$ we get back the element we started with.

A homotopy from the other composition $g_{r+l}\circ f_r$ to $\iota_X$ is also straightforward to construct.
\end{proof}

\begin{cor}
The two symmetric monoidal structures on $\Bord_n$ constructed in the last two sections are equivalent.
\end{cor}

\begin{rem}
It is straightforward to get a similar result for $\Bord_n^{(\infty,d)}$.
\end{rem}

\section{Interpretation of bordisms as manifolds with corners and the homotopy category}\label{hocat}

In this section we compare our definitions to the definitions of higher bordisms from \cite{Laures, LaudaPfeiffer, Schommer, BoekstedtMadsen} which are certain manifolds with corners. We first recall the definition and show that every bordism in that sense leads to an element in our space of bordisms from the previous section and vice versa. Then we prove that every path in our space of bordisms leads to diffeomorphic bordisms as manifolds with corners, and explain that the spaces indeed are the disjoint union of classifying spaces of diffeomorphisms of bordisms, as suggested in \ref{sec space PBord}. Finally, we show that the homotopy category of our symmetric monoidal Segal space of bordisms recovers the usual bordism category.

\subsection{Bordisms as manifolds with corners and embeddings thereof}\label{sec (cubical) bordisms}

For the definition and notation for $\<k\>$-manifolds used in this section we refer to \cite{Laures}. In brief, a $\<k\>$-manifold is a manifold $M$ with ``faces''\footnote{For an $n$-dimensional manifold $M$ with corners, for any $x\in M$ the number of zeros $c(x)$ in the coordinates of $x$ in any chart is independent of the chart (by a chart we mean a diffeomorphism $x\ni U\to V\subset \mathbb{R}_{\geq0}^n$). A connected face of $M$ is the closure of a component of $\{x\in M: c(x)=1\}$, i.e.~of the $(n-1)$-dimensional part of the boundary. An $n$-dimensional manifold with faces is an $n$-dimensional manifold $M$ with corners such that any $x\in M$ is in exactly $c(x)$ faces, i.e.~if $x$ is in the $(n-k)$-dimensional part of the boundary, then it lies in $k$ different faces.}
together with an ordered $n$-tuple $(\partial_1M,\ldots, \partial_k M)$ satisfying
\begin{enumerate}
\item $\partial_1M \cup \ldots\cup \partial_k M = \partial M$, and
\item $\partial_iM \cap \partial_jM$ is a face of $\partial_i M$ and $\partial_j M$ for every $i\neq j$.
\end{enumerate}
The number $k$ indicates that the manifold has corners of codimension $k$, which are exactly the components of $\partial_1M \cap \ldots\cap \partial_k M$.
\begin{ex}
Consider the biangle as a manifold with corners:
$$
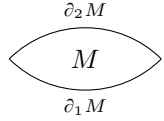
\begin{tikzpicture}
\node[inner sep=0] (A) at (-1,0) {};
\node[inner sep=0] (B) at (1,0) {};
\path[font=\scriptsize,>=angle 90]
(-1,0) edge [bend left=45] node[anchor=south] {$\partial_2 M$} node[above] {} (1,0)
(-1,0) edge [bend right=45] node[anchor=north] {$\partial_1 M$} node[below] {} (1,0);
\draw (0,0) node {$M$};
\end{tikzpicture}
$$
It is a manifold with faces: every point $x$ in the interior has $c(x)=0$, every point in $\partial_1M\cap\partial_2M$ has $c(x)=2$ and is a face of both $\partial_1M$ and $\partial_2M$, and every other point has $c(x)=1$ and lies either in $\partial_1M$ or in $\partial_2M$. Moreover, the ordered pair $(\partial_1M, \partial_2M)$ gives $M$ the structure of a $\<2\>$-manifold.
\end{ex}

An $k$-bordism is a $\<k\>$-manifold such that for each $\partial_iM$ we distinguish between an ``incoming'' and an ``outgoing'' part. We will see later that we can think of it as having $k$ ``directions'' in which there is an ``incoming'' and an ``outgoing'' part of the boundary.
\begin{defn}\label{defn k-bordism}\
\begin{itemize}
\item A {\em (cubical) 0-bordism} is a closed manifold.
\item An $n$-dimensional {\em cubical $k$-bordism} is an $n$-dimensional $\<k\>$-manifold whose $k$-tuple of faces is denoted by $(\partial_1M,\ldots, \partial_k M)$ together with decompositions
$$\partial_i M= \partial_{i,in} M \amalg \partial_{i, out} M,$$
such that $\partial_{i,in} M$ and $\partial_{i,out} M$ are $(n-1)$-dimensional cubical $(k-1)$-bordisms.
\item An $n$-dimensional {\em $k$-bordism} is an $n$-dimensional cubical $k$-bordism such that $\partial_{i,in} M$ and $\partial_{i,out} M$ are trivial in the sense that there are $(n-k-1+i)$-dimensional $(i-1)$-bordisms $M_{i, in}$ and $M_{i, out}$ such that there are diffeomorphisms
$$\partial_{i,in} M \cong M_{i, in} \times[0,1]^{k-i} \quad\mbox{and}\quad \partial_{i,out} M \cong M_{i, out} \times[0,1]^{k-i}$$
for $1\leq i\leq n-1$.
\end{itemize}
\end{defn}

\begin{rem}
For $k=2$ our definition of 2-bordism agrees with that in \cite{Schommer}. One should think of $M_{i, in}$ and $M_{i, out}$ as the $i$-source and $i$-target of $M$.
\end{rem}

\begin{ex}
An example of a 2-dimensional 2-bordism is illustrated in the following picture.
$$
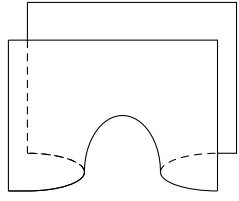
\begin{tikzpicture}[scale=2.5]
\begin{scope}[yscale=1,xscale=-1]

\draw (2,2) arc (-90: 0: 0.3cm and 0.1cm);
\draw [densely dashed] (2.3, 2.1) arc (0:90: 0.3cm and 0.1cm);
\draw (2, 2.2) -- (1.9, 2.2) -- (1.9, 3) -- (3,3);
\draw (2,2) -- (2, 2.8) -- (3.1, 2.8) 
	 (3.1, 2) -- (3,2);
\draw (3,2) arc (270: 180: 0.3cm and 0.1cm);
\draw [densely dashed]  (3, 2.2) arc (90: 180: 0.3cm and 0.1cm) (3, 2.2);

\draw (2.3, 2.1) arc (180: 0: 0.2cm and 0.3cm);
\draw (2.7, 2.1) arc (180:270: 0.3cm and 0.1cm) --+(0.1,0);
\draw [densely dashed] (2.7, 2.1) arc (180: 90: 0.3cm and 0.1cm);

\draw (3.1, 2.8) -- (3.1, 2);
\draw (3,3)--(3,2.8);
\draw[densely dashed] (3, 2.8) -- (3, 2.2);
\end{scope}
\end{tikzpicture}
$$
Its tuple $(\partial_1 M, \partial_2 M)$ of faces is given by the vertical and the horizontal faces, respectively.
\end{ex}

\begin{ex}\label{ex cube}
Let $M=[0,1]^k$. It is a $k$-bordism with
$$\partial_{i, in} M = [0,1]^{i-1}\times\{0\} \times [0,1]^{k-i} \quad\mbox{and}\quad \partial_{i, out} M = [0,1]^{i-1}\times\{1\} \times [0,1]^{k-i}.$$
\end{ex}

A $\<k\>$-manifold $M$ determines a functor
$$M: [1]^k \to \mathcal{T}\mathpzc{op}, \quad a=(a_1,\ldots, a_k) \longmapsto \begin{cases} M(a)=\bigcap_{\{i:\, a_i=0\}} \partial_iM,\quad \mbox{if}\, a\neq (1,\ldots, 1),\\ M(1,\ldots, 1) = M. \end{cases}$$

Recall the following embedding theorem via ``neat'' embeddings for $\<k\>$-manifolds. 
\begin{thm}[Proposition 2.1.7 in \cite{Laures}]\label{thm Laures embedding}
Any compact $\<k\>$-manifold admits a neat embedding in $\R_+^k \times \R^m$ for some $m$, i.e.~a natural transformation $\iota: M\to \R^k_+\times \R^m$ such that
\begin{enumerate}
\item $\iota(a)$ is an inclusion of a submanifold for all $a\in [1]^k$
\item the intersections $M(a) \cap (\R^m \times\R^k_+(b)) = M(b)$ are perpendicular for all $b<a$.
\end{enumerate}
\end{thm}

We adapt the definition of a neat embedding for $\<k\>$-manifolds for bordisms.
\begin{defn}
A {\em neat embedding} $\iota$ of a (cubical) k-bordism $M$ is a natural transformation of $M$ to $\R^m \times[0,1]^k$ for some $m$, both viewed as functors $[1]^k \to \mathcal{T}\mathpzc{op}$ such that
\begin{enumerate}
\item $\iota(a)$ is an inclusion of a submanifold for all $a\in [1]^k$ respecting the prescribed decomposition of the faces of the bordism
\item the intersections $M(a) \cap (\R^m \times[0,1]^k(b)) = M(b)$ are transverse for all $b<a$.
\end{enumerate}
\end{defn}

To prove that any $k$-bordism admits a neat embedding we use that any $\<k\>$-manifold admits a compatible collaring:
\begin{lemma}[Lemma 2.1.6 in \cite{Laures}]
For $a\in [1]^k$ we write $1-a=(1,\ldots, 1)-a$. Any $\<k\>$-manifold $M$ admits a $\<k\>$-diagram $C$ of embeddings
$$C(a<b): \R_+^k(1-a) \times M(a) \hookrightarrow \R_+^k(1-b) \times M(b)$$
with the property that $C(a<b)$ restricted to $\R_+^k(1-b) \times M(a)$ is the inclusion map $id\times M(a<b)$.
\end{lemma}

Now the embedding theorem for $\<k\>$-manifolds leads to an embedding theorem for (cubical) $k$-bordisms. Such embedded cubical $k$-bordisms appear in \cite{BoekstedtMadsen}.
\begin{thm}\label{thm embedding k-bordism}
Any $n$-dimensional (cubical) $k$-bordism $M$ admits a neat embedding into $\R^m \times[0,1]^k$.
\end{thm}

\begin{proof}
Let $M$ be an $n$-dimensional $k$-bordism. By the above theorem, there is a (neat) embedding $\iota: M \hookrightarrow \R_+^k \times \R^{m'} \subset \R^{k+m'} =\R^m$ for some $m'$ and $m=m'+k$. We will use that the product of an embedding with any smooth map still is an embedding.
For this, we will construct a smooth map $h: M\to [0,1]^k$ such that its product with $\iota$ is a neat embedding.

The idea for $h$ is that the decomposition into disjoint unions of the boundary components of the $k$-bordism determine a decomposition of the collars as well. Starting with the lowest dimensional corners $M(0)$, we use this decomposition to define $h$ on each component using either the collar coordinate or one minus the collar coordinate in each coordinate direction. Then we proceed by induction on the dimension of the corner and define $h$ successively on $M(a)$.

Recall from Example \ref{ex cube} that $[0,1]^k$ is a $k$-bordism. We fix a collaring, e.g.~the one given by the embeddings determined by diffeomorphisms
$$\R_+^\alpha \times [0,1]^{k-\alpha} \cong [0,\frac16)^\alpha \times [0,1]^{k-\alpha}
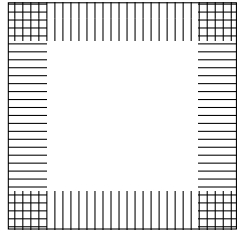
$$
$$
\begin{tikzpicture}
\draw (0,0) -- (0,3) -- (3,3) -- (3,0) -- cycle;
\fill[pattern= vertical lines] (0,0) -- (0,0.5) -- (3,0.5) -- (3,0);
\fill[pattern= vertical lines] (0,2.5) -- (0,3) -- (3,3) -- (3,2.5);
\fill[pattern= horizontal lines] (0,0) -- (0.5,0) -- (0.5,3) -- (0,3);
\fill[pattern= horizontal lines] (2.5,0) -- (3,0) -- (3,3) -- (2.5,3);
\end{tikzpicture}
$$

Let $a=(a_i) \in[1]^k$. Denote by $|a|=\sum a_i$ and $S(a)=\{i: a_i =0\}\subset \{1,\ldots, k\}$. Note that $|S(a)|=k-|a|$.

By the above lemma, there is a collaring of the $\<k\>$-manifold $M$.
The collaring gives an embedding $C(a<1): \R_+^{S(a)} \times M(a) \hookrightarrow M(1-0)=M$ whose image is a neighborhood $U(a)$ of the corner $M(a)$.
The decompositions $\partial_i M=\partial_{i,in} M \amalg \partial_{i, out} M$ give a decomposition of $M(a)$ into $2^{|S(a)|}$ disjoint components: 
$$M(a) = \bigcap_{\{i: a_i\neq 0\}} \partial_i M = \bigcap_{\{i: a_i\neq 0\}} \partial_{i, in} M \amalg \partial_{i, out} M = \bigsqcup_{\alpha \in [1]^{|S(a)|}} M(a,\alpha),$$
i.e.~an element $c\in M(a)$ lies in $M(a,\alpha)$ if and only if
$$ \alpha_i = \begin{cases} 0, & c \in \partial_{i,in} M\\
1, & c \in \partial_{i,out} M.
\end{cases}$$
This decomposition also determines a decomposition of $U(a)$ into into $2^{k-|a|}$ disjoint components $U(a,\alpha)$ for $\alpha\in [1]^{S(a)}$ such that $U(a,\alpha)$ is the image of $\R_+^{S(a)} \times M(a,\alpha)$ under $C(a<1)$.
The chosen collaring of $[0,1]^k$ induces one on $[0,1]^{S(a)}$, which in turn determines an embedding $\iota_\alpha:\R_+^{S(a)}\hookrightarrow [0,1]^{S(a)}$ for any particular corner $\alpha\in[1]^{S(a)}\subset[0,1]^{S(a)}$. Note that the images of these embeddings for varying corners are disjoint.
We define $h_a$ on $U(a, \alpha)$ to be the composition
$$U(a,\alpha) \cong \R_+^{S(a)} \times M(a, \alpha) \overset{\pr}{\longrightarrow} \R_+^{S(a)} \overset{\iota_{\alpha}}{\longrightarrow} [0,1]^{S(a)}.$$

For $a=0\in [1]^k$, i.e~the lowest (=$(n-k)$-)dimensional corners of $M$, the function $h_0: U(0)\to[0,1]^k$ is the restriction $h|_{U(0)}$ to $U(0)$ of the desired function $h$.

For $a>0$, assume $h$ is already defined on $U(b)$ for $b<a$ with $|b|=|a|-1$. Fix $\alpha\in [1]^{S(a)}$ and let $\beta \in [1]^{S(b)}$ such that $\beta_i=\alpha_i$ for every $i\in S(a)\subset S(b)$. Then $U(b,\beta) \subset U(a,\alpha)$. Since $U(b,\beta)$ are disjoint for varying $\beta$ and the collarings restrict compatibly, we can choose a smooth function 
$$h_{\beta,\alpha}: U(a,\alpha) \to [0,1]^{\{1,\ldots,k\}\setminus S(a)}$$ such that the product with $h_a$ agrees with $h$ on $U(b, \beta)$ for all such $\beta$. This defines a smooth map $h:M\to [0,1]^k$.

We claim that the product $\iota\times h:M \hookrightarrow \R^m \times [0,1]^k$ is a neat embedding of $k$-bordisms.
The first condition is fulfilled by construction, as $h$ is defined so that $M(a,\alpha)$ is sent to $\alpha(c)\in[0,1]^{S(a)}$.
For the second condition note that by construction, $M(a,\alpha) = h_a^{-1}(\alpha)$, $M(b, \beta) = h_b^{-1}(\beta)$, and $h_b=h_a\times h_{\beta, \alpha}$  on $U(b,\beta)$. But on $U(b,\beta)$ the function $h_{\beta, \alpha}$ is a projection onto the extra collar coordinate, with $M(b,\beta)$ the preimage of $0\in\R_+$. Thus the intersection is transversal.
\end{proof}

Conversely, 
\begin{prop}\label{prop composed bordisms}
For $l\geq -n$ and $d=n+l$ any element in $(\PBord_n^l)_{k_1,\ldots, k_n}$ leads to a $(k_1,\ldots,k_d)$-fold composition of $n$-dimensional $d$-bordisms.
\end{prop}

\begin{proof}
Let $(M\subset V\times B(\oul{I}), \oul{I})$ be an element in $(\PBord_n^l)_{k_1,\ldots, k_d}$. As usual, we use the notation $\pi: M\hookrightarrow V\times B(\oul{I}) \twoheadrightarrow B(\oul{I})$. Then for $(1\leq j_i \leq k_i)_{1\leq i\leq d}$ define
$$M_{j_1,\ldots, j_d}  = \pi^{-1} \left( \prod_i [\frac{2a^i_{j-1} + b^i_{j-1}}{3} , \frac{a^i_j+2b^i_j}{3}] \right).$$
They are $n$-dimensional cubical $d$-bordisms since they are manifolds with corners with a decomposition of the boundary given by the preimages of the faces of the cube, similarly to Example \ref{ex cube}:
$$\partial_{i_0,in} M_{j_1,\ldots, j_d} = \pi^{-1} \left( \prod_{i<i_0} [\frac{2a^i_{j-1} + b^i_{j-1}}{3} , \frac{a^i_j+2b^i_j}{3}] \times \{\frac{2a^i_{j-1} + b^i_{j-1}}{3}\} \times  \prod_{i>i_0} [\frac{2a^i_{j-1} + b^i_{j-1}}{3} , \frac{a^i_j+2b^i_j}{3}] \right)$$
and
$$\partial_{i_0,out} M_{j_1,\ldots, j_d} = \pi^{-1} \left( \prod_{i<i_0} [\frac{2a^i_{j-1} + b^i_{j-1}}{3} , \frac{a^i_j+2b^i_j}{3}] \times \{\frac{a^i_j+2b^i_j}{3}\} \times  \prod_{i>i_0} [\frac{2a^i_{j-1} + b^i_{j-1}}{3} , \frac{a^i_j+2b^i_j}{3}] \right)$$
The triviality condition to be a $d$-bordism follows from condition \eqref{cond 3} in Definition \ref{def PBord_n^l}. Note that we essentially extracted the underlying $k$-bordisms from Remark \ref{rem composed bordisms} and Notation \ref{rem composed bordisms 3}.

Moreover, they are composable along the faces in the sense that $\partial_{i, out} M_{j_1,\ldots, j_i-1,\ldots, j_d}$ and $\partial_{i, in} M_{j_1,\ldots, j_d}$ can be glued along their collar to form a new $k$-bordism given by
$$\pi^{-1} \left( \prod_{i'<i} [\frac{2a^i_{j-1} + b^i_{j-1}}{3} , \frac{a^i_j+2b^i_j}{3}] \times [\frac{2a^i_{j-2} + b^i_{j-2}}{3} , \frac{a^i_j+2b^i_j}{3}] \times \prod_{i'>i} [\frac{2a^i_{j-1} + b^i_{j-1}}{3} , \frac{a^i_j+2b^i_j}{3}]\right).$$
\end{proof}

\subsection{A time-dependent Morse lemma and spaces of bordisms}\label{sec Morse}

We have already seen in Remark \ref{rem composed bordisms} and Notation \ref{rem composed bordisms 3}, and in Corollary \ref{cor composed bordisms} and Proposition \ref{prop composed bordisms}, that the Morse lemma allows interpreting an element in $(\PBord_n^l)_{k_1,\ldots, k_n}$ as a composition of $k_1\cdots k_n$ bordisms. In this section we will see that paths in that space lead to diffeomorphisms of the composed bordisms and remark on why this space should be thought of as the classifying space of these diffeomorphisms.

The following theorem is the classical Morse lemma, as can be found e.g.~in \cite{Milnor}.
\begin{thm}[Morse lemma]\label{Morse lemma}
Let $f$ be a smooth proper real-valued function on a manifold $M$. Let $a<b$ and suppose that the interval $[a,b]$ contains no critical values of $f$. Then $M^a=f^{-1}((-\infty, a])$ is diffeomorphic to $M^b=f^{-1}((-\infty, b])$.
\end{thm}

We repeat the proof here since later on in this section we will adapt it to the situation we need.

\begin{proof}
Choose a metric on $M$, and consider the vector field
$$V= \frac{\nabla_y f}{|\nabla_y f|^2},$$
where $\nabla_y$ is the gradient on $M$. Since $f$ has no critical value in $[a,b]$, $V$ is defined in $f^{-1}((a-\epsilon, b+\epsilon))$, for suitable $\epsilon$. Choose a smooth function $\tilde g:\R \to \R$ which is 1 on $(a-\frac{\epsilon}{2}, b+\frac{\epsilon}{2})$ and compactly supported in $(a-\epsilon, b+\epsilon)$. Lift $\tilde g$ to a function $g:M\to\R$ by setting $g(y)=\tilde g(f(y))$. Then
$$\mathcal{V} = g \frac{\nabla_y f}{|\nabla_y f|^2}$$
is a compactly supported vector field on $M$ and hence generates a 1-parameter group of diffeomorphisms
$$\psi_t: M \longrightarrow M.$$
Viewing $f-(a+t)$ as a function on $\R\times M$, $(t,y)\mapsto f(y)-(a+t)$, we find that in $f^{-1}((a-\frac{\epsilon}{2}, b+\frac{\epsilon}{2}))$,
$$\partial_t (f-(a+t)) = 1 = \frac{\nabla_y f}{|\nabla_y f|^2}\cdot (f-(a+t)) = V\cdot (f-(a+t)),$$
and so the flow preserves the set
$$\{(t, y): f(y)=a+t \}.$$
Thus, the diffeomorphism $\psi_{b-a}$ restricts to a diffeomorphism
$$\psi_{b-a}|_{M^a}:M^a \longrightarrow M^b.$$
\end{proof}

In Lemma 3.1 in \cite{gaywehrheimwoodward} Gay, Wehrheim, and Woodward prove a time-dependent Morse lemma which shows that a smooth family of composed bordisms in their (ordinary) category of (connected) bordisms gives rise to a diffeomorphism which intertwines with the bordisms. We adapt this lemma to a variant which will be suitable for our situation in the higher categorical setting.

We start by defining some rescaling data to compare bordisms with different families of underlying intervals.
\begin{defn}\label{defn comp intervals intertwine}
Let $(I_0(s) \leq \cdots \leq I_k(s))\to \eDelta{l}$ be an $l$-simplex in $\Int_k$. A smooth family of strictly monotonically increasing diffeomorphisms
$$\big(\varphi_{s,t}:(a_0(s), b_{k}(s) )\to (a_0(t), b_{k}(t) )\big)_{s,t\in|\Delta^l|}$$
is said to {\em intertwine with the composed intervals} if the following condition is satisfied for every morphism $f:[m]\to[l]$ in the simplex category $\Delta$: Let $|f|: |\Delta^m| \to |\Delta^l|$ be the induced map between standard simplices. For every $0\leq j <k$ such that
\begin{itemize}
\item either for every $s\in |f|(|\Delta^m|)$ the intersection $I_j(s) \cap I_{j+1}(s)$ is empty
\item or for every $s\in |f|(|\Delta^m|)$ the intersection $I_j(s) \cap I_{j+1}(s)$ contains only one element,
\end{itemize}
we require that for every $s\in |f|(|\Delta^m|)$,
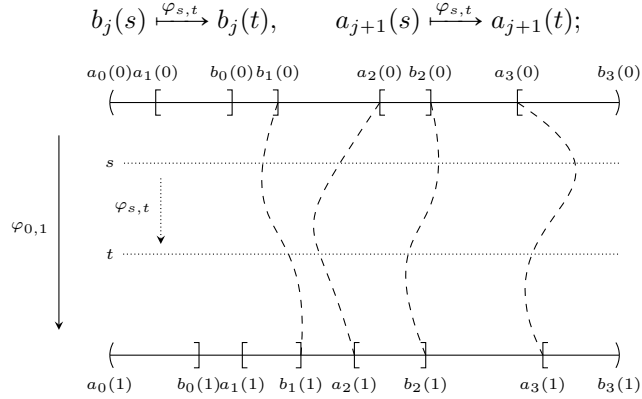
$$b_j(s) \xmapsto{\varphi_{s,t}} b_j(t), \qquad a_{j+1}(s) \xmapsto{\varphi_{s,t}} a_{j+1}(t);$$
\begin{center}
\begin{tikzpicture}[scale=0.67]
\draw (-5,0) -- (5,0);
\draw (-5,-5) -- (5,-5);

\draw[densely dotted]
(-5,-1.2) node {\tiny $s$} edge (5,-1.2)
(-5,-3) node {\tiny $t$} edge (5,-3);

\draw (5, -5.25) node [anchor=north] {\tiny $b_3(1)$};
\draw (5,-5) arc (0:30:0.5);
\draw (5,-5) arc (0:-30:0.5);

\draw (-5, -5.25) node [anchor=north] {\tiny $a_0(1)$};
\draw (-5,-5) arc (0:30:-0.5);
\draw (-5,-5) arc (0:-30:-0.5);

\draw[->] (-6,-0.65) -- (-6,-4.5);
\draw (-6,-2.5) node [anchor=east] {\tiny $\varphi_{0,1}$};

\draw[->, densely dotted] (-4,-1.5) -- (-4,-2.8);
\draw (-4,-2.1) node [anchor=east] {\tiny $\varphi_{s,t}$};

\draw (5, 0.25) node [anchor=south] {\tiny $b_3(0)$};
\draw (5,0) arc (0:30:0.5);
\draw (5,0) arc (0:-30:0.5);

\draw (-5, 0.25) node [anchor=south] {\tiny $a_0(0)$};
\draw (-5,0) arc (0:30:-0.5);
\draw (-5,0) arc (0:-30:-0.5);

\draw (-2.7, 0.25) -- (-2.6, 0.25) node[anchor=south] {\tiny $b_0(0)$} -- (-2.6, -0.25) -- (-2.7, -0.25);
\draw (-3.35, -4.75) -- (-3.25, -4.75) -- (-3.25, -5.25) node[anchor=north] {\tiny$b_0(1)$} -- (-3.35, -5.25);

\draw (-4, 0.25) -- (-4.1, 0.25) node[anchor=south] {\tiny $a_1(0)$} -- (-4.1, -0.25) -- (-4, -0.25);
\draw (-2.3, -4.75) -- (-2.4, -4.75) -- (-2.4, -5.25) node[anchor=north] {\tiny$a_1(1)$} -- (-2.3, -5.25);

\draw (-1.8, 0.25) -- (-1.7, 0.25) node[anchor=south] {\tiny $b_1(0)$} -- (-1.7, -0.25) -- (-1.8, -0.25);
\draw (-1.35, -4.75) -- (-1.25, -4.75) -- (-1.25, -5.25) node[anchor=north] {\tiny$b_1(1)$} -- (-1.35, -5.25);
\draw[dashed] (-1.7,0) .. controls (-2.7, -3) and (-1, -2) .. (-1.25,-5);

\draw (0.4, 0.25) -- (0.3, 0.25) node[anchor=south] {\tiny $a_2(0)$} -- (0.3, -0.25) -- (0.4, -0.25);
\draw (-0.1, -4.75) -- (-0.2, -4.75) -- (-0.2, -5.25) node[anchor=north] {\tiny$a_2(1)$} -- (-0.1, -5.25);
\draw[dashed] (0.3,0) .. controls (-1.7, -3) and (-1, -2) .. (-0.2,-5);

\draw (1.2, 0.25) -- (1.3, 0.25) node[anchor=south] {\tiny $b_2(0)$} -- (1.3, -0.25) -- (1.2, -0.25);
\draw (1.1, -4.75) -- (1.2, -4.75) -- (1.2, -5.25) node[anchor=north] {\tiny$b_2(1)$} -- (1.1, -5.25);
\draw[dashed] (1.3,0) .. controls (2, -3) and (0, -2) .. (1.2,-5);

\draw (3.1, 0.25) -- (3, 0.25) node[anchor=south] {\tiny $a_3(0)$} -- (3, -0.25) -- (3.1, -0.25);
\draw (3.6, -4.75) -- (3.5, -4.75) -- (3.5, -5.25) node[anchor=north] {\tiny$a_3(1)$} -- (3.6, -5.25);
\draw[dashed] (3,0) .. controls (6, -2) and (2, -2) .. (3.5,-5);
\end{tikzpicture}
\end{center}
\end{defn}

\begin{rem}
Note that it is enough to check this condition for $m\leq l$.
\end{rem}

\begin{defn}
Let $(I_0(s) \leq \cdots \leq I_k(s))\to [0,1]$ be a 1-simplex in $\Int_k$. A {\em rescaling datum for $\ul{I}$} is a is a smooth family of strictly monotonically increasing diffeomorphisms $\varphi_{s,t}: (a_0(s), b_{k}(s) )\to (a_0(t), b_{k}(t) )$ for $s,t\in[0,1]$ such that
\begin{enumerate}
\item $\varphi_{s,s} = id \quad \text{for every }s\in [0,1],$
\item $\varphi_{s,t} = \varphi_{t,s}^{-1} \quad \text{for every }s,t\in [0,1],$
\item $(\varphi_{s,t})_{s,t\in\eDelta{1}}$ intertwines with the composed intervals.
\end{enumerate}

\end{defn}

\begin{thm}\label{thm time Morse lemma}
Let $(M\subset \R^r\times B(\oul{I})\times\eDelta{1}, \oul{I})$ be a 1-simplex in $(\PBord_n^l)_{k_1,\ldots, k_n}$. Then,
\begin{enumerate}
\item for every $1\leq i\leq n$, there is a rescaling datum $\varphi^i_{s,t}$ for $\ul{I}^i$, and 
\item there is a smooth family of diffeomorphisms
$$(\psi_{s,t}:M_s\longrightarrow M_t)_{s,t\in[0,1]},$$
such that $\psi_{s,s}=id_{M_s}$ and $\psi_{s,t} = \psi_{t,s}^{-1}$, which {\em intertwine with the composed bordisms} with respect to the product of the rescaling data $\varphi_{s,t}= (\varphi^i_{s,t})_{i=1}^n : B(\oul{I}(s)) \to B(\oul{I}(t))$. By this we mean the following: denoting by $\pi_s$ the composition $M_s \hookrightarrow V\times B(\oul{I}(s)) \twoheadrightarrow B(\oul{I}(s))$, for $1\leq i\leq n$ and $0\leq j_i,l_i\leq k_i$ let
$$
\begin{array}{rlc}
t^i_{j_i}\in I^i_{j_i}(s) \mbox{ such that} & \varphi_{s,t}^i(t^i_{j_i})\in I^i_{j_i}(t), &\mbox{and} \\
t^i_{l_i}\in I^i_{l_i}(s) \mbox{ such that} & \varphi_{s,t}^i(t^i_{l_i})\in I^i_{l_i}(t).
\end{array}
$$
Then $\psi_{s,t}$ restricts to a diffeomorphism
$$\pi_s^{-1}\left( \prod_{i=1}^n [t^i_{j_i}, t^i_{l_i}] \right) \xrightarrow{\psi_{s,t}} \pi_s^{-1}\left( \prod_{i=1}^n [\varphi_{s,t}(t^i_{j_i}), \varphi_{s,t}(t^i_{l_i})] \right),$$
i.e.~denoting $B=\prod_{i=1}^n [t^i_{j_i}, t^i_{l_i}]$,
\begin{equation*}
\begin{tikzcd}[column sep=large]
&M_t \arrow[hookleftarrow] {r} & \pi_t^{-1}(\varphi_{s,t}(B)) \arrow {d} {\pi_t}\\
M_s \arrow {ru} [description] {\psi_{s,t}} \arrow{d} {\pi_s} \arrow[hookleftarrow] {r} & \pi_s^{-1}(B) \arrow{d}{\pi_s} \arrow[dashed] {ru} [description] {\psi_{s,t}} & \varphi_{s,t}(B)\\
B(\oul{I}(s)) \arrow[hookleftarrow] {r} & B \arrow {ru} [description] {\varphi_{s,t}}
\end{tikzcd}
\end{equation*}
\end{enumerate}
\end{thm}

\begin{proof}
The main strategy of the proof is the same as for the classical Morse lemma. Namely, we will construct a suitable vector field whose flow gives the desired diffeomorphisms. First, we fix the metric on $M$ induced by the restriction of the standard metric on $\R^r\times B(\oul{I})\times\eDelta{1}$. Recall from Remark \ref{rem simplex trivial fiber bundle} that the map $M\to \eDelta{1}$ exhibits $M$ as a trivial fiber bundle, so there is a diffeomorphism $M\cong \eDelta{1}\times N$ as abstract manifolds. For every $s\in[0,1]$, the fiber $M_s$ is diffeomorphic to $N$ as abstract manifolds. We fix the metric on $N$ induced by the diffeomorphism $N\cong M_0$, and use the notation $f_s: N \cong M_s\hookrightarrow V\times B(\oul{I}(s)) \twoheadrightarrow B(\oul{I}(s))$.

For steps 1-3 assume that $l=-(n-1)$. The general case applies these arguments in each direction separately.
\paragraph{Step 1: disjoint intervals}
First assume that for all $0\leq j\leq k$ and for every $s\in[0,1]$ we have
$$I_j(s)\cap I_{j+1}(s) =\emptyset.$$

We first define suitable vector fields $V_j$ and $W_j$ in neighborhoods of the preimage under $f$ of $\bigcup_{s\in[0,1]} \{s\}\times I_j(s)$, such that their flows will preserve the preimages of the left and right endpoints of the intervals, respectively. Then we use a partition of unity to obtain a vector field $\mathcal V$ defined on $[0,1] \times N$ which gives rise to the desired diffeomorphisms.

Let 
$$A_j = \bigcup_{s\in [0,1]} \{s\}\times f_s^{-1}(a_j(s))\subset [0,1]\times N, \qquad B_j= \bigcup_{s\in [0,1]} \{s\}\times f_s^{-1}(b_j(s))\subset [0,1]\times N.$$
Now for $0\leq j\leq k$ consider the vector fields
$$V_j=\left(\partial_s, \partial_s(a_j(s)-f_s)\frac{\nabla_y f_s}{|\nabla_y f_s|^2}\right), \qquad W_j=\left(\partial_s, \partial_s(b_j(s)-f_s)\frac{\nabla_y f_s}{|\nabla_y f_s|^2}\right),$$
where $\nabla_y$ is the gradient on $N$. Since $f_s$ has no critical value in $I_j(s)$, the vector fields $V_j$ and $W_j$ are defined on $f^{-1}(U_j)$, where $U_j$ is a neighborhood of $\bigcup_{s\in[0,1]} \{s\}\times I_j(s)$. Moreover, viewing $a_j: (s,y)\mapsto a_j(s)$ as a function on $[0,1]\times N$,
$$V_j(f-a_j) = \partial_s(f - a_j) + \partial_s (a_j - f)\frac{\nabla_y f}{|\nabla_y f|^2}(f-a_j) = \partial_s(f - a_j) + \partial_s (a_j - f) = 0,$$
So the vector field $V_j$ is tangent to $A_j$ and similarly, $W_j$ is tangent to $B_j$.

We would now like to construct a vector field $\mathcal{V}$ on $[0,1] \times N$ which for every $0\leq j\leq k$, at $A_j$ restricts to $V_j$ and at $B_j$ restricts to $W_j$, and such that there exists a family of functions $\big( c_x: [0,1]\to \R \big)_{x\in I_j(0)}$ such that
\begin{itemize}
\item[-] $c_x(0)=x$, $c_x(s)\in I_j(s)$,
\item[-] the graphs of $c_x$ for varying $x$ partition $\bigcup_{s\in[0,1]} \{s\} \times [a_j(s), b_j(s)]$, and
\item[-] $\mathcal{V}$ is tangent to $C_x=\bigcup_{s\in[0,1]} \{s\}\times f_s^{-1}(c_x(s))$.
\end{itemize}
We will use $c_x$ to define $\varphi_{0,s}(x)=c_x(s)$ and $\varphi_{s,t} = \varphi_{0,t}\circ\varphi_{0,s}^{-1}$. Moreover, the diffeomorphisms $\psi_{s,t}$ will arise as the flow along $\mathcal{V}$.

Fix smooth functions $\tilde g_j, \tilde h_j:\bigcup_{s\in[0,1]} \{s\}\times B(\ul{I}(s)) \to\R_{\geq0}$ which satisfy the following conditions:
\begin{enumerate}
\item $\tilde g_j, \tilde h_j$ are compactly supported in $U_j$,
\item $\tilde g_j=1$ in a neighborhood of $\mathrm{graph}\,a_j = \{(s, a_j(s)): s\in[0,1]\}$,
\\$\tilde h_j=1$ in a neighborhood of $\mathrm{graph}\,b_j$
\item $\tilde g_j+\tilde h_j=1$ in $\bigcup_{s\in[0,1]} \{s\}\times I_j(s)$, and the supports of the $\tilde g_j+\tilde h_j$ are disjoint.
\end{enumerate}
Lift the functions $\tilde g_j, \tilde h_j$ to functions $g_j, h_j: [0,1]\times N\to \R$ by setting $g_j(s,y):=\tilde g_j(s,f_s(y))$ and $h_j(s,y):=\tilde h_j(s,f_s(y))$. Then consider the following vector field on $f^{-1}(U_j)$:
$$\mathcal{V}_j = \left( \partial_s, \left( g_j \partial_s (a_j)+ h_j \partial_s (b_j)-\partial_s (f) \right) \frac{\nabla_y f}{|\nabla_y f|^2} \right)$$
This vector field is supported on the support of $g_j+h_j$ and thus extends to a vector field on $N$. Note that for $(s,y)\in A_j$, $\mathcal{V}_j(s,y)=V_j(s,y)$, and for $(s,y)\in B_j$, $\mathcal{V}_j(s,y)=W_j(s,y)$.

Now let $\mathcal{V}$ be the vector field on $[0,1]\times N$ constructed by combining the above vector fields as follows:
\begin{equation}\label{eqn formula vector field}
\mathcal{V} = \left( \partial_s, \sum_{0\leq j\leq k} \left( g_j \partial_s (a_j)+ h_j \partial_s (b_j)-\partial_s (f) \right) \frac{\nabla_y f_s}{|\nabla_y f_s|^2} \right).
\end{equation}
Note that in $\bigcup_{s\in[0,1]} \{s\}\times f_s^{-1}(I_j(s))$, it restricts to $\mathcal{V}_j$. 

In order for $\mathcal{V}$ to be tangent to $C_x$, the functions $c_x$ must satisfy $\mathcal{V}_j(f-c_x)=0$ at points in $C_x$. Expanding the left hand side as
\begin{eqnarray*}
\mathcal{V}_j(f-c_x) & = & \partial_s(f-c_x) + \left( g_j \partial_s (a_j)+ h_j \partial_s (b_j)-\partial_s (f) \right) \frac{\nabla f}{|\nabla f|^2}(f-c_x)\\
& = & -\partial_s(c_x) + g_j\partial_s (a_j)+ h_j \partial_s (b_j)
\end{eqnarray*}
leads to the ordinary differential equation with smooth coefficients on $[0,1]$
\begin{eqnarray*}
\partial_s(c_x)(s) & = & g_j(s, c_x(s))\partial_s (a_j)(s)+ h_j(s, c_x(s)) \partial_s (b_j)(s),\\
c_x(0) & = & x.
\end{eqnarray*}
By Picard-Lindel\"of, it has a unique, a priori local, solution. To see that it extends to every $s\in[0,1]$, consider the smooth function $F: [0,1]\times N\to \bigcup_{s\in[0,1]} \{s\}\times B(\ul{I}(s))$ defined to be $\pi$ under the diffeomorphism $M\cong [0,1]\times N$, so $F(s,y)=(s, f(s,y))=(s, f_s(y))$. Since $\pi$ is proper, so is $F$. Moreover, $C_x= F^{-1}(\mathrm{graph}\,c_x)$. For fixed $x$, $\mathrm{graph}\,c_x$ sits inside the support of $\tilde g_j+\tilde h_j$, for some $j$, and 
thus is compact in $\bigcup_{s\in[0,1]} \{s\}\times B(\ul{I}(s))$. Therefore $C_x$  lies in a compact part of $[0,1] \times N$ 
and thus the local solution exists for all $s\in [0,1]$.

We now define our rescaling data essentially by following the curve $c_x$. Explicitly, let $\varphi_{0,s}:B(\ul{I}(0)) \to B(\ul{I}(s))$ be defined on $[a_j(0), b_j(0)]$ by sending $x_0$ to $c_{x_0}(s)$. Note that by construction, it sends $a_j(0), b_j(0)$ to $a_j(s), b_j(s)$. Since the solution $c_x$ of the ODE varies smoothly with respect to the initial value $x$ this map is a diffeomorphism. So we can define $\varphi_{s,t}:B(\ul{I}(s))\to B(\ul{I}(t))$ on $[a_j(s), b_j(s)]$ by sending $x_s=c_{x_0}(s)$ to $c_{x_0}(t)$. We extend $\varphi_{s,t}$ to a diffeomorphism in between these intervals in the following way. Let $\tilde{\tilde{g}}_j,\tilde{\tilde h}_j: [b_j(0), a_{j+1}(0)] \to \R$ be a partition of unity such that $\tilde{\tilde g}_j$ is strictly decreasing, $\tilde{\tilde g}_j(b_j(s))=1$, and $\tilde{\tilde h}_j(a_{j+1}(s))=1$. Then, for $x_0\in [b_j(0), a_{j+1}(0)]$ set
$$c_{x_0}(s) = \tilde{\tilde g}_j(x_0)c_{b_j(0)}(s) + \tilde{\tilde h}_j(x_0) c_{a_{j+1}(0)}(s) \quad\mbox{and} \quad \varphi_{s,t}(c_{x_0}(s))= c_{x_0}(t).$$

As mentioned above, we obtain the diffeomorphisms $\psi_{s,t}$ by flowing along the vector field $\mathcal{V}$. Since $\mathcal{V}$ is tangent to the sets $C_x = \bigcup_{s\in[0,1]} \{s\}\times f_s^{-1}(c_x(s))$ for $x\in I_0(0)\cup\cdots\cup I_k(0)$, the flow preserves $C_x$, and $\bigcup_{s\in[0,1]} \{s\} \times f_s^{-1}([b_j(s), a_{j+1}(s)])$ in between. Again, this implies that the flow exists for all $s\in [0,1]$. It is of the form $\Psi(t-s, (s,y)) = (t, \psi_{s,t}(y))$ for $0\leq s\leq t\leq 1$, where $(\psi_{s,t})_{s,t\in [0,1]}$ is a family of diffeomorphisms on $N$. We transport them under the diffeomorphism $M\cong [0,1]\times N$ to diffeomorphisms $(\psi_{s,t}:M_s\to M_t)_{s,t\in [0,1]}$, which by construction intertwine with the composed bordisms with respect to the rescaling data $\varphi_{s,t}$. 

\paragraph{Step 2: common endpoints}

Now consider the case that for $0\leq j\leq k$ we have that either for every $s\in[0,1]$, $I_j(s)\cap I_{j+1}(s) =\emptyset$ as in the previous case or for every $s\in[0,1]$ we have
$$|I_j(s)\cap I_{j+1}(s)| = 1.$$

In this case, one can modify the above argument. We explain this for the case of two intervals with one common endpoint, i.e.~$b_j(s)=a_{j+1}(s)$.

Instead of choosing smooth functions $\tilde g_j, \tilde h_j, \tilde g_{j+1}, \tilde h_{j+1}:\bigcup_{s\in[0,1]} \{s\}\times B(\ul{I}(s)) \to\R$ such that the supports of $\tilde g_j+\tilde h_j$ and $\tilde g_{j+1}+\tilde h_{j+1}$ are disjoint (which now is not possible), we fix three smooth functions $\tilde f_j, \tilde g_j, \tilde h_j:\bigcup_{s\in[0,1]} \{s\}\times B(\ul{I}(s)) \to\R$ which satisfy the following conditions:
\begin{enumerate}
\item $\tilde f_j, \tilde g_j, \tilde h_j$ are compactly supported in $U_j \cup U_{j+1}$,
\item $\tilde f_j=1$ in a neighborhood of $\mathrm{graph}\,a_j = \{(s, a_j(s)): s\in[0,1]\}$,
\\
$\tilde g_j=1$ in a neighborhood of $\mathrm{graph}\,b_j=\mathrm{graph}\,a_{j+1}$,\\
$\tilde h_j=1$ in a neighborhood of $\mathrm{graph}\,b_{j+1}$,
\item $\tilde f_j+\tilde g_j+\tilde h_j=1$ in $\bigcup_{s\in[0,1]} \{s\}\times (I_j(s)\cup I_{j+1}(s))$, and the support of the $\tilde f_j+\tilde g_j+\tilde h_j$ is disjoint from the sums associated to the other intervals.
\end{enumerate}
Now continue the proof similarly to above.

\paragraph{overlapping intervals with non-empty interior}

If for some $0\leq j\leq k$ the intersection $I_j(s) \cap I_{j+1}(s)$ has non-empty interior for every $s\in[0,1]$, then one can do the above construction with the intervals $I_j(s), I_{j+1}(s)$ replaced by the interval $I_j(s) \cup I_{j+1}(s)$.

\paragraph{Step 4: partial overlaps -- mixed cases}

When the above cases are mixed, we can combine the cases treated so far. We will illustrate this in the case this in the case when the intervals first are disjoint and then start overlapping. The other cases are treated similarly.

Let us assume that there is an $\tilde{s}$ such that for $s<\tilde{s}$, $I_j(s) \cap I_{j+1}(s) = \emptyset$ and for $s\geq \tilde{s}$, $I_j(s) \cap I_{j+1}(s)\neq \emptyset$. In this case, $\tilde{x} = b_j(\tilde{s}) = a_{j+1}(\tilde{s})$, which is a regular value of $f_{\tilde{s}}$. Since $f$ is smooth, there is an open ball $U_j$ centered at $(\tilde{s}, \tilde{x})$ in $\bigcup_{s\in[0,1]} \{s\}\times B(\ul{I}(s))$ such that for $(s, x) \in U$, $x$ is a regular value of $f_s$.  Let $\bar{s}< \tilde{s}$ be such that for every $\bar{s} \leq s\leq\tilde{s}$, the set $\{s\}\times [a_j(s), b_{j+1}(s)]$ is covered by $U \cup \big(\{s\}\times (I_j(s)\cup I_{j+1}(s))\big)$. Choose $s_0$ and $t_0$ such that $\bar{s} \leq s_0 < t_0$.
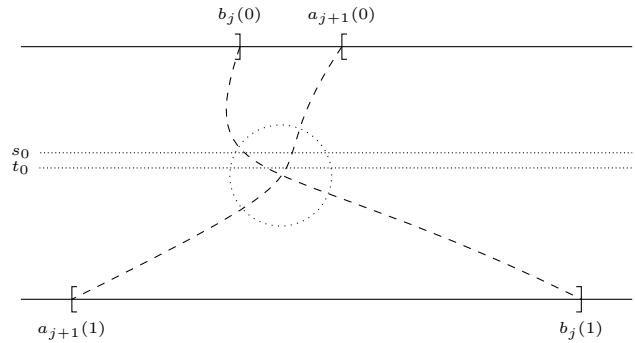
\begin{center}
\begin{tikzpicture}[scale=0.67]
\draw (-6,0) -- (6,0);
\draw (-6,-5) -- (6,-5);

\draw[densely dotted]
(-6,-2.1) node {\tiny $s_0$} edge (6,-2.1)
(-6,-2.4) node {\tiny $t_0$} edge (6,-2.4);

\draw (-1.8, 0.25) -- (-1.7, 0.25) node[anchor=south] {\tiny $b_j(0)$} -- (-1.7, -0.25) -- (-1.8, -0.25);
\draw (-4.9, -4.75) -- (-5, -4.75) -- (-5, -5.25) node[anchor=north] {\tiny $a_{j+1}(1)$} -- (-4.9, -5.25);
\draw[dashed] (-1.7,0) .. controls (-2.7, -3) and (-1, -2) .. (5,-5);

\draw (0.4, 0.25) -- (0.3, 0.25) node[anchor=south] {\tiny $a_{j+1}(0)$} -- (0.3, -0.25) -- (0.4, -0.25);
\draw (4.9, -4.75) -- (5, -4.75) -- (5, -5.25) node[anchor=north] {\tiny $b_j(1)$} -- (4.9, -5.25);
\draw[dashed] (0.3,0) .. controls (-1.7, -3) and (1, -2) .. (-5,-5);

\draw[dotted] (-0.9, -2.55) circle (1);

\end{tikzpicture}
\end{center}

In $[0,t_0]$, we are in the situation of disjoint intervals and can use the first construction to obtain $c_x^{(1)}(s)$ and $\mathcal{V}^{(1)}(s,y)$ for $s\leq t_0$.

In $[s_0, 1]$, we apply the construction from step 1 to the intervals $I_j(s)$ and $I_{j+1}(s)$ replaced by the interval $[a_j(s), b_{j+1}(s)]$ to obtain $c_x^{(2)}(s)$ and $\mathcal{V}^{(2)}(s,y)$ for $s\geq s_0$.

Now choose a partition of unity $G,H: [0,1] \to \R$ such that $G|_{[0,s_0]}=1, H|_{[t_0,1]}=1$, and $G$ is strictly decreasing on $[s_0,t_0]$. For $s<t$ define
$$c_x(s) = G(s) c_x^{(1)}(s) + H(s) c_x^{(2)}(s), \qquad \mathcal{V}(s,y) = G(s) \mathcal{V}^{(1)}(s,y) + H(s) \mathcal{V}^{(2)}(s,y).$$
Then define $\varphi_{s,t}$ and $\psi_{s,t}$ as before.

\paragraph{Step 5: several directions}

Assume now that $l>-(n-1)$. Let
$$\pi_s: N \cong M_s\hookrightarrow V\times B(\oul{I}(s)) \twoheadrightarrow B(\oul{I}(s))$$
and for $1\leq i \leq n$ denote by $(p_i)_s:N\to B(\ul{I^i}(s))$ the composition of $\pi_s$ with the projection to the $i$th coordinate. Note that by condition \ref{cond 3} in Definition \ref{def PBord}, the function $(p_i)_s$ does not have a critical point in $I^i_0(s)\cup\ldots\cup I^i_{k_i}(s)$.

By steps 1-3 for each $i$ we got a vector field
$$
\mathcal{V}^i = \left(\partial_s, \Pi_i(s,y)\frac{\nabla_y (p_i)_s}{|\nabla_y (p_i)_s|^2}\right),
$$
e.g.~see \eqref{eqn formula vector field}. We combine them to obtain a new vector field on $[0,1]\times N$ given by
$$\tilde{\mathcal{V}}= \left(\partial_s, \sum_{i=1}^n \Pi_i(s,y)\frac{\nabla_y (p_i)_s}{|\nabla_y (p_i)_s|^2}\right).$$
For $i\neq j$ the projections $(p_i)_0$ and $(p_j)_0$ are orthogonal with respect to the metric on $N$ and moreover, $(p_i)_s$, $(p_j)_s$ stay orthogonal along the path, because the change of metric on $N\cong M_s$ induced by the embedding of $M_s$ respects orthogonality on $B(\oul{I})$. This implies that
$$\frac{\nabla_y (p_i)_s}{|\nabla_y (p_i)_s|^2} p_j = \delta_{ij},$$
and so $\tilde{\mathcal{V}}$ still is tangent to the respective $C^i_x$ in each direction and thus its flow, if it exists globally, will give rise to the desired diffeomorphisms and rescaling data.

The global existence follows from the special form of the vector field. Given a point $(t,y_t)\in N$, the flow will preserve a set of the form
$$\{ (s,y): \pi_s(y_s) = (c_{x_0}^1(s), \ldots, c_{x_0}^n(s)) = (\xi_1(s), \ldots, \xi_n(s)) \},$$
where the right hand side is in the notation of Example \ref{ex cutoff path}, and $\vec c_{x_0}(t) = \vec\xi(t) = y_t$. One can show, as in the example, that this set lies in a compact part of $N$ and thus the flow exists globally.
\end{proof}

We can now relate the spaces of bordisms to diffeomorphisms of bordisms in a more classical sense.
\begin{defn}
Building upon the previous section, in particular Proposition \ref{prop composed bordisms}, we define a {\em diffeomorphism of a $(k_1,\ldots, k_d)$-fold composition of $n$-dimensional $d$-bordisms} to be a diffeomorphism of the composition which ``intertwines with'', i.e.~restricts to, the composed bordisms. 
\end{defn}

The above theorem shows that a path in $(\PBord_n^l)_{k_1,\ldots, k_d}$ leads to such an intertwining diffeomorphism of the compositions at the start and at the end of the path. Actually, much more is true.
\begin{prop}\label{prop BDiff}
For fixed $k_1,\ldots, k_d$ and a $(k_1,\ldots, k_d)$-fold composition $M$ of $n$-dimensional $d$-bordisms $M_{j_1,\ldots, j_d}$ which we denote by $(M, (M_{j_1,\ldots, j_d}))$, consider the group of such intertwining diffeomorphisms $\mathrm{Diff}(M, (M_{j_1,\ldots, j_d}))$. Then $(\PBord_n^l)_{k_1,\ldots, k_n}$ is the disjoint union of classifying spaces of $\mathrm{Diff}(M, (M_{j_1,\ldots, j_d}))$, where the disjoint union is taken over diffeomorphism classes.
\end{prop}

\begin{proof}
We sketch the argument, essentially following the one for showing that $\mathrm{Sub}(M, \R^\infty)$ is a classifying space for the group of diffeomorphisms of $M$, and its modifications in \cite{GMTW} and \cite{Lurie}:
Consider the space $\mathrm{Emb}((M, (M_{j_1,\ldots, j_d})), \R^\infty\times [0,1]^d)$ of neat embeddings of the composition which restricts to neat embeddings of the composed bordisms. It is non-empty by the embedding theorem for $d$-bordisms (Theorem \ref{thm embedding k-bordism}) and contractible, which can be seen similarly to $\mathrm{Emb}(M,\R^\infty)$ being contractible. We get a principal $\mathrm{Diff}(M, (M_{j_1,\ldots, j_d}))$-bundle
$$\mathrm{Emb}((M, (M_{j_1,\ldots, j_d})), \R^\infty\times [0,1]^d) \to \mathrm{Emb}((M, (M_{j_1,\ldots, j_d})), \R^\infty\times [0,1]^d) / \mathrm{Diff}(M, (M_{j_1,\ldots, j_d})).$$
The disjoint union over all diffeomorphism classes of the right hand side is equivalent to~$(\PBord_n^l)_{k_1,\ldots, k_n}$.
\end{proof}

\subsection{The homotopy category \texorpdfstring{$h_1(\Bord_n^{(\infty,1)})$}{}}

The goal of the Section is to show that there is a equivalence of symmetric monoidal categories between the homotopy category of the $(\infty,1)$-category $\Bord_n^{(\infty,1)}$ and the usual unoriented bordism category $n\Cob$. In fact, even more is true: one can show that there is an equivalence of symmetric monoidal bicategories between the homotopy bicategory of the $(\infty,2)$-category $\Bord_n^{(\infty,2)})$ and the unoriented bordism bicategory defined in \cite{Schommer}. This will be proven in a subsequent article.

\subsubsection{The symmetric monoidal structure \texorpdfstring{on $h_1(\Bord_n^{(\infty,1)})$}{}}\label{sec symm mon h_1(Bord)}

$\Bord_n^{(\infty,1)} =\Bord_n^{-(n-1)} \simeq \myloopnop[n-1]{ \Bord_n }$ is an $(\infty,1)$-category with a symmetric monoidal structure defined in two ways similarly to that of $\Bord_1$. Both induce a symmetric monoidal structure on the homotopy category $h_1(\Bord_n^{(\infty,1)})$. We now make this symmetric monoidal structure more explicit for later purposes.

\paragraph{...coming from a $\Gamma$-object}
We can either obtain the symmetric monoidal structure as a $\Gamma$-object on $\Bord_n^{(\infty,1)} \simeq \myloopnop[n-1]{\Bord_n}$ by iterating the construction of the symmetric monoidal structure on the looping from Example \ref{ex symm mon looping} or by constructing a functor from an assignment $[m] \mapsto \Bord_n^{-(n-1)}[m]$ as mentioned in Remark \ref{rem Gamma for d-cat Bord}. In the second case, $\Bord_n^{-(n-1)}[m]$ arises, similarly to $\Bord_n[m]$, from the spaces $(\PBord_n^{V, -(n-1)}[m])_{k_1,\ldots, k_n}$, whose set of 0-simplices is the collection of tuples
$$(M_1,\ldots, M_m, (I_0\leq\ldots\leq I_k)),$$
where $M_1,\ldots , M_m$ are disjoint $n$-dimensional submanifolds of $V\times B(\ul{I})=(a_0, b_k)$, and each $(M_\beta, (I_0\leq\ldots\leq I_k))\in (\PBord_n^{V,-(n-1)})_{k_1,\ldots, k_n}$.

We saw in Example \ref{symm mon h_1} that a $\Gamma$-object endows the homotopy category of its underlying Segal space with a symmetric monoidal structure. Explicitly, in the second case, it comes from the following maps.
$$\begin{array}{ccccc}
\Bord_n^{-(n-1)}\langle 1\rangle \times \Bord_n^{-(n-1)}\langle 1\rangle&\overunder[\gamma_1\times\gamma_2]{\simeq}{\xleftarrow{\hphantom{\gamma_1\times\gamma_2}}} & \Bord_n^{-(n-1)}\langle 2\rangle & \overset{\gamma}{\longrightarrow} & \Bord_n^{-(n-1)}\langle 1\rangle,\\
(M_1, \ul{I}), (M_2, \ul{I}) & \xleftarrow{\hphantom{\gamma_1\times\gamma_2}}\!\shortmid & (M_1,M_2, \ul{I}) & \longmapsto & (M_1\amalg M_2, \ul{I})
\end{array}$$

\paragraph{...coming from a tower}
To understand the symmetric monoidal structure on $h_1(\Bord_n^{(\infty,1)})$ coming from a symmetric monoidal structure as a tower, we use that $\Bord_n^{(\infty,1)}=\Bord_n^{-(n-1)}$ has a symmetric monoidal structure coming from the collection of $k$-hybrid $(k+1)$-fold Segal spaces given (essentially) by the $k$-hybrid completion of 
$$\deloop[k]{ \PBord_n^{k-(n-1)} }{\emptyset }.$$
This symmetric monoidal structure induces one on the homotopy category $h_1(\Bord_n^{(\infty,1)}) \simeq h_1(\Bord_n^{-(n-1)})$. Since completion is a Dwyer-Kan equivalence, see \ref{sec completion}, it is enough to understand the symmetric monoidal structure on $h_1(\PBord_n^{-(n-1)})$.

The monoidal structure arises from composition in $\PBord_n^{1-(n-1)}$, the next layer of the tower $\PBord_n^{2-(n-1)}$ gives a braiding and the higher layers show that it is symmetric monoidal, see Section \ref{sec monoidal comparison}. Consider the diagram
$$(\PBord_n^{1-(n-1)})_{1,\bullet} \overunder[{(\PBord_n^{1-(n-1)})_{0,\bullet}}]{h}{\times} (\PBord_n^{1-(n-1)})_{1,\bullet} \overunder[d^1_0 \times d^1_2]{\simeq}{\xleftarrow{\hphantom{d^1_0\times d^1_2}}} (\PBord_n^{1-(n-1)})_{2,\bullet} \overset{d^1_1}{\longrightarrow} (\PBord_n^{1-(n-1)})_{1,\bullet}.$$

Similarly to Remark \ref{rem hybrid easy looping}, we find that
$$\deloop{\PBord_n^{1-(n-1)} }{ \emptyset }_{1,\bullet} \simeq \myloopnop{\PBord_n^{1-(n-1)} }_\bullet \simeq (\PBord_n^{-(n-1)})_\bullet$$
and together with the maps above this gives a monoidal structure
$$h_1(\PBord_n^{-(n-1)})\times h_1(\PBord_n^{-(n-1)}) \longrightarrow h_1(\PBord_n^{-(n-1)}).$$
We spell this structure out explicitly. Consider two objects or 1-morphisms represented by elements 
$$(M)=(M\subseteq V\times B(\ul{I}), \ul{I}), \quad (N)=(N\subseteq W\times B(\tilde{\ul{I}}), \tilde{\ul{I}}))$$
in $(\PBord_n^{-(n-1)})_k$ for $k=0$ or $k=1$. Without loss of generality we can assume that $V = W =\R^r$, that $(M), (N) \in (\PBord_n^{-(n-1), \R^r})_k$, and that (perhaps after rescaling) $\ul{I}=\tilde{\ul{I}}$. Furthermore, choose $c>0$ such that $(M), (N) \in (\PBord_n^{-(n-1), \bound{r}})_k$. What follows will be independent of the choice of $c$.

Under the map $\ell_r(c): \PBord_n^{-(n-1),\bound{r}} \to \myloopnop{\PBord_n^{1-(n-1), \R^{r-1}}}$ from Proposition \ref{thm looping}, $(M)$ and $(N)$ are sent to
$$(M_1) = (M\subseteq \R^{r-1}\times(-(c+1), c+1)\times B(\ul{I}), (-(c+1), -c] \leq [c, c+1), \ul{I}),$$
$$(N_1)=(N\subseteq \R^{r-1}\times (-(c+1), c+1) \times B({\ul{I}}), (-(c+1), -c] \leq [c, c+1), \ul{I}).$$
Now we can use the gluing procedure as in the proof of the Segal condition for $\PBord_n$ in Proposition \ref{prop PBord Ssp}. In this case, the sources and targets of $(M_1) $ and $(N_1)$ are all empty, so the construction of the glued element is as follows: we choose a path from $(N_1)$ to another element $(N_2)$ by moving the first coordinate in the box such that the pair $\big((M), (N_2)\big)$ lies in
$$(\PBord_n^{1-(n-1)})_{1,\bullet} \overunder[{(\PBord_n^{1-(n-1)})_{0,\bullet}}]{}{\times} (\PBord_n^{1-(n-1)})_{1,\bullet},$$
i.e.~such that 
$$(N_2)=(N\subseteq \R^{r-1}\times(c, 3c+2) \times B({\ul{I}}), (c,c+1] \leq [3c+1, 3c+2), \ul{I}).$$
Since $d^1_1((M_1)) = d^1_0 ((N_1)) =\emptyset$ we have that $M$ and $N_2$ are disjoint as submanifolds of $\R^{r-1} \times (-(c+1), 3c+2)\times B(\tilde{\ul{I}})$. So the glued element is
$$\left(M \amalg N \subset V_{d-1}\times (-(c+1), 3c+2)\times B({\ul{I}}), (-(c+1), -c] \leq [c, c+1] \leq [3c+1, 3c+2), \ul{I} \right).$$
The third face map $d^1_1$ sends it to
$$\left(M \amalg N \subset V_{d-1}\times (-(c+1), 3c+2)\times B({\ul{I}}), (-(c+1), -c]  \leq [3c+1, 3c+2), \ul{I} \right)$$
which by $u_r:\myloopnop{\PBord_n^{1-(n-1), \R^r} } \to \PBord_n^{-(n-1)}$ is sent to
$$\left(M \amalg N \subset \R^r\times B({\ul{I}}), \ul{I} \right).$$

\subsubsection{The homotopy category and \texorpdfstring{$n\Cob$}{nCob}}

Our higher categories of bordisms give back the ordinary categories of $n$-bordisms, as we see in the main proposition in this section. First, let us briefly recall the definition of the symmetric monoidal category of bordisms. A good reference for the details and subtleties is e.g.~\cite{Kock}.
\begin{defn}
The symmetric monoidal category of $n$-dimensional bordisms $n\Cob$ is defined as follows:
\begin{itemize}
\item Objects are closed $(n-1)$-dimensional smooth manifolds. 
\item A morphism from $M$ to $N$ is a diffeomorphism class of $n$-dimensional cobordisms from $M$ to $N$, where an {\em $n$-dimensional bordism} from $M$ to $N$ is a smooth manifold $\Sigma$ with boundary, together with a specified diffeomorphism $\partial \Sigma \cong M \amalg N$. 
\item Composition of morphisms $\Sigma_1:M_0\to M_1$ and $\Sigma_2:M_1\to M_2$ is given by the diffeomorphism class of the gluing $\Sigma_1 \amalg_{M_1} \Sigma_2$.
\item The identity morphism on $M$ is the diffeomorphism class of the cylinder $M\times [0,1]$ viewed as a morphism from $M$ to $M$. 
\item The symmetric monoidal structure is given by taking disjoint unions of objects and morphisms.
\end{itemize}
\end{defn}

\begin{rem}
An $n$-dimensional bordism $\Sigma$ from $M$ to $N$ is exactly an $n$-dimensional 1-bordism $\Sigma$ as in Definition \ref{defn k-bordism} with $\partial_{in}\Sigma = M$ and $\partial_{out}\Sigma = N$.
\end{rem}

\begin{prop}\label{h_1 nCob}
There is an equivalence of symmetric monoidal categories between the homotopy category of the $(\infty,1)$-category $\Bord_n^{(\infty,1)}$ and the category of $n$-bordisms,
$$h_1(\Bord_n^{(\infty,1)})\simeq n\Cob.$$
\end{prop} 

\begin{proof} We first show that there is an equivalence of categories $h_1(\Bord_n^{(\infty,1)})\simeq n\Cob$ and then show that it respects the symmetric monoidal structures.

Rezk's completion functor is a Dwyer-Kan equivalence of Segal spaces, and thus by definition induces an equivalence of the homotopy categories. So it is enough to show that
$$h_1(\PBord_n^{-(n-1)})\simeq n\Cob.$$
We define a functor
$$F:h_1(\PBord_n^{-(n-1)}) \longrightarrow n\Cob$$
and show that it is essentially surjective and fully faithful.

\paragraph{Definition of the functor}
By definition, an object in $h_1(\PBord_n^{-(n-1)})$ is an element $(M)=(M\subset V\times (a,b), I=(a,b)) \in \big(\PBord_n^{-(n-1)}\big)_0$. Since $\pi: M\to (a,b)$ is submersive and proper, in particular $\frac{a+b}{2}$ is a regular value of $\pi$ and $\pi^{-1}(\frac{a+b}{2})$ is a closed $(n-1)$-dimensional manifold. We define
$$F((M)) = \pi^{-1}(\frac{a+b}{2}).$$

A morphism in $h_1(\PBord_n^{-(n-1)})$ is an element in $\pi_0((\PBord_n^{-(n-1)})_1)$, represented by
$$(N)= \big(N\subset V\times (a_0,b_1), I_0=(a_0, b_1] \leq I_1=[a_1, b_1) \big) \in \big(\PBord_n^{-(n-1)}\big)_1.$$

We let $F$ send $(N)$ to the isomorphism class of
$$\bar{N}=\pi^{-1}\left(  \big[\frac{2 a_0 + b_0}{3}, \frac{ a_1 + 2 b_1}{3} \big] \right).$$
This is an $n$-dimensional manifold with boundary
$$\pi^{-1}( \frac{2a_0 + b_0}{3}) \amalg \pi^{-1}(  \frac{a_1+2b_1}{3}).$$
Since $\pi$ only has regular values in $I_0$ and $I_1$, the Morse lemma gives diffeomorphisms
$$\pi^{-1}( \frac{2a_0 + b_0}{3}) \cong \pi^{-1}( \frac{a_0 + b_0}{2}) \quad\mbox{and}\quad \pi^{-1}(  \frac{a_1+2b_1}{3}) \cong \pi^{-1}(  \frac{a_1+b_1}{2}),$$
Thus $F((N))$ is an $n$-dimensional bordism from the image of the source $F(d_0(N))$ to the image of the target $F(d_1(N))$.

We need to check that this assignment is well-defined, i.e.~independent of the choice of the representative of the isomorphism class. Any two representatives $(N_0), (N_1)$ are connected by a path in $(\PBord_n^{-(n-1)})_1$. From this path we can obtain another one which has ``shorter'' intervals, namely just by shrinking them to $(a_0(s), \frac{2a_0(s)+b_0(s)}{3}]$ and $[\frac{a_1(s)+2b_1(s)}{3}, b_1(s))$. Now Theorem \ref{thm time Morse lemma} gives a diffeomorphism $\psi_{0,1}:N_0\to N_1$ which restricts to a diffeomorphism $\bar{\psi}_{0,1}:\bar{N_0}\to \bar{N_1}$.

Note that the Morse lemma implies that any image of the degeneracy map in $(\PBord_n^{-(n-1)})_1$ is sent to an identity morphism in $n\Cob$ and that $F$ behaves well with composition.

\paragraph{The functor is an equivalence of categories}
Whitney's embedding theorem shows that $F$ is essentially surjective. Moreover, it is injective on morphisms: Let $\iota_0:N_0\hookrightarrow V\times B(\ul{I})$ and $\iota_1:N_1\hookrightarrow W\times B(\ul{\tilde{I}})$ be embeddings which are representatives of two 1-morphisms $(N_0\subset V\times B(\ul{I}) , \ul{I})$ and $(N_1\subset W\times B(\tilde{\ul{I}}), \ul{\tilde{I}})$ which have diffeomorphic images. Without loss of generality we can assume that $V=W$ and $\ul{I}=\ul{\tilde{I}}$. Then there is a diffeomorphism $\psi: \bar{N_0} \to \bar{N_1}$, which can be extended to the rest of the collars, i.e.~we get a diffeomorphism $\psi:N_0\to N_1$. Since $\mathrm{Emb}(N_0, \R^\infty\times B(\ul{I}))$ is contractible, there is a path from $\iota_0$ to $\iota_1\circ \psi$, which induces a 1-simplex $(N \subset V \times B(\ul{I})\times [0,1], B(\ul{I}))$ in $(\PBord_n^{-(n-1)})_1$ such that the fiber at $s=0$ is $(N_0\subset V\times B(\ul{I}) , \ul{I})$ and the fiber at 1 is $\big(im(\iota_1\circ\psi)(N_0)=N_1\subset V\times B(\ul{I}) , \ul{I} \big)$.

It remains to show that $F$ is full. In the case $n=1,2$ this is easy to show, as we have a classification theorem for 1- and 2-dimensional manifolds with boundary. In the 1-dimensional case it is enough to show that an open line, the circle and the half-circle, once as a bordism from 2 points to the empty set and once vice versa, lie in the image of the map, which is straightforward. In the 2-dimensional case, the pair-of-pants decomposition tells us how to embed the manifold.

For general $n$ we need to find a suitable embedding of our bordism. Theorem \ref{thm embedding k-bordism} provides one for much more general $k$-bordisms, but as there is a much simpler argument for $k=1$, we provide it here.

We first embed the manifold with boundary into $\R^+\times \R^{2n}$ using Laures' embedding theorem \ref{thm Laures embedding} for manifolds with boundary. Then the boundary of the half-space is $\partial( \R^+\times \R^{2n})=\R^{2n}$. We want to transform this embedding into an embedding into $(0,1)\times \R^{2n}$ such that the incoming boundary is sent into $\{\epsilon\}\times\R^{2n}$ and the outgoing boundary is sent into $\{1-\epsilon\}\times\R^{2n}$.

We first show that the boundary components can be separated by a hyperplane in $\R^{2n}$. The boundary components are compact, so they can be embedded into (large enough) balls $B^{2n}$. By perhaps first applying a suitable ``stretching'' transformation, one can assume that these balls do not intersect. Now, since $2n>1$ we have that the configuration space of these balls $\pi_0(\mathrm{Conf}(B^{2n}, \R^{2n}))\cong *$ is contractible, there is a transformation to a configuration in which the boundary components are separated by a hyperplane, without loss of generality given by the equation $\{x_1=0\}\subset\R^{2n}$.

Consider the restriction of the (holomorphic) logarithm function with branch cut $-i\R^+$ to $(\R^+\times \R)\setminus\{(0,0)\}\cong\mathbb{H}\setminus 0\subseteq\mathbb{C}$. It is a homeomorphism to $\{(x,y)\in \R^2: 0\leq y \leq \pi\}$. We can apply $\log\times id_{\R^{2n-1}}$ to $(\R^+\times \R_{x_1})\times \R^{2n-1}$ and, composing this with a suitable rescaling, obtain an embedding into $(\epsilon,1-\epsilon)\times \R^{2n}$. Now choose a collaring of the bordism to extend the embedding to $(0,1)\times \R^{2n}$.

\paragraph{The functor is a symmetric monoidal equivalence}

In the case of the structure coming from a $\Gamma$-object, one can similarly to the previous paragraph define an equivalence of categories
$$F[m]: h_1\Bord_n^{-(n-1)}[m] \longrightarrow n\Cob^m.$$
Then one can easily check that the following diagram commutes.
$$
\begin{tikzcd}
h_1\Bord_n^{-(n-1)}[1] \times h_1\Bord_n^{-(n-1)}[1] \arrow{d}{F\times F} & h_1\Bord_n^{-(n-1)}[2] \arrow{d}{F[2]} \arrow{r}{} \arrow{l}{\simeq} & h_1\Bord_n^{-(n-1)}[1] \arrow{d}{F}\\
n\Cob\times n\Cob \arrow[equal]{r}{} & n\Cob\times n\Cob \arrow{r}{\amalg}& n\Cob
\end{tikzcd}
$$
Thus we have a functor of $\Gamma$-objects in categories. Finally, there is an equivalence of categories between $\Gamma$-objects in categories and symmetric monoidal categories, which is a direct consequence of MacLane's coherence theorem for symmetric monoidal categories \cite{MacLane}.

For the case of the structure coming from a tower, we explicitly saw that the symmetric monoidal structure on $h_1(\Bord_n^{(\infty,1)})$ sends two objects or 1-morphisms determined by
$$(M)=(M\subseteq V_d\times B(\ul{I}), \ul{I}), \quad (N)=(N\subseteq V_d\times B(\ul{I}), \ul{I})$$
to
$$(M\amalg N)=\left(M \amalg N \hookrightarrow V \times B(\ul{I}), \ul{I} \right),$$
where the images of $M$ and $N$ lie in disjoint ``heights'' in the $v_1$-direction in $V_d$. Thus, under the functor $F$ the element $(M\amalg N)$ is sent to $F((M))\amalg F((N))$.

Finally, in both cases, any element represented by $(\emptyset, \ul{I})$ is sent to $\emptyset$.
\end{proof}

\section{Bordisms with additional structure: orientations and framings}\label{sec decorations}

In the study of fully extended topological field theories, one is particularly interested in manifolds with extra structure, especially that of a framing. In this section we explain how to define the $(\infty,n)$-category of structured $n$-bordisms, in particular for the structure of an orientation or a framing.

\subsection{Structured manifolds}\label{structured mfld}

We first recall the definition of structured manifolds and the topology on their morphism spaces making them into a topological category. In the next subsection we will see that the smooth singular simplices on these topological spaces essentially will give rise to the spatial structure of the levels of the $n$-fold Segal space of structured bordisms similarly to the construction in Section \ref{sec space PBord}.

Throughout this subsection, let $M$ be an $n$-dimensional smooth manifold.

\begin{defn}
Let $X$ be a topological space and $E\to X$ a topological $n$-dimensional vector bundle which corresponds to a (homotopy class of) map(s) $e: X\to B\mathrm{GL}(\R^n)$ from $X$ to the classifying space of the topological group $\mathrm{GL}(\R^n)$. More generally, we could also consider a map $e:X\to B\mathrm{Homeo}(\R^n)$ to the classifying space of the topological group of homeomorphisms of $\R^n$, but for our purposes vector bundles are enough. An {\em $(X,E)$-structure or, equivalently, an $(X,e)$-structure on an $n$-dimensional manifold $M$} consists of the following data:
\begin{enumerate}
\item a map $f: M\to X$, and
\item an isomorphism of vector bundles
$$triv: TM \cong f^*(E).$$
\end{enumerate}
Denote the set of $(X,E)$-structured $n$-dimensional manifolds by $\Man_n^{(X,E)}$.
\end{defn}

An interesting class of such structures arises from topological groups with a morphism to $O(n)$.
\begin{defn}
Let $G$ be a topological group together with a continuous homomorphism $e:G\to O(n)$, which induces $e: BG\to B\mathrm{GL}(\R^n)$. As usual, let $BG=EG/G$ be the classifying space of $G$, where $EG$ is total space of its universal bundle, which is a weakly contractible space on which $G$ acts freely. Then consider the vector bundle $E=(\R^n\times EG)/G$ on $BG$.
A $(BG, E)$-structure or, equivalently, a $(BG, e)$-structure on an $n$-dimensional manifold $M$ is called a {\em $G$-structure on $M$}. The set of $G$-structured $n$-dimensional manifolds is denoted by $\Man_n^G$.
\end{defn}

For us, the most important examples will be the following three examples.
\begin{ex}
If $G$ is the trivial group, $X=BG=*$ and $E$ is trivial. Then a $G$-structure on $M$ is a trivialization of $TM$, i.e.~a framing.
\end{ex}

\begin{example}
Let $G=O(n)$ and $e = id_{O(n)}$. Then, since the inclusion $O(n) \to \mathrm{Diff}\,(\R^n)$ is a deformation retract, an $O(n)$-structured manifold is just a smooth manifold.
\end{example}

\begin{example}
Let $G=SO(n)$ and $e:SO(n)\to O(n)$ is the inclusion. Then an $SO(n)$-structured manifold is an oriented manifold.
\end{example}

\begin{defn}
Let $M$ and $N$ be $(X, E)$-structured manifolds. Then let the space of morphisms from $M$ to $N$ be
$$\mathrm{Map}^{(X,E)}(M,N)= \mathrm{Emb}(M,N) \overunder[\mathrm{Map}_{/B\mathrm{Homeo}(\R^n)}(M,N)]{h}{\times} \mathrm{Map}_{/X}(M,N).$$
Taking (singular or differentiable) simplices leads to a space, i.e.~a simplicial set of morphisms from $M$ to $N$.
Thus we get a topological (or simplicial) category $\CMan_n^{(X,E)}$ of $(X,E)$-structured manifolds. Disjoint union gives $\CMan_n^{(X,E)}$ a symmetric monoidal structure.
\end{defn}

\begin{rem}
For $G=O(n)$ we recover $\mathrm{Emb}(M,N)$, and for $G=SO(n)$, the space of orientations on a manifold is discrete, so an element in $\mathrm{Map}^{SO(n)}(M,N)$ is an orientation preserving map.

If $G$ is the trivial group we saw above that a $G$-structure is a framing. In this case, the above homotopy fiber product reduces to
$$\mathrm{Map}^{(X,E)}(M,N)= \mathrm{Emb}(M,N) \overunder[\mathrm{Map}_{GL(d)}(\mathrm{Fr}(TM), \mathrm{Fr}(TN))]{h}{\times} \mathrm{Map}(M,N).$$
Thus, a framed embedding is a pair $(f,h)$, where $f:M\to N$ lies in $\mathrm{Emb}(M,N)$ and $h$ is a homotopy between between the trivialization of $TM$ induced by the framing of $M$ and that induced by the pullback of the framing on $N$.
\end{rem}

\subsection{The \texorpdfstring{$(\infty,n)$}{(infty,n)}-category of structured bordisms}

Fix a type of structure given by the pair $(X,E)$. In this subsection we define the $n$-fold (complete) Segal space of $(X,E)$-structured bordisms $\Bord_n^{(X,E)}$.

Compared to Definition \ref{def PBord} we add an $(X,E)$-structure to the data of an element in a level set.

\begin{defn}
Let $V$ be a finite-dimensional vector space. For every $n$-tuple $k_1,\ldots, k_n\geq0$, let $\big(\PBord_n^{(X,E), V}\big)_{k_1,\ldots, k_n}$ be the collection of tuples $(M,f, triv, (I^i_0\leq\cdots\leq I^i_{k_i})_{i=1,\ldots ,n})$, where
\begin{enumerate}
\item $(M, (I^i_0\leq\cdots\leq I^i_{k_i})_{i=1}^n)$ is an element in the set $(\PBord_n^V)_{k_1,\ldots, k_n}$, and
\item $(f, triv)$ is an $(X,E)$-structure on the (abstract) manifold $M$.
\end{enumerate}
\end{defn}

\begin{rem}
Note that there is a forgetful map
$$U: \big(\PBord_n^{(X,E), V}\big)_{k_1,\ldots, k_n} \to (\PBord_n^V)_{k_1,\ldots, k_n}$$
forgetting the $(X,E)$-structure.
\end{rem}

\begin{defn}
An $l$-simplex of $\big(\PBord_n^{(X,E), V}\big)_{k_1,\ldots, k_n}$ consists of tuples $(M, f, triv, \oul{I}(s)= (I^i_0(s) \leq\cdots\leq I^i_{k_i}(s))_{s\in\eDelta{l}}$ such that
\begin{enumerate}
\item $\oul{I}=(I^i_0 \leq\cdots\leq I^i_{k_i})_{1\leq i \leq n} \to \eDelta{l}$ is an $l$-simplex  in $\Int^n_{k_1,\ldots, k_n}$,
\item $M$ is a closed and bounded $(n+l)$-dimensional submanifold of $V\times B(\oul{I})$ such that\footnote{Recall from Section \ref{sec boxing} that $B(\oul{I})$ denotes the total space of $B(\oul{I})\to\eDelta{l}$ and is the subspace $\bigcup_{s\in\eDelta{l}} B(\oul{I}(s)) \times \{s\}$ of $\R^n\times\eDelta{l}$.}
\begin{enumerate}
\item the composition $\pi: M \hookrightarrow V\times B(\oul{I}) \twoheadrightarrow B(\oul{I})$ of the inclusion with the projection is proper,
\item its composition with the projection onto $\eDelta{l}$ is a submersion $\pi_l: M\to \eDelta{l}$, 
\item  $(f,triv): \mathrm{ker}(D\pi_l:TM\to T\eDelta{l}) \to f^*E$ is a fiberwise linear isomorphism.
\end{enumerate}
\item for every $S\subseteq\{1,\ldots, n\}$, let $p_S:M\xrightarrow{\pi}B(\oul{I}) \subseteq \R^n\times\eDelta{l} \xrightarrow{\pi_S}\R^S \times \eDelta{l}$ be the composition of $\pi$ with the projection $\pi_S$ onto the $S$-coordinates. Then for every $1\leq i\leq n$ and $0\leq j_i\leq k_i$,  at every $x\in p_{\{i\}}^{-1}(I^i_{j_i} (s) \times\{s\})$, the map $p_{\{i,\ldots,n\}}$ is submersive.
\end{enumerate}
\end{defn}

Similarly as for $\PBord_n$ the levels can be given a spatial structure with the above $l$-simplices and then the collection of levels can be made into a complete $n$-fold Segal space $\Bord_n^{(X,E)}$.

Moreover, $\Bord_n^{(X,E)}$ has a symmetric monoidal structure given by $(X,E)$-structured versions of the $\Gamma$-object and of the tower giving $\Bord_n$ a symmetric monoidal structure.

\subsection{Example: Objects in \texorpdfstring{$\Bord_2^{fr}$}{Bord2fr} are 2-dualizable}

In dimension one, a framing is the same as an orientation. Thus the first interesting case is the two-dimensional one. In this case, the existence of a framing is a rather strong condition. However, we will see that any object in $\Bord_2^{fr}$ is 2-dualizable. Being 2-dualizable means that it is dualizable with evaluation and coevaluation maps themselves have adjoints, see \cite{Lurie}.

Consider an object in $\Bord_2^{fr}$, which, since in this case $\Bord_2^{fr} = \PBord_2^{fr}$ by remark \ref{rem PBord incomplete}, is an element of the form
$$\big(M\subseteq V\times (a^1, b^1)\times (a^2, b^2), F, (a^1, b^1), (a^2, b^2) \big),$$
where $F$ is a framing of $M$. By the submersivity condition \ref{cond 3} in the Definition \ref{def PBord} of $\PBord_2$, $M$ is a disjoint union of manifolds which are diffeomorphic to $(0,1)^2$. Thus, it suffices to consider an element of the form
$$\big((0,1)^2\subseteq (0,1)^2, F, (0,1), (0,1) \big),$$
where $F$ is a framing of $(0,1)^2$. Depict this element by
$$
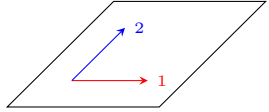
\begin{tikzpicture}
\draw (0,0) -- (2, 0) -- (3.4, 1.4) -- (1.4, 1.4) -- cycle;
\draw[->, red] (0.85,0.35) -- (1.85,0.35) node[anchor = west] {\tiny 1};
\draw[->, blue] (0.85,0.35) -- (1.55, 1.05) node[anchor = west] {\tiny 2};
\end{tikzpicture}
$$
One should think of this as a point together with a 2-framing, 
$$
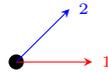
\begin{tikzpicture}
\fill (0.85,0.35) circle (0.1);
\draw[->, red] (0.85,0.35) -- (1.85,0.35) node[anchor = west] {\tiny 1};
\draw[->, blue] (0.85,0.35) -- (1.55, 1.05) node[anchor = west] {\tiny 2};
\end{tikzpicture}
$$

We claim that its dual is the same underlying unstructured manifold together with the opposite framing
$$
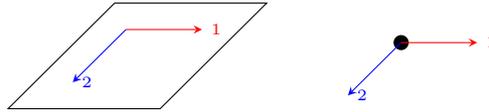
\begin{tikzpicture}
\draw (0,0) -- (2, 0) -- (3.4, 1.4) -- (1.4, 1.4) -- cycle;
\draw[->, red] (1.55, 1.05) -- (2.55, 1.05) node[anchor = west] {\tiny 1};
\draw[->, blue] (1.55, 1.05) -- (0.85, 0.35) node[anchor = west] {\tiny 2};
\end{tikzpicture}
\hspace{3em}
\begin{tikzpicture}
\fill (1.55, 1.05) circle (0.1);
\draw[->, red] (1.55, 1.05) -- (2.55, 1.05) node[anchor = west] {\tiny 1};
\draw[->, blue] (1.55, 1.05) -- (0.85, 0.35) node[anchor = west] {\tiny 2};
\end{tikzpicture}
$$

An evaluation 1-morphism $ev_{\ptpos}$ between them is given by the element in $(\Bord_2^{fr})_{1,0}$ which is a strip, i.e.~$(0,1)^2$, with the framing given by slowly rotating the framing by $180^\circ$, and is embedded into $\R\times (0,1)^2$ by folding it over once as depicted further down.

$$
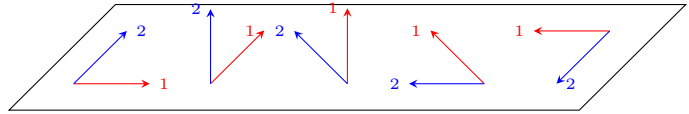
\begin{tikzpicture}

\begin{scope}[shift = {(-0.85,-0.35)}]
\draw (0,0) -- (7.5, 0) -- (8.9, 1.4) -- (1.4, 1.4) -- cycle;
\draw[->, red] (0.85,0.35) -- (1.85,0.35) node[anchor = west] {\tiny 1};
\draw[->, blue] (0.85,0.35) -- (1.55, 1.05) node[anchor = west] {\tiny 2};
\end{scope}

 \foreach \t in {1,2,3}{
\tikzset{shift={(1.8*\t,0)},rotate=45*\t}
\begin{scope}[shift = {(-0.85,-0.35)}]
\draw[->, red] (0.85,0.35) -- (1.85,0.35) node[anchor = east] {\tiny 1};
\draw[->, blue] (0.85,0.35) -- (1.55, 1.05) node[anchor = east] {\tiny 2};
\end{scope}
}
\begin{scope}[shift={(4.5,-0.35)}]
\draw[->, red] (2.55, 1.05) -- (1.55, 1.05) node[anchor = east] {\tiny 1};
\draw[->, blue] (2.55, 1.05) -- (1.85, 0.35) node[anchor = west] {\tiny 2};
\end{scope}

\end{tikzpicture}
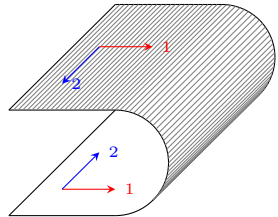
$$

$$
\begin{tikzpicture}[scale=0.7]

\draw (0,-1) arc (-90:90:1);
\draw (0.707, - 0.707) -- (2.707, 1.293);
\draw (2.707, 1.293) arc (-45:90:1);

\foreach \t in {-45, -40,..., 90}
	\draw[gray] ({cos(\t)},{sin(\t)}) -- ({2+cos(\t)},{2+sin(\t)});

\draw (-2,-1) -- (0,-1);
\draw (-2,1) -- (0,1);
\draw (0, 3) -- (2,3);
\draw (-2, 1) -- (0,3);
\draw (-2,-1) -- (0,1);

\foreach \t in {1,2,...,22}
	\draw[gray] (-0.087*\t,1) -- (2-0.087*\t, 3);

\draw[->, red] (-1,-0.5) -- (0, -0.5) node[anchor = west] {\tiny 1};
\draw[->, blue] (-1, -0.5) -- (-0.3, 0.2) node[anchor = west] {\tiny 2};

\draw[->, red] (-0.3, 2.2) -- (0.7, 2.2) node[anchor = west] {\tiny 1};
\draw[->, blue] (-0.3, 2.2) -- (-1, 1.5) node[anchor = west] {\tiny 2};
\end{tikzpicture}
$$

A coevaluation $coev_{\ptpos}$ is given similarly by rotating the framing along the strip in the other direction, by -$180^\circ$.

The composition
$$
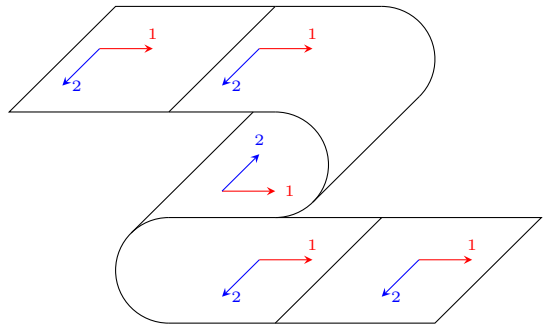
\begin{tikzpicture}[scale=0.7]

\draw (-2, 1) -- (-5, 1) -- (-3, 3) -- (0,3);

\draw (0,-1) arc (-90:90:1);
\draw (0.707, - 0.707) -- (2.707, 1.293);
\draw (2.707, 1.293) arc (-45:90:1);

\draw (-2,-1) -- (0,-1);
\draw (-2,1) -- (0,1);
\draw (0, 3) -- (2,3);
\draw (-2, 1) -- (0,3);

\draw (-2.707, -1.293) -- (-0.414 ,1);

\draw (-2,-1) arc (-90: 90: -1);
\draw (-2, -3) -- (0, -3) -- (2, -1) -- (0, -1);

\draw (0,-3) -- (3, -3) -- (5, -1) -- (2, -1);

\draw[->, red] (-1,-0.5) -- (0, -0.5) node[anchor = west] {\tiny 1};
\draw[->, blue] (-1, -0.5) -- (-0.3, 0.2) node[anchor = south] {\tiny 2};

\draw[->, red] (-0.3, 2.2) -- (0.7, 2.2) node[anchor = south] {\tiny 1};
\draw[->, blue] (-0.3, 2.2) -- (-1, 1.5) node[anchor = west] {\tiny 2};

\begin{scope}[shift={(0,-4)}]
\draw[->, red] (-0.3, 2.2) -- (0.7, 2.2) node[anchor = south] {\tiny 1};
\draw[->, blue] (-0.3, 2.2) -- (-1, 1.5) node[anchor = west] {\tiny 2};
\end{scope}

\begin{scope}[shift={(3,-4)}]
\draw[->, red] (-0.3, 2.2) -- (0.7, 2.2) node[anchor = south] {\tiny 1};
\draw[->, blue] (-0.3, 2.2) -- (-1, 1.5) node[anchor = west] {\tiny 2};
\end{scope}

\begin{scope}[shift={(-3,0)}]
\draw[->, red] (-0.3, 2.2) -- (0.7, 2.2) node[anchor = south] {\tiny 1};
\draw[->, blue] (-0.3, 2.2) -- (-1, 1.5) node[anchor = west] {\tiny 2};
\end{scope}

\end{tikzpicture}
$$
is connected by a path to the flat strip with the following framing given by pulling at the ends of the strip to flatten it.
$$
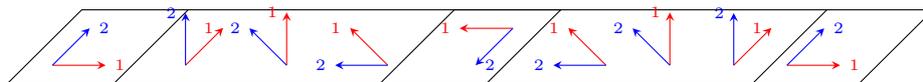
\begin{tikzpicture}[scale=0.7]

\begin{scope}[shift = {(-0.85,-0.35)}]
\draw (0,0) -- (16, 0) -- (17.4, 1.4) -- (1.4, 1.4) -- cycle;
\draw (7, 0) -- (9, 0) -- (10.4, 1.4) -- (8.4, 1.4) -- cycle;
\draw (2,0) -- (3.4, 1.4);
\draw (14,0) -- (15.4, 1.4);
\draw[->, red] (0.85,0.35) -- (1.85,0.35) node[anchor = west] {\tiny 1};
\draw[->, blue] (0.85,0.35) -- (1.55, 1.05) node[anchor = west] {\tiny 2};
\end{scope}

\begin{scope}[shift={(0.6,0)}]
 \foreach \t in {1,2,3}{
\tikzset{shift={(1.9*\t,0)},rotate=45*\t}
\begin{scope}[shift = {(-0.85,-0.35)}]
\draw[->, red] (0.85,0.35) -- (1.85,0.35) node[anchor = east] {\tiny 1};
\draw[->, blue] (0.85,0.35) -- (1.55, 1.05) node[anchor = east] {\tiny 2};
\end{scope}
}
\end{scope}

\begin{scope}[shift={(6.1,-0.35)}]
\draw[->, red] (2.55, 1.05) -- (1.55, 1.05) node[anchor = east] {\tiny 1};
\draw[->, blue] (2.55, 1.05) -- (1.85, 0.35) node[anchor = west] {\tiny 2};
\end{scope}

\begin{scope}[shift= {(9.2, 0)}]
 \foreach \t in {1,2,3}{
\tikzset{shift={(1.2*\t,0)},rotate=180-45*\t}
\begin{scope}[shift = {(-0.85,-0.35)}]
\draw[->, red] (0.85,0.35) -- (1.85,0.35) node[anchor = east] {\tiny 1};
\draw[->, blue] (0.85,0.35) -- (1.55, 1.05) node[anchor = east] {\tiny 2};
\end{scope}
}
\end{scope}

\begin{scope}[shift= {(13.8, 0)}]
\begin{scope}[shift = {(-0.85,-0.35)}]
\draw[->, red] (0.85,0.35) -- (1.85,0.35) node[anchor = west] {\tiny 1};
\draw[->, blue] (0.85,0.35) -- (1.55, 1.05) node[anchor = west] {\tiny 2};
\end{scope}
\end{scope}

\end{tikzpicture}
$$

This strip is homotopic to the same strip with the trivial framing. Thus the composition is connected by a path to the identity and thus is the identity in the homotopy category. Similarly,
$$\big(ev_{\ptpos} \otimes id_{\ptpos}\big) \circ \big(id_{\ptpos} \otimes coev_{\ptpos} \big)\simeq id_{\ptpos}.$$

In the above construction, we used $ev_{\ptpos}$ and $coev_{\ptpos}$ which arose from strips with framing rotating by $\pm 180^\circ$. A similar argument holds if you use for the evaluation any strip with the framing rotating by $\alpha\pi$ for any odd integer $\alpha$ and for the coevaluation rotation by $\beta\pi$ for any odd $\beta$. Denoting these by $ev(\alpha)$ and $coev(\beta)$, they will be adjoints to each other if $\alpha+\beta = 2$. 

The counit of the adjunction is given by the cap with the framing coming from the trivial framing on the (flat) disk.

$$
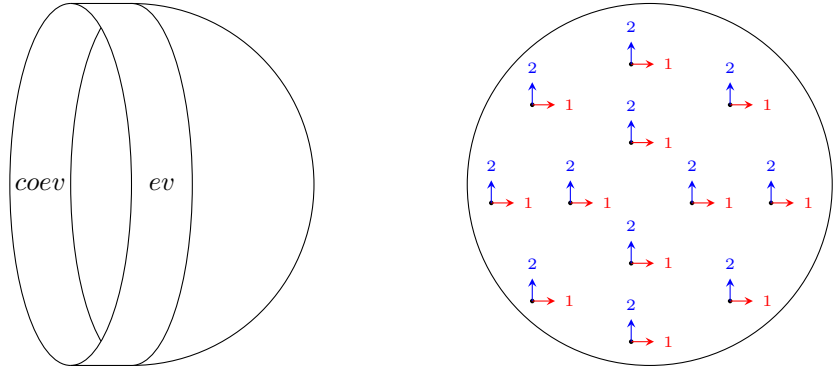
\begin{tikzpicture}[scale=0.8]
\draw (0,3) arc [start angle=90, end angle=270, x radius=1, y radius=3];
\draw (1,3) arc [start angle=90, end angle=-90, radius=3];
\draw (1,3) arc [start angle=90, end angle=270, x radius=1, y radius=3];
\filldraw[fill=white, draw=black] (0, 3) -- (1, 3) arc [start angle=90, end angle=-90, x radius=1, y radius=3] -- (0,-3) arc [start angle=-90, end angle=90, x radius=1, y radius=3];
\draw (1.5, 0) node {$ev$};
\draw (-0.5, 0) node {$coev$};
\end{tikzpicture}
\hspace{2cm}
\begin{tikzpicture}[scale=0.8]
\draw (0,0) circle (3);
 \foreach \t in {0,1,...,7}{
\tikzset{rotate=45*\t}
\draw (2.3, 0) node {$\ptframe$};
}
 \foreach \t in {0,1,...,3}{
\tikzset{rotate=90*\t}
\draw (1, 0) node {$\ptframe$};
}
\end{tikzpicture}
$$

Similarly, the unit of the adjunction is given by a saddle with the framing coming from the one of the torus which turns by $2\pi$ along each of the fundamental loops.
$$
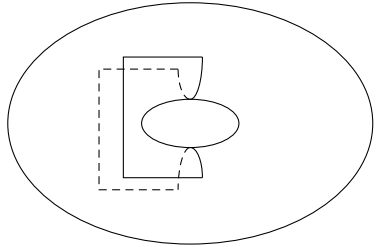
\begin{tikzpicture}[scale=0.8]
\begin{scope}[rotate= -90]
\draw (0,0) circle [x radius=2, y radius=3];
\draw (0,0) circle [x radius=0.4, y radius=0.8];

\draw[densely dashed] (-0.4, 0) arc [start angle=0, end angle=-90, x radius=0.5, y radius=0.2];
\draw (-0.4, 0) arc [start angle=0, end angle=90, x radius=0.7, y radius=0.2];

\draw (0.4, 0) arc [start angle=180, end angle=90, x radius=0.5, y radius=0.2];
\draw[densely dashed] (0.4, 0) arc [start angle=180, end angle=270, x radius=0.7, y radius=0.2];

\draw[densely dashed] (-0.9, -0.2) -- (-0.9, -1.5) -- (1.1, -1.5) -- (1.1, -0.2);
\draw (-1.1, 0.2) -- (-1.1, -1.1) -- (0.9, -1.1) -- (0.9, 0.2);
\end{scope}
\end{tikzpicture}
$$

Then the following 2-bordism is also framed and exhibits the adjunction.
$$
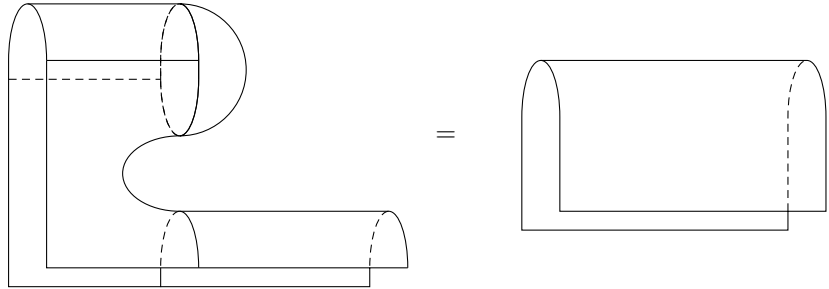
\begin{tikzpicture}[scale=2.5]

\begin{scope} [rotate=90]
\node (X) at (2.7, 1.75) {};

\draw (2,2) arc (-90: 0: 0.3cm and 0.1cm);
\draw [densely dashed] (2.3, 2.1) arc (0:90: 0.3cm and 0.1cm);
\draw (2, 2.2) -- (1.9, 2.2) -- (1.9, 3) -- (3.1,3) arc (90: -90: 0.3cm and 0.1cm);
\draw (2,2) -- (2, 2.8) -- (3.1, 2.8) 
	(3.4, 2.9) -- (3.4, 2.1)
	 (3.1, 2) -- (3,2) (3.1, 2) arc (-90: 0: 0.3cm and 0.1cm)  ;
\draw (3,2) arc (270: 180: 0.3cm and 0.1cm);
\draw [densely dashed]  (3, 2.2) arc (90: 180: 0.3cm and 0.1cm) (3, 2.2) -- (3.1, 2.2) arc (90: 0: 0.3cm and 0.1cm);

\draw (2.3, 1) -- (2.3, 2.1) arc (180: 0: 0.2cm and 0.3cm);
\draw (1.9, 2.2) -- (1.9, 1.1) -- (2, 1.1) (2, 2) -- (2, 0.9);
\draw (2,0.9) arc (-90: 0: 0.3cm and 0.1cm);
\draw [densely dashed] (2.3, 1) arc (0:90: 0.3cm and 0.1cm);

\draw (2.7, 2.1) arc (180: 360: 0.35 cm) arc (0: -90: 0.3cm and 0.1cm) -- +(-0.1, 0) arc (270: 180: 0.3cm and 0.1cm);
\draw [densely dashed] (2.7, 2.1) arc (180: 90: 0.3cm and 0.1cm) -- +(0.1, 0) arc (90: 0: 0.3cm and 0.1cm);

\draw (3.1, 2.8) -- (3.1, 2);
\draw[densely dashed] (3, 3) -- (3, 2.2);

\draw (2.7, 0.7) node {=};

\begin{scope}[shift={(2.2, -1.3)}]
	\node (A) at (0.9, 0.75) {};
	\draw (.1,.2) -- (0, 0.2) --(0,1.6) -- (.6, 1.6) arc (90: -90:  0.3cm and 0.1cm) -- (0.1, 1.4) 
	-- (0.1,0) -- (.6, 0) arc (-90: 0:  0.3cm and 0.1cm)  -- (.9, 1.5);
	\draw [densely dashed] (.1, .2) -- (.6, .2) arc (90: 0: 0.3cm and 0.1cm);
\end{scope}
\end{scope}
\end{tikzpicture}
$$

\begin{rem}
One can use a similar, but much longer, argument to show that objects in $\Bord_n^{fr}$ are in fact $n$-dualizable.
\end{rem}

\section{Fully extended topological field theories}\label{TFT}

Now that we have a good definition of a symmetric monoidal $(\infty,n)$-category of bordisms modelled as a symmetric monoidal complete $n$-fold Segal space, we can define fully extended topological field theories \`a la Lurie. 

\subsection{Definition}

\begin{defn}
A {\em fully extended unoriented $n$-dimensional topological field theory} is a symmetric monoidal functor of $(\infty,n)$-categories with source $Bord_n$.
\end{defn}

\begin{rem}
Consider a fully extended unoriented $n$-dimensional topological field theory
$$Z: \Bord_n \longrightarrow \C,$$
where $\C$ is a symmetric monoidal complete $n$-fold Segal space. We have seen in Corollary \ref{cor looping is again a Bord} and Section \ref{hocat} that there is a map $n\Cob \simeq h_1(\Bord_n^{(\infty,1)}) \to h_1( \myloopnop[n-1]{\Bord_n})$. Precomposition of $\myloop[n-1]{Z}{\emptyset}$ with this map induces a symmetric monoidal functor
$$n\Cob \to h_1(\myloopnop[n-1]{\Bord_n}) \longrightarrow h_1( \myloop[n-1]{\C}{ Z(*)} ),$$
i.e.~an ordinary $n$-dimensional topological field theory. 
\end{rem}

\paragraph{Additional structure} Recall from the previous section that there are variants of $\Bord_n$ which require that the underlying manifolds of their elements are endowed with some additional structure, e.g.~an orientation or a framing. These variants lead to the following definitions.

\begin{defn}
Fix a type of structure given by the pair $(X,E)$. A {\em fully extended $n$-dimensional $(X,E)$-topological field theory} is a symmetric monoidal functor of $(\infty,n)$-categories with source $\Bord_n^{(X,E)}$.
\end{defn}

In particular, the most interesting cases are the following:
\begin{defn}
A {\em fully extended $n$-dimensional framed topological field theory} is a symmetric monoidal functor of $(\infty,n)$-categories with source $\Bord_n^{fr}$.
\end{defn}

\begin{defn}
A {\em fully extended $n$-dimensional oriented topological field theory} is a symmetric monoidal functor of $(\infty,n)$-categories with source $\Bord_n^{or}$.
\end{defn}

\subsection{\texorpdfstring{$n$}{n}-TFT yields \texorpdfstring{$k$}{k}-TFT}
Every fully extended $n$-dimensional (unoriented, oriented, framed) TFT yields a fully extended $k$-dimensional (unoriented, oriented, framed) TFT for any $k\leq n$ by truncation from subsection \ref{sec truncation}.

Note that for $k<n$, we have a map of $k$-fold Segal spaces
$$\PBord_k \longrightarrow \tau_k(\PBord_n)=(\PBord_n)_{\underbrace{\bullet,\dots,\bullet}_{k\textrm{ times}},\underbrace{0,\dots,0}_{n-k\textrm{ times}}}$$
induced by sending $\big(M \hookrightarrow V\times B(\oul{I}), \oul{I}=(I^i_0\leq \cdots \leq I^i_{j_k})_{i=1}^k \big)\in (\PBord_k)_{j_1,\ldots, j_k}$ to 
$$\big( M\times (0,1)^{n-k} \hookrightarrow V\times (0,1)^{n-k}\times B(\oul{I}), \oul{I}, \underbrace{(0,1),\ldots, (0,1)}_{n-k} \big).$$

The completion map $\PBord_n\to \Bord_n$ induces a map on the truncations. Precomposition with the above map yields a map of (in general non-complete) $n$-fold Segal spaces
$$\PBord_k \longrightarrow \tau_k(\PBord_n) \longrightarrow \tau_k(\Bord_n).$$
Recall from \ref{sec truncation} that since $\tau_k(\Bord_n)$ is complete, by the universal property of the completion we obtain a map $\Bord_k\to\tau_k(\Bord_n)$, which is compatible with the symmetric monoidal structure (for both approaches).

\begin{rem}
Note that this map is usually not an equivalence, since completion does not commute with truncation. Moreover, if we equip the bordisms with an orientation or a framing, the image of $\PBord_k$ in $\tau_k(\PBord_n)$ consists of those $n$-oriented or $n$-framed bordisms whose orientation or framing is a stabilization of a $k$-orientation or $k$-framing.
\end{rem}

We conclude that any fully extended $n$-dimensional (unoriented, oriented, framed) TFT with values in a complete $n$-fold Segal space $\C$, $\Bord_n\to\C$ leads to a $k$-dimensional (unoriented, $n$-oriented, $n$-framed) TFT given by the composition
$$\Bord_k\longrightarrow \tau_k(\Bord_n) \longrightarrow \tau_k(\C)$$
with values in the complete $k$-fold Segal space $\tau_k(\C)$.

\bibliographystyle{alpha}
\bibliography{literature}

\end{document}